\documentclass[12pt,leqno]{article}

\usepackage{amsthm}
\usepackage[usenames]{color}

\usepackage{amsmath}
\usepackage{amscd}
\usepackage{amssymb}
\usepackage{latexsym}
\makeatletter

\@addtoreset{equation}{section}
\makeatother

\usepackage{mathtools,mathrsfs}
\usepackage{enumitem}
\usepackage{xcolor}
\usepackage{tikz-cd}
\usepackage{hyperref}
\usepackage[nameinlink,noabbrev]{cleveref}
\hypersetup{
colorlinks=true,
linkcolor=red!68!black,
citecolor=blue!68!black,
urlcolor=blue!68!black,
}

\newtheorem{thm}{Theorem}[subsection]
\newtheorem{pr}[thm]{Proposition}
\newtheorem{lm}[thm]{Lemma}
\newtheorem{cor}[thm]{Corollary}

\theoremstyle{definition}

\newtheorem{df}[thm]{Definition}

\newcommand{\sm}{\raisebox{2.33pt}{~\rule{6.4pt}{1.3pt}~}}

\newcommand{\ssm}{\raisebox{1.00pt}{~\!\rule{3.4pt}{0.8pt}\!~}}

\input xy \xyoption{all} \CompileMatrices

\numberwithin{equation}{subsection}
\theoremstyle{plain}
\newtheorem{theorem}[thm]{Theorem}
\newtheorem{proposition}[thm]{Proposition}
\newtheorem{lemma}[thm]{Lemma}
\newtheorem{corollary}[thm]{Corollary}
\newtheorem{conjecture}[thm]{Conjecture}

\theoremstyle{definition}
\newtheorem{definition}[thm]{Definition}
\newtheorem{remark}[thm]{Remark}

\newtheorem{construction}[thm]{Construction}
\newtheorem{convention}[thm]{Convention}
\makeatletter
\newcommand{\SXYauthornote}[1]{%
\begingroup
\let\@footnotetext\H@@footnotetext
\renewcommand{\thefootnote}{}%
\footnotetext{#1}%
\endgroup
}
\makeatother

\begin{document}

\title{Singular supports, characteristic cycles, and microlocalization of constructible sheaves}
\author{Takeshi Saito${}^\ast$ \and Jiangnan Xiong${}^\dag$ \and Enlin Yang${}^{\ddag}$}
\maketitle

\SXYauthornote{${}^\ast$takeshisaito@g.ecc.u-tokyo.ac.jp}
\SXYauthornote{${}^\ast$Department of Mathematical Sciences, University of Tokyo, Komaba 153-8914 Tokyo, Japan.}

\SXYauthornote{${}^\dag$xiongjiangnan@stu.pku.edu.cn}
\SXYauthornote{${}^\dag$School of Mathematical Sciences, Peking University,
No. 5 Yiheyuan Road Haidian District, Beijing 100871, P. R. China.}

\SXYauthornote{${}^{\ddag}$yangenlin@cnu.edu.cn}
\SXYauthornote{${}^\ddag$School of Mathematical Sciences, Capital Normal University,
No. 105 XiSanHuan North Road, Beijing 100048, P. R. China.}

\SXYauthornote{\textit{2010 Mathematics Subject Classification.} Primary 14F20; Secondary 14C17.}
\SXYauthornote{\textit{Key words and phrases.} singular support,  characteristic cycle, microlocalization, specialization,
Fourier transform.}

\begin{abstract}

For a constructible sheaf $\mathcal{F}$ on 
a smooth scheme $X$ over a perfect field,
the singular support $SS(\mathcal{F})$
and the characteristic cycle $CC(\mathcal{F})$
are defined on the cotangent bundle.
In analogy with the construction of Kashiwara--Schapira via microlocalization in the transcendental setting,
we construct
the microlocal versions
${ SS}_\mu(\mathcal{F})$ and $CC_\mu(\mathcal{F})$ 
of $SS(\mathcal{F})$ and $CC(\mathcal{F})$ respectively,
and
prove 
\[
SS_\mu(\mathcal{F})= SS(\mathcal{F}),\qquad   CC_\mu(\mathcal{F})= {\rm cl}CC(\mathcal{F}),
\]
where $\operatorname{cl}$ is the cycle class map.

As a consequence of the second equality, we show
that  the cohomological characteristic classes are the cycle classes associated with the corresponding characteristic cycles, as
conjectured in (\emph{Invent.~Math.~207: 597--695, 2017}).
As further applications, we give new proofs of the index formula and Bloch's  conductor formula in the proper case,  independently of Abe's homotopical approach via $\infty$-categories.
\end{abstract}

\tableofcontents

\section{Introduction}

For a constructible sheaf ${\cal F}$ on 
a smooth scheme $X$ over a field $k$ of
characteristic $p>0$,
the singular support $SS({\cal F})$
and the characteristic cycle $CC({\cal F})$
are defined on the cotangent bundle $T^*X$
\cite{SS}, \cite{CC}.
In a transcendental context,
the singular supports
and the characteristic cycles
allow a construction using
the microlocalization $\mu{\cal H}om({\cal F},{\cal F})$
defined on the cotangent bundle \cite{KS}.
We study a similar construction 
in the algebraic setting.

Let $\eta_0$ be the generic point of the henselization of
$\mathbf A_k^1$ at the origin. Choose a geometric point
$\bar\eta_0$ above $\eta_0$ and a compatible separable closure
$\bar k$ of $k$, and write $G_{\eta_0}={\rm Gal}(\bar \eta_0/\eta_0)$ and $G_k={\rm Gal}(\bar k/k)$.
Let $\Lambda$ be a finite local ring of residue characteristic $\ell\neq p$. We fix a non-trivial character $\psi: \mathbf F_p\to\Lambda^\times$.
For ${\cal F}\in D_{\rm ctf}(X,\Lambda)$, the microlocalization
$\mu{\cal H}om({\cal F},{\cal F})$ is an object of the derived category $D_{\rm ctf}(T^\ast X\times_k\eta_0,\Lambda)$ of bounded constructible complexes of finite tor-dimension
on $T^*X_{\bar k}$ equipped with a continuous semilinear
$G_{\eta_0}$-action, where the action on the base is induced
by the quotient $G_{\eta_0}\to G_k$.
It is defined as the Fourier transform of the specialization
$\nu{\cal H}om({\cal F},{\cal F})\in D_{\rm ctf}(TX_{\bar k},\Lambda)$ of
\begin{align}
	\mathscr H_{\mathcal F}
	:={\cal H}om({\rm pr}_1^*{\cal F},{\rm pr}_2^!{\cal F})
\end{align}
along the diagonal $\delta\colon X\hookrightarrow X\times X$.
Here and in what follows, we omit $R$ from the notation for derived functors.
The support of $\mu{\cal H}om({\cal F},{\cal F})$
descends to a unique closed conical subset
$SS_\mu({\cal F})\subseteq T^*X$,
which we define in Definition~\ref{dfSSmu}.
The construction of the cohomology class
$CC_\mu({\cal F})$
supported on $SS_\mu({\cal F})$, given
in Definition \ref{dfCCmu},
refines  the construction of the characteristic class
$C_{X/k}({\cal F})$ in \cite{AS}.
Recall that
$C_{X/k}({\cal F})
\in {\rm H}^0(X,K_X)$
is defined 
as the composition
\begin{align}
	\Lambda\overset{1_{\cal F}}
	\to \delta^!\mathscr H_{\mathcal F}\to 
	\delta^*\mathscr H_{\mathcal F}
	\overset{\rm ev}\to K_X,
\end{align}
where  $K_X=a^!\Lambda$
for the structure morphism $a\colon X\to {\rm Spec}\, k$.
Taking  adjoints and applying
the microlocalization functor,
we obtain
\begin{align}
	\Lambda\to \mu{\cal H}om({\cal F},{\cal F})
	\to (e^{\vee*}K_X)|_{T^\ast X_{\bar k}},
\end{align}
where $e^\vee\colon T^*X\to X$ is the
projection.
The composition defines a
 $G_k$-invariant cohomology class supported on $SS_\mu(\mathcal F)_{\bar k}$, which descends uniquely to a class
$CC_\mu({\cal F})
\in {\rm H}^0_{SS_\mu({\cal F})
}(T^*X,e^{\vee*}K_X)$.

We prove the following equality between $SS_\mu(\cal F)$ and $SS({\cal F})$.
\begin{theorem}[Theorem \ref{thm:singular-support-equality} and  Theorem \ref{thm:singular-support-equality-Ecoeff}]\label{thm:main}
Let $X$ be a smooth $k$-scheme and 
${\mathcal F}\in D_{\mathrm{ctf}}(X,\Lambda)$. We have
\begin{equation}\label{eq:main}
{SS}_\mu({\mathcal F})={SS}({\mathcal F}).
\end{equation}
The assertion also holds for constructible complexes with coefficients in a
finite extension $E/\mathbf Q_\ell$, interpreted on the pro-\'etale site.
\end{theorem}
Following a suggestion of Beilinson, one can use Wang's Fourier
transform \cite{Wang} to construct $SS_\mu$ and $CC_\mu$
uniformly over fields of arbitrary characteristic, including
characteristic zero; see Definition~\ref{def:SSmuB},
Proposition~\ref{prop:Wang-SS-comparison}, and
Remark~\ref{rem:Wang-CC-comparison}. 
This approach simplifies some arguments and avoids both the
choice of an additive character and the associated coefficient
extension in Lemma~\ref{lm:coeff-character}.

We outline our proof of \eqref{eq:main}. 
The first step is to prove the following inclusion  (cf. Theorem \ref{thm:singular-support-inclusion}):
\begin{equation}\label{eq:main-inclusion} 
{SS}_\mu({\mathcal F})\subseteq{SS}({\mathcal F}).
\end{equation}
By a standard reduction, we may assume that $X=P=\mathbf P^n$ is a projective space of dimension $n> 1$. Let $P^\vee$ be the dual projective space and let
\begin{equation}
Q=\{(x,[H]):x\in H\}\subseteq P\times P^\vee,\quad p:Q\to P,\quad p^\vee:Q\to P^\vee,
\end{equation}
be the incidence correspondence and its projections respectively. 
Let $U$ be the largest open subset of $Q$ such that $p|_U$ is universally locally acyclic relatively to ${\mathcal G}=p^{\vee\ast}Rp_*^\vee p^*{\mathcal F}[2n-2](n-1)$.
Beilinson's results on Radon transforms \cite[Theorem 3.2]{SS} identify the projectivization of ${SS}({\mathcal F})$ with $Q\setminus U$. It therefore suffices to show that $\mu\mathcal{H}om({\mathcal F},{\mathcal F})$ vanishes on the cotangent lines parametrized by $U$.

We consider the base change of the diagonal $\delta_P:P\hookrightarrow P\times P$ along $p\times p$
\begin{equation}\label{eq:intro-collision-center}
\begin{tikzcd}
Q\times_PQ\arrow[r,hook,"i_p"]\arrow[d]&Q\times Q\arrow[d,"p\times p"]\\
P\arrow[r,hook,"\delta_P"]&P\times P.
\end{tikzcd}
\end{equation}
It induces  natural morphisms  $\Pi: T_{Q\times_PQ}(Q\times Q)\simeq Q\times_PQ\times_PTP\to T_P(P\times P)\simeq TP$ and $\Pi^\vee:  Q\times_PQ\times_PT^\ast P\to T^\ast P$, which are the projections.
A key step is to construct the following retraction (see Corollary \ref{cor:Radon-retract-of-nuHom}):
\begin{equation}\label{eq:intro-specialized-retraction}
\nu_{\delta_P}(\mathscr H_{\mathcal F})
\longrightarrow \Pi_*\nu_{i_p}(\mathscr H_{\mathcal G})
\longrightarrow\nu_{\delta_P}(\mathscr H_{\mathcal F}).
\end{equation}
Applying the Fourier transform to \eqref{eq:intro-specialized-retraction}, we obtain $\mu\mathcal{H}om({\mathcal F},{\mathcal F})$ as a retract of
\begin{equation}
F_\psi \Pi_*\nu_{i_p}(\mathscr H_{\mathcal G})\simeq \Pi^\vee_\ast F_\psi \nu_{i_p}(\mathscr H_{\mathcal G}).
\end{equation}
We then estimate the support 
\begin{align}
\Sigma:={\rm supp}_k(F_\psi \nu_{i_p}(\mathscr H_{\mathcal G}))\subseteq Q\times_PQ\times_PT^\ast P,
\end{align}
where ${\rm supp}_k(F_\psi \nu_{i_p}(\mathscr H_{\mathcal G}))$ is the unique closed subset of $Q\times_PQ\times_PT^\ast P$ such that 
\begin{align}
	({\rm supp}_k(F_\psi \nu_{i_p}(\mathscr H_{\mathcal G})))_{\bar k}={\rm supp}(F_\psi \nu_{i_p}(\mathscr H_{\mathcal G})).
\end{align}
We use a deformation method to construct a linear action of $T(Q\times Q/P^\vee\times P^\vee)$ on $T_{Q\times_PQ}(Q\times Q)$ under which $\nu_{i_p}(\mathscr H_{\mathcal G})$ is invariant (cf. Lemma \ref{lem:calculatekerbetavee}). We use this linear action to deduce that $\Sigma$ is contained in the union
of the zero section and the tautological line bundle on the diagonal $Q\subseteq Q\times_PQ$.
We use the universal local acyclicity of $p|_U$ with respect to $\mathcal G$ to deduce that $\Sigma|_{U\times_PU}$ is contained in the zero section $0_{T^\ast P}$. 
Combining the two estimations, we obtain \eqref{eq:main-inclusion}.

The next step is to prove a support estimate for microlocalization along a smooth closed subscheme.
\begin{theorem}[Theorem \ref{thm:support-of-microlocalization-to-closed-subscheme}]\label{thm:intro-Saito2}
Let $Z\hookrightarrow X$ be a closed immersion of smooth $k$-schemes and ${\mathcal F}\in D_{\rm ctf}(X,\Lambda)$. Let $T_ZX$ be the normal bundle of $Z\hookrightarrow X$ and $F_{\psi,T_ZX/Z}$ the Fourier transform on it. We define
\begin{equation}\label{eq:intro-specialization-support}
\mu_{Z/X}({\mathcal F})=F_{\psi,T_ZX/Z}(\nu_{Z/X}{\mathcal F})\in D_{\rm ctf}(T^*_ZX\times_k\eta_0,\Lambda),
\end{equation}
where $\nu_{Z/X}({\mathcal F})$ is the specialization of ${\mathcal F}$ along $Z\hookrightarrow X$, and $\eta_0$ is the generic point of the henselization of $\mathbf A^1$ at $0$. We have:		 
\begin{equation}\label{eq:suppMircClosed}
\mathrm{supp}_k(\mu_{Z/X}({\mathcal F}))
\subseteq{SS}_\mu({\mathcal F})\cap T_Z^*X
\subseteq{SS}({\mathcal F})\cap T_Z^*X,
\end{equation}
where $\mathrm{supp}_k(\mu_{Z/X}({\mathcal F}))\subseteq T_Z^\ast X$ is the unique closed subset
of $T_Z^\ast X$ such that 
\begin{align}
	(\mathrm{supp}_k(\mu_{Z/X}({\mathcal F})))_{\bar k}=\mathrm{supp}(\mu_{Z/X}({\mathcal F}))
	\quad{\rm in}\quad T_{Z_{\bar k}}X_{\bar k}.
\end{align}
\end{theorem}
This is proved by constructing a retraction
\begin{equation}
\mu_{Z/X}({\mathcal F})\to	(q^{\vee})^\ast\mu\mathcal{H}om({\mathcal F},{\mathcal F})\otimes\mu_{Z/X}({\mathcal F})\longrightarrow\mu_{Z/X}({\mathcal F}),
\end{equation}
where $q^\vee:T^*_ZX\hookrightarrow T^*X\times_XZ$ is dual to the quotient bundle map $q:TX\times_XZ\to T_ZX$.
Then \eqref{eq:main-inclusion} implies \eqref{eq:suppMircClosed}.

We establish a K\"unneth formula for $\mu\mathcal{H}om$ in Lemma \ref{lem:Kunneth-for-muHom}.  Let $X$ and $Y$ be smooth
$k$-schemes.  For constructible complexes ${\mathcal F}\in D_{\rm ctf}(X,\Lambda)$ and ${\mathcal G}\in D_{\rm ctf}(Y,\Lambda)$, we construct a canonical isomorphism
\begin{align}
\mu\mathcal{H}om({\mathcal F}\boxtimes {\mathcal G},{\mathcal F}\boxtimes {\mathcal G})
\simeq\mu\mathcal{H}om({\mathcal F},{\mathcal F})\boxtimes\mu\mathcal{H}om({\mathcal G},{\mathcal G}).
\end{align}
This gives the product formulas
\begin{align}\label{eq-intro:SSCCKunneth}
{SS}_\mu({\mathcal F}\boxtimes {\mathcal G})={SS}_\mu({\mathcal F})\times{SS}_\mu({\mathcal G}),\qquad
{CC}_\mu({\mathcal F}\boxtimes {\mathcal G})={CC}_\mu({\mathcal F})\boxtimes{CC}_\mu({\mathcal G}).
\end{align}
Finally, we prove the equality \eqref{eq:main} using \eqref{eq:main-inclusion}, Theorem \ref{thm:intro-Saito2} together with the K\"unneth formula \eqref{eq-intro:SSCCKunneth} for ${SS}_\mu$.

\begin{theorem}[Theorem  \ref{prMilnor} and Theorem \ref{thm:CC-cycle}]\label{thm:main2}
Let $k$ be a perfect field and	let $X$ be a smooth $k$-scheme of pure dimension $d$. 
For any
${\mathcal F}\in D_{\mathrm{ctf}}(X,\Lambda)$, we have
\begin{equation}\label{eq:mainCC}
{CC}_\mu({\mathcal F})=\mathrm{cl}{CC}({\mathcal F}),
\end{equation}
where ${\rm cl}: {\rm Z}_{d}({SS}({\mathcal F}))\to {\rm H}^0_{{SS}({\mathcal F})}(T^\ast X, e_X^{\vee\ast}  K_{X/k})$ is the cycle class map.

The assertion also holds for constructible complexes with coefficients in a
finite extension $E/\mathbf Q_\ell$, interpreted on the pro-\'etale site.
\end{theorem}
We sketch  the proof of \eqref{eq:mainCC}.
First, we prove the equality in Corollary \ref{cortame} 
for a tamely ramified sheaf ${\cal F}$ on a curve
by computing $\nu{\cal H}om({\cal F},{\cal F})$.
Then, using a global argument
based on a theorem of Katz--Gabber \cite{KG}, we 
prove Proposition \ref{prdim1} for curves.
For the proof, we use an index formula
for $CC_\mu$ proved in Corollary \ref{corindex}
and compare it with
the Grothendieck--Ogg--Shafarevich formula.
The index formula in Corollary \ref{corindex}
is a special case of the compatibility of
$CC_\mu$ with proper push-forward in
Proposition \ref{prpush}.

To prove Theorem \ref{prMilnor}
in higher dimension,
we prove that $CC_\mu({\cal F})$ also
satisfies the characterization
of $CC({\cal F})$ by the Milnor formula.
By  the equality
$SS_\mu ({\cal F})=
SS({\cal F})$,
the theorem of Beilinson
$\dim SS({\cal F})= \dim X$
implies
$\dim SS_\mu({\cal F})\leqq \dim X$.
By the weak cohomological purity,
this makes it suffice
to compare the coefficients
in $CC_\mu({\cal F})$ and in $CC({\cal F})$
of irreducible components of dimension
$\dim X$.
By the compatibility with push-forward
for a closed immersion,
it is reduced to the case where $X$ is projective.
Then, by taking a Lefschetz pencil
and by the compatibility with proper push-forward in
Proposition \ref{prpush},
the Milnor formula is reduced to
the equality $CC_\mu={\rm cl}\, CC$ for the direct image on ${\mathbf P}^1$
proved in Proposition \ref{prdim1}.

As an application of Theorem~\ref{thm:main2}, we prove a
conjecture on cohomological characteristic classes formulated by the first named author.  
\begin{conjecture}[Saito, {\cite[Conjecture~6.8]{CC}}] 
\label{conj:saito-characteristic-intro}
Assume $k$ is a perfect field.
Let $Z$ be a $k$-scheme, and let $i: Z\to X$ be a closed immersion into a smooth $k$-scheme $X$.
Let $\Lambda$ be a finite local ring whose residue characteristic is invertible in $k$.
For every
${\mathcal F}\in D_{\mathrm{ctf}}(Z,\Lambda)$, we have
\begin{equation}\label{eq:saito-characteristic-conjecture-intro}
C_{Z/k}({\mathcal F})=\operatorname{cl}_Z(cc_{Z,0}({\mathcal F}))
\quad\text{in }{\rm H}^0(Z,{K}_{Z/k}),
\end{equation}
where $cc_{Z,0}(\mathcal F)=0_X^! CC(i_\ast \mathcal F)\in {\rm CH}_0(Z)$ is the refined Gysin pull-back of $CC(i_\ast \mathcal F)$ by the zero section $0_X\colon X\to T^\ast X$  $($see \cite[Definition 6.7]{CC}$)$.
\end{conjecture}

In \cite[Theorem 1.2]{YZ}, the third named author and Zhao prove this conjecture in the case where
$Z$ is smooth and quasi-projective by using non-acyclicity classes and the fibration method. Using microlocalization, we prove Conjecture \ref{conj:saito-characteristic-intro} in full generality.

\begin{theorem}[Theorem~\ref{thm:characteristic-class}]
\label{thm:characteristic-intro}
Conjecture~\ref{conj:saito-characteristic-intro} holds.
\end{theorem}

We apply the  results for $\ell$-adic coefficients to establish two results for finite coefficients. The first is the index formula for the Euler--Poincar\'e characteristic, which is proved in \cite[Theorem 7.13]{CC} under the assumption that the variety is projective.

\begin{corollary}[Corollary \ref{cor:index}]\label{intro:index}
Let $X$ be a  smooth scheme over a perfect field  $k$ of characteristic $p>0$, and let $\Lambda$ be a finite local ring over $\mathbf Z/\ell^n\mathbf Z$ for a prime number $\ell\neq p$. 
Let  $i\colon Z\hookrightarrow X$ be a closed immersion and $\mathcal F\in D_{\rm ctf}(Z,\Lambda)$.
Assume $Z$ is proper over $k$.
Then we have an equality in $\mathbf Z$
\begin{align}\label{eq:intro-indexFormulaforProper}
	\chi(Z_{\bar k},\mathcal F)={\rm deg}cc_{Z,0}(\mathcal F).
\end{align}
\end{corollary}

The second is the geometric case of the conductor formula conjectured by Bloch in \cite[Conjecture]{Blo87},
which is proved in \cite[Corollary 2.2.4]{SaitoCond} under the assumption that the morphism $f$ is projective.
\begin{corollary}[Corollary \ref{cor:bloch-constant}]\label{intro:bloch-constant}
Let $f: X\to Y$ be a proper flat morphism of smooth schemes over a perfect field $k$ of characteristic $p>0$, and let $\Lambda$ be a finite local ring over $\mathbf Z/\ell^m\mathbf Z$ for a prime number $\ell\neq p$.  Let $y\in Y$ be a closed point.
Assume that ${\rm dim}X=n$,  ${\rm dim}Y=1$ and that there is a dense open subset $V\subseteq Y$ such that
the base change $f_V:X\times_YV\to  V$ is smooth.
Then we have an equality  in $\mathbf Z$
\begin{align}\label{eq:intro-bloch-constant}
-a_y(f_\ast\Lambda)=(-1)^n{\rm deg}_{k(y)}\left( {c_n}^X_{X_y}(\Omega^1_{X/Y})\cap[X]\right),
\end{align}
where ${c_n}^X_{X_y}(\Omega^1_{X/Y})$ is the localized Chern class defined in \cite[Section 1]{Blo87}.
\end{corollary}	
We also note that Abe \cite{Abe} has studied the index formula and the conductor formula from a homotopical perspective  using $\infty$-categories.

We briefly sketch the contents of each section.
We recall basic properties of Fourier
transforms
and prove some complements in Section \ref{ssFourier}.
In Section \ref{sssp}, we introduce deformation to normal bundle.
The specialization functor 
is introduced in Section \ref{sssp-1}
following Verdier
\cite{Verdier}.
The specialization is defined by applying
the nearby cycles functor
to the deformation to normal bundle. The microlocalization is defined as the Fourier
transform of the specialization.
In Section \ref{subsec:Radon-transform-notations}, we prove some results on retraction.

In Section \ref{ssSS}, we define $SS_\mu({\cal F})$
using 
the microlocalization $\mu{\cal H}om({\cal F},{\cal F})$.
We prove \eqref{eq:main-inclusion} and Theorem \ref{thm:intro-Saito2} in Section \ref{sec:inclusion-of-singluar-supports}, and we establish the equality \eqref{eq:main} in Section \ref{sec:comparison-on-singular-supports}.

In Section \ref{susec:constructionOfCCmu}, we define $CC_\mu({\cal F})$
using 
the microlocalization $\mu{\cal H}om({\cal F},{\cal F})$.
We compute $CC_\mu$ 
in Section \ref{sstame}
for a tamely ramified sheaf on a curve; this computation is used in the proof of the curve case of \eqref{eq:mainCC}.
In Sections \ref{sspush} and \ref{sspull},
we study the compatibility with
proper push-forward and
with smooth pull-back respectively.
We prove the K\"unneth formula \eqref{eq-intro:SSCCKunneth} in Section \ref{sec:external-products}.
We prove \eqref{eq:mainCC} in Section \ref{ssCC}.
In Section \ref{ssAS}, we give an explicit and independent computation of $CC_\mu$ 
for a wildly ramified Artin--Schreier sheaf on a curve.
As an application of Theorem \ref{thm:main2}, we prove Theorem \ref{thm:characteristic-intro} in Section \ref{sec:characteristic-classes}.
In Section \ref{sec:NA-equality},
we extend the results to $\ell$-adic sheaves, and prove the index formula \eqref{eq:intro-indexFormulaforProper} and Bloch's conductor formula \eqref{eq:intro-bloch-constant}.

When the first named author studied the microlocalization of \'etale sheaves
with Ahmed Abbes
in 2004, they thought that this approach would not work
since an alternative approach using the blow-up
taken in \cite{AS} worked well.
The first named author recently computed the example 
in Section \ref{ssAS}
and realized that they had reached the negative conclusion
too quickly.
The same question was also considered independently
by D.-C.~Cisinski and A.~Khan.

During 2016--2018, when the third named author was a postdoctoral researcher at the University of Regensburg, Denis-Charles Cisinski suggested developing a theory of singular support and microlocalization in the motivic setting. While a notion of singular support for motives was successfully constructed, the difficulties arising from wild ramification led the third named author to remain skeptical about the existence of a satisfactory theory of microlocalization. After completing the joint work with Zhao \cite{YZ}, however, he came to believe that an appropriate form of microlocalization might exist, although no substantial progress was made at that stage.\footnote{There is a link between $\mu\mathcal Hom(\cal F,\cal F)$ and the object $\delta_1^\ast\delta^\Delta\mathscr H_{\mathcal F}$;  this link will be investigated in a later paper.} 
It was only after the first named author had taken the first step in this work that the key idea emerged. 
More specifically, the idea of deforming after the base change \eqref{eq:intro-collision-center} was inspired by the construction of non-acyclicity classes in \cite[4.2--4.3]{YZ}.

\subsection*{Acknowledgments}
The authors thank Alexander Beilinson for suggesting the use
of Wang's Fourier transform.
They also thank Adeel Khan, Yigeng Zhao, Haoyu Hu, Fangzhou Jing,
and Xiangyu Pan for their comments.
The first named author (T.~S.) is supported by
Kakenhi 24K06683 and 26H01993.
The second and third named authors (J.~X. and E.~Y.) are partially supported by the National Key
R\&D Program of China (Grant No.~2021YFA1001400), NSFC Grant No.~12522101 and NSFC Grant
No.~12271006.

\subsection*{Notation and Conventions}

\begin{enumerate}[label=\textup{(\arabic*)},leftmargin=2.2em]
\item Throughout this article,
$k$ denotes a field of characteristic $p>0$, and $\bar k$ denotes a separable closure of $k$.
 Unless otherwise stated, all schemes are separated and of finite type over  $k$.

\item Let $\Lambda$ be a finite local ring
over ${\mathbf Z}/\ell^n{\mathbf Z}$
for a prime number $\ell\neq p$
and an integer $n\geqq 1$.
For a noetherian scheme $X$ over $k$,
$D^b_c(X,\Lambda)\supset
D_{\rm ctf}(X,\Lambda)$
denote
the derived category of bounded constructible
complexes of $\Lambda$-modules
on $X$ and
its full subcategory of complexes
of finite Tor-dimension.

\item 
We fix a
non-trivial character
$\psi:\mathbf F_p\to\Lambda^\times$ whenever such a character exists; otherwise use the coefficient extension of Lemma \ref{lm:coeff-character}.
 Let $\mathcal L_\psi$ be the Artin--Schreier sheaf on $\mathbf A^1_k$ corresponding to the character $\psi$. For a morphism $f\colon X\to \mathbf A^1_k$, we put $\mathcal L_\psi(f)=f^\ast \mathcal L_\psi$.

\item For a smooth $k$-scheme $X$, we write $TX$ and $T^\ast X$ for its tangent and cotangent bundles, respectively. If $f\colon X\to Y$ is a smooth morphism, we write $T(X/Y)$ and $T^\ast(X/Y)$ for the relative tangent and cotangent bundles, respectively. For a regular immersion $i\colon Z\hookrightarrow X$, we write $T_ZX$ and $T^\ast_ZX$ for the normal and conormal bundles of $Z$ in $X$, respectively.

\item Unless explicitly stated otherwise, we omit $R$ and $L$ from the notation for	derived functors.

\item For a separated morphism $a\colon X\to S$, write
${K}_{X/S}=a^!\Lambda$ and
${D}_{X/S}(-)=\mathcal{H}om(-,{K}_{X/S})$.  We write ${D}_X$ when the base is
$\operatorname{Spec}k$.
\end{enumerate}

\section{Fourier transforms and the specialization}\label{SXY:saito-section1}

\subsection{Complements on Fourier transforms}\label{ssFourier}

Let $k$ be a field of characteristic $p>0$.
Let $\Lambda$ be a finite local ring
over ${\mathbf Z}/\ell^n{\mathbf Z}$
for a prime number $\ell\neq p$
and an integer $n\geqq 1$.
If such a character exists, we fix a non-trivial character
$\psi\colon {\mathbf F}_p\to \Lambda^\times$.
Let $\psi^\vee\colon {\mathbf F}_p\to \Lambda^\times$
be the character defined by
$\psi^\vee(x)=\psi(-x)=\psi(x)^{-1}$.

In general,
the following lemma shows that such a $\psi$ exists after a finite
\'etale faithfully flat extension of the coefficient ring $\Lambda$.

\begin{lm}\label{lm:coeff-character}
	There exist a finite \'etale faithfully flat local $\Lambda$-algebra
	$\Lambda'$ and a non-trivial character
	$\psi:\mathbf F_p\to(\Lambda')^\times$.
\end{lm}

\begin{proof}
	Since $p$ is invertible in $\Lambda$, the algebra
	$B=\Lambda[T]/(1+T+\cdots+T^{p-1})$
	is finite \'etale over $\Lambda$. 
	Since $B$ is Artinian, it is a finite product of nonzero
	local rings. Choose one local factor $\Lambda'$.
	It is finite \'etale over $\Lambda$, hence finite free
	of positive rank because $\Lambda$ is local.
	In particular, $\Lambda'$ is faithfully flat over $\Lambda$.
	
The image $\zeta$ of $T$ in $\Lambda'$ satisfies
$\zeta^p=1$, and $\zeta-1$ is invertible.
Hence $\zeta$ has order $p$, and
\begin{align}
\psi:\mathbf F_p\longrightarrow(\Lambda')^\times,
\qquad \bar a\longmapsto\zeta^a,
\end{align}
is a non-trivial character.
\end{proof}

\begin{df}\label{dfFD}
Let $k$ be a field of characteristic $p>0$ and
$\psi\colon {\mathbf F}_p\to \Lambda^\times$
be a non-trivial character.

{\rm 1.}
Let ${\mathbf A}^1={\rm Spec}\, k[t]\to 
{\mathbf A}^1={\rm Spec}\, k[x]$
be the Artin--Schreier covering
defined by $t^p-t=x$ and identify the Galois group
with ${\mathbf F}_p$ by the action of $a\in {\mathbf F}_p$
defined by $t\mapsto t+a$.
Let 
${\cal L}_\psi$ be
the locally constant constructible sheaf
of free $\Lambda$-modules of rank $1$
on ${\mathbf A}^1={\rm Spec}\, k[x]$
trivialized by the Artin--Schreier covering
${\mathbf A}^1\to {\mathbf A}^1$
corresponding to 
the character $\psi\colon {\mathbf F}_p\to \Lambda^\times$.

{\rm 2.}
Let $X$ be a noetherian scheme over $k$
and let $E$ be a vector bundle over $X$.
Let $E^\vee$ be the dual vector bundle
and 
\begin{equation}
	\begin{CD}
		E@<{{\rm pr}_1}<<E\times_XE^\vee 
		@>{{\rm pr}_2}>>E^\vee\\
		@.@V\mu=<,> VV@.\\
		@.{\mathbf A}^1@.
	\end{CD}
\end{equation}
be the projections and the canonical pairing.
We define the (naive) Fourier transform
\begin{equation}
F=F_\psi=F_{\psi, E/X}\colon D^b_c(E,\Lambda)\to 
D^b_c(E^\vee,\Lambda)
\label{eqFpsi}
\end{equation}
by $F_\psi {\cal F}={\rm pr}_{2!}({\rm pr}_1^*{\cal F}\otimes
\mu^*{\cal L}_{\psi})$
and the (naive) dual Fourier transform
\begin{equation}
F^\vee=F_{\psi^\vee}=F_{\psi^\vee,E^\vee/X}\colon D^b_c(E^\vee,\Lambda)\to 
D^b_c(E,\Lambda)
\label{eqFvee}
\end{equation}
by $F_{\psi^\vee} {\cal G}={\rm pr}_{1!}({\rm pr}_2^*{\cal G}\otimes
\mu^*{\cal L}_{\psi^\vee})$.
\end{df}

The functor
$F_\psi\colon D^b_c(E,\Lambda)\to 
D^b_c(E^\vee,\Lambda)$ induces
$F_\psi\colon D_{\rm ctf}(E,\Lambda)\to 
D_{\rm ctf}(E^\vee,\Lambda)$.
Let $0\colon X\to E$ and
$0^\vee\colon X\to E^\vee$ be the $0$-sections
and let $e\colon E\to X$ 
and 
$e^\vee \colon E^\vee\to X$ denote the canonical
morphisms.
Fourier transforms satisfy the following elementary
properties.

\begin{lm}\label{lmF0}
Let $e\colon E\to X$
be a vector bundle of rank $n$
over a noetherian scheme $X$ over $k$
and let $e^\vee\colon E^\vee\to X$
be the dual vector bundle.

{\rm 1.}
We have a canonical isomorphism
\begin{equation}
0^{\vee*}F_\psi\to e_!
\label{eq0Fe}
\end{equation}
of functors $D^b_c(E,\Lambda)
\to D^b_c(X,\Lambda)$.

{\rm 2.}
We have a canonical isomorphism
\begin{equation}
F_\psi 0_*\to e^{\vee*},
\label{eqF1}
\end{equation}
of functors $D^b_c(X,\Lambda)
\to D^b_c(E^\vee,\Lambda)$.

{\rm 3.}
The adjoint 
\begin{equation}
F_\psi e^!\to 0^\vee_*
\label{eqF0!}
\end{equation}
of the composition 
$0^{\vee*}F_\psi e^!\to
e_!e^!\to1$
of the morphism 
induced by
{\rm (\ref{eq0Fe})}
and the adjunction
is an isomorphism
of functors $D^b_c(X,\Lambda)
\to D^b_c(E^\vee,\Lambda)$.

The isomorphism
{\rm (\ref{eqF0!})} is equivalently formulated 
as an isomorphism
\begin{equation}
F_\psi e^*\to 0^\vee_*(-\otimes e_!\Lambda).
\label{eqF0}
\end{equation}
In particular, for any ${\mathcal M}\in D^b_{c}(X,\Lambda)$, we have
\begin{equation}
\mathrm{supp}F_{\psi}(e^*{\mathcal M})\subseteq 0^\vee(X)\subseteq E^\vee.
\end{equation}

\end{lm}

\proof{
1.
By proper base change theorem,
we obtain an isomorphism
(\ref{eq0Fe}).

2.
Let ${\cal F}\in D^b_c(X,\Lambda)$.
Since the restriction
$\mu^*{\cal L}_{\psi}|_{X\times_XE^\vee}$
is canonically identified with $\Lambda$,
we have an isomorphism
$F_\psi 0_*{\cal F}=
{\rm pr}_{2!}({\rm pr}_1^*0_*{\cal F}\otimes
\mu^*{\cal L}_{\psi})\to e^{\vee*}{\cal F}$
(\ref{eqF1}).

3.
Let ${\cal F}\in D^b_c(X,\Lambda)$.
We have an isomorphism
$F_\psi (e^!{\cal F})=
{\rm pr}_{2!}({\rm pr}_2^*e^{\vee!}{\cal F}\otimes
\mu^*{\cal L}_{\psi})\to e^{\vee!}{\cal F}
\otimes F_\psi \Lambda$
by the projection formula.
By the proper base change theorem,
the restriction of
$F_\psi \Lambda=
{\rm pr}_{2!}
\mu^*{\cal L}_{\psi}$
on the complement $E^\vee\sm 0^\vee(X)$
of the $0$-section vanishes.
Hence the isomorphism
(\ref{eq0Fe}) induces
an isomorphism
(\ref{eqF0!}).

By the canonical isomorphism
$e^*(-)\otimes e^!\Lambda
\to e^!$
of Poincar\'e duality,
by identifying
$e^{\vee*}\!-\!\otimes 0^\vee_*e_!\Lambda
=
0^\vee_*(-\!\otimes e_!\Lambda)$,
we obtain (\ref{eqF0}).
\qed

}
\medskip

The following is 
a fundamental result of Katz and Laumon.

\begin{pr}\label{prGL}
{\rm 1. (\cite[Th\'eor\`eme (2.4.4)]{KL})}.
Let $j\colon {\mathbf G}_m\times
{\mathbf A}^1
\to {\mathbf A}^1\times {\mathbf P}^1$
be the open immersion over $k$
and let ${\cal L}_\psi(y/x)$ denote the pull-back
of ${\cal L}_\psi$ by the morphism
$y/x\colon {\mathbf G}_m\times
{\mathbf A}^1={\rm Spec}\, k[x^{\pm 1},y]
\to {\mathbf A}^1$.
Then, the projection
${\rm pr}_1\colon {\mathbf A}^1\times {\mathbf P}^1
\to {\mathbf A}^1$
is universally locally acyclic relatively to 
$j_!{\cal L}_\psi(y/x)$.

{\rm 2. (\cite[(a) in the proof of Th\'eor\`eme (2.4.1)]{KL})}.
Let $X$ be a noetherian scheme
and consider the cartesian diagram
\begin{equation}
	\begin{CD}
		{\mathbf A}^2_X
		@>{j_2}>> {\mathbf P}^1_X\times_X
		{\mathbf A}^1_X\\
		@V{{\rm pr}_1}VV@VV{\overline{\rm pr}_1}V
		\\
		{\mathbf A}^1_X
		@>j>> {\mathbf P}^1_X.
	\end{CD}
\end{equation}
Then, the morphism
$j_{2!}({\rm pr}_1^*(-)
\otimes \mu^*{\cal L}_\psi)
\to
j_{2*}({\rm pr}_1^*(-)
\otimes \mu^*{\cal L}_\psi)$
is an isomorphism.
\end{pr}

\proof{
1. 
First, we show the universal local acyclicity
on ${\mathbf A}^2={\rm Spec}\, k[x,y]
\subseteq {\mathbf A}^1\times {\mathbf P}^1$. 
It suffices to show the universal local acyclicity
after the base change by the purely inseparable morphism
${\mathbf A}^1\to {\mathbf A}^1$
sending $x$ to
$x^p$.
The normalization of the covering of
${\mathbf A}^2$ defined by
$t^p-t=y/x^p$
is ${\mathbf A}^2={\rm Spec}\, k[x,s]
\to {\mathbf A}^2={\rm Spec}\, k[x,y]$
given by 
$s^p-x^{p-1}s=y$
for $s=xt$ and 
the projection
${\mathbf A}^2={\rm Spec}\, k[x,s]
\to {\mathbf A}^1={\rm Spec}\, k[x]$
is smooth.
Hence the assertion follows from the local acyclicity
of smooth morphisms.

Next, we show the universal local acyclicity
on ${\mathbf A}^1\times ({\mathbf P}^1\sm
\{0\})$.
Changing the coordinates and
notation,
it suffices to show that
the projection
${\mathbf A}^2={\rm Spec}\, k[x,y]
\to {\mathbf A}^1={\rm Spec}\, k[x]$
is universally locally acyclic
with respect to
$j_!{\cal L}_\psi(1/xy)$
for the open immersion
$j\colon {\mathbf G}_m^2\to {\mathbf A}^2$.
Let $g\colon 
{\mathbf A}^1\times 
{\mathbf G}_m={\rm Spec}\, k[x,u^{\pm1}]
\to {\mathbf A}^2={\rm Spec}\, k[x,u]$
be the open immersion.
By the generic universal local acyclicity
\cite[Corollaire 2.16]{TF},
the projection
$p_1\colon {\mathbf A}^2={\rm Spec}\, k[x,u]
\to 
{\mathbf A}^1={\rm Spec}\, k[x]$
is universally locally acyclic
relatively to $g_!{\cal L}_\psi(1/u)$.
The blow-up $q\colon P\to {\mathbf A}^2$
at the origin is the union of
two affine planes
$A={\mathbf A}^2={\rm Spec}\, k[x,y]$
and
$B={\mathbf A}^2={\rm Spec}\, k[u,v]$
where $u=xy$
and $x=uv$ respectively.
Since $q$ is an isomorphism
outside the exceptional divisor $E$,
the composition
$p_1\circ q\colon P\to {\mathbf A}^1$
is locally acyclic
on the complement $P\sm E$.
Further,
by the first step above,
the composition
$p_1\circ q\colon P\to {\mathbf A}^1$
is locally acyclic
relatively to $q^*g_!{\cal L}_\psi(1/u)$
except possibly at
the origins of $A$ and $B$.
Since $q\colon P\to {\mathbf A}^2$
is proper,
by proper base change theorem,
the composition
$p_1\circ q\colon P\to {\mathbf A}^1$
is everywhere locally acyclic
relatively to $q^*g_!{\cal L}_\psi(1/u)$.
Since the restriction of
$q^*g_!{\cal L}_\psi(1/u)$ on $A$
is the same as
$j_!{\cal L}_\psi(1/xy)$,
the assertion follows.

2.
Let ${\cal F}\in D^b_c({\mathbf A}^1_X,\Lambda)$.
By 1,
the projection 
$\overline{\rm pr}_1\colon
{\mathbf P}^1_X\times_X
{\mathbf A}^1_X
\to
{\mathbf P}^1_X$
is locally acyclic relatively to 
$j_{2!}\mu^*{\cal L}_\psi$.
Hence by \cite[Proposition 2.10]{App},
the morphism
$j_{2!}({\rm pr}_1^*{\cal F}
\otimes \mu^*{\cal L}_\psi)
=
\overline {\rm pr}_1^*
j_*{\cal F}\otimes
j_{2!}\mu^*{\cal L}_\psi
\to
j_{2*}({\rm pr}_1^*
{\cal F}\otimes
\mu^*{\cal L}_\psi)$
is an isomorphism.
\qed

}
\medskip

Proposition \ref{prGL}.2
implies that
the canonical morphism
\begin{equation}
F_\psi{\cal F}=
{\rm pr}_{2!}({\rm pr}_1^*{\cal F}\otimes\mu^*{\cal L}_\psi)
\to
{\rm pr}_{2*}({\rm pr}_1^*{\cal F}\otimes\mu^*{\cal L}_\psi)
\label{eqF!*}
\end{equation}
is an isomorphism
\cite[Th\'eor\`eme (2.4.1)]{KL}.

\begin{cor}\label{corFj}
Let $E$ be a line bundle on $X$
and $E^\vee$ be the dual.
Let $g\colon E^\times=E\sm X\to E$ and
$g^\vee\colon E^{\vee\times}=E^\vee\sm X\to E^\vee$
be the open immersions
of the complements
of the $0$-sections.
Let $0\colon X\to E$ be the $0$-section
and define a morphism
$0_*\Lambda_X\to g_!\Lambda_{E^\times}[1]$
by the exact sequence
$0\to g_!\Lambda_{E^\times}
\to \Lambda_E\to 0_*\Lambda_X\to 0$.
Then, there exists a unique
isomorphism
\begin{equation}
F_\psi g_!\Lambda_{E^\times}[1]
\to g^\vee_*\Lambda_{E^{\vee\times}}
\label{eqFj!}
\end{equation}
fitting in the commutative diagram
\begin{equation}
\begin{CD}
F_\psi0_*\Lambda_X@>{(\ref{eqF1})}>>
\Lambda_{E^\vee}\\
@VVV@VVV\\
F_\psi g_!\Lambda_{E^\times}[1]
@>{(\ref{eqFj!})}>>
g^\vee_*\Lambda_{E^{\vee\times}}.
\end{CD}
\label{eqFg!}
\end{equation}
\end{cor}

\proof{
By the isomorphism
$F_\psi \Lambda_E\to 0^\vee_* \Lambda_X(-1)[-2]$,
the exact sequence
$0\to g_!\Lambda_{E^\times}
\to \Lambda_E\to \Lambda_X\to 0$
induces 
an isomorphism
$(F_\psi0_*\Lambda_X)|_{E^{\vee\times}}
=\Lambda_{E^{\vee\times}}
\to 
F_\psi g_!\Lambda_{E^\times}[1]|_{E^{\vee\times}}$.
By adjunction, there exists a unique morphism
$F_\psi g_!\Lambda_{E^\times}[1]
\to g^\vee_*\Lambda_{E^{\vee\times}}$
making (\ref{eqFg!}) commutative.
Its restriction on $E^{\vee\times}$ is
an isomorphism.

It suffices to show that the morphism
$F_\psi g_!\Lambda
\to 
g^\vee_*g^{\vee*}F_\psi g_!\Lambda$ is an isomorphism.
We consider the commutative diagram
\begin{equation}
	\begin{CD}
		F_\psi g_!\Lambda
		@>>> 
		{\rm pr}_{2*}({\rm pr}_1^*g_!\Lambda
		\otimes \mu^*{\cal L})\\
		@VVV@VVV\\
		g^\vee_*g^{\vee*}F_\psi g_!\Lambda
		@>>>
		g^\vee_*{\rm pr}_{2*}(({\rm pr}_1^*g_!\Lambda
		\otimes \mu^*{\cal L})|_{E\times_XE^{\vee\times}}).
	\end{CD}
\end{equation}
By the isomorphism (\ref{eqF!*}),
the horizontal arrows are isomorphisms.
It suffices to show that the right vertical arrow is an isomorphism.
Since $e_*g_!\Lambda=0$,
the morphism
${\rm pr}_{2*}({\rm pr}_1^*g_!\Lambda
\otimes \mu^*{\cal L})
\to
{\rm pr}_{2*}((g_!\Lambda\boxtimes g^\vee_*\Lambda)
\otimes \mu^*{\cal L})$ is an isomorphism.
Let
$g_2^\vee\colon E
\times_XE^{\vee\times}
\to E\times_XE^\vee$
be the open immersion.
Then 
${\rm pr}_{2*}((g_!\Lambda\boxtimes g^\vee_*\Lambda)
\otimes \mu^*{\cal L})
\to 
{\rm pr}_{2*}g_{2*}^\vee(({\rm pr}_1^*g_!\Lambda
\otimes \mu^*{\cal L})|_{E\times_XE^{\vee\times}})
$ is an isomorphism and
the assertion follows.
\qed

}

\begin{pr}[{\cite[Th\'eor\`eme (1.2.2.1)]{La}}]\label{prFF}
Let $e\colon E\to X$
be a vector bundle 
over a noetherian scheme $X$ over $k$
and let $e^\vee\colon E^\vee\to X$
be the dual vector bundle.
There exists a canonical isomorphism
\begin{equation}
i\colon F_\psi F_{\psi^\vee}\to -\otimes e^{\vee*}e_!\Lambda
\label{eqinv}
\end{equation}
of functors $D^b_c(E^\vee,\Lambda)\to D^b_c(E^\vee,\Lambda)$.
\end{pr}

\proof{
Let ${\cal L}_\psi(x(y^\vee-x^\vee))$
denote the pull-back of ${\cal L}_\psi$
by the morphism $E^\vee\times_XE\times_XE^\vee
\to {\mathbf A}^1$ sending $(x^\vee,x,y^\vee)$
to $\mu(x,y^\vee-x^\vee)$.
We consider the cartesian diagram
\begin{equation}
	\begin{aligned}
\xymatrix{
E&E\times_XE^\vee\ar[r]^{{\rm pr}_2}
\ar[l]_{{\rm pr}_1}&
E^\vee\\
E^\vee\times_XE\ar[u]^{{\rm pr}_2}\ar[d]_{{\rm pr}_1}&
E^\vee\times_XE
\times_XE^\vee\ar[u]^{{\rm pr}_{23}}
\ar[l]_-{{\rm pr}_{12}}\ar[r]^-{{\rm pr}_{13}}
\ar[ld]_{{\rm pr}_1}
\ar[ru]^{{\rm pr}_3}
&
E^\vee\times_XE^\vee\ar[u]_{{\rm pr}_2}
\\
E^\vee&
E^\vee\times_XE
\ar[r]^{{\rm pr}_1}
\ar[u]^{\delta \times 1}
&E^\vee.\ar[u]_{\delta}
}
	\end{aligned}
\end{equation}
By the proper base change theorem
applied to the lower right square,
the pull-back by the lower middle vertical
arrow induces isomorphisms
\begin{equation}
{\rm pr}_{13!}{\cal L}_\psi(x(y^\vee-x^\vee))\to
{\rm pr}_{13!}(\delta\times 1)_*
(\delta\times 1)^*{\cal L}_\psi(x(y^\vee-x^\vee))\to
\delta_*{\rm pr}_{1!}\Lambda
\to
\delta_*e^{\vee*}e_!\Lambda.
\label{eqpr13}
\end{equation}
By
proper base change theorem,
the projection formula
and by (\ref{eqpr13}),
we obtain isomorphisms
\begin{align}
	\begin{aligned}
		F_\psi F_{\psi^\vee}
		\to {\rm pr}_{3!}({\rm pr}_1^*-\otimes 
		{\cal L}_\psi(x(y^\vee-x^\vee)))
		&\,\to {\rm pr}_{2!}({\rm pr}_1^*-\otimes 
		{\rm pr}_{13!}
		{\cal L}_\psi(x(y^\vee-x^\vee)))\\
		&\,
		\to {\rm pr}_{2!}(
		\delta_*(-\otimes e^{\vee*}e_!\Lambda))
		\to -\otimes e^{\vee*}e_!\Lambda.
	\end{aligned}
\end{align}
\qed

}
\medskip

In the following, we fix $\psi$ and drop it from the notation.
The isomorphism (\ref{eqinv}) is equivalently reformulated as
an isomorphism
\begin{equation}
i\colon F (F^\vee(-)\otimes e^!\Lambda)
\to 1.
\label{eqinvb}
\end{equation}

\begin{pr}\label{prFF-}
Let $e\colon E\to X$
be a vector bundle of rank $n$
over a noetherian scheme $X$ over $k$
and let $e^\vee\colon E^\vee\to X$
be the dual vector bundle.
The diagram
\begin{equation}
\begin{CD}
F (F^\vee e^{\vee!}(-)
\otimes e^!\Lambda)
@>i>> e^{\vee!}\\
@V{\rm (\ref{eqF0})}VV@AA{(-1)^n{\rm can}}A\\
F(0_*(-)
\otimes e^!\Lambda)
@>{\rm (\ref{eqF1})}>> 
e^{\vee*}(-)\otimes
e^{\vee*}0^*e^!\Lambda
\end{CD}
\label{eqFF-1}
\end{equation}
of functors $D^b_c(X,\Lambda)\to 
D^b_c(E^\vee,\Lambda)$
is commutative.
The upper horizontal arrow is
defined by 
{\rm (\ref{eqinvb})}
and the right vertical arrow
is $(-1)^n$-times 
the morphism induced by
the canonical isomorphism
$e^{\vee*}0^*e^!\Lambda\to 
e^{\vee!}\Lambda$.
\end{pr}

\proof{
We consider closed immersions
\begin{equation}
\begin{CD}
E^\vee\times_X E^\vee
@>{(1,0,1)}>>
E^\vee\times_X E\times_X E^\vee
@<{\delta\times 1}<<
E^\vee\times_X E
\end{CD}
\label{eqimm2}
\end{equation}
defined by
$(\delta\times 1)(x^\vee,x)=
(x^\vee,x,x^\vee).$
Let $p_1,p_3\colon E^\vee\times_X E^\vee
\to E^\vee$
and
$q\colon E^\vee\times_X E
\to E^\vee$
be the restrictions of
the projections
${\rm pr}_1,
{\rm pr}_3\colon
E^\vee\times_X E\times_X E^\vee
\to E^\vee$.
By the proof of Proposition \ref{prFF},
the upper horizontal morphism
$i\colon F (F^\vee e^{\vee!}{\mathcal F}
\otimes e^!\Lambda)
\to e^{\vee!}{\mathcal F}$
is induced by the composition via upper right
in the diagram
\begin{equation}
\begin{aligned}
\xymatrix{
	{\rm pr}_{3!}
	({\rm pr}_1^*e^{\vee!}{\mathcal F}\otimes
	{\cal L}((y^\vee-x^\vee)x)
	\otimes {\rm pr}_2^*e^!\Lambda)
	\ar[r]\ar[d]
	&
	q_!(
	e^{\vee!}{\mathcal F}
	\boxtimes e^!\Lambda)
	\ar[d]
	\\
	p_{3!}({\rm pr}_1^*e^{\vee!}{\mathcal F}
	\otimes (e^\vee\times e^\vee)^*0^*e^!\Lambda)
	\ar[r]
	&e^{\vee!}{\mathcal F}
}
\end{aligned}
\label{eqpr3}
\end{equation}
for ${\cal F}\in D^b_c(X,\Lambda)$.
The upper horizontal arrow and the left vertical
arrow are induced by the pull-back by the
immersions in (\ref{eqimm2}) respectively.
The right vertical arrow is induced by the isomorphism
${\rm Tr}_{E/X}\colon
e_!e^!\to 1$.
The lower horizontal
arrow is induced by the isomorphism
${\rm Tr}_{E^\vee/X}\colon
e^\vee_!e^{\vee!}\to 1$
for the first factor $E^\vee$
and the canonical isomorphism
$e^{\vee*}0^*e^!\Lambda\to 
e^{\vee!}\Lambda$.
The arrows in (\ref{eqFF-1})
via lower left are induced by
the corresponding arrows
in (\ref{eqpr3}).
Hence it suffices to show that
(\ref{eqpr3}) is $(-1)^n$-commutative.
By the projection formula,
we may assume ${\cal F}=\Lambda$.

After changing the coordinate
$(x^\vee,x,y^\vee)$
by
$(y^\vee-x^\vee,x,y^\vee)$,
the immersions in (\ref{eqimm2}) 
are the base change by $E^\vee\to X$ of
the immersions
$E^\vee\to E\times_X E^\vee\gets E$
and
${\cal L}((y^\vee-x^\vee)x)$
is the pull-back of
${\cal L}(z^\vee x)$.
Hence, it suffices to show Lemma \ref{lmEEv}.1 below.
\qed

}

\begin{lm}\label{lmEEv}
{\rm 1.}
Let $e\colon E\to X$ be a vector bundle of
rank $n$ and $e^\vee\colon E^\vee\to X$
be the dual.
Then, the diagram
\begin{equation}
\begin{CD}
(e\times e^\vee)_!\mu^*{\cal L}_\psi(n)[2n]
@=
e^\vee_!
F_\psi e^!\Lambda
@>{\rm (\ref{eqF0!})}>>
e^\vee_!0^\vee_*\Lambda\\
@|@.@VVV\\
e_!
F^\vee_{\psi} e^{\vee!}\Lambda
@>{\rm (\ref{eqF0!})}>>
e_!0_*\Lambda
@>>>
\Lambda
\end{CD}
\end{equation}
of isomorphisms of sheaves on $X$
is $(-1)^n$-commutative.

{\rm 2.}
Let $\mu\colon {\mathbf A}^2={\rm Spec}\, k[x,y]
\to
{\mathbf A}^1={\rm Spec}\, k[t]$
be the morphism
defined by $t=xy$
and let $\Phi_0\Lambda$ denote
the vanishing cycles complex
at $0\in {\mathbf A}^2$ with respect to $\mu$.
Then, we have an isomorphism
$R^1\mu_!\Lambda\to \Lambda$,
an exact sequence
$0\to \Phi_0^1\Lambda
\to R^2\mu_!\Lambda
\to \Lambda(-1)\to 0$
and $R^q\mu_!\Lambda=0$
for $q\neq 1,2$.

{\rm 3.}
Let $k$ be an algebraically closed field and
let $\sigma$ be the involution of
${\mathbf A}^2={\rm Spec}\, k[x,y]$
switching $x$ and $y$.
Then $\sigma$ acts on
${\rm H}^2_c({\mathbf A}^2,{\cal L}_\psi(xy))$ by
$-1$.
\end{lm}

\proof{
1. 
Since the assertion is local on $X$,
we may assume that
$E={\mathbf A}^n_X$.
By the K\"unneth formula,
we may assume that $n=1$.
Further, we may assume that
$X={\rm Spec}\, k$ for an algebraically closed field $k$.
Hence, the assertion is reduced to 3.

2.
We construct
a relative compactification 
$\overline \mu\colon
P\to {\mathbf A}^1$.
The second projection
${\mathbf P}^1\times {\mathbf A}^1
\to {\mathbf A}^1={\rm Spec}\, k[t]$
is proper.
Let $P\to
{\mathbf P}^1\times {\mathbf A}^1$
be the blow-up at
the origin of
${\mathbf A}^2={\rm Spec}\, k[y,t]
\subseteq {\mathbf P}^1\times {\mathbf A}^1$.
Let $L_0$ be the proper transform of
$(y=0)\subseteq {\mathbf A}^2$
and $L_\infty$ be the inverse image  of
$({\mathbf P}^1\times {\mathbf A}^1)
\sm {\mathbf A}^2$.
Then, the complement
$P\sm (L_0\cup L_\infty)$
is identified with
${\mathbf A}^2={\rm Spec}\, k[x,y]$
by $xy=t$
and
hence $P\to {\mathbf P}^1\times {\mathbf A}^1
\to {\mathbf A}^1={\rm Spec}\, k[t]$
is a relative compactification of
$\mu\colon {\mathbf A}^2={\rm Spec}\, k[x,y]
\to
{\mathbf A}^1$.

Let $j\colon {\mathbf A}^2={\rm Spec}\, k[x,y]
=P\sm (L_0\cup L_\infty)
\to P$ be the open immersion.
The morphism $\overline \mu\colon P\to {\mathbf A}^1$
is smooth except at $0\in {\mathbf A}^2$
and $L_0$ and $L_\infty$
are sections of $\overline \mu\colon P\to {\mathbf A}^1$.
Hence the morphism $P\to {\mathbf A}^1$
is locally acyclic relatively to $j_!\Lambda$
except at the origin $0\in {\mathbf A}^2$.
Let $\bar\eta_0$ be the generic geometric
point of the strict henselization
at $0\in {\mathbf A}^1$ and
let $\Phi_0\Lambda$
denote the vanishing cycles complex at 0
with respect to $\mu$.
By the Picard--Lefschetz formula,
the complex
$\Phi_0\Lambda$
is acyclic except at degree $1$
and
the $\Lambda$-module
$\Phi^1_0\Lambda$
is free of rank 1.
Let
$+={\mathbf A}^2\times_{{\mathbf A}^1}0$
and
${\mathbf G}_{m,\bar\eta_0}=
{\mathbf A}^2\times_{{\mathbf A}^1}
\bar\eta_0$ be the fibers of $\mu\colon
{\mathbf A}^2\to{\mathbf A}^1$.
Then, the assertion follows from 
the distinguished triangle
$\Gamma_c(+,\Lambda)
\to
\Gamma_c({\mathbf G}_{m,\bar\eta_0},\Lambda)
\to 
\Phi_0\Lambda\to $.

3.
By the projection formula,
we have a canonical isomorphism
$\mu_!{\cal L}_\psi(xy)
\to {\cal L}_\psi\otimes \mu_!\Lambda$.
Hence by 2.~and
$\Gamma_c({\mathbf A}^1,{\cal L}_\psi)=0$,
we have a canonical isomorphism
$\Phi^1_0\Lambda
\to 
{\rm H}^0_c({\mathbf A}^1,{\cal L}_\psi
\otimes R^2\mu_!\Lambda)\to 
{\rm H}^2_c({\mathbf A}^2,{\cal L}_\psi(xy))$.
Since
$\sigma$ acts on $\Phi^1_0\Lambda$ by $-1$
by the Picard--Lefschetz formula,
the assertion follows.
\qed

}
\medskip

We study relations of Fourier transforms
with pull-back and push-forward
by linear morphisms of vector bundles.

\begin{pr}[{\cite[Th\'eor\`eme (1.2.2.4)]{La}}]\label{prFa}
Let $a\colon E\to E'$ be a linear morphism of
vector bundles on a noetherian scheme $X$ over $k$
and $a^\vee\colon E'^\vee\to E^\vee$ be the dual.
Let $F$ and $F'$ denote the Fourier transforms
for $E$ and $E'$ respectively.
We have a canonical isomorphism
\begin{equation}
d\colon F'a_!\to a^{\vee*}F
\label{eqd}
\end{equation}
of functors $D^b_c(E,\Lambda)
\to D^b_c(E'^\vee,\Lambda)$.
\end{pr}

The morphism (\ref{eqF1}) is
the case where $a\colon X\to E$ is the $0$-section.

\proof{
By applying the proper base change theorem
to the cartesian squares in the commutative diagram
\begin{align}
\begin{aligned}
\xymatrix{
	&E\times_XE^\vee\ar[r]\ar[ld]&
	E^\vee
	\\
	E\ar[d]_a&E\times_XE'^\vee
	\ar[l]\ar[r]\ar[u]\ar[d]&
	E'^\vee\ar[u]_{a^\vee}
	\\
	E'&E'\times_XE'^\vee\ar[l]
	\ar[ur]&
}
\end{aligned}
\end{align}
and by the projection formula,
we obtain isomorphisms
\begin{align}
\begin{aligned}
F'a_!
\,=
{\rm pr}'_{2!}
({\rm pr}'^*_1(a_!-)
\otimes {\cal L}_\psi(x'x'^\vee))
&\to 
{\rm pr}'_{2!}
(a\times 1_{E'^\vee})_!
({\rm pr}'^*_1(-)
\otimes {\cal L}_\psi(a(x)x'^\vee))\\
&\,\to
{\rm pr}'_{2!}
(1_E\times a^\vee)^*(
{\rm pr}^*_1(-)
\otimes {\cal L}_\psi(xx^\vee))\\
&\to
a^{\vee*}
{\rm pr}_{2!}
({\rm pr}^*_1(-)
\otimes {\cal L}_\psi(xx^\vee))
=
a^{\vee*}
F.
\end{aligned}
\end{align}
\qed

}

\begin{cor}\label{cordv}
Let the notation be as in Proposition
{\rm \ref{prFa}}
and let
$F^\vee$ and $F'^\vee$ denote the dual Fourier transforms
for $E$ and $E'$ respectively (cf. \eqref{eqFvee}).

{\rm 1.}
There exists a unique isomorphism
\begin{equation}
d^\vee \colon 
F(a^*(-)
\otimes e^!\Lambda)
\to a^\vee_!F'(-
\otimes e'^!\Lambda)
\label{eqd2}
\end{equation}
of functors $D^b_c(E',\Lambda)
\to D^b_c(E^\vee,\Lambda)$
such that the diagram
\begin{equation}
\begin{CD}
F(a^*F'^\vee(-)
\otimes e^!\Lambda)
@>{d^\vee}>>
a^\vee_!F'(F'^\vee(-)
\otimes e'^!\Lambda)
\\
@AdAA@VV{i_{E'}}V\\
F(F^\vee a_!^\vee(-)
\otimes e^!\Lambda)
@>{i_E}>>
a^\vee_!
\end{CD}
\label{eqd22}
\end{equation}
is commutative.
The left vertical arrow is an isomorphism induced
by the isomorphism $d$ for $a^\vee
\colon E'^\vee\to E^\vee$
under the canonical identifications
$E^{\vee\vee}=E$
and 
$E'^{\vee\vee}=E'$.

If $a\colon E\to E'$ is smooth,
{\rm (\ref{eqd2})} is equivalently reformulated as an isomorphism
\begin{equation}
d^\vee \colon 
Fa^!
\to a^\vee_*F'.
\label{eqd2a}
\end{equation}

{\rm 2.}
The isomorphism
$Fe^!
\to 0^\vee_!$
{\rm (\ref{eqF0!})}
is the special case of $d^\vee$ 
{\rm (\ref{eqd2a})} where
$a\colon E\to X$ is
the structure morphism.
Namely, the diagram
\begin{equation}
\begin{aligned}
\xymatrix{
	F(e^*(-)
	\otimes e^!\Lambda)=Fe^!
	\ar[r]^-{(\ref{eqF0!})}
	&
	0^\vee_!
	\\
	F(F^\vee 0_!^\vee(-)
	\otimes e^!\Lambda)
	\ar[u]^d
	\ar[ur]_{i_E}}
\label{eqd1.5}
\end{aligned}
\end{equation}
is commutative.
\end{cor}

\proof{
1.
The functor
$F'^\vee 
\colon
D^b_c(E'^\vee,\Lambda)
\to
D^b_c(E',\Lambda)$
is an equivalence of categories
by Proposition \ref{prFF}.
Hence the commutative diagram (\ref{eqd22})
uniquely defines a natural isomorphism
$d^\vee$.

2.
By the characterization of $d^\vee$,
it suffices to show that
the diagram (\ref{eqd1.5})
is commutative.
By the proof of Proposition \ref{prFF-},
$i_E\colon F(F^\vee0^\vee_!(-)\otimes e^!\Lambda)
\to 0^\vee_!$
equals the composition of the isomorphism
$F(F^\vee0^\vee_!(-)\otimes e^!\Lambda)
\to 
{\rm pr}_{1!}(0^\vee_!(-)\boxtimes e^!\Lambda)
$ induced by the
pull-back by $\delta\times 1\colon
E^\vee\times E\to E^\vee\times E\times E^\vee$
and the isomorphism
${\rm pr}_{1!}(0^\vee_!(-)\boxtimes e^!\Lambda)
\to 
0^\vee_!(-)\otimes e_!e^!\Lambda
\to 
0^\vee_!$
induced by ${\rm Tr}_{E/X}$.
Since
$F(e^*(-)\otimes e^!\Lambda)\to
Fe^!
\to 0^\vee_!$
(\ref{eqF0!})
is also induced by ${\rm Tr}_{E/X}$,
we have the equality as required.
\qed

}

\medskip

Let $a\colon E\to E'$ be a linear morphism of
vector bundles on a noetherian scheme $X$.
Canonical morphisms 
\begin{align}
&{\rm adj}_a\colon
a_!e^! \to e'^!,
\label{eqadje}
\\
&{\rm adj}_a\colon
0_! \to a^!0'_!
\label{eqadj0}
\end{align}
are defined to be the adjoints of 
the canonical isomorphisms $e^! \to a^!e'^!$
and $a_!0_!\to 0'_!$.
In particular, 
(\ref{eqadje}) for the $0$-section
$a=0'\colon X\to E'$ defines
${\rm adj}_{e'}\colon
0'_! \to e'^!$.
Similarly,
\begin{align}
&{\rm res}_a\colon
a^*0'_*
\to 0_*,
\label{eqres0}
\\
&{\rm res}_a\colon
e'^*
\to a_*e^*
\label{eqrese}
\end{align}
are defined as the adjoint of 
$0'_* \to a_*0_*$
and
$a^*e'^* \to e^*$.
In particular, (\ref{eqres0}) for
the canonical morphism
$a\colon E\to E'=X$
defines
${\rm res}_e\colon
e^* \to 0_*$.
We show that the Fourier transforms
of adjunction morphisms are
identified with the restriction morphisms.
We also show that dual statements involve signs.

\begin{pr}\label{prEE'}
Let $a\colon E\to E'$ be a linear morphism of
vector bundles on a noetherian scheme $X$ over $k$
and $a^\vee\colon E'^\vee\to E^\vee$ be the dual.
Let $F$ and $F'$ denote the Fourier transforms
for $E$ and $E'$ respectively.

{\rm 1.}
The diagram 
\begin{equation}
\begin{aligned}
\xymatrix{
	F'a_!e^! \ar[d]_{{\rm adj}_a}
	\ar[r]^{d}&
	a^{\vee*}Fe^!
	\ar[r]^{{\rm (\ref{eqF0!})}_E}
	&a^{\vee*}0^\vee_*
	\ar[ld]^{{\rm res}_{a^\vee}}
	\\
	F'e'^!
	\ar[r]^{{\rm (\ref{eqF0!})}_{E'}}
	&0'^\vee_*&
}
\label{eqresF}
\end{aligned}
\end{equation}
is commutative.
In particular, the diagram
\begin{equation}
\begin{CD}
F_\psi e^!
@>{\rm (\ref{eqF0!})}>>
0^\vee_*
\\
@A{{\rm adj}_e}AA
@AA{{\rm res}_{e^\vee}}A
\\
F_\psi 0_!
@>{\rm (\ref{eqF1})}>>
e^{\vee*}
\end{CD}
\label{eqresF64}
\end{equation}
is commutative.

{\rm 2.}
Assume that $a\colon E\to E'$
is an immersion
of codimension $r=n'-n={\rm rank}\, E'-{\rm rank}\, E$.
Then, the diagram 
\begin{equation}
\begin{aligned}
\xymatrix{
	F'a_*e^* 
	\ar[r]^{d}&
	a^{\vee*}Fe^*
	\ar[r]^-{{\rm (\ref{eqF0})}_E}
	&a^{\vee*}0^\vee_*(-\otimes e_!\Lambda)
	\\
	F'e'^*\ar[u]^{{\rm res}_a}
	\ar[r]^-{{\rm (\ref{eqF0})}_{E'}}
	&0'^\vee_*(-\otimes e'_!\Lambda)
	\ar[ru]_{{\rm adj}_{a^\vee}}
	&
}
\end{aligned}
\label{eqresFe}
\end{equation}
is $(-1)^r$-commutative.
The slant arrow is defined by
canonically identifying the
upper right corner
$a^{\vee*}0^\vee_*(-\otimes e_!\Lambda)$
with
$a^{\vee!}0^\vee_!(-\otimes e'_!\Lambda)$.
In particular, the diagram
\begin{equation}
\begin{CD}
F_\psi 0'_*
@>{\rm (\ref{eqF1})}>>
e'^{\vee*}
\\
@A{{\rm res}}AA
@AA{{\rm adj}}A
\\
F_\psi e'^*
@>{\rm (\ref{eqF0})}>>
0'^\vee_*(-\otimes e'_!\Lambda)
\end{CD}
\label{eqresF1.30}
\end{equation}
is $(-1)^{n'}$-commutative.
\end{pr}

\proof{
1.
By taking the adjoint, it suffices to show that
the diagram 
\begin{equation}
\begin{aligned}
\xymatrix{
	0'^{\vee*}F'a_!e^! \ar[d]_{{\rm adj}_a}
	\ar[r]^{d}&
	0'^{\vee*}a^{\vee*}Fe^!
	\ar[r]^{{\rm (\ref{eqF0!})}_E}
	&
	0'^{\vee*}a^{\vee*}0^\vee_*
	\ar[ld]^{{\rm res}_{a^\vee}}
	\\
	0'^{\vee*}F'e'^!
	\ar[r]^{{\rm (\ref{eqF0!})}_{E'}}
	&1&
}
\end{aligned}
\label{eqresFad0}
\end{equation}
is commutative.
By the definition of (\ref{eqF0!}) as the adjoint
and by the isomorphism (\ref{eq0Fe}),
it is reduced to the commutative diagram
\begin{equation}\label{eqresFad1}
\begin{aligned}
\xymatrix{
e'_!a_!e^!\ar[d]_{{\rm adj}_a}
\ar[r]
&
e_!e^!
\ar[d]
\\
e'_!e'^!
\ar[r]
&1&
}
\end{aligned}
\end{equation}
where the arrows without labels are canonical isomorphisms.

For (\ref{eqresF64}), it is
enough to consider the case
$a=0\colon X\to E$.

2.
By taking the adjoint,
it suffices to show that the
diagram
\begin{equation}
\begin{aligned}
\xymatrix{
	a^\vee_!F'a_*e^* 
	\ar[r]^{d}&
	a^\vee_!a^{\vee*}Fe^*
	\ar[r]^-{{\rm (\ref{eqF0})}_E}
	&0^\vee_!(-\otimes e'_!\Lambda)
	\\
	a^\vee_!F'e'^*\ar[u]^{{\rm res}_a}
	\ar[r]^-{{\rm (\ref{eqF0})}_{E'}}
	&a^\vee_!0'^\vee_!(-\otimes e'_!\Lambda)
	\ar[ru]_{{\rm can}}
	&
}
\end{aligned}
\label{eqresFea}
\end{equation}
is $(-1)^r$-commutative.
Since every term is supported on
the $0$-section of $E^\vee$,
it suffices to show the 
$(-1)^r$-commutativity after applying
$e^\vee_!$.
Then by the projection formula,
it suffices to show the 
$(-1)^r$-commutativity after evaluating
the functors at $\Lambda$.
Namely, it is reduced to showing
that the diagram
\begin{equation}
\begin{aligned}
\xymatrix{
	e'^\vee_!F'a_*\Lambda_E
	\ar[r]^{d}&
	e'^\vee_!a^{\vee*}F\Lambda_E
	\ar[d]^-{{\rm (\ref{eqF0})}_E}
	\\
	e'^\vee_!F'\Lambda_{E'}\ar[u]^{{\rm res}_a}
	\ar[r]^-{{\rm (\ref{eqF0})}_{E'}}
	&
	e'_!\Lambda_{E'}
}
\end{aligned}
\label{eqresFeaL}
\end{equation}
is $(-1)^r$-commutative.
Since the assertion is local on $X$,
by the K\"unneth formula,
we may further assume that
$a\colon X=E\to E'$
is the $0$-section.
In this case, 
the $(-1)^r=(-1)^{n'}$-commutativity 
follows from
Lemma \ref{lmEEv}.1.

For (\ref{eqresF1.30}), it is
enough to consider the case
$a=0\colon X\to E'$.
\qed

}
\medskip

We deduce an analogue of
Proposition \ref{prEE'} for $d^\vee$.

\begin{pr}\label{prFFdd}
Let $a\colon E\to E'$ be a linear morphism of
vector bundles on a noetherian scheme $X$ over $k$
and $a^\vee\colon E'^\vee\to E^\vee$ be the dual.
Let $n={\rm rank}\, E,\,
n'={\rm rank}\, E'$ and
$r=n-n'$.

{\rm 1.}
The diagram
\begin{equation}
\begin{aligned}
\xymatrix{
	F (a^*0'_*(-)\otimes e^!\Lambda)
	\ar[r]^{d^\vee}
	\ar[d]_{{\rm res}_a}&
	a^\vee_!F'(0'_*(-)
	\otimes e'^!\Lambda)
	\ar[r]^-{(\ref{eqF1})_{E'}}
	&
	a^\vee_!e'^{\vee!}
	\ar[ld]^{(-1)^r{\rm adj}_{a^\vee}}
	\\
	F (0_*(-)\otimes e^!\Lambda)
	\ar[r]^-{{\rm (\ref{eqF1})}_E}
	&
	e^{\vee!}&
}
\end{aligned}
\label{eqcores}
\end{equation}
is commutative.
The bottom horizontal arrow
{\rm (\ref{eqF1})}$_E$
is the composition 
$F(0_*(-)
\otimes e^!\Lambda)
\overset{(\ref{eqF1})} \longrightarrow
e^{\vee*}(-)\otimes
e^{\vee*}0^*e^!\Lambda
\to
e^{\vee*}(-)\otimes
e^{\vee!}\Lambda
\to
e^{\vee!}$
and 
{\rm (\ref{eqF1})}$_{E'}$
in the upper line
is defined similarly.

{\rm 2.}
Assume that $a$ is smooth.
Then the diagram 
\begin{equation}
\begin{aligned}
\xymatrix{
	Fa^!0'_!\ar[r]^{d^\vee}
	&
	a^\vee_*F'0'_!
	\ar[r]^{{\rm (\ref{eqF1})}_{E'}}&
	a^\vee_*e'^{\vee*}
	\\
	F0_!\ar[u]^{{\rm adj}_a}
	\ar[r]^{{\rm (\ref{eqF1})}_E}&
	e^{\vee*}
	\ar[ur]_{{\rm res}_{a^\vee}}&}
\end{aligned}
\label{eqresFb}
\end{equation}
is commutative.
\end{pr}

In \cite[Lemma 13.5]{KW},
the diagram (\ref{eqcores})
is studied with an unspecified isomorphism.

\proof{
1.
Applying the functor
$F(-\otimes e^!\Lambda)$
to the diagram (\ref{eqresF})
for $a^\vee\colon E'^\vee\to E^\vee$,
we obtain a commutative diagram 
\begin{equation}
\begin{aligned}
\xymatrix{
	F(F^\vee a^\vee_!e'^{\vee!}(-)
	\otimes e^!\Lambda)
	\ar[r]^-{d} \ar[d]_{{\rm adj}_{a^\vee}}&
	F(a^*F'^\vee e'^{\vee!}(-)
	\otimes e^!\Lambda)
	\ar[r]^-{{\rm (\ref{eqF0!})}_{E'^\vee}}
	&
	F (a^*0'_*(-)\otimes e^!\Lambda)
	\ar[ld]^{{\rm res}_a}
	\\
	F(F^\vee e^{\vee!}(-)
	\otimes e^!\Lambda)
	\ar[r]^{{\rm (\ref{eqF0!})}_{E^\vee}}
	&F (0_*(-)\otimes e^!\Lambda).
	&
}
\end{aligned}
\label{eqresF'}
\end{equation}
We consider the composition of this with
(\ref{eqcores})
by identifying the slant arrow in
(\ref{eqresF'})
and the left vertical arrow in (\ref{eqcores}).
By Proposition \ref{prFF-},
the composition 
$F(F^\vee e^{\vee!}(-)
\otimes e^!\Lambda)
\to
F (0_*(-)\otimes e^!\Lambda)
\to e^{\vee!}$
of the lower line is $(-1)^n i_E$.
Hence by Proposition \ref{prEE'}.1,
it suffices to show
that the composition of the upper line
$F(F^\vee a^\vee_!e'^{\vee!}(-)
\otimes e^!\Lambda)
\to  a^\vee_!e'^{\vee!}$
is $(-1)^{n'} i_E$.
To prove this,
we consider the diagram
\begin{align}
\begin{aligned}
\xymatrix{
	F(F^\vee a^\vee_!e'^{\vee!}(-)
	\otimes e^!\Lambda)
	\ar[r]^-{d} \ar[d]_{i_E}&
	F(a^*F'^\vee e'^{\vee!}(-)
	\otimes e^!\Lambda)
	\ar[r]^-{{\rm (\ref{eqF0!})}_{E'^\vee}}
	\ar[d]^{d^\vee}
	&
	F (a^*0'_*(-)\otimes e^!\Lambda)
	\ar[d]^{d^\vee}
	\\
	a^\vee_!e'^{\vee!}
	&
	a^\vee_!F'(F'^\vee e'^{\vee!}(-)
	\otimes e'^!\Lambda)
	\ar[r]^-{{\rm (\ref{eqF0!})}_{E'^\vee}}
	\ar[l]_-{i_{E'}}
	&
	a^\vee_!F'(0'_*(-)
	\otimes e'^!\Lambda).
}
\end{aligned}
\end{align}
The upper line is the same as that of (\ref{eqresF'})
and the right vertical arrow is
the upper left horizontal arrow in (\ref{eqcores}).
By Proposition \ref{prFF-},
the composition of the lower line
is $(-1)^{n'}$-times the upper right
horizontal arrow in (\ref{eqcores})
labelled (\ref{eqF1})$_{E'}$.
The left square is commutative
by Corollary \ref{cordv}
and the right square is commutative
by the functoriality of $d^\vee$.
Hence the composition of
the upper lines in
(\ref{eqresF'})
and 
(\ref{eqcores})
is $(-1)^{n'}i_E$
and 
the diagram  (\ref{eqcores})
is commutative.

2.
Applying the functor
$F(-\otimes e^!\Lambda)$
to the diagram (\ref{eqresFe})
for $a^\vee\colon E'^\vee\to E^\vee$,
we obtain a $(-1)^r$-commutative diagram 
\begin{equation}
\begin{aligned}
\xymatrix{
	F(F^\vee a^\vee_*e'^{\vee*}(-)
	\otimes e^!\Lambda)
	\ar[r]^-{d} &
	F(a^*F'^\vee e'^{\vee*}(-)
	\otimes e^!\Lambda)
	\ar[r]^-{{\rm (\ref{eqF0})}_{E'^\vee}}
	&
	F (a^*0'_*(-)\otimes a^!\Lambda)
	\\
	F(F^\vee e^{\vee*}(-)
	\otimes e^!\Lambda)
	\ar[r]^-{{\rm (\ref{eqF0})}_{E^\vee}}
	\ar[u]^{{\rm res}_{a^\vee}}
	&F 0_*
	\ar[ur]_{{\rm adj}_a}.
	&
}
\end{aligned}
\label{eqresFadj}
\end{equation}
At the upper right corner
and at the lower right corner,
we identify $a^*(e'^*e'^\vee_!\Lambda)
\otimes e^!\Lambda$ with $a^!\Lambda$
and $e^*(e^\vee_!\Lambda)
\otimes e^!\Lambda$ with $\Lambda$
respectively
by the canonical isomorphisms.
We consider the composition of
(\ref{eqresFadj}) with
(\ref{eqresFb})
by identifying the slant arrow in
(\ref{eqresFadj})
and the left vertical arrow in (\ref{eqresFb}),
similarly as in the proof of 1.
By Proposition \ref{prFF-},
the composition 
$F(F^\vee e^{\vee*}(-)
\otimes e^!\Lambda)
\to
F 0_*
\to e^{\vee*}$
of the lower line is $(-1)^n i_E$.
Hence by Proposition \ref{prEE'}.2,
it suffices to show
that the composition of the upper line
$F(F^\vee a^\vee_*e'^{\vee*}(-)
\otimes e^!\Lambda)
\to  a^\vee_*e'^{\vee*}$
is $(-1)^{n'} i_E$.
To prove this,
we consider the diagram
\begin{align}
\begin{aligned}
\xymatrix{
	F(F^\vee a^\vee_*e'^{\vee*}(-)
	\otimes e^!\Lambda)
	\ar[r]^-{d} \ar[d]_{i_E}&
	F(a^*F'^\vee e'^{\vee*}(-)
	\otimes e^!\Lambda)
	\ar[r]^-{{\rm (\ref{eqF0})}_{E'^\vee}}
	\ar[d]^{d^\vee}
	&
	F (a^*0'_*(-)\otimes a^!\Lambda)
	\ar[d]^{d^\vee}
	\\
	a^\vee_*e'^{\vee*}
	&
	a^\vee_*F'(F'^\vee e'^{\vee*}(-)
	\otimes e'^!\Lambda)
	\ar[r]^-{{\rm (\ref{eqF0})}_{E'^\vee}}
	\ar[l]_-{i_{E'}}
	&
	a^\vee_!F'0'_*.
}
\end{aligned}
\end{align}
The left square is commutative
by Corollary \ref{cordv}
and the right square is commutative
by the functoriality of $d^\vee$.
By Proposition \ref{prFF-},
the composition of the lower line
is $(-1)^{n'}$-times the upper right
horizontal arrow in (\ref{eqresFb})
labelled (\ref{eqF1})$_{E'}$.
Hence the composition of
the upper lines in
(\ref{eqresFadj})
and 
(\ref{eqresFb})
is $(-1)^{n'}i_E$
and 
the diagram  (\ref{eqresFb})
is commutative.
\qed

}
\medskip

Let $X\to Y$ be a separated morphism 
of finite type of noetherian schemes over $k$
and $E$ be a vector bundle on $Y$.
Let 
$b\colon E\times_YX\to E$ and
$b^\vee\colon E^\vee\times_YX\to E^\vee$
be the base change morphisms for $E$
and for the dual vector bundle.
Let $F_X$ and $F_Y$
denote the Fourier transforms
for $E\times_YX$ and $E$ respectively.
Then
we have canonical isomorphisms
\begin{align}
b^\vee_!F_X&\, \to F_Yb_!,
\label{eqFb!}
\\
b^{\vee*}F_Y&\, \to F_Xb^*
\label{eqFb*}
\end{align}
by the proper base change theorem.

\medskip

Let $X$ be a smooth curve over
a field $k$
and $0\in X$ be a $k$-rational point. Let $j\colon U=X\sm \{0\}\to X$ be the open immersion.
We identify 
$R^1j_*\Lambda(1)$
with $\Lambda_0$
by the isomorphism
\begin{equation}
\Lambda_0\to R^1j_*\Lambda(1)
\label{eqLam0}
\end{equation}
defined by the image of a uniformizer
by the boundary morphism
$j_*{\mathbf G}_m\to R^1j_*\Lambda(1)$.
Let $e\colon E\to X$ be a line bundle over $X$
and $e^\vee\colon E^\vee\to X$ be the dual.
Let $E_0=E\times_X0$ and
$E^\vee_0=E^\vee\times_X0$ be the fibers.
Let
$j_E\colon E\sm E_0\to E$
and
$j_{E^\vee}\colon E^\vee\sm E^\vee_0\to E^\vee$
be the base change.
Let $E_0^\times=E_0\sm\{0\}$
and $E^{\vee\times}_0=E^\vee_0\sm\{0\}$
be the complement of the origin and let
$g_0\colon E_0^\times\to E_0$ and
$g_0^\vee\colon E_0^{\vee\times}\to E^\vee_0$
be the open immersions.
We identify the non-trivial cohomology sheaf
${\cal H}^{-1}(e^*j_*K_U)=e^*R^1j_*\Lambda(1)$
at the top degree
with $\Lambda_{E_0}$
by the isomorphism (\ref{eqLam0}).
Define a subcomplex
${\cal N}={\cal N}_\Lambda
\in D_{\rm ctf}(E,\Lambda)$
of $e^*j_*K_U$ on $E$
by replacing
${\cal H}^{-1}(e^*j_*K_U)=e^*R^1j_*\Lambda(1)
=\Lambda_{E_0}$
by the subsheaf
$g_{0!}\Lambda_{E_0^\times}$.
Here and in the following,
by abuse of notation,
we identify a sheaf on a closed
subscheme of $E$ with
its direct image on $E$
and similarly for $E^\vee$.
By the definition,
we have a canonical isomorphism
\begin{equation}
0^*{\cal N}_\Lambda
=0^*e^*\mathcal H^{-2}(j_*K_U)[2]\to K_X.
\label{eq0KX}
\end{equation}

\begin{lm}\label{lmNL}
{\rm 1.}
We have a distinguished triangle
\begin{equation}
{\cal N}_\Lambda
\to 
e^*j_*K_U
\oplus
g_{0!}\Lambda_{E_0^\times}[1]
\to
\Lambda_{E_0}[1]\to.
\label{eqNL1}
\end{equation}

{\rm 2.}
Let $\,+\colon
0^\vee(X)\cup E^\vee_0
\to E^\vee$ be the closed immersion
of the union of the $0$-section
and the fiber $E^\vee_0$.
Then, there exists a unique morphism
\begin{equation}
F_\psi {\cal N}_\Lambda\to +_!+^!e^{\vee*} K_X
\label{eqFNL+}
\end{equation}
fitting in the commutative diagram
\begin{equation}
\begin{CD}
F_\psi {\cal N}_\Lambda
@>{\rm  (\ref{eqFNL+})}>> +_!+^!e^{\vee*} K_X\\
@VVV
@VV{\rm adj}V\\
F_\psi 0_*K_X
@>{\rm (\ref{eqF1})}>> e^{\vee*} K_X.
\end{CD}
\label{eqFNL++}
\end{equation}
The left vertical arrow
is induced by the adjoint of
$0^*{\cal N}_\Lambda\to K_X$
{\rm (\ref{eq0KX})}
and the right vertical arrow
is the adjunction morphism.
\end{lm}

\proof{
1.
We have a canonical morphism
${\cal N}_\Lambda
\to 
e^*j_*K_U
\oplus
g_{0!}\Lambda_{E_0^\times}[1]$
by construction.
The morphisms
of cohomology sheaves are
isomorphisms except for degree $-1$.
For degree $-1$,
the morphism
is an injection and the cokernel is
$\Lambda_{E_0}$.
Hence, we obtain (\ref{eqNL1}).

2.
By abuse of notation,
we identify a sheaf on a closed
subscheme of $E^\vee$ with
its direct image on $E^\vee$.
Applying the Fourier transform
to (\ref{eqNL1}),
we obtain a distinguished triangle
\begin{equation}
F_\psi {\cal N}_\Lambda
\to 
0^\vee_*j_*\Lambda \oplus
g^\vee_{0*}\Lambda
\to 
\Lambda_0(-1)[-1]\to.
\label{eqFNLdt}
\end{equation}
Hence $F_\psi {\cal N}_\Lambda$
is supported on the union
$+=0^\vee(X)\cup E^\vee_0$
and the composition 
$F_\psi {\cal N}_\Lambda
\to F_\psi 0_*K_X
\to e^{\vee*}K_X$
of the arrows in (\ref{eqFNL++})
via lower left induces
$F_\psi {\cal N}_\Lambda
\to +_!+^!e^{\vee*} K_X$
(\ref{eqFNL+})
making
(\ref{eqFNL++}) commutative.
\qed

}

\begin{pr}\label{prNL}
Let $i\colon E^\vee_0\to E^\vee$
be the closed immersion
and let
$\Lambda_{0^\vee(X)}\to 0^{\vee!}e^{\vee*} K_X$
and
$\Lambda_{E^\vee_0}\to i^!e^{\vee*} K_X$
be the canonical isomorphisms
defined by the cycle classes
$[0^\vee(X)]$ and $[E^\vee_0]$.
Define
$+_!+^!e^{\vee*} K_X\to
0^\vee_*j_*\Lambda
\oplus
g_{0*}^\vee\Lambda$
to be that induced by
the inverses of these isomorphisms.
Then, the diagram
\begin{equation}
\begin{CD}
F_\psi {\cal N}_\Lambda
@>{\rm  (\ref{eqFNL+})}>> +_!+^!e^{\vee*} K_X\\
@V{\rm (\ref{eqNL1})}VV
@VVV\\
F_\psi (e^*j_*K_U
\oplus
g_{0!}\Lambda_{E_0^\times}[1])
@>{\rm (\ref{eqF0})\oplus
(\ref{eqFj!})}>> 
0^\vee_*j_*\Lambda
\oplus
g_{0*}^\vee\Lambda
\end{CD}
\label{eqFNL+++}
\end{equation}
is $(-1)$-commutative
and the morphism {\rm (\ref{eqFNL+})}
is an isomorphism.
\end{pr}

\proof{
First, we show the $(-1)$-commutativity for
the first component.
It suffices to consider the restriction on
$U=X\sm \{0\}$,
by adjunction.
The restriction of
$0_!\Lambda\to {\cal N}_\Lambda\to 0_*K_X$
on $U$ equals that of the composition of
the adjoint
$0_!\Lambda\to e^*K_X$
and the restriction
$e^*K_X\to 0_*K_X$.
By Proposition \ref{prFFdd}.1,
its Fourier transform is 
$(-1)$-times the composition of
the restriction
$\Lambda\to 0^\vee_*\Lambda$
and the adjoint
$0^\vee_!\Lambda\to e^{\vee*}K_X$.
Since the adjoint
$0^\vee_!\Lambda\to e^{\vee*}K_X$
is defined by the cycle class $[0^{\vee}(X)]$,
the $(-1)$-commutativity for the first component follows.

We show the $(-1)$-commutativity for
the second component.
Since the identification $\Lambda\to R^1j_*\Lambda(1)$
(\ref{eqLam0}) is defined by
sending $1$ to the class of a uniformizer $t$,
its image by the boundary map
$R^1j_*\Lambda(1)
\to {\rm H}^2_0(X,\Lambda(1))$
is the minus $-[0]$ of the cycle class
by \cite[2.1.3]{Cycle}.
Thus, the $(-1)$-commutativity for the second component follows.

We show that
{\rm (\ref{eqFNL+})}
is an isomorphism.
By cohomological purity,
$+_!+^!e^{\vee*} K_X$ fits in a distinguished
triangle
$+_!+^!e^{\vee*} K_X
\to 0^\vee_*j_*\Lambda \oplus
g^\vee_{0*}\Lambda
\to 
\Lambda_0(-1)[-1]\to$.
The lower horizontal arrow
is an isomorphism by Corollary \ref{corFj}.
Since the diagram
(\ref{eqFNL+++}) is $(-1)$-commutative,
by comparing this
with
(\ref{eqFNLdt}),
we see that
{\rm (\ref{eqFNL+})}
is an isomorphism.
%
%
\qed

}

\medskip

By purity,
the $\Lambda$-module
${\rm H}^0_{
0^\vee(X)\cup E^\vee_0}(E^\vee,e^{\vee*}K_X)$
is free of rank $2$
generated by the cycle classes
$[0^\vee(X)]$
and $[E^\vee_0]$.

\begin{cor}\label{cor113}
Let $I$ denote the inertia group at $0\in E_0$.

{\rm 1.}
We have a distinguished triangle
\begin{equation}
0^!{\cal N}_\Lambda
\to 
j_*\Lambda
\oplus
\Gamma(I,\Lambda)_0
\to
{\rm H}^1(I,\Lambda)_0[-1]\to.
\label{eqNL0}
\end{equation}
The last two terms are placed
at $0\in E_0\cap 0(X)\subseteq E$.

{\rm 2.}
Let $\Lambda\to 0^!{\cal N}_\Lambda$
be a morphism
and let
$\Lambda\to 
F_\psi{\cal N}_\Lambda
\to +^!e^{\vee*}K_X$
be the composition of 
the Fourier transform
of its adjoint $0_!\Lambda\to {\cal N}_\Lambda$ and
{\rm  (\ref{eqFNL+})}.
If the composition 
$\Lambda\to 0^!{\cal N}_\Lambda
\to 
j_*\Lambda
\oplus
\Gamma(I,\Lambda)_0$
with the first arrow in {\rm (\ref{eqNL0})}
maps $1\in \Lambda$
to $(a,b)\in \Lambda
\oplus
{\rm H}^0(I,\Lambda)$,
then the composition
$\Lambda\to 
+^!e^{\vee*}K_X$
defines an element
$-a[0^\vee(X)]
-b[E^\vee_0]\in
{\rm H}^0_{
0^\vee(X)\cup E^\vee_0}(E^\vee,e^{\vee*}K_X)$.
\end{cor}

\proof{
1.
For terms in (\ref{eqNL1}),
we have canonical isomorphisms
$0^!e^*j_*K_U\to j_*\Lambda$,
$0^!g_{0!}\Lambda_{E_0^\times}
\to
\Gamma(I,\Lambda)_0[-1]$
and $0^!\Lambda_{E_0}
\to
{\rm H}^1(I,\Lambda)_0[-2]$.
Hence
(\ref{eqNL1})
induces (\ref{eqNL0}).

2.
Let $g^\vee_0\colon
E^\vee_0\sm \{0\}
\to E^\vee_0$
be the open immersion
to the fiber at $0$.
Then the Fourier transform of the first arrow in
(\ref{eqNL0})
defines the upper line in the commutative diagram
\begin{equation}
\begin{CD}
	F_\psi 0^!{\cal N}_\Lambda
	@>>>
	F_\psi (j_*\Lambda
	\oplus
	\Gamma(I,\Lambda)_0)
	@>{(\ref{eqF1})}>>
	e^{\vee*}j_*\Lambda
	\oplus
	e^{\vee*}\Gamma(I,\Lambda)_0
	\\
	@VVV@VVV@VV{\rm res}V\\
	F_\psi {\cal N}_\Lambda
	@>>>
	F_\psi (e^*j_*K_U
	\oplus
	g_{0!}\Lambda_{E_0^\times}[1])
	@>>>
	0^\vee_*j_*\Lambda
	\oplus
	g^\vee_{0*}\Lambda.
\end{CD}
\end{equation}
The lower line is 
(\ref{eqFNL+++}) via lower left.
Thus the assertion follows from
the $(-1)$-commutativity
of (\ref{eqFNL+++}).
\qed

}

\begin{proposition}\label{prop:Kunneth-for-Fourier}
Let $X$ and ${X'}$ be noetherian schemes over $k$.
Let $e:E\to X$ and $e':{E'}\to {X'}$ be vector bundles. For ${\mathcal F}\in D_{\rm ctf}(E,\Lambda)$ and ${\mathcal G}\in D_{\rm ctf}({E'},\Lambda)$, there is a natural isomorphism
\begin{equation}\label{eq:Fourier-product}
F_{\psi,E\times {E'}/X\times {X'}}({\mathcal F}\boxtimes {\mathcal G})\simeq F_{\psi,E/X}({\mathcal F})\boxtimes F_{\psi,{E'}/{X'}}({\mathcal G}).
\end{equation}
\end{proposition}
\begin{proof}
We have natural isomorphisms
\begin{equation}
(E\times {E'})\times_{X\times {X'}}(E\times {E'})^\vee\simeq(E\times {E'})\times_{X\times {X'}}(E^\vee\times {E'}^\vee)\simeq(E\times_XE^\vee)\times({E'}\times_{X'}{E'}^\vee).
\end{equation}
Consider the following commutative diagram
\begin{equation}
\begin{tikzcd}
	(E\times {E'})\times_{X\times {X'}}(E\times {E'})^\vee\arrow[r,"\simeq"]\arrow[d,"{<,>_{E\times {E'}}}"']&(E\times_XE^\vee)\times({E'}\times_{X'}{E'}^\vee)\arrow[d,"{(<,>_E,<,>_{E'})}"]\\
	\mathbf{A}^1&\mathbf{A}^1\times\mathbf{A}^1.\arrow[l,"\mathrm{add}"]
\end{tikzcd}
\end{equation}
Since $\mathrm{add}^*\mathcal{L}_\psi=\mathcal{L}_\psi\boxtimes\mathcal{L}_\psi$, we obtain an isomorphism
\begin{equation}\label{eq:addLpsi}
\mathcal{L}_\psi(\langle\cdot,\cdot\rangle_{E\times {E'}})\simeq\mathcal{L}_\psi(\langle\cdot,\cdot\rangle_E)\boxtimes\mathcal{L}_\psi(\langle\cdot,\cdot\rangle_{E'}).
\end{equation}
Now \eqref{eq:Fourier-product} is the following composition of isomorphisms
\begin{align}
\begin{aligned}
	&F_{\psi,E\times {E'}/X\times {X'}}({\mathcal F}\boxtimes {\mathcal G})\\
	&\simeq(\mathrm{pr}_{E^\vee\times {E'}^\vee})_!(\mathrm{pr}_{E\times {E'}}^*({\mathcal F}\boxtimes {\mathcal G})\otimes\mathcal{L}_\psi(\langle\cdot,\cdot\rangle_{E\times {E'}}))\\
	&\overset{\eqref{eq:addLpsi}}{\simeq}(\mathrm{pr}_{E^\vee}\times\mathrm{pr}_{{E'}^\vee})_!((\mathrm{pr}_E^*{\mathcal F}\boxtimes\mathrm{pr}_{E'}^*{\mathcal G})\otimes(\mathcal{L}_\psi(\langle\cdot,\cdot\rangle_E)\boxtimes\mathcal{L}_\psi(\langle\cdot,\cdot\rangle_{E'})))\\
	&\overset{(1)}{\simeq}\mathrm{pr}_{E^\vee!}(\mathrm{pr}_E^*{\mathcal F}\otimes \mathcal{L}_\psi(\langle\cdot,\cdot\rangle_E))\boxtimes \mathrm{pr}_{{E'}^\vee!}(\mathrm{pr}_{E'}^*{\mathcal G}\otimes \mathcal{L}_\psi(\langle\cdot,\cdot\rangle_{E'}))\\
	&\simeq F_{\psi,E/X}({\mathcal F})\boxtimes F_{\psi,{E'}/{X'}}({\mathcal G}),
\end{aligned}
\end{align}
where (1) is an isomorphism by the K\"{u}nneth formula.
\end{proof}
\begin{corollary}[{\cite[proof of Proposition 1.2.2.7]{La}}]\label{cor:relative-Kunneth-for-Fourier}
Let $X$ be a noetherian scheme over $k$.
Let $e:E\to X$ and $e':{E'}\to X$ be vector bundles. For ${\mathcal F}\in D_{\rm ctf}(E,\Lambda)$ and ${\mathcal G}\in D_{\rm ctf}({E'},\Lambda)$, there is a natural isomorphism
\begin{equation}\label{eq:relative-Fourier-product}
F_{\psi,E\times_X {E'}/X}({\mathcal F}\boxtimes_X {\mathcal G})\simeq F_{\psi,E/X}({\mathcal F})\boxtimes_X F_{\psi,{E'}/{X}}({\mathcal G}).
\end{equation}
\end{corollary}
\begin{proof}
Let $i:E\times_XE'\hookrightarrow E\times E'$ and $i^\vee:E^\vee\times_XE^{\prime\vee}\hookrightarrow E^\vee\times E^{\prime\vee}$ be inclusions.
By the  isomorphism \eqref{eqFb*} and Proposition \ref{prop:Kunneth-for-Fourier}, we have natural isomorphisms
\begin{align}
\begin{aligned}
F_{\psi,E\times_X {E'}/X}({\mathcal F}\boxtimes_X {\mathcal G})
&\simeq F_{\psi,E\times_X {E'}/X}(i^*({\mathcal F}\boxtimes {\mathcal G}))\\
&\overset{\eqref{eqFb*}}{\simeq} i^{\vee\ast}F_{\psi,E\times {E'}/X\times X}({\mathcal F}\boxtimes {\mathcal G})\\
&\overset{\eqref{eq:Fourier-product}}{\simeq } i^{\vee\ast}(F_{\psi,E/X}(\mathcal{F})\boxtimes F_{\psi,E'/X}(\mathcal{G}))\\
&\simeq F_{\psi,E/X}(\mathcal{F})\boxtimes_X F_{\psi,E'/X}(\mathcal{G}).			
\end{aligned}
\end{align}
\end{proof}
\begin{corollary}[{\cite[Proposition 1.2.2.7]{La}}]\label{cor:convolution-of-Fourier}
Let $e:E\to X$ be a vector bundle, and let
\begin{equation}
\mathrm{add}:E\times_XE\to E
\end{equation}
be the addition morphism. For $\mathcal F,\mathcal G\in D_{\mathrm{ctf}}(E,\Lambda)$, we have a natural isomorphism
\begin{equation}\label{eq:convolutionofFourier}
{F}_{\psi,E/X}(\mathrm{add}_!(\mathcal F\boxtimes_X \mathcal G))\simeq {F}_{\psi,E/X}(\mathcal F)\otimes {F}_{\psi,E/X}(\mathcal G).
\end{equation}
\end{corollary}
\begin{proof}
The dual morphism
\begin{equation}
\mathrm{add}^\vee=\delta_{e^\vee}:E^\vee\to E^\vee\times_XE^\vee
\end{equation}
is the diagonal morphism of $e^\vee:E^\vee\to X$. By Proposition \ref{prFa} and Corollary \ref{cor:relative-Kunneth-for-Fourier}, we have the following natural isomorphisms
\begin{align}
\begin{aligned}
	{F}_{\psi,E/X}(\mathrm{add}_!(\mathcal F\boxtimes_X \mathcal G))
	&\simeq \delta_{e^\vee}^*{F}_{\psi,E\times_XE/X}(\mathcal F\boxtimes_X\mathcal  G)\\
	&\simeq \delta_{e^\vee}^*({F}_{\psi,E/X}(\mathcal F)\boxtimes_X{F}_{\psi,E/X}(\mathcal G))\\
	&\simeq {F}_{\psi,E/X}(\mathcal F)\otimes {F}_{\psi,E/X}(\mathcal G).
\end{aligned}
\end{align}
\end{proof}
\begin{definition}\label{def:defineinvariantsheaf}
Let $X$ be a noetherian scheme over $k$.
Let $e:E\to X$ and $w:W\to X$ be vector bundles. Let $\beta:W\to E$ be a linear morphism of vector bundles over $X$.
Let $\mathrm{add}_E$ be the addition map
\begin{equation}
\mathrm{add}_E:E\times_XE\to E,(v_1,v_2)\mapsto v_1+v_2.
\end{equation}
The morphism $\beta$ induces an action $\alpha$ of $W$ on $E$:
\begin{equation}
\alpha=\mathrm{add}_E\circ(\beta\times\mathrm{id}_E):W\times_XE\to E,(\ell,v)\mapsto \beta(\ell)+v.
\end{equation}
Write $\mathrm{pr}_E=w\times\mathrm{id}_E:W\times_XE\to E$ for the projection.
For ${\mathcal K}\in D_{\rm ctf}(E,\Lambda)$, a $W$-invariant structure on ${\mathcal K}$ is an isomorphism
\begin{equation}\label{eq:defineinvariantstructure}
\alpha^*{\mathcal K}\simeq \mathrm{pr}_E^*{\mathcal K}.
\end{equation}
If ${\mathcal K}$ admits a $W$-invariant structure, then
we say that ${\mathcal K}$ is $W$-invariant.
\end{definition}

\begin{remark}
If the isomorphism \eqref{eq:defineinvariantstructure} satisfies the usual cocycle condition, then ${\mathcal K}$ is $W$-equivariant. We do not need this stronger condition.
\end{remark}

\begin{proposition}\label{prop:Fouriersupportofinvariantsheaf} 
We use the notation of Definition \ref{def:defineinvariantsheaf}.
Assume ${\mathcal K}\in D_{\rm ctf}(E,\Lambda)$ is $W$-invariant.
Consider the dual morphism
\begin{equation}
\beta^\vee:E^\vee\to W^\vee,
\end{equation}
and its kernel $\ker(\beta^\vee)\subseteq E^\vee$. Then we have
\begin{equation}\label{eq:Fouriersupportofinvariantsheaf}
\mathrm{supp}(F_{\psi,E/X}{\mathcal K})\subseteq\ker(\beta^\vee).
\end{equation}
\end{proposition}

\begin{proof}
Consider the dual morphisms
\begin{align}
&\alpha^\vee=(\beta^\vee\times\mathrm{id}_{E^\vee})\circ\delta_{e^\vee}
=(\beta^\vee,\mathrm{id}_{E^\vee}):E^\vee\to W^\vee\times_XE^\vee,\\
&\mathrm{pr}_E^\vee=(0_{W^\vee},\mathrm{id}_{E^\vee}):E^\vee\to W^\vee\times_XE^\vee,
\end{align}
where $0_{W^\vee}:E^\vee\to W^\vee$ is the zero morphism.
Note that $\alpha$ and $\mathrm{pr}_E$ are smooth of the same relative dimension.
By Poincar\'e duality and \eqref{eq:defineinvariantstructure}, we have an isomorphism
\begin{equation}
\alpha^!{\mathcal K}\simeq\mathrm{pr}_E^!{\mathcal K}.
\end{equation}
Then by \eqref{eqd2a}, we obtain a natural isomorphism
\begin{equation}\label{eq:Retaid}
(\beta^\vee,\mathrm{id})_*{F}_{\psi,E/X}(\mathcal K)
\simeq
(0_{W^\vee},\mathrm{id})_* {F}_{\psi,E/X}(\mathcal K).
\end{equation}
Taking the supports of both sides of \eqref{eq:Retaid}, we obtain
\begin{align}
(\beta^\vee,\mathrm{id})
\bigl(\mathrm{supp}\,{F}_{\psi,E/X}(\mathcal K)\bigr)
=
(0_{W^\vee},\mathrm{id})
\bigl(\mathrm{supp}\,{F}_{\psi,E/X}(\mathcal K)\bigr).
\end{align}
By taking the first projection, we obtain
$\beta^\vee
\bigl(\mathrm{supp}\,{F}_{\psi,E/X}(\mathcal K)\bigr)
\subseteq 0_{W^\vee}$.
This implies \eqref{eq:Fouriersupportofinvariantsheaf}.
\end{proof}

\begin{lemma}\label{lem:Fourier-transform-of-induced-action}
With the notation of Definition \ref{def:defineinvariantsheaf},
for ${\mathcal F}\in D_{\mathrm{ctf}}(W,\Lambda)$ and ${\mathcal G}\in D_{\mathrm{ctf}}(E,\Lambda)$, we have a natural isomorphism
\begin{equation}\label{eq:Fourier-transform-of-induced-action}
F_{\psi,E/X}(R\alpha_!(\mathrm{pr}_1^*{\mathcal F}\otimes\mathrm{pr}_2^*{\mathcal G}))
\simeq
\beta^{\vee\ast}F_{\psi,W/X}({\mathcal F})\otimes F_{\psi,E/X}({\mathcal G}).
\end{equation}
\end{lemma}

\begin{proof}
Consider the commutative diagram
\begin{equation}
\begin{tikzcd}
	&W\times_XE\arrow[d,"\beta\times\mathrm{id}"]\arrow[r,"\mathrm{pr}_1"]\arrow[dl,"{{\rm pr}_2}"']&W\arrow[d,"\beta"]\\
	E&E\times_XE\arrow[r,"\mathrm{pr}_1"']\arrow[l,"{{\rm pr}_2}"]&E,
\end{tikzcd}
\end{equation}
where the right square is cartesian.
By Corollary \ref{cor:convolution-of-Fourier} and Proposition \ref{prFa}, we have
\begin{align}
\begin{aligned}
	{F}_{\psi,E/X}(\alpha_!(\mathcal F\boxtimes \mathcal G))
	&\simeq {F}_{\psi,E/X}(\mathrm{add}_{E!}(\beta\times\mathrm{id})_!(\mathcal F\boxtimes \mathcal G))\\
	&\simeq {F}_{\psi,E/X}(\mathrm{add}_{E!}(\beta_!\mathcal F\boxtimes\mathcal G))\\
	&\overset{\eqref{eq:convolutionofFourier}}{\simeq}
	{F}_{\psi,E/X}(\beta_!\mathcal F)\otimes {F}_{\psi,E/X}(\mathcal G)\\
	&\overset{\eqref{eqd}}{\simeq}
	\beta^{\vee\ast} {F}_{\psi,W/X}(\mathcal F)\otimes {F}_{\psi,E/X}(\mathcal G).
\end{aligned}
\end{align}
\end{proof}

\subsection{Deformation to normal bundle}\label{sssp}

First, we study dilatations of schemes
smooth over an affine line
${\mathbf A}^1={\rm Spec}\, k[u]$.

\begin{df}\label{dfdil}
Let $k$ be a field
and let $Z\to X$ be a closed immersion
of schemes smooth over the affine line
${\mathbf A}^1$.
Let $\{0\}\subseteq {\mathbf A}^1={\rm Spec}\, k[u]$
be the closed subscheme defined by
the ideal $(u)$ and
define closed subschemes
$Z_0=Z\times_{{\mathbf A}^1}\{0\}\subseteq 
X_0=X\times_{{\mathbf A}^1}\{0\}\subseteq X$.
Let $\overline D_ZX\to X$ be the blow-up
at $Z_0\subseteq X$
and
define the dilatation $D_ZX
\subseteq \overline D_ZX$
to be the open subscheme
obtained by 
removing the proper transform
of $X_0\subseteq X$. 
\end{df}

If $X={\rm Spec}\, A$ is affine
and $Z\subseteq X$
is defined by an ideal $I\subseteq A$,
then $u$ is a non-zero divisor in $A$ and
$D_ZX={\rm Spec}\, A[I/u]$
for the subring $A[I/u]\subseteq A[1/u]$.

\begin{lm}\label{lmXZsm}

{\rm 1.}
The immersion $Z\to X$ is lifted
uniquely to $Z\to D_ZX$.

{\rm 2.}
The closed fiber
$D_ZX\times_{{\mathbf A}^1}\{0\}$
is canonically identified with
the normal bundle $T_{Z_0}X_0$.

{\rm 3.}
The scheme $D_ZX$ is smooth over
${\mathbf A}^1$.
\end{lm}

\proof{
1.
Since $Z_0\subseteq Z$ is a Cartier divisor,
the assertion follows from the universality of dilatations.

2.
The exceptional divisor
of the blow-up
$\overline D_ZX\to X$ is
canonically identified with
the projective space bundle
${\mathbf P}(T_{Z_0}X)$.
Since
$T_{Z_0}X
=
T_{Z_0}X_0\oplus {\mathbf A}^1_{Z_0}$,
the intersection
${\mathbf P}(T_{Z_0}X)
\cap D_ZX$
is identified with
$T_{Z_0}X_0$.

3.
We may assume that
$X={\rm Spec}\, A$
and that $Z\subseteq X$
is defined by an ideal $I\subseteq A$.
Then $D_ZX={\rm Spec}\, A[I/u]$
for the subring $A[I/u]\subseteq A[1/u]$
and is flat over ${\mathbf A}^1$.
Since 
$D_ZX\times_{{\mathbf A}^1}0
=T_{Z_0}X_0$
is smooth over $k$,
$D_ZX$ is smooth over ${\mathbf A}^1$.
\qed

}

\medskip

Let \begin{equation}
\begin{CD}
Z@>>>X\\
@VVV@VVV\\
W@>>>Y
\end{CD}
\label{eqXZfun}
\end{equation}
be a commutative diagram of smooth schemes
over ${\mathbf A}^1$
such that the horizontal arrows are closed immersions.
Then, we have a canonical morphism
$D_ZX\to D_WY$.
We say that
the diagram (\ref{eqXZfun})
is transversal
if it is cartesian and
if the morphism
${\cal O}_X
\otimes^L_{{\cal O}_Y}
{\cal O}_W\to {\cal O}_Z$
is an isomorphism.
Further if the diagram (\ref{eqXZfun})
is transversal,
then the diagram
\begin{equation}
\begin{CD}
D_ZX@>>>X\\
@VVV@VVV\\
D_WY@>>>Y
\end{CD}
\label{eqXZsim}
\end{equation}
is cartesian.

\begin{lm}\label{lmXZfun}
Let \begin{equation}
\begin{CD}
Z@>>>X\\
@VVV@VVV\\
W@>>>Y
\end{CD}
\label{eqXZfunlm}
\end{equation}
be a commutative diagram of smooth schemes
over ${\mathbf A}^1$
such that the horizontal arrows are closed immersions.

{\rm 1.}
Assume that $X\to Y$ is a closed immersion
and that
the diagram {\rm (\ref{eqXZfunlm})} is cartesian.
Then the morphism
$D_ZX\to D_WY$
is a closed immersion.

{\rm 2.}
Assume that $X\to Y$
and $Z\to W$ are smooth.
Then the morphism
$D_ZX\to D_WY$
is smooth.
\end{lm}

\proof{
1.
We may assume that $Y={\rm Spec}\, B$
and $X={\rm Spec}\, A$
are affine and
$Z\subseteq X$ and
$W\subseteq Y$ are defined by ideals
$I\subseteq A$ and $J\subseteq B$.
By the assumption that the diagram
{\rm (\ref{eqXZfunlm})} is cartesian,
the surjection $B\to A$ induces a surjection $J\to I$.
Hence the induced morphism
$B[J/u]\to A[I/u]$ is a surjection.

2.
By the assumption
that $X\to Y$ is smooth,
the base change of 
the morphism
$D_ZX\to D_WY$
by ${\mathbf G}_m\to {\mathbf A}^1$
is smooth.
Further by the assumption
that $Z\to W$ is smooth,
the morphism
$T_{Z_0}X_0\to T_{W_0}Y_0$
is smooth.
Hence the morphism
$D_ZX\to D_WY$
of smooth schemes over
${\mathbf A}^1$ is smooth.
\qed

}
\medskip

We apply the construction above
to define the deformation to normal bundle.

\begin{df}\label{dfAZX}
Let $k$ be a field
and let $i\colon Z\to X$ be a closed immersion
of smooth schemes over $k$.
We define the deformation $A_ZX$
to the normal bundle
$T_ZX$ to be the dilatation $D_{Z\times {\mathbf A}^1}
(X\times {\mathbf A}^1)$
for the closed immersion
$Z\times {\mathbf A}^1\to
X\times {\mathbf A}^1$
of smooth schemes over the affine line
${\mathbf A}^1={\rm Spec}\, k[u]$. 
\end{df}

More explicitly, if the closed immersion $i$ is defined by the coherent ideal sheaf $I\subseteq \mathcal O_X$, then
\begin{equation}\label{eq:def-of-deformation-by-Rees}
A_ZX=\mathrm{Spec}_X(\mathcal{R}_{I/\mathcal{O}_X}),
\end{equation}
where $\mathcal{R}_{I/\mathcal{O}_X}$ is the extended Rees algebra
\begin{equation}
\mathcal{R}_{I/\mathcal{O}_X}=\mathcal{O}_X[u,Iu^{-1}]=\bigoplus\limits_{m\in\mathbf{Z}}I^mu^{-m}\subseteq\mathcal{O}_X[u,u^{-1}],
\end{equation}
with the convention that $I^m=\mathcal{O}_X$ for $m\leq0$.

By Lemma~\ref{lmXZsm}.3,
the element $u$ defines a smooth morphism
\begin{equation}\label{eq:structure-map-of-deformation-space}
A_ZX\to\mathbf{A}^1,
\end{equation}
which fits in a cartesian diagram
\begin{equation}
	\begin{CD}
		T_ZX@>>>
		A_ZX
		@<<<
		X\times {\mathbf G}_m\\
		@VVV@VVV@VVV\\
		0@>>>{\mathbf A}^1@<<<
		{\mathbf G}_m.
	\end{CD}
\end{equation}
Together with the natural projection $A_ZX\to X$, the morphism \eqref{eq:structure-map-of-deformation-space} gives
\begin{equation}
\pi_{Z/X}:A_ZX\to X\times\mathbf{A}^1.
\end{equation}
We consider the diagonal immersion
$\delta\colon \Delta_X=X\to X\times X$.
Then, the normal bundle
$T_X(X\times X)$
is the tangent bundle $TX$.
Hence, we have a cartesian diagram
\begin{equation}
	\begin{CD}
		TX@>>>
		A_X(X\times X)
		@<<<
		X\times X\times {\mathbf G}_m\\
		@VVV@VVV@VVV\\
		0@>>>{\mathbf A}^1@<<<
		{\mathbf G}_m.
	\end{CD}
\end{equation}

\begin{lemma}\label{lem:Saito-lemma-on-deformation-spaces}
Let
\begin{equation}\label{eq:morphism-of-regular-immersions}
\begin{tikzcd}
	Z\arrow[r,hook,"i"]\arrow[d,"g"]&X\arrow[d,"f"]\\
	W\arrow[r,hook,"j"]&Y
\end{tikzcd}
\end{equation}
be a commutative diagram of smooth schemes such that the horizontal arrows are closed immersions.

{\rm 1.} 
There is a natural morphism
\begin{equation}\label{eq:Dgf}
	A_g(f):A_ZX\to A_WY,
\end{equation}
which fits into a commutative diagram
\begin{equation}\label{eq:morphismofdeformationspacesoverA}
	\begin{tikzcd}
		A_ZX\arrow[r,"{A_g(f)}"]\arrow[d,"{\pi_{Z/X}}"]&A_WY\arrow[d,"{\pi_{W/Y}}"]\\
		X\times\mathbf{A}^1\arrow[r,"{f\times\mathrm{id}}"]&Y\times\mathbf{A}^1
	\end{tikzcd}
\end{equation}

{\rm 2. }
If $f$ and $g$ are smooth, then $A_g(f)$ is smooth. In particular, if $f$ is smooth, then the morphism
$A_X(X\times X)
\to
A_Y(Y\times Y)$
is smooth.

{\rm 3. }
If $f$ and $g$ are closed immersions and if \eqref{eq:morphism-of-regular-immersions} is cartesian, then $A_g(f)$ is a closed immersion. In particular, if $f$ is a closed immersion, then  the morphism
$A_X(X\times X)
\to
A_Y(Y\times Y)$
is a closed immersion.

{\rm 4. }
If \eqref{eq:morphism-of-regular-immersions} is transversal, i.e., it is cartesian and $\mathcal{O}_Z\simeq\mathcal{O}_X\otimes^L_{\mathcal{O}_Y}\mathcal{O}_W$, then \eqref{eq:morphismofdeformationspacesoverA} is cartesian.

\end{lemma}
\begin{proof}
Consider the diagram
\begin{equation}
\begin{tikzcd}
	Z\times \mathbf A^1\arrow[r,hook]\arrow[d]&X\times \mathbf A^1\arrow[d]\\
	W\times \mathbf A^1\arrow[r,hook]&Y\times \mathbf A^1.
\end{tikzcd}
\end{equation}
The assertions follow by applying \eqref{eqXZsim} and Lemma~\ref{lmXZfun} to this diagram.
\end{proof}

Consider the diagram \eqref{eq:morphism-of-regular-immersions}. Its induced morphism $A_g(f)$ in \eqref{eq:Dgf}  is defined over $\mathbf{A}^1$. Over $\mathbf G_m$, the morphism \eqref{eq:Dgf} is identified with $f\times{\rm id}: X\times \mathbf G_m\to Y\times\mathbf G_m$:
\begin{equation}
\begin{tikzcd}[column sep=40]
A_Z(X)\times_{\mathbf{A}^1}\mathbf{G}_m\arrow[r,"A_g(f)\times\mathrm{id}"]\arrow[d,equals]&A_W(Y)\times_{\mathbf{A}^1}\mathbf{G}_m\arrow[d,equals]\\
X\times\mathbf{G}_m\arrow[r,"f\times\mathrm{id}"]&Y\times\mathbf{G}_m.
\end{tikzcd}
\end{equation}
On the special fiber, the morphism \eqref{eq:Dgf} is the natural map
$T_ZX\to T_WY$
between normal bundles:
\begin{equation}
\begin{tikzcd}[column sep=40]
A_Z(X)\times_{\mathbf{A}^1}\{0\}\arrow[r,"A_g(f)\times\mathrm{id}"]\arrow[d,equals]&A_W(Y)\times_{\mathbf{A}^1}\{0\}\arrow[d,equals]\\
T_ZX\arrow[r,"T_gf"]&T_WY.
\end{tikzcd}
\end{equation}

\begin{remark}
Let $i:Z\hookrightarrow X$ be a closed immersion between smooth schemes. 
The structure morphism $\pi_{Z/X}:A_Z(X)\to X\times\mathbf{A}^1$ can be described by the following diagram:
\begin{equation}
\begin{tikzcd}
A_Z(X)\arrow[r,"A_i(\mathrm{id})"]\arrow[d,"\pi_{Z/X}"']&A_X(X)\arrow[d,"\pi_{X/X}","\simeq"swap]\\
X\times\mathbf{A}^1\arrow[r,equal]&X\times\mathbf{A}^1.
\end{tikzcd}
\end{equation}

\end{remark}

\begin{proposition}\label{prop:product-of-deformation-spaces}
Let $i:Z\hookrightarrow X$ and $j:W\hookrightarrow Y$ be closed immersions between smooth schemes over $k$. We have a cartesian diagram
\begin{align}
\begin{aligned}
\xymatrix{
	A_{Z\times W}(X\times Y)\ar[r]\ar[d]&A_Z(X)\ar[d]\\
	A_W(Y)\ar[r]&\mathbf{A}^1.
}
\end{aligned}
\end{align}
\end{proposition}
\begin{proof}
We work locally on $X$ and on $Y$, and then we may assume that $X=\mathrm{Spec}(A)$ and $Y=\mathrm{Spec}(B)$ are affine.
Let $I\subseteq A$ be the ideal defining $Z$ in $X$, and let $J\subseteq B$ be the ideal defining $W$ in $Y$. Then 
\begin{equation}
K=I\otimes_kB+A\otimes_kJ
\end{equation}
is the ideal defining $Z\times W$ in $X\times Y$. It suffices to prove that
\begin{equation}\label{eq:tensor-product-of-Rees-algebras}
\mathcal{R}_{K/A\otimes_kB}\simeq\mathcal{R}_{I/A}\otimes_{k[u]}\mathcal{R}_{J/B}.
\end{equation}
By definition, both sides of $\eqref{eq:tensor-product-of-Rees-algebras}$ are subrings of
\begin{equation}
A\otimes_kB[u,u^{-1}]\simeq A[u,u^{-1}]\otimes_{k[u]}B[u,u^{-1}]
\end{equation}
generated by $A\otimes_kB[u],(I\otimes_kB)u^{-1}$ and $(A\otimes_kJ)u^{-1}$.
\end{proof}

\begin{corollary}\label{cor:pullback-of-deformation-spaces-in-triangle}
Let $j:Z\hookrightarrow X$ and $i:X\hookrightarrow Y$ be closed immersions between smooth schemes, and let $f:Y\to X$ be a smooth morphism such that $f\circ i=\mathrm{id}$.
Then the following diagram is cartesian:
\begin{equation}\label{eq:pullback-of-deformation-spaces-in-triangle}
\begin{tikzcd}[column sep=5em, row sep=3em] 
	A_Z(Y)\arrow[d,"{A_j(\mathrm{id})}"]\arrow[r,"{A_{\mathrm{id}}(f)}"]&A_Z(X)\arrow[d,"\pi_{Z/X}"]\\
	A_X(Y)\arrow[r,"(f\times {\rm id})\circ\pi_{X/Y}"]&X\times\mathbf{A}^1.
\end{tikzcd}
\end{equation}
\end{corollary}
\begin{proof}
Since the formation of deformation spaces is compatible with \'etale base change, we may work \'etale locally on $X$ and $Y$.
Using the smoothness of \(f\) together with the section \(i\), we may therefore reduce to the following case by \cite[Theorem 2.5.8]{FuLei}
\begin{align}
X=X\times \{w\}\hookrightarrow Y=X\times W \xrightarrow{f={\rm pr}_X}X
\end{align}
where $W$ is a smooth scheme over $k$ and $w\in W$ is a $k$-rational point. 
By Proposition \ref{prop:product-of-deformation-spaces}, we obtain cartesian diagrams
\begin{align}
\begin{aligned}
\xymatrix{
A_Z(Y)= A_{Z\times\{w\}}(X\times W)\ar[r]\ar[d]&A_{\{w\}}(W)\ar[d]& A_{X\times\{w\}}(X\times W)= A_X(Y)\ar[l]\ar[d]\\
A_Z(X)\ar[r]&\mathbf{A}^1&A_X(X)= X\times\mathbf{A}^1\ar[l],
}
\end{aligned}
\end{align}
and the assertion follows.
\end{proof}
The following construction will be used to define an invariant structure on specialization.
\begin{construction}\label{cons:deformation-space-action-notations}
Let $f:X\to Y$ be a smooth morphism, and let $i:Z\hookrightarrow X$ be a regular immersion between smooth $k$-schemes.
Let $\delta_f: X\hookrightarrow X\times_YX$ be the diagonal morphism.
Applying Corollary~\ref{cor:pullback-of-deformation-spaces-in-triangle} to the following diagram
\begin{align}
\begin{aligned}
	\xymatrix{
		Z\ar@{^(->}[r]^-i&X\ar@{^(->}[r]^-{\delta_f}\ar@{=}[rd]&X\times_YX\ar[d]^-{{\rm pr}_{f,1}}\ar[r]^-{{\rm pr}_{f,2}}&X\ar[d]^-f\\
		&&X\ar[r]_-f&Y,
	}
\end{aligned}
\end{align}
we obtain a cartesian diagram
\begin{equation}\label{eq:dzxyxddzx}
\begin{tikzcd}
	A_Z(X\times_YX)\arrow[d,"{A_i(\mathrm{id})}"]\arrow[r,"{A_{\mathrm{id}}(\mathrm{pr}_{f,1})}"]&A_Z(X)\arrow[d,"\pi_{Z/X}"]\\
	A_X(X\times_YX)\arrow[r,"\pi_{\delta_f}"]&X\times\mathbf{A}^1,
\end{tikzcd}
\end{equation}
where $\pi_{\delta_f}=({\rm pr}_{f,1}\times{\rm id})\circ \pi_{X/X\times_YX}$. 
The second projection $\mathrm{pr}_{f,2}:X\times_YX\to X$ induces the following action:
\begin{equation}\label{eq:actionofDXDeltaonD}
\tilde{\alpha}_{Z/X/Y}:A_X(X\times_YX)\times_{X\times\mathbf{A}^1}A_Z(X)\xleftarrow[\simeq]{(A_i(\mathrm{id}),A_{\mathrm{id}}(\mathrm{pr}_{f,1}))}A_Z(X\times_YX)\xrightarrow{A_{\mathrm{id}}(\mathrm{pr}_{f,2})}A_Z(X).
\end{equation}
Unwinding the definitions, on the generic fiber, \eqref{eq:actionofDXDeltaonD} is given by
\begin{equation}
\mathrm{pr}_{f,2}:X\times_YX\to X,
\end{equation}
whereas on the special fiber, \eqref{eq:actionofDXDeltaonD} induces an action
\begin{equation}\label{eq:actionofTXYZonNZX}
\alpha_{Z/X/Y}:T(X/Y)|_Z\times_ZT_ZX=T(X/Y)\times_XT_ZX\to T_ZX.
\end{equation}
Similarly, the morphism induced by $\mathrm{pr}_{f,1}$ is the projection onto the $A_Z(X)$-factor:
\begin{equation}\label{eq:projectionofactionofDXDeltaonD}
\mathrm{pr}_{A_Z(X)}:A_X(X\times_YX)\times_{X\times\mathbf{A}^1}A_Z(X)\xleftarrow[\simeq]{(A_i(\mathrm{id}),A_{\mathrm{id}}(\mathrm{pr}_{f,1}))}A_Z(X\times_YX)\xrightarrow{A_{\mathrm{id}}(\mathrm{pr}_{f,1})}A_Z(X).
\end{equation}
On the generic fiber, the morphism \eqref{eq:projectionofactionofDXDeltaonD} is given by
\begin{equation}
\mathrm{pr}_{f,1}:X\times_YX\to X.
\end{equation}
On the special fiber, \eqref{eq:projectionofactionofDXDeltaonD} is the projection
\begin{equation}
\mathrm{pr}_{T_{Z}X}:T({X/Y})|_Z\times_ZT_ZX=T({X/Y})\times_XT_ZX\to T_ZX,
\end{equation}
where $T({X/Y})$ is the relative tangent bundle of $X\to Y$ and  is isomorphic to the normal bundle of $\delta_f: X\to X\times_YX$.
\end{construction}

\subsection{The specialization functor}\label{sssp-1}

We briefly recall the definition and functorialities
of the nearby cycles functor.
Let ${\cal O}_K$ be a henselian discrete valuation
ring.
Let $S^{\rm ur}={\rm Spec}\, {\cal O}_{K^{\rm ur}}$
be the maximal unramified extension of
$S={\rm Spec}\, {\cal O}_K$.
Let 
$s,\eta\in S$ be the closed point and
the generic point.
Let
$\bar s\in S^{\rm ur}$ be the closed point
and
$\bar \eta
={\rm Spec}\,
K_{\rm sep}
\to S^{\rm ur}$ be a geometric point
above $\eta$.
Let $\Lambda$ be a finite ring 
and assume
that the cardinality is invertible
in ${\cal O}_K$.
For a scheme $X$ of finite type over $S$,
let
\begin{equation}
\begin{CD}
	X_{\bar s}@>{\bar i}>> X\times_SS^{\rm ur}
	@<{\bar j}<<X_{\bar \eta}
	@>j>>X_\eta\\
	@VVV@VVV@VVV@VVV\\
	{\bar s}@>>> S^{\rm ur}@<<<\bar\eta
	@>>>\eta
\end{CD}
\end{equation}
be a cartesian diagram.
Let $D^b_c(X_s\times_s\eta,\Lambda)$
be the derived category
of constructible sheaves
on $X_{\bar s}$ with 
continuous action of ${\rm Gal}(\bar\eta/\eta)$
compatible with the action on $X_{\bar s}$
as in \cite[Expos\'e XIII,
Construction 1.2.4 (b)]{SGA7II}.
Then the nearby cycles functor
\begin{equation}
	\Psi\colon D^b_c(X_\eta,\Lambda)
	\to D^b_c(X_s\times_s\eta,\Lambda)
\end{equation}
is defined by $\Psi=\bar i^*\bar j_*j^*$.
Its composition $D^b_c(X,\Lambda)
\to D^b_c(X_s\times_s\eta,\Lambda)$
with the pull-back
is also denoted by $\Psi$.

Let $f\colon X\to Y$ be a morphism
of schemes separated of finite type over $S$.
Then the base change morphisms define morphisms
of functors
\begin{align}
\alpha\colon f_s^* \Psi_Y &\to \Psi_Xf_{ \eta}^*,
\label{eqbcpsi1}
\\
\Psi_Yf_{ \eta*} &\to  f_{s*} \Psi_X,
\label{eqbcpsi2}
\\
\gamma\colon f_{s!}\Psi_X&\to \Psi_Yf_{ \eta!}. 
\label{eqbcpsi}
\end{align}
The morphisms (\ref{eqbcpsi1})
and (\ref{eqbcpsi2}) are adjoint to each other.
If $f$ is proper,
(\ref{eqbcpsi2}) and (\ref{eqbcpsi}) are 
isomorphisms 
inverse to each other
by the proper base change theorem.
If $f$ is smooth, (\ref{eqbcpsi1}) is an isomorphism
by the smooth base change theorem.
A canonical morphism
\begin{equation}
\beta\colon \Psi_X f_{ \eta}^!\to f_s^!\Psi_Y
\label{eqbcpsi4}
\end{equation}
is defined as the adjoint of (\ref{eqbcpsi}).

Let us describe $\alpha$ and $\beta$ more explicitly.
Consider the cartesian diagram
\begin{align}
	\begin{aligned}
		\begin{CD}
			X_{\bar s}@>{\bar i_X}>> X\times_SS^{\rm ur}
			@<{\bar j_X}<<X_{\bar \eta}
			@>{j_X}>>X_\eta\\
			@V{f_s}VV@VV{\widetilde f}V@VV{f_{\bar \eta}}V@VV{f_\eta}V\\
			Y_{\bar s}@>{\bar  i_Y}>> Y\times_SS^{\rm ur}
			@<{\bar  j_Y}<<Y_{\bar \eta}
			@>{j_Y}>>Y_\eta.
		\end{CD}
	\end{aligned}
\end{align}
Let $\mathcal G\in D_{\rm ctf}(Y_\eta,\Lambda)$.
The map $\alpha(\mathcal G): f_s^* \Psi_Y( \mathcal G)\to \Psi_X(f_{ \eta}^*\mathcal G)$ is obtained by
restricting the base change map
\begin{equation}
	\widetilde{f}^\ast \bar j_{Y\ast}j_Y^\ast\mathcal G\to \bar j_{X\ast} f_{\bar\eta}^\ast j_Y^\ast\mathcal G
\end{equation}
to the special fiber. The map $\beta(\mathcal G):  \Psi_X (f_{ \eta}^!\mathcal G)\to f_s^!\Psi_Y(\mathcal G)$ is adjoint to 
\begin{align}\label{eq:defbetaadj}
	f_{s!}\Psi_X (f_{ \eta}^!\mathcal G)\xrightarrow{\gamma(f_\eta^!\mathcal G)}\Psi_Y(f_{ \eta!}f_{ \eta}^!\mathcal G)\xrightarrow{\Psi_Y(\varepsilon_{f_\eta,\mathcal G})}\Psi_Y(\mathcal G),
\end{align}
where $\varepsilon_{f_\eta}:=\varepsilon_{f_\eta,\mathcal G}: f_{ \eta!}f_\eta^!\mathcal G\to \mathcal G$ is the counit.

The morphisms $\alpha$ \eqref{eqbcpsi1} and $\gamma$ \eqref{eqbcpsi} are compatible with the isomorphism of projection formula as follows.
\begin{lemma}\label{lem:eqbcpsiPF}
Let $f\colon X\to Y$ be a morphism
of separated schemes  of finite type over $S={\rm Spec}\,\mathcal O_K$.
Let $\mathcal F\in D_{\rm ctf}(X_\eta,\Lambda)$ and $\mathcal G\in D_{\rm ctf}(Y_\eta,\Lambda)$. Then the  diagram 
\begin{align}\label{eq:eqbcpsiPF}
\begin{aligned}
\xymatrix{
f_{s!}\Psi_X(\mathcal F)\otimes\Psi_Y(\mathcal G)\ar[r]_-\simeq^-{{\rm PF}_{f_s}}\ar[d]^-{\eqref{eqbcpsi}}_-{\gamma(\mathcal F)\otimes 1}&f_{s!}(\Psi_X(\mathcal F)\otimes f_s^\ast \Psi_Y(\mathcal G))\ar[d]_-{\eqref{eqbcpsi1}}^-{f_{s!}(1\otimes\alpha(\mathcal G))}\\
\Psi_Y(f_{\eta!}\mathcal F)\otimes \Psi_Y(\mathcal G)\ar[d]^-m
&f_{s!}(\Psi_X(\mathcal F)\otimes \Psi_X(f^\ast_\eta{\mathcal G}))\ar[d]^-m\\
\Psi_Y(f_{\eta!}\mathcal F\otimes\mathcal G)\ar[d]_-{\Psi_Y({\rm PF}_{f_{\eta}})}^-\simeq& f_{s!}\Psi_X({\mathcal F}\otimes f_\eta^\ast\mathcal G)\ar[d]^-{\gamma(\mathcal F\otimes f_\eta^\ast\mathcal G)}_-{\eqref{eqbcpsi}}\\
\Psi_Y(f_{\eta!}(\mathcal F\otimes f_\eta^\ast\mathcal G))\ar@{=}[r]&
\Psi_Y(f_{\eta!}(\mathcal F\otimes f_\eta^\ast\mathcal G))
}
\end{aligned}
\end{align}
is commutative, where $m$ is the canonical tensor map for the nearby cycles functor, and
${\rm PF}_{f_s}$ and ${\rm PF}_{f_\eta}$ are the isomorphisms defined by the projection formula for the morphisms $f_s$ and $f_\eta$ respectively.
\end{lemma}
\begin{proof}
We denote the counits
$
\bar 	j_X^*\bar j_{X*}\xrightarrow{\varepsilon} {\rm id}$ and 
$\bar  j_Y^*\bar j_{Y*}\xrightarrow{\varepsilon} {\rm id}$
 both by $\varepsilon$. Let
 $a\colon
 \widetilde f^*\bar j_{Y*}
 \longrightarrow
 \bar 	j_{X*}f_{\bar\eta}^*$ be the base change morphism, which 
 is adjoint to
 \begin{align}
 	\bar 	j_X^*\widetilde f^*\bar j_{Y*}
 	\simeq
 	f_{\bar\eta}^*\bar j_Y^*\bar j_{Y*}
 	\xrightarrow{\,f_{\bar\eta}^*\varepsilon\,}
 	f_{\bar\eta}^*.
 \end{align}
Let 
$c\colon
\widetilde f_!\bar j_{X*}
\longrightarrow
\bar j_{Y*}f_{\bar\eta!}$
be the morphism adjoint to
\begin{align}
\begin{aligned}
	b\colon
	\bar 	j_Y^*\widetilde f_!\bar j_{X*}
	&\xrightarrow[\simeq]{\mathrm{BC}}
	f_{\bar\eta!}\bar j_X^*\bar j_{X*}
	\xrightarrow{\,f_{\bar\eta!}\varepsilon\,}
	f_{\bar\eta!}.
\end{aligned}
\end{align}
	For $\bar j=\bar j_X$ or $\bar j_Y$, let
$
\mu\colon
\bar 	j_*A\otimes\bar  j_*B
\longrightarrow
\bar 	j_*(A\otimes B)$
be the canonical tensor morphism, which is adjoint to
\begin{align}
	\bar 	j^*(\bar j_*A\otimes \bar j_*B)
	\simeq
	\bar 	j^*\bar j_*A\otimes\bar  j^*\bar j_*B
	\xrightarrow{\varepsilon\otimes\varepsilon}
	A\otimes B.
\end{align}
By the definitions of $a$ and $\mu$, we have commutative diagrams
\begin{equation}\label{eq:coma-mu}
\begin{aligned}
\xymatrix{
f_{\bar\eta}^*\bar j_Y^*\bar j_{Y*}\ar[d]_-\simeq\ar[r]^-{f_{\bar \eta}^\ast\varepsilon}&f_{\bar \eta}^\ast&
\bar j_X^\ast(\bar j_{X\ast}A\otimes \bar j_{X\ast}B)\ar[d]_-\simeq\ar[r]^-{\bar j_X^\ast \mu}& \bar j_X^\ast \bar j_{X\ast}(A\otimes B)\ar[d]^-\varepsilon\\
\bar j_X^*\widetilde f^*\bar j_{Y*}\ar[r]_-{\bar j_X^\ast a}&\bar j^\ast_X\bar j_{X\ast}f^\ast_{\bar \eta}\ar[u]_-\varepsilon&\bar j_X^\ast \bar j_{X\ast} A\otimes \bar j_X^\ast \bar j_{X\ast} B\ar[r]_-{\varepsilon\otimes \varepsilon}&A\otimes B.
}
\end{aligned}
\end{equation}

Let $\mathcal F'=\mathcal F|_{X_{\bar \eta}}$ and $\mathcal G'=\mathcal G|_{Y_{\bar \eta}}$.
Recall that $\Psi_X(\mathcal F)
	=\bar i_X^*\bar  j_{X*}\mathcal F'$,
$\Psi_Y(\mathcal G)
	=\bar i_Y^*\bar  j_{Y*}\mathcal G'$ and
$\bar i_Y^*\widetilde f_!
	\simeq f_{s!}\bar i_X^*$. 
Under these identifications, \eqref{eq:eqbcpsiPF} is obtained
by applying $\bar i_Y^*$ to the diagram
\eqref{eq:lifted-nearby-cycle-diagram} below, using the compatibility
of the projection formula with restriction to closed subschemes.
Thus, it suffices to prove that the following diagram commutes.
\begin{equation}\label{eq:lifted-nearby-cycle-diagram}
	\begin{tikzcd}[column sep=large,row sep=large]
		\widetilde f_!\bar j_{X*}\mathcal F'
		\otimes \bar j_{Y*}\mathcal G'
		\arrow[r,"\mathrm{PF}"]
		\arrow[d,"c\otimes 1"']
		&
		\widetilde f_!\bigl(
		\bar j_{X*}\mathcal F'
		\otimes\widetilde f^*\bar j_{Y*}\mathcal G'
		\bigr)
		\arrow[d,"\widetilde f_!(1\otimes a)"]
		\\
		\bar j_{Y*}f_{\bar\eta!}\mathcal F'
		\otimes \bar j_{Y*}\mathcal G'
		\arrow[d,"\mu"']
		&
		\widetilde f_!\bigl(
		\bar j_{X*}\mathcal F'
		\otimes \bar j_{X*}f_{\bar\eta}^*\mathcal G'
		\bigr)
		\arrow[d,"\widetilde f_!\mu"]
		\\
		\bar j_{Y*}\bigl(f_{\bar\eta!}\mathcal F'\otimes\mathcal G'\bigr)
		\arrow[d,"\bar j_{Y*}(\mathrm{PF})"']
		&
		\widetilde f_!\bar j_{X*}
		\bigl(\mathcal F'\otimes f_{\bar\eta}^*\mathcal G'\bigr)
		\arrow[d,"c"]
		\\
		\bar j_{Y*}f_{\bar\eta!}
		\bigl(\mathcal F'\otimes f_{\bar\eta}^*\mathcal G'\bigr)
		\arrow[r,equal]
		&
		\bar j_{Y*}f_{\bar\eta!}
		\bigl(\mathcal F'\otimes f_{\bar\eta}^*\mathcal G'\bigr).
	\end{tikzcd}
\end{equation}
By the adjunction $\bar j_Y^*\dashv\bar j_{Y*}$,
it suffices to show that the two composites in
\eqref{eq:lifted-nearby-cycle-diagram} have the same adjoint.

By the definitions of $c$ and $\mu$, 
the adjoint of the composite
along the left-hand path is
\begin{equation}\label{eq:left-adjoint}
	\begin{aligned}
		\bar 	j_Y^*\bigl(
		\widetilde f_!\bar j_{X*}\mathcal F'
		\otimes\bar  j_{Y*}\mathcal G'
		\bigr)&\simeq
		\bar 	j_Y^*\widetilde f_!\bar j_{X*}\mathcal F'
		\otimes \bar j_Y^*\bar j_{Y*}\mathcal G'
		\\
		&\xrightarrow{\,
			b\otimes\varepsilon
			\,}
		f_{\bar\eta!}\mathcal F'\otimes\mathcal G'\xrightarrow{\ \mathrm{PF}\ }
		f_{\bar\eta!}
		\bigl(\mathcal F'\otimes f_{\bar\eta}^*\mathcal G'\bigr).
	\end{aligned}
\end{equation}
By \eqref{eq:coma-mu}, the adjoint of the composite
along the top and right-hand edges is
\begin{equation}
\begin{aligned}
	\bar 	j_Y^*\bigl(
	\widetilde f_!\bar j_{X*}\mathcal F'
	\otimes\bar  j_{Y*}\mathcal G'
	\bigr)&\xrightarrow{\rm PF}\bar j_Y^\ast \widetilde f_!\bigl(
	\bar j_{X*}\mathcal F'
	\otimes\widetilde f^*\bar j_{Y*}\mathcal G'
	\bigr)\\
	&\simeq
	f_{\bar\eta!}\left(
	\bar 	j_X^*\bar j_{X*}\mathcal F'
	\otimes
	f_{\bar\eta}^*\bar j_Y^*\bar j_{Y*}\mathcal G'
	\right)\\
	&\xrightarrow{
		f_{\bar\eta!}\left(
		\varepsilon
		\otimes f_{\bar\eta}^*\varepsilon
		\right)}
	f_{\bar\eta!}
	\bigl(\mathcal F'\otimes f_{\bar\eta}^*\mathcal G'\bigr).
\end{aligned}
\end{equation}
By the compatibility of the projection formula with restriction
to open subschemes, this composite coincides with the composite
in \eqref{eq:left-adjoint}.
\end{proof}

We recall a canonical construction (cf. \cite[(8.3)]{CC}). Let $f:X\to Y$ be a separated morphism of schemes, and let ${\mathcal F},\mathcal G\in D_{\mathrm{ctf}}(Y,\Lambda)$. We have a relative purity morphism
\begin{equation}\label{eq:defcfF}
	c_f=c_{f,{\mathcal F},\mathcal G}:f^*{\mathcal F}\otimes f^!\mathcal G\to f^!({\mathcal F}\otimes\mathcal G),
\end{equation}
which is obtained by adjunction from the following composition
\begin{equation}\label{eq:relative-purity-morphism}
	f_!(f^*{\mathcal F}\otimes f^!\mathcal G)\simeq {\mathcal F}\otimes f_!f^!\mathcal G\to {\mathcal F}\otimes\mathcal G.
\end{equation}
By the definition of $c_{f,\mathcal F,\mathcal G}$, we have a
commutative diagram
\begin{align}\label{eq:equivdeficifg}
	\begin{aligned}
		\begin{CD}
			f_!(f^*\mathcal{F}\otimes f^!\mathcal G)
			@>{f_!c_{f,\mathcal{F},\mathcal G}}>>
			f_!f^!(\mathcal{F}\otimes\mathcal G)\\
			@V{\mathrm{PF}_f}V{\simeq}V
			@VV{\varepsilon_{f,\mathcal{F}\otimes\mathcal G}}V\\
			\mathcal{F}\otimes f_!f^!\mathcal G
			@>{1\otimes\varepsilon_{f,\mathcal G}}>>
			\mathcal{F}\otimes\mathcal G,
		\end{CD}
	\end{aligned}
\end{align}
where  $\mathrm{PF}_f$ is the isomorphism  defined by the projection formula for the morphism $f$.
We put $c_{f,\mathcal F}:=c_{f,\mathcal F,\Lambda}$.
We say that $f$ is ${\mathcal F}$-transversal if $c_{f,\mathcal F}$ is an isomorphism (cf. \cite[Definition 8.5]{CC}).

\begin{lemma}\label{lem:PsiTrans}
Let $f\colon X\to Y$ be a morphism
of separated schemes  of finite type over $S={\rm Spec}\,\mathcal O_K$.
For any $\mathcal F\in D_{\rm ctf}(Y_\eta,\Lambda)$ and $\mathcal G\in D_{\rm ctf}(Y_\eta,\Lambda)$, we have a commutative diagram in $D_{\rm ctf}(X_s\times_s\eta,\Lambda)$
\begin{equation}\label{eq:specialization-transversality-general}
	\begin{tikzcd}[column sep=40]
		\Psi_X(f_\eta^\ast {\mathcal F})\otimes \Psi_X(f_\eta^!\mathcal G)\arrow[r,"m"]&\Psi_X(f_\eta^\ast {\mathcal F}\otimes f_\eta^!\mathcal G)\arrow[r,"{\Psi_X(c_{f_\eta})}"]&\Psi_X(f_\eta^!({\mathcal F}\otimes\mathcal G))\arrow[dd,"\beta({\mathcal F\otimes\mathcal G})","{\eqref{eqbcpsi4}}"']\\
		f_s^*\Psi_Y({\mathcal F})\otimes \Psi_X(f_\eta^!\mathcal G)\arrow[u,"\alpha({\mathcal F})\otimes 1","{\eqref{eqbcpsi1}}"']\arrow[d,"1\otimes \beta({\mathcal G})"',"{\eqref{eqbcpsi4}}"]&&\\
		f_s^*\Psi_Y({\mathcal F})\otimes  f_s^!\Psi_Y(\mathcal G)\arrow[r,"c_{f_s}"]&f_s^!(\Psi_Y({\mathcal F})\otimes  \Psi_Y(\mathcal G))\arrow[r,"f_s^!m"]&f_s^!\Psi_Y({\mathcal F}\otimes\mathcal G)
	\end{tikzcd}
\end{equation}
where  $m$ is the canonical tensor map for the nearby cycles functor.
\end{lemma}
\begin{proof}
	We show the following equality
	\begin{align}\label{eq:equalityforcom-}
		\beta({\mathcal{F}\otimes\mathcal G})
		\circ \Psi_X(c_{f_\eta})\circ m
		\circ(\alpha({\mathcal{F}})\otimes1)
		=
		f_s^!m\circ c_{f_s}
		\circ(1\otimes\beta({\mathcal G}))
	\end{align}
	by comparing the adjoints under	$f_{s!}\dashv f_s^!$.
	
	To compute the adjoint of the left hand side,
	we consider the diagram
	\begin{align}\label{eq:help4-}
		\begin{aligned}
			\xymatrix@C=1.1cm{
				f_{s!}\bigl(f_s^*\Psi_Y({\mathcal F})\otimes \Psi_X(f_\eta^!\mathcal G)\bigr)
				\ar[r]^-{f_{s!}(\alpha({\mathcal{F}})\otimes1)}
				\ar[dd]_{\mathrm{PF}_{f_s}}^{\simeq}\ar@{}|{(3)}[rddd]&
				f_{s!}(\Psi_X(f_\eta^\ast {\mathcal F})\otimes \Psi_X(f_\eta^!\mathcal G))\ar[d]^-{f_{s!}m}&\\
				&f_{s!}\Psi_X(f_\eta^\ast {\mathcal F}\otimes f_\eta^!\mathcal G)\ar[d]_-{\gamma_{f_\eta^\ast\mathcal F\otimes f_\eta^!\mathcal G}}\ar[r]^-{f_{s!}\Psi_X(c_{f_\eta})}\ar@{}|{(1)}[rd]&f_{s!}\Psi_X(f_\eta^!({\mathcal F}\otimes\mathcal G))\ar[d]^-{\gamma_{f_\eta^!(\mathcal F\otimes\mathcal G)}}\\
				\Psi_Y(\mathcal{F})\otimes f_{s!} \Psi_X(f_\eta^!\mathcal G)
				\ar[d]_{1\otimes\gamma_{f_\eta^!\mathcal G}}&
				\Psi_Y(f_{\eta!}(f_\eta^\ast\mathcal F\otimes f_\eta^!\mathcal G))\ar[r]^-{\Psi_Y(f_{\eta!}c_{f_\eta})}\ar[d]_-{\Psi_Y({\rm PF}_{f_\eta})}^-\simeq\ar@{}|{(2)}[rd]&
				\Psi_Y(f_{\eta!}f_\eta^!(\mathcal F\otimes\mathcal G))\ar[d]^-{\Psi_Y(\varepsilon_{f_\eta})}\\
				\Psi_Y(\mathcal{F})\otimes\Psi_Y( f_{\eta!}f_\eta^!\mathcal G)
				\ar[r]^m&
				\Psi_Y(\mathcal F\otimes f_{\eta!}f_\eta^!\mathcal G)\ar[r]^-{\Psi_Y(1\otimes\varepsilon_{f_\eta})}& 
				\Psi_Y(\mathcal F\otimes\mathcal G).
			}
		\end{aligned}
	\end{align}
	The composition of
	the right column is the adjoint of
	$\beta({\mathcal{F}\otimes\mathcal G})$
	by \eqref{eq:defbetaadj}.
	Hence 
	the adjoint of the left hand side
	$\beta({\mathcal{F}\otimes\mathcal G})
	\circ \Psi_X(c_{f_\eta})\circ m
	\circ(\alpha({\mathcal{F}})\otimes1)$
	of the equality \eqref{eq:equalityforcom-}
	equals the composition via upper right.
	By naturality of the morphism $\gamma$ \eqref{eqbcpsi} and the commutative diagram \eqref{eq:equivdeficifg}, 
	the  squares (1) and (2)
	are commutative respectively.
	By Lemma \ref{lem:eqbcpsiPF} applied to the morphism $f$, 
	$f_\eta^!{\cal G}\in D_{\rm ctf}(X_\eta,\Lambda)$
	and
	$\mathcal  F\in D_{\rm ctf}(Y_\eta,\Lambda)$,
	the left rectangle (3) is commutative. 
	Thus the diagram \eqref{eq:help4-} is
	commutative
	and the adjoint of the left side
	of the equality \eqref{eq:equalityforcom-}
	equals the composition via the lower left.

	To compute the adjoint of the right hand side,
	we consider the diagram
\begin{align}\label{eq:secondh1-}
	\begin{gathered}
		\xymatrix{
			f_{s!}({f_s^\ast} \Psi_Y(\mathcal F)\otimes \Psi_X( f_\eta^!\mathcal G))\ar[r]^-{f_{s!}(1\otimes\beta({\mathcal G}))}\ar[d]_-{{\rm PF}_{f_s}}&f_{s!}({f_s^*}\Psi_Y({\mathcal F})\otimes f_s^!\Psi_Y(\mathcal G))\ar[r]^-{f_{s!}c_{f_s}}\ar[d]^-{{\rm PF}_{f_s}}&f_{s!}f_s^!(\Psi_Y({\mathcal F})\otimes \Psi_Y(\mathcal G))\ar[d]^-{\varepsilon_{f_s}}\\
			\Psi_Y(\mathcal F)\otimes f_{s!}\Psi_X(f_\eta^!\mathcal G)\ar[r]^-{1\otimes f_{s!}\beta({\mathcal G})}
			\ar[d]_-{1\otimes \gamma_{f_\eta^!\mathcal G}}&
			\Psi_Y(\mathcal F)\otimes f_{s!}f_s^!\Psi_Y(\mathcal G)\ar[r]^-{1\otimes \varepsilon_{f_s}}&\Psi_Y(\mathcal{F})\otimes\Psi_Y(\mathcal G)\ar[d]^m\\
			\Psi_Y(\mathcal F)\otimes \Psi_Y(f_{\eta!}f_\eta^!\mathcal G)
			\ar[rru]_-{1\otimes \Psi_Y(\varepsilon_{f_\eta})}
			\ar[r]^m&
			\Psi_Y(\mathcal{F}\otimes f_{\eta!}f_\eta^!\mathcal G)
			\ar[r]^{\Psi_Y(1\otimes\varepsilon_{f_\eta})}&
			\Psi_Y(\mathcal{F}\otimes\mathcal G).
		}
	\end{gathered}
\end{align}
	The adjoint of the right hand side
	$f_s^!m\circ c_{f_s}
	\circ(1\otimes\beta({\mathcal G}))$
	of the equality \eqref{eq:equalityforcom-}
	equals the composition through the upper right.
	By \eqref{eq:equivdeficifg}, the upper right square is commutative.
	By the functoriality of the isomorphism of projection formula,
	the upper left square is commutative.
	By naturality of the tensor map and its compatibility with the
	unit, the lower triangle is commutative.
	The definition of $\beta({\mathcal G})$ 
	as the adjoint of
	\eqref{eq:defbetaadj}
	implies that the upper triangle (without $\Psi_Y({\cal F}$))
	is commutative.
	Thus the diagram \eqref{eq:secondh1-} is
	commutative
	and the adjoint of the right side
	of the equality \eqref{eq:equalityforcom-}
	equals the composition via the lower left.
	Thus, we get the equality \eqref{eq:equalityforcom-}.
\end{proof}
We recall the definition of monodromic sheaves
from \cite[D\'efinition 3.1]{Verdier}.

\begin{df}\label{dfmono}
Let $X$ be a noetherian scheme over
a field $k$ and $E$ be a vector bundle over $X$.
For a geometric point
$x$ of $E$,
define a morphism $w_x\colon {\mathbf G}_{m,x}\to E$
by $w_x(t)=t\cdot x$.

We say that ${\cal F}\in D^b_c(E,\Lambda)$
is monodromic if for every geometric point
$x$ of $E$ and for every integer $q\in {\mathbf Z}$,
the cohomology sheaf
${\cal H}^q(w_x^*{\cal F})$ on
${\mathbf G}_{m,x}$ is locally constant
and tamely ramified
along $0,\infty$.
\end{df}

We study a property of 
a locally constant sheaf on ${\mathbf G}_m$
tamely ramified
along $0,\infty$.
Recall that $k$ is a field of characteristic $p>0$
and $\Lambda$ is a finite ring
such that the cardinality
is invertible in $k$.

\begin{lm}\label{lmjM}
Let ${\cal M}$ be a locally constant constructible sheaf
on
${\mathbf G}_m={\rm Spec}\, k[x^{\pm 1}]$
tamely ramified
along $0,\infty\in {\mathbf P}^1$
and let $j\colon {\mathbf G}_m
\to {\mathbf A}^1$ be the open immersion.
Let $\overline k$ be a separable closure of $k$.

{\rm 1.}
The canonical morphism
$\Gamma({\mathbf G}_{m,\overline k},{\cal M})
\to (j_*{\cal M})_{\overline 0}$
is an isomorphism.

{\rm 2.}
$\Gamma_c({\mathbf A}^1_{\overline k},j_*{\cal M})=0$.
\end{lm}

\proof{
1.
Let $I_0^{\rm t}$ be the tame inertia group at $0$
and let $M$ be the stalk of ${\cal M}$
at the generic geometric point of ${\mathbf G}_m$.
Then, the canonical morphism
$I_0^{\rm t}\to \pi_1({\mathbf G}_{m,\overline k})^{\rm t}$
to the tame fundamental group is an isomorphism
and we obtain a commutative diagram
\begin{equation}
\begin{CD}
	\Gamma({\mathbf G}_{m,\overline k},{\cal M})
	@>>>
	j_*{\cal M}_{\overline 0}\\
	@AAA@AAA\\
	\Gamma(\pi_1({\mathbf G}_{m,\overline k})^{\rm t},M)
	@>>>
	\Gamma(I_0^{\rm t},M)
\end{CD}
\end{equation}
of isomorphisms.

2.
Let $j_\infty\colon {\mathbf A}^1\to
{\mathbf P}^1$ be the open immersion.
Then, we have a distinguished triangle
$\Gamma_c({\mathbf A}^1_{\overline k},j_*{\cal M})
\to 
\Gamma({\mathbf P}^1_{\overline k},j_{\infty*}j_*{\cal M})
\to 
j_{\infty*}j_*{\cal M}|_{\overline \infty}\to$.
The second morphism is identified with
$\Gamma({\mathbf G}_{m,\overline k},{\cal M})
\to 
j_{\infty*}j_*{\cal M}|_{\overline \infty}$
and 
is an isomorphism
similarly as in 1.
Hence we have
$\Gamma_c({\mathbf A}^1_{\overline k},j_*{\cal M})
=0$.
\qed

}

\medskip

We recall the definition of
the specialization functor
from \cite{Verdier}.
Let $i\colon Z\to X$ be a closed immersion
of smooth schemes over a field $k$.
We consider the cartesian diagram
\begin{equation}
\begin{CD}
T_ZX@>>>
A_ZX
@<<<
X\times {\mathbf G}_m\\
@VVV@VVV@VVV\\
0@>>>{\mathbf A}^1@<<<
{\mathbf G}_m.
\end{CD}
\label{eqAZX}
\end{equation}
Let $\eta_0$ be the generic point
of the henselization of ${\mathbf A}^1$
at $0$, and let $\bar k$ be a separable closure of
 $k$.
Let $\Psi$ be the nearby cycles functor
with respect to the middle vertical arrow
in (\ref{eqAZX})
at $0\in {\mathbf A}^1$
and let $p_1\colon 
X\times {\mathbf G}_m
\to X$ denote the projection.
We define the specialization functor
\begin{equation}
\nu_i=\nu_{Z/X}\colon 
D^b_c(X,\Lambda)
\to
D^b_c(T_ZX
\times_k\eta_0,\Lambda)
\label{eqnZ}
\end{equation}
by
$\nu_{Z/X}=\Psi\circ p_1^*$. 
We adopt the following convention throughout the paper.
\begin{convention}\label{con:delignetopoi}
For a morphism $r:U\to V$ of schemes over $k$, 
we also write $r$ for the induced morphism of topoi
\begin{align}\label{eq:sitesMorphism}
	r: U\times_k\eta_0\to V\times_k\eta_0.
\end{align}
For $\mathcal F\in D_c^b(U,\Lambda)$, we also denote by
$\mathcal F$ its pull-back along the canonical projection
$U\times_k\eta_0\to U$.
Its underlying object in $D_c^b(U_{\bar k}, \Lambda)$ is
$\mathcal F|_{U_{\bar k}}$, equipped with the semilinear
$G_{\eta_0}$-action induced by its natural $G_k$-action
via the quotient $G_{\eta_0}\to G_k$.
\end{convention}
\begin{lemma}\label{lem:specialization-in-ULA-case}
Consider a cartesian diagram of smooth $k$-schemes
\begin{equation}
\begin{tikzcd}
	Z\arrow[r,"i",hook]\arrow[d,"g"]&X\arrow[d,"f"]\\
	W\arrow[r,"j",hook]&Y,
\end{tikzcd}
\end{equation}
where $i$ and $j$ are closed immersions.
Let $p:T_ZX\to Z$ be the projection map, and let ${\mathcal F}\in D_{c}^b(X,\Lambda)$. Assume that $f$ and $g$ are smooth and that $f$ is universally locally acyclic relatively to ${\mathcal F}$. Then there is a canonical isomorphism
\begin{equation}
\nu_{Z/X}({\mathcal F})\simeq  p^* i^*{\mathcal F} \qquad{\rm in}\qquad D_c^b(T_ZX\times_k\eta_0,\Lambda),
\end{equation}
where the right-hand side is interpreted according to
Convention~\ref{con:delignetopoi}.
\end{lemma}
\begin{proof}
By Lemma \ref{lem:Saito-lemma-on-deformation-spaces}, we have a cartesian diagram
\begin{equation}
\begin{tikzcd}
	A_Z(X)\arrow[d]\arrow[r,"A_g(f)"]&A_W(Y)\arrow[d]\\
	X\arrow[r,"f"]&Y.
\end{tikzcd}
\end{equation}
Then $A_g(f)$ is universally locally acyclic with respect to ${\mathcal F}|_{A_Z(X)}$. Since $A_W(Y)\to\mathbf{A}^1$ is smooth, it follows from \cite[Th. finitude, Lemme 2.14]{TF} (see also \cite[Corollaire 2.7]{App}) that the morphism
\begin{equation}
A_Z(X)\xrightarrow{A_g(f)}A_W(Y)\to\mathbf{A}^1
\end{equation}
is universally locally acyclic with respect to ${\mathcal F}|_{A_Z(X)}$.
Applying the nearby cycles functor yields a natural isomorphism
$\nu_{Z/X}({\mathcal F})\simeq {\mathcal F}|_{T_ZX}\simeq p^* i^*{\mathcal F}$.
\end{proof}

Let \begin{equation}
\begin{CD}
Z@>>>X\\
@VgVV@VVfV\\
W@>>>Y
\end{CD}
\label{eqXZfunb}
\end{equation}
be a commutative diagram of smooth schemes
over $k$
such that the horizontal arrows are closed immersions.
Then, we have a commutative
diagram
\begin{equation}
\begin{CD}
T_ZX@>>>A_ZX@<<< X\times {\mathbf G}_m\\
@V{T_f}VV@V{A_f=A_g(f)}VV@VVV\\
T_WY@>>>A_WY@<<<Y\times {\mathbf G}_m
\end{CD}
\end{equation}
and  the morphisms
(\ref{eqbcpsi1}),
(\ref{eqbcpsi2}),
(\ref{eqbcpsi}),
(\ref{eqbcpsi4}) define
morphisms
\begin{align}
\alpha\colon  T_f^* \nu_{W/Y} &\to \nu_{Z/X}f^*,
\label{eqnufun1}
\\
\nu_{W/Y}f_*  &\to T_{f*}\nu_{Z/X},
\label{eqnufun2}
\\
\gamma\colon T_{f!} \nu_{Z/X} &\to \nu_{W/Y}f_!,
\label{eqnufun3}
\\
\beta\colon\,\, \nu_{Z/X}f^!  &\to T_f^!\nu_{W/Y}.
\label{eqnufun4}
\end{align}

\begin{pr}\label{prverdier}
Let $i\colon Z\to X$ be a closed immersion
of smooth schemes over $k$
and ${\cal F}\in D^b_c(X,\Lambda)$.

{\rm 1 (\cite[Section 8 (SP1)]{Verdier}).}
The specialization $\nu_{Z/X}{\cal F}$
on $T_ZX$ is monodromic.

{\rm 2 (cf.~\cite[Section 8 (SP5)]{Verdier}).}
The morphisms
$ 0^*\nu_{Z/X}{\cal F}\to  i^*{\cal F}$
{\rm (\ref{eqnufun1})}
and
$ i^!{\cal F}\to  0^!\nu_{Z/X}{\cal F}$
{\rm (\ref{eqnufun4})}
for the immersion $A_ZZ=Z\times {\mathbf A}^1\to A_ZX$
are isomorphisms.
\end{pr}

\proof{

2.
The first isomorphism is
\cite[Section 8 (SP5)]{Verdier}.
The following proof of the second isomorphism
is similar to the proof
for the first isomorphism loc.~cit.
It suffices to show that
$0^!\nu_{Z/X}{\cal F}=0$
assuming that
$i^!{\cal F}=0$.
Let $j\colon U=X\sm Z\to X$
be the open immersion of the complement.
Then, the assumption $i^!{\cal F}=0$
means that the canonical morphism
${\cal F}\to j_*j^*{\cal F}$ is an isomorphism.
By blow-up,
we may assume that $Z\subseteq X$
is a Cartier divisor of $X$
and that there exists a cartesian
diagram
$$\begin{CD}
Z@>>>X\\
@VVV@VVV\\
0@>>>{\mathbf A}^1
\end{CD}$$
where the vertical arrows are smooth.
We prove the assertion by induction
on the dimension $d$ of the support of ${\cal F}$.
If $d=0$, then ${\cal F}=0$
on a neighborhood of $Z$ and the assertion holds.
Assume $d\geqq 1$. 

First, we
show $0^!\nu_{Z/X}{\cal F}=0$
generically on $Z$.
By replacing $X$ by an open neighborhood
of the generic point of $Z$,
we may assume that
$j^*{\cal F}$ is locally constant.
By further replacing $X$ by
a finite scheme over $X$ and
shrinking $X$
and using the isomorphism
(\ref{eqnufun2}) for finite morphisms,
we may assume that the cohomology sheaves of
$j^*{\cal F}$ are constant
and further $j^*{\cal F}={\cal H}^0j^*{\cal F}[0]$.
Let $j_1\colon U\times {\mathbf G}_m\to 
A_ZX \sm A_ZZ$
be the open immersion.
Then, since $A_ZX \sm A_ZZ\to {\mathbf A}^1$
is smooth, the direct image
$R^0j_{1*}{\rm pr}_1^*j^*{\cal F}$ is a constant sheaf.
Let $j_2\colon 
A_ZX \sm A_ZZ\to
A_ZX$
be the open immersion
and let
$j_{2,0}\colon
T_ZX\sm Z
\to T_ZX$
be its restriction on the closed fiber.
Further, since
$A_ZX$ and $A_ZZ$ are smooth over ${\mathbf A}^1$,
we have $\nu_{Z/X}{\cal F}=
(j_{2*}R^0j_{1*}{\rm pr}_1^*j^*{\cal F})|_{T_ZX}
=j_{2,0*}(
(R^0j_{1*}{\rm pr}_1^*j^*{\cal F})|_{T_ZX\ssm Z})$.
Thus we have
$0^!\nu_{Z/X}{\cal F}=0$
generically on $Z$.

We
show $0^!\nu_{Z/X}{\cal F}=0$
on the whole $Z$.
Since the assertion is local on $X$,
we may take an \'etale morphism
$X\to {\mathbf A}^n$
such that
$Z\subseteq X$ is the inverse image of
${\mathbf A}^{n-1}\subseteq {\mathbf A}^n$.
By considering the normalization of 
${\mathbf A}^n$ in $X$
and the compatibility with finite direct image,
we may assume $Z={\mathbf A}^{n-1}
\subseteq X={\mathbf A}^n$.
Since the support of
$0^!\nu_{Z/X}{\cal F}$ is of dimension
$<n-1$,
there exists a linear projection
$X={\mathbf A}^n\to X'={\mathbf A}^{n-1}$
such that
$Z={\mathbf A}^{n-1}\to Z'={\mathbf A}^{n-2}$
is finite on the support $S$ of
$0^!\nu_{Z/X}{\cal F}$.
Let $\overline X=X'\times {\mathbf P}^1\to X'$
and $\overline q\colon \overline Z=Z'\times {\mathbf P}^1\to Z'$
be projective completions
and let $\overline {\cal F}=j_{X*}{\cal F}$
for the open immersion $j_X\colon X\to \overline X$.
Then, we have $0^!\nu_{Z/X}{\cal F}=
(0^!\nu_{\overline Z/\overline X}\overline {\cal F})|_Z$
and  $\overline q\colon \overline Z\to Z'$
is finite on the union $S\cup(\overline Z\sm Z)$ containing
the support of
$0^!\nu_{\overline Z/\overline X}\overline {\cal F}$.
By the induction hypothesis and the compatibility
with finite direct images,
we have
$\overline q_*0^!\nu_{\overline Z/\overline X}\overline {\cal F}=0$
and hence
$0^!\nu_{\overline Z/\overline X}\overline {\cal F}=0$.
Thus, we obtain
$0^!\nu_{Z/X}{\cal F}=0$.
\qed

}

\medskip

Let $i:Z\hookrightarrow X$ be a closed immersion of smooth schemes.
We denote by
\begin{equation}
	0:Z\hookrightarrow T_ZX
\end{equation}
the zero section.
Consider the map
\begin{equation}\label{eq:zero-section-of-deformation-spaces}
	\widetilde{i}=	A_{\mathrm{id}}(i):A_Z(Z)=Z\times \mathbf A^1\hookrightarrow A_Z(X).
\end{equation}
Let $S$ be the strict henselization of $\mathbf A^1$ at $0$
and let $\bar s$ and $\bar \eta$ be the closed
point and a geometric generic point of $S$ respectively.
The  morphism
$\widetilde{i}=A_{\mathrm{id}}(i)$ gives the diagram
\begin{align}
	\begin{gathered}
		\begin{CD}
			Z_{\bar s} @>{s_Z}>> Z\times S @<{j_Z}<< Z_{\bar\eta}\\
			@V{ 0}VV @VV{\widetilde{i}}V @VV{i_{\bar \eta}}V\\
			T_{Z_s}X_{\bar s} @>{s_X}>> A_Z(X)_S @<{j_X}<< X_{\bar\eta}.
		\end{CD}
	\end{gathered}
\end{align}
Here $j_X$ and $j_Z$ are the geometric generic fiber maps.

Let ${\mathcal F}\in D_{\mathrm{ctf}}(X,\Lambda)$
and $\mathcal{G}\in D_{\rm ctf}(Z,\Lambda)$. 
Applying  \eqref{eqnufun1}, 
and \eqref{eqnufun4} to  the map $\widetilde{i}=	A_{\mathrm{id}}(i)$, we get natural morphisms
\begin{equation}\label{eq:zero-section-of-nu}
	\alpha({\mathcal F})\colon	 0^*\nu_{Z/X}({\mathcal F})\xrightarrow{\eqref{eqnufun1}}  i^*{\mathcal F},\quad 
	\beta({\mathcal F})\colon
	 i^!{\mathcal F}\xrightarrow{\eqref{eqnufun4}}  0^!\nu_{Z/X}({\mathcal F}),
\end{equation}
which are isomorphisms by Proposition~\ref{prverdier}. 
Also by applying  
\eqref{eqnufun3},
we get an isomorphism
\begin{align}\label{eq:defgamma}
	\begin{aligned}
		\gamma({\mathcal{G}}): 0_!\mathcal{G}
		&\rightarrow\nu_{Z/X}(i_!\mathcal{G}).
	\end{aligned}
\end{align}
Taking $\mathcal F=\Lambda$, we get isomorphisms
\begin{align}\label{eq:zero-section-of-nu-forLambda}
	 i^!\Lambda\xrightarrow[\simeq]{\beta(\Lambda)}  0^!\nu_{Z/X}(\Lambda)\simeq  0^!\Lambda.
\end{align}

\begin{proposition}\label{prop:specialization-detect-transversality}
Let $i:Z\hookrightarrow X$ be a closed immersion of smooth $k$-schemes, and let ${\mathcal F}\in D_{\mathrm{ctf}}(X,\Lambda)$ and ${\mathcal G}\in D_{\mathrm{ctf}}(X,\Lambda)$. 
Let $\eta_0$ be the generic point of the henselization of
$\mathbf A^1$ at $0$, and write $0: Z\hookrightarrow T_ZX$ for the zero section. 
Using the notation of Convention~\ref{con:delignetopoi},
we have a commutative diagram in $D_{\mathrm{ctf}}(Z\times_k \eta_0,\Lambda)$
\begin{equation}\label{eq:specialization-detect-transversality}
	\begin{tikzcd}[column sep=35]
		 i^*{{\mathcal F}}\otimes i^!{\mathcal G}\arrow[rr,"c_{ i,{{\mathcal F}},{\mathcal G}}"]&& i^!({\mathcal F}\otimes{\mathcal G})\arrow[dd,"\beta_{\mathcal F\otimes\mathcal G}","\simeq"']\\
		 0^*\nu_{Z/X}({\mathcal F})\otimes  i^!{\mathcal G}\arrow[u,"\alpha_{\mathcal F}\otimes 1","\simeq"swap]\arrow[d,"1\otimes \beta_{\mathcal G}","\simeq"swap]&&\\
		 0^*\nu_{Z/X}({\mathcal F})\otimes  0^!\nu_{Z/X}(\mathcal G)\arrow[r,"c_{ 0}"]& 0^!(\nu_{Z/X}({\mathcal F})\otimes \nu_{Z/X}(\mathcal G))\arrow[r," 0^!m"]& 0^!\nu_{Z/X}({\mathcal F}\otimes\mathcal G),
	\end{tikzcd}
\end{equation}
where $m$ is the canonical tensor map for the nearby cycles functor.

In particular, setting  $\mathcal G=\Lambda$, we obtain  that the morphism $i$ is ${\mathcal F}$-transversal if and only if $0$ is $\nu_{Z/X}({\mathcal F})$-transversal.
\end{proposition}

\begin{proof}
This follows from Lemma \ref{lem:PsiTrans}, applied to the morphism
$\widetilde{i}=	A_{\mathrm{id}}(i):A_Z(Z)=Z\times \mathbf A^1\hookrightarrow A_Z(X)$.
\end{proof}

Now we construct an invariant structure on specializations.

\begin{lemma}\label{lem:invariant-structure-on-specialization}
Let $f:X\to Y$ be a smooth morphism, and let $i:Z\hookrightarrow X$ be a regular immersion between smooth $k$-schemes.
For ${\mathcal K}\in D_{\rm ctf}(Y,\Lambda)$, the complex $\nu_{Z/X}(f^*{\mathcal K})\in D_{\rm ctf}(T_ZX\times_k\eta_0,\Lambda)$ is $T(X/Y)|_Z$-invariant for the action of $T(X/Y)|_Z$ on $T_ZX$ defined by \eqref{eq:actionofTXYZonNZX}.
\end{lemma}
\begin{proof}
Since $f\circ\mathrm{pr}_{f,1}=f\circ\mathrm{pr}_{f,2}$, there is an isomorphism
\begin{equation}\label{eq:twopullbacksoffKareequal}
\mathrm{pr}_{f,1}^*f^*{\mathcal K}\simeq\mathrm{pr}_{f,2}^*f^*{\mathcal K}.
\end{equation}
By Lemma~\ref{lem:Saito-lemma-on-deformation-spaces}.2, the morphisms \eqref{eq:actionofDXDeltaonD} and \eqref{eq:projectionofactionofDXDeltaonD} are smooth  over $\mathbf{A}^1$. Applying \eqref{eqbcpsi1} to these two morphisms yields natural isomorphisms
\begin{align}
\nu_{Z/X\times_YX}\mathrm{pr}_{f,2}^*f^*{\mathcal K}\simeq \alpha_{Z/X/Y}^*\nu_{Z/X}f^*{\mathcal K},\label{eq:comparespecializationactionside}\\
\nu_{Z/X\times_YX}\mathrm{pr}_{f,1}^*f^*{\mathcal K}\simeq{\mathrm{pr}}_{T_{Z}X}^*\nu_{Z/X}f^*{\mathcal K}.\label{eq:comparespecializationprojectionside}
\end{align}
Combining \eqref{eq:twopullbacksoffKareequal}, \eqref{eq:comparespecializationactionside}, and \eqref{eq:comparespecializationprojectionside}, we obtain an isomorphism
\begin{equation}
\alpha_{Z/X/Y}^*\nu_{Z/X}f^*{\mathcal K}\simeq{\mathrm{pr}}_{T_{Z}X}^*\nu_{Z/X}f^*{\mathcal K}\quad\text{in }D_{\rm ctf}((T(X/Y)\times_XT_ZX)\times_k\eta_0,\Lambda).
\end{equation}
This proves the assertion.
\end{proof}

We recall the internal Hom object $\mathscr H=\mathcal Hom$ for cohomological correspondences (cf. \cite{LuZheng}) and define its specialization version.
\begin{definition}
Let $X$ be a smooth scheme over a field $k$.
We identify the normal bundle
$T_X(X\times X)$
for the diagonal $\delta=\delta_X\colon X\to X\times X$
with the tangent bundle
$TX$.
Let $\eta_0$ denote the generic point
of the henselization of ${\mathbf A}^1$
at $0$.
We define bifunctors
\begin{align}
&\mathscr{H}=\mathscr{H}_X\colon 
D_{\rm ctf}(X,\Lambda)^{\rm op}
\times D_{\rm ctf}(X,\Lambda)
\to
D_{\rm ctf}(X\times X,\Lambda),\label{eq:defHom}\\
&\nu{\cal H}om
\colon 
D_{\rm ctf}(X,\Lambda)^{\rm op}
\times D_{\rm ctf}(X,\Lambda)
\to
D_{\rm ctf}(TX\times_k\eta_0,\Lambda),
\label{eqnHom}
\end{align}
by $\mathscr{H}_X({\mathcal F}_1,{\mathcal F}_2)=\mathcal{H}om_{X\times X}(\mathrm{pr}_1^*{\mathcal F}_1,\mathrm{pr}_2^!{\mathcal F}_2)$ and
$\nu{\cal H}om=\nu_{X/X\times X}(\mathscr{H}_X)$
(\ref{eqnZ}).
If ${\mathcal F}_1={\mathcal F}_2={\mathcal F}$, we also write
\begin{align}\label{eq:mathscrHF}
\mathscr{H}_{\mathcal F}=\mathscr{H}_X({\mathcal F},{\mathcal F}).
\end{align}
\end{definition}

Applying the construction of \ref{cons:deformation-space-action-notations} to $i:Z\hookrightarrow X$ and $X\to\mathrm{Spec}(k)$, we obtain
\begin{align}
&\tilde{\alpha}:A_X(X\times X)\times_{X\times\mathbf{A}^1}A_Z(X)\to A_Z(X),\\
&\alpha:TX|_Z\times_ZT_ZX=TX\times_XT_ZX\to T_ZX.
\end{align}
We denote by $q:TX|_Z\to T_ZX$ the quotient bundle morphism.
Unwinding the definitions, we see that for a geometric point $z\in Z$, the action
\begin{equation}
\alpha_z:T_zX\times(T_ZX)_z\to(T_ZX)_z
\end{equation}
is given as follows:
\begin{equation}
\alpha_z=\mathrm{add}\circ(q_z\times\mathrm{id}):T_zX\times T_zX/T_zZ\to T_zX/T_zZ.
\end{equation}
Consider the following commutative diagram
\begin{align}\label{eq:strange}
\begin{aligned}
\xymatrix{
	&A_Z(X)\ar@{=}@/_10pc/[dd]\ar[r]^-{A_{i}({\rm id})}\ar[d]^-{A_{\rm id}(\delta_X)}&A_X(X)\ar[d]_-{A_{\rm id}(\delta_X)}\ar[d]\ar@{=}@/^5pc/[dd]\\
A_Z(X)	&A_Z(X\times X)\ar[l]^-{A_{\rm id}({\rm pr}_2)}\ar[r]_-{A_{i}({\rm id})}\ar[d]^-{A_{\rm id}({\rm pr}_1)}&A_X(X\times X)\ar[d]\\
	&A_Z(X)\ar[r]& A_X(X)=X\times \mathbf A^1.
}
\end{aligned}
\end{align}
The bottom square of \eqref{eq:strange} is cartesian by Corollary \ref{cor:pullback-of-deformation-spaces-in-triangle}. Hence the upper square  of \eqref{eq:strange} is also cartesian, which gives the following commutative diagram over the special fiber (note that $T_Z(X\times X)\simeq TX\times_X T_ZX$ by \eqref{eq:dzxyxddzx})
\begin{align}\label{eq:strange-special}
\begin{aligned}
\xymatrix{
	&T_ZX\ar[d]_-{(0_{TX},{\rm id})}\ar[r]& X\ar[d]^-{0_{TX}}\\
T_ZX	&TX\times_X T_ZX\ar[l]_-{\alpha}\ar[r]_-{{\rm pr}_{TX}}\ar[d]_-{{\rm pr}_{T_ZX}}&TX\ar[d]\\
	&T_ZX\ar[r]&X,
}
\end{aligned}
\end{align}	
where the squares are cartesian.
\begin{proposition}\label{prop:retractpropfornu}
Let $i:Z\hookrightarrow X$ be a closed immersion of smooth schemes.
For all ${\mathcal F}\in D_{\mathrm{ctf}}(X,\Lambda)$, the complex
$\nu_{Z/X}({\mathcal F})\in D_{\mathrm{ctf}}(T_ZX\times_k\eta_0,\Lambda)$
is a retract of
\begin{equation}\label{eq:a-sheaf-whose-retract-is-nu_Z/X}
\begin{aligned}
	&\alpha_!({\mathrm{pr}}_{TX}^*\nu\mathcal{H}om({\mathcal F},{\mathcal F})\otimes{\mathrm{pr}}_{T_ZX}^*\nu_{Z/X}({\mathcal F}))\\
	\simeq &\alpha_!({\mathrm{pr}}_{TX|_Z}^*(\nu\mathcal{H}om({\mathcal F},{\mathcal F})|_{TX|_Z})\otimes{\mathrm{pr}}_{T_ZX}^*\nu_{Z/X}({\mathcal F})).
\end{aligned}
\end{equation}
The retraction maps are given  by the following two compositions:
\begin{equation}\label{eq:defRetract1}
\begin{aligned}
\nu_{Z/X}(\mathcal{F})&\simeq \alpha_! {(0_{TX},\mathrm{id})}_! {(0_{TX},\mathrm{id})}^* {\mathrm{pr}}_{T_ZX}^*\nu_{Z/X}(\mathcal{F})\\
&\simeq \alpha_!( {(0_{TX},\mathrm{id})}_!\Lambda\otimes {\mathrm{pr}}_{T_ZX}^*\nu_{Z/X}(\mathcal{F}))\\
&\simeq \alpha_!( {\mathrm{pr}}_{TX}^*  0_{TX!}\nu_{X/X}(\Lambda)\otimes {\mathrm{pr}}_{T_ZX}^*\nu_{Z/X}(\mathcal{F}))\\
&\xrightarrow{\eqref{eqnufun3}} \alpha_!( {\mathrm{pr}}_{TX}^*\nu_{X/X\times X}(\delta_{X!}\Lambda)\otimes {\mathrm{pr}}_{T_ZX}^*\nu_{Z/X}(\mathcal{F}))\\
&\xrightarrow{\rm coev} \alpha_!( {\mathrm{pr}}_{TX}^*\nu_{X/X\times X}(\mathscr{H}_{\mathcal F})\otimes {\mathrm{pr}}_{T_ZX}^*\nu_{Z/X}(\mathcal{F}))\\
&= \alpha_!( {\mathrm{pr}}_{TX}^*\nu\mathcal{H}om(\mathcal{F},\mathcal{F})\otimes {\mathrm{pr}}_{T_ZX}^*\nu_{Z/X}(\mathcal{F})),
\end{aligned}
\end{equation}
\begin{equation}\label{eq:defRetract2}
\begin{aligned}
& \alpha_!( {\mathrm{pr}}_{TX}^*\nu\mathcal{H}om(\mathcal{F},\mathcal{F})\otimes {\mathrm{pr}}_{T_ZX}^*\nu_{Z/X}(\mathcal{F}))\\
&= \alpha_!( {\mathrm{pr}}_{TX}^*\nu_{X/X\times X}(\mathscr{H}_{\mathcal F})\otimes {\mathrm{pr}}_{T_ZX}^*\nu_{Z/X}(\mathcal{F}))\\
&\xrightarrow{\eqref{eqnufun1}} \alpha_!(\nu_{Z/X\times X}(\mathscr{H}_{\mathcal F})\otimes\nu_{Z/X\times X}(\mathrm{pr}_1^*\mathcal{F}))\\
&\xrightarrow{\quad m\quad} \alpha_!\nu_{Z/X\times X}(\mathscr{H}_{\mathcal F}\otimes\mathrm{pr}_1^*\mathcal{F})\\
&\xrightarrow{\eqref{eqnufun3}}\nu_{Z/X}(\mathrm{pr}_{2!}(\mathscr{H}_{\mathcal F}\otimes\mathrm{pr}_1^*\mathcal{F}))\\
&\xrightarrow{\rm ev}\nu_{Z/X}(\mathrm{pr}_{2!}\mathrm{pr}_2^!\mathcal{F})\\
&\xrightarrow{\rm counit}\nu_{Z/X}(\mathcal{F}),
\end{aligned}
\end{equation}	
where $m$ is induced by the canonical tensor map for the nearby cycles functor.
\end{proposition}
\begin{proof}
Consider the morphisms $X\xhookrightarrow{\delta_X}X\times X\xrightarrow{{\rm pr}_2}X$.		
We first note that the morphisms (cf. \eqref{eq1})
\begin{align}
{\delta}_{X!}\Lambda\otimes {\rm pr}_1^\ast \mathcal F\xrightarrow{{\rm coev}\otimes {\rm id}}\mathscr H_{\mathcal F}\otimes {\rm pr}_1^\ast \mathcal F\xrightarrow{{\rm ev}}{\rm pr}_2^!\mathcal F
\end{align}
induce a retraction
\begin{align}\label{eq:retractionbyev}
\mathcal F\simeq {\rm pr}_{2!}({\delta}_{X!}\Lambda\otimes {\rm pr}_1^\ast \mathcal F)\to {\rm pr}_{2!}(\mathscr H_{\mathcal F}\otimes {\rm pr}_1^\ast \mathcal F)\to \mathcal F.
\end{align}
By naturality, we have a commutative diagram
\begin{align}\label{eq:retractionbyev2}
\begin{aligned}
\xymatrix{
\alpha_!({\mathrm{pr}}_{TX}^\ast \nu_{X/X\times X}(\delta_{X!}\Lambda)\otimes {\mathrm{pr}}^\ast_{T_ZX}\nu_{Z/X}(\mathcal F))\ar[d]_-{\eqref{eqnufun1}}\ar[r]^-{\rm coev}&\alpha_!({\mathrm{pr}}_{TX}^\ast \nu_{X/X\times X}(\mathscr H_{\mathcal F})\otimes {\mathrm{pr}}^\ast_{T_ZX}\nu_{Z/X}(\mathcal F))\ar[d]^-{\eqref{eqnufun1}}\\
\alpha_!(\nu_{Z/X\times X}(\delta_{X!}\Lambda)\otimes\nu_{Z/X\times X}(\mathrm{pr}_1^*\mathcal{F}))\ar[d]_-m\ar[r]^-{\rm coev}&\alpha_!(\nu_{Z/X\times X}(\mathscr{H}_{\mathcal F})\otimes\nu_{Z/X\times X}(\mathrm{pr}_1^*\mathcal{F}))\ar[d]^-m\\
\alpha_!\nu_{Z/X\times X}(\delta_{X!}\Lambda\otimes\mathrm{pr}_1^*\mathcal{F})\ar[d]_-{\eqref{eqnufun3}}\ar[r]^-{\rm coev}&\alpha_!\nu_{Z/X\times X}(\mathscr{H}_{\mathcal F}\otimes\mathrm{pr}_1^*\mathcal{F})\ar[d]^-{\eqref{eqnufun3}}\\
\nu_{Z/X}(\mathrm{pr}_{2!}(\delta_{X!}\Lambda\otimes\mathrm{pr}_1^*\mathcal{F}))\ar[r]^	-{\rm coev}&\nu_{Z/X}(\mathrm{pr}_{2!}(\mathscr{H}_{\mathcal F}\otimes\mathrm{pr}_1^*\mathcal{F})).
}
\end{aligned}
\end{align}
To prove Proposition \ref{prop:retractpropfornu}, it suffices to show that the following diagram commutes:
\begin{align}\label{eq:finaldiagram3}
\begin{aligned}
\xymatrix{
	\nu_{Z/X}(\mathcal F)\ar[d]_-\simeq\ar@{=}[r]\ar@{}|{(1)}[rdd]&\nu_{Z/X}(\mathcal F)\ar[dd]^-\simeq \ar@{=}[dddr]\\
	\alpha_!({\mathrm{pr}}_{TX}^\ast  0_{TX!}\Lambda\otimes {\mathrm{pr}}^\ast_{T_ZX}\nu_{Z/X}(\mathcal F))\ar[d]^-\simeq&	\\
	\alpha_!({\mathrm{pr}}_{TX}^\ast \nu_{X/X\times X}(\delta_{X!}\Lambda)\otimes {\mathrm{pr}}^\ast_{T_ZX}\nu_{Z/X}(\mathcal F))\ar[r]\ar[d]\ar@{}|{(2)}[rd]& \nu_{Z/X}({\rm pr}_{2!}(\delta_{X!}\Lambda\otimes{\rm pr}_1^\ast\mathcal F))\ar[d]&\\
	\alpha_!({\mathrm{pr}}_{TX}^\ast \nu_{X/X\times X}(\mathscr H_{\mathcal F})\otimes {\mathrm{pr}}^\ast_{T_ZX}\nu_{Z/X}(\mathcal F))\ar[r]& \nu_{Z/X}({\rm pr}_{2!}(\mathscr H_{\mathcal F}\otimes{\rm pr}_1^\ast\mathcal F))\ar[r]\ar@{}|{(3)}[ur]&\nu_{Z/X}(\mathcal F),
}
\end{aligned}
\end{align}
where (2) is given by the commutative diagram \eqref{eq:retractionbyev2}.
Commutativity of (3) follows from \eqref{eq:retractionbyev}.
It remains to verify that the diagram (1) commutes. 

Applying Lemma \ref{lem:eqbcpsiPF} to
the morphism $A_{\rm id}(\delta_X):A_Z(X)\to A_Z(X\times X)$ and $\Lambda\in D_{\rm ctf}(X,\Lambda)$, ${\rm pr}_1^\ast\mathcal F\in D_{\rm ctf}(X\times X,\Lambda)$, we obtain the following commutative diagram (note that $\nu_{X/X\times X}(\delta_{X!}\Lambda)\simeq 0_{TX!}\Lambda$ and $\nu_{Z/X}(\Lambda)\simeq \Lambda$)
\begin{align}\label{eq:finaldiagram1}
\begin{aligned}
\xymatrix{
{\mathrm{pr}}_{TX}^{\ast}\nu_{X/X\times X}(\delta_{X!}\Lambda)
\otimes
{\mathrm{pr}}_{T_ZX}^{\ast}\nu_{Z/X}(\mathcal{F})\ar[d]_-\simeq&\\
{(0_{TX},{\rm id})}_!\nu_{Z/X}(\Lambda)\otimes \nu_{Z/X\times X}({\rm pr_1^\ast\mathcal F})\ar[r]\ar[d]_-\simeq\ar@{}[rd]^-{\eqref{eq:eqbcpsiPF}}&\nu_{Z/X\times X}(\delta_{X!}\Lambda\otimes {\rm pr}_1^\ast\mathcal F)\ar[d]^-\simeq\\
{(0_{TX},{\rm id})}_!(\nu_{Z/X}(\Lambda)\otimes {(0_{TX},{\rm id})}^\ast\nu_{Z/X\times X}({\rm pr_1^\ast\mathcal F}))\ar[r]\ar[d]_-\simeq&\nu_{Z/X\times X}(\delta_{X!}(\Lambda\otimes \delta_X^\ast {\rm pr}_1^\ast\mathcal F))\ar[d]^-\simeq\\
{(0_{TX},{\rm id})}_!{(0_{TX},{\rm id})}^\ast\nu_{Z/X\times X}({\rm pr_1^\ast\mathcal F})\ar[r]&\nu_{Z/X\times X}(\delta_{X!}\delta_X^\ast {\rm pr}_1^\ast\mathcal F).
}
\end{aligned}
\end{align}
We apply the natural transformation \eqref{eqnufun3}
associated with
$A_Z(X\times X)\xrightarrow{A_{\mathrm{id}}(\mathrm{pr}_2)} A_Z(X)$
to the composite of the right-hand vertical arrows in
\eqref{eq:finaldiagram1}.
By naturality, we obtain a commutative diagram
\begin{align}\label{eq:finaldiagram2}
\begin{aligned}
\xymatrix{
\alpha_!\nu_{Z/X\times X}(\delta_{X!}\Lambda\otimes {\rm pr}_1^\ast\mathcal F)\ar[r]^-{\eqref{eqnufun3}}\ar[d]_-\simeq &	\nu_{Z/X}\bigl(
\mathrm{pr}_{2!}
(\delta_{X!}\Lambda\otimes\mathrm{pr}_1^{\ast}\mathcal{F})
\bigr)\ar[d]^-\simeq\\
\alpha_!\nu_{Z/X\times X}(\delta_{X!}\delta_X^\ast {\rm pr}_1^\ast\mathcal F)\ar[r]^-{\eqref{eqnufun3}}&\nu_{Z/X}\bigl(
\mathrm{pr}_{2!}\delta_{X!}\delta_X^{\ast}
\mathrm{pr}_1^{\ast}\mathcal{F}
\bigr).
}
\end{aligned}
\end{align}
Applying $\alpha_!$ to \eqref{eq:finaldiagram1} and combining
the resulting diagram with \eqref{eq:finaldiagram2}, we obtain the commutativity
of the upper square in the following diagram:
\begin{align}\label{eq:verifyretractfinal}
\begin{aligned}
\xymatrix@C=3pc@R=2.5pc{
	\alpha_!\bigl(
	{\mathrm{pr}}_{TX}^{\ast}\nu_{X/X\times X}(\delta_{X!}\Lambda)
	\otimes
	{\mathrm{pr}}_{T_ZX}^{\ast}\nu_{Z/X}(\mathcal{F})
	\bigr)
	\ar[r]^-{\eqref{eq:retractionbyev2}}
	\ar[d]^-{\simeq}_-{\rm projection~formula}
	&
	\nu_{Z/X}\bigl(
	\mathrm{pr}_{2!}
	(\delta_{X!}\Lambda\otimes\mathrm{pr}_1^{\ast}\mathcal{F})
	\bigr)
	\ar[d]^-{{\rm projection~formula}}_\simeq
	\\
	\alpha_!{(0_{TX},\mathrm{id})}_!
	{(0_{TX},\mathrm{id})}^{\ast}
	{\mathrm{pr}}_{T_ZX}^{\ast}\nu_{Z/X}(\mathcal{F})
	\ar[r]^-\simeq
	\ar[d]_-{\simeq}
	&
	\nu_{Z/X}\bigl(
	\mathrm{pr}_{2!}\delta_{X!}\delta_X^{\ast}
	\mathrm{pr}_1^{\ast}\mathcal{F}
	\bigr)
	\ar[d]^-{\simeq}
	\\
	\nu_{Z/X}(\mathcal{F})
	\ar@{=}[r]
	&
	\nu_{Z/X}(\mathcal{F}).
}
\end{aligned}
\end{align}
The lower square commutes by the transitivity of
\eqref{eqnufun3} for
$\mathrm{pr}_{2!}\circ\delta_{X!}={\rm id}_{X!}$
and of \eqref{eqnufun1} for
$\delta_X^*\circ\mathrm{pr}_1^*={\rm id}_X^\ast$.

Now  \eqref{eq:verifyretractfinal} implies that  the diagram (1)  in \eqref{eq:finaldiagram3} commutes. This finishes the proof. 
\end{proof}

Let $f: X\to Y$ be a proper smooth morphism between smooth schemes over a field $k$.	
Consider the following cartesian diagrams:
\begin{equation}\label{eq:prphpp}
\begin{tikzcd}
X\times_Y X\arrow[r,"\mathrm{pr}_{f,1}"]\arrow[d,"\mathrm{pr}_{f,2}"]&X\arrow[d,"f"]\\
X\arrow[r,"f"]&Y,
\end{tikzcd}\quad\begin{tikzcd}
X\times_Y X\arrow[r,"i_f",hook]\arrow[d,"h_f"]&X\times X\arrow[d,"f\times f"]\\
Y\arrow[r,"\delta_Y",hook]&Y\times Y.
\end{tikzcd}
\end{equation}
By Lemma~\ref{lem:Saito-lemma-on-deformation-spaces}.(4), the following diagram is cartesian:
\begin{equation}
\label{eq:pullback-deformation-space-of-X-times_Y-X}
\begin{tikzcd}[column sep=40]
A_{X\times_Y X}(X\times X)\arrow[r,"A_{h_f}(f\times f)"]\arrow[d]&A_Y(Y\times Y)\arrow[d]\\
X\times X\times\mathbf{A}^1\arrow[r,"f\times f\times\mathrm{id}"]&Y\times Y\times\mathbf{A}^1.
\end{tikzcd}
\end{equation}
The morphism $A_{h_f}(f\times f)$ over  the special fiber is the following morphism
\begin{align}
T_{X\times_Y X}(X\times X)\xrightarrow{T_{h_f}(f\times f)}T_Y(Y\times Y).
\end{align}
Since $f$ is a smooth morphism, we have an isomorphism
\begin{align}
T_{X\times_YX}(X\times X)\simeq X\times_YX\times_Y TY,
\end{align}
together with a commutative diagram
\begin{equation}\label{eq:hpidtpDef}
\begin{tikzcd}[column sep=50]
T_{X\times_Y X}(X\times X)\arrow[r,"T_{h_f}(f\times f)"]\arrow[d,"\simeq"']&T_Y(Y\times Y)\arrow[d,equals]\\
X\times_Y X\times_Y TY\arrow[r,"{h_f\times\mathrm{id}_{TY}={\rm pr}_{TY}}"]&TY.
\end{tikzcd}
\end{equation}

\begin{lemma}\label{lem:proper-pushforward-of-Hom-and-nuHom}
Let $f: X\to Y$ be a proper smooth morphism between smooth schemes over a field $k$.	
Let ${\mathcal F}_1,{\mathcal F}_2\in D_{\rm ctf}(X,\Lambda)$.
We put 
\begin{equation}
\mathscr{H}_X({\mathcal F}_1,{\mathcal F}_2)=\mathcal{H}om_{X\times X}(\mathrm{pr}_1^*{\mathcal F}_1,\mathrm{pr}_2^!{\mathcal F}_2)\in D_{\rm ctf}(X\times X,\Lambda),
\end{equation}
and similarly for $\mathscr{H}_Y$. Then
there are natural isomorphisms
\begin{align}
&(f\times f)_*\mathscr{H}_X({\mathcal F}_1,{\mathcal F}_2)\simeq\mathscr{H}_Y(f_*{\mathcal F}_1,f_*{\mathcal F}_2),\label{eq:proper-pushforward-of-internal-Hom}\\
&(T_{h_f}(f\times f))_*\nu_{i_f}\mathscr{H}_X({\mathcal F}_1,{\mathcal F}_2)\simeq\nu_{\delta_Y}\mathscr{H}_Y(f_*{\mathcal F}_1,f_*{\mathcal F}_2)\label{eq:proper-pushforward-of-Verdier-specialization}.
\end{align}
\end{lemma}
\begin{proof}
By \cite[Proposition 2.11]{LuZheng} and the K\"unneth formula, we have
\begin{equation}
\begin{aligned}
	&(f\times f)_*\mathscr{H}_X({\mathcal F}_1,{\mathcal F}_2)\simeq(f\times f)_*({D}_X{\mathcal F}_1\boxtimes {\mathcal F}_2)\\
	\simeq &f_*{D}_X{\mathcal F}_1\boxtimes f_*{\mathcal F}_2\simeq {D}_Y(f_*{\mathcal F}_1)\boxtimes f_*{\mathcal F}_2\simeq \mathscr{H}_Y(f_*{\mathcal F}_1,f_*{\mathcal F}_2).
\end{aligned}
\end{equation}
By \eqref{eq:pullback-deformation-space-of-X-times_Y-X}, the morphism $A_{X\times_Y X}(X\times X)\to A_Y(Y\times Y)$ is proper over $\mathbf{A}^1$.
By proper base change for nearby cycles \eqref{eqbcpsi2}, we obtain
\begin{equation}
({T_{h_f}(f\times f)})_*\nu_{i_f}\mathscr{H}_X({\mathcal F}_1,{\mathcal F}_2)\simeq\nu_{\delta_Y}(f\times f)_*\mathscr{H}_X({\mathcal F}_1,{\mathcal F}_2)\simeq\nu_{\delta_Y}\mathscr{H}_Y(f_*{\mathcal F}_1,f_*{\mathcal F}_2).
\end{equation}
\end{proof}

\begin{lemma}\label{lem:Radon-retract-of-nuHom}
Let $f:X\to Y$ be a proper smooth morphism between smooth schemes over $k$.
Consider the diagrams \eqref{eq:prphpp}, \eqref{eq:pullback-deformation-space-of-X-times_Y-X} and \eqref{eq:hpidtpDef}.
Let ${\mathcal F}\in D_{\rm ctf}(X,\Lambda)$ and ${\mathcal G}\in D_{\rm ctf}(Y,\Lambda)$. Assume there is a retraction
$\mathcal G\xrightarrow{i}f_\ast\mathcal  F\xrightarrow{r}\mathcal G$ (i.e., $ri={\rm id}$).
Then
$\mathscr{H}_{Y}(\mathcal G,\mathcal G)$ is  a retract of $\mathscr{H}_{Y}(f_*{\mathcal F},f_*{\mathcal F})$, and
$\nu\mathcal{H}om_Y({\mathcal G},{\mathcal G}):=\nu_{\delta_Y}\mathscr{H}_Y({\mathcal G},{\mathcal G})$ is a retract of
\begin{align}
\nu_{\delta_Y}\mathscr{H}_{Y}(f_*{\mathcal F},f_*{\mathcal F})\overset{{\eqref{eq:proper-pushforward-of-Verdier-specialization}}}{\simeq }({T_{h_f}(f\times f)})_*\nu_{i_f}\mathscr{H}_{X}(\mathcal F,\mathcal F).
\end{align}
\end{lemma}
The retraction $\mathscr{H}_{Y}(\mathcal G,\mathcal G)\to \mathscr{H}_{Y}(f_*{\mathcal F},f_*{\mathcal F})\to \mathscr{H}_{Y}(\mathcal G,\mathcal G)$ is defined by \eqref{eq:defRetration-nuHom1} and \eqref{eq:defRetration-nuHom2} below.
\begin{proof}
Consider the following two morphisms
\begin{align}\label{eq:defRetration-nuHom1}
&\mathcal{H}om_{Y\times Y}(\mathrm{pr}_1^*\mathcal{G},\mathrm{pr}_2^!\mathcal{G})\to\mathcal{H}om_{Y\times Y}(\mathrm{pr}_1^*f_*\mathcal{F},\mathrm{pr}_2^!f_*\mathcal{F}),\\
\label{eq:defRetration-nuHom2}		&\mathcal{H}om_{Y\times Y}(\mathrm{pr}_1^*f_*\mathcal{F},\mathrm{pr}_2^!f_*\mathcal{F})\to\mathcal{H}om_{Y\times Y}(\mathrm{pr}_1^*\mathcal{G},\mathrm{pr}_2^!\mathcal{G}),
\end{align}
where \eqref{eq:defRetration-nuHom1} is defined by the functoriality for $r$ on the first factor and $i$ for the second factor, and the map \eqref{eq:defRetration-nuHom2} is defined analogously.
By construction, the composition of \eqref{eq:defRetration-nuHom1} and \eqref{eq:defRetration-nuHom2} is the identity.
Therefore, $\mathscr{H}_{Y}(\mathcal G,\mathcal G)$ is  a retract of $\mathscr{H}_{Y}(f_*{\mathcal F},f_*{\mathcal F})$. 
Applying the functor $\nu_{\delta_Y}$, we get that $\nu\mathcal{H}om_Y({\mathcal G},{\mathcal G})=\nu_{\delta_Y}\mathscr{H}_{Y}(\mathcal G,\mathcal G)$ is a retract of 
\[
\nu_{\delta_Y}\mathscr{H}_{Y}(f_*{\mathcal F},f_*{\mathcal F})\overset{{\eqref{eq:proper-pushforward-of-Verdier-specialization}}}{\simeq }({T_{h_f}(f\times f)})_*\nu_{i_f}\mathscr{H}_{X}(\mathcal F,\mathcal F).
\]
This finishes the proof.
\end{proof}

We compute the nearby cycles
complex for
certain tamely ramified sheaves.

\begin{lm}\label{lmpsitame}
Let $X$ be a scheme smooth over
the semi-stable curve
${\rm Spec}\, k[s,t,u]/(st-u)$ over 
${\mathbf A}^1={\rm Spec}\, k[u]$
and let $D_1,D_2\subseteq X$
be the divisors defined by $s$ and $t$
respectively.
Let ${\cal F}$ be a locally constant constructible
sheaf
of $\Lambda$-modules on
$U=X\sm (D_1\cup D_2)$ tamely ramified along 
$D_1\cup D_2$
and let $\Psi=\Psi {\cal F}$
be the nearby cycles complex 
on $X\times_{{\mathbf A}^1}0=D_1\cup D_2$ with respect 
to the morphism $X\to {\mathbf A}^1={\rm Spec}\, k[u]$
at $u=0$.
Let $j_1\colon D_1^\circ=D_1\sm (D_1\cap D_2)\to D_1$
and $j_2\colon D_2^\circ=D_2\sm (D_1\cap D_2)\to D_2$
be the open immersions.

{\rm 1.}
The restrictions
$\Psi_1^\circ=\Psi|_{D_1^\circ}$
and
$\Psi_2^\circ=\Psi|_{D_2^\circ}$
are locally constant and
tamely ramified along
$D_1\cap D_2$.

{\rm 2.}
The canonical morphisms
$\Psi|_{D_1}\to j_{1*}\Psi_1^\circ$
and
$\Psi|_{D_2}\to j_{2*}\Psi_2^\circ$
are isomorphisms.
\end{lm}

\proof{
1.
Let ${\mathbf A}^{1\, \rm t}={\rm Spec}\, k[u^{1/n};p\nmid n]$
be the maximal tamely ramified covering
and $j^{\rm t}\colon U\times_{{\mathbf A}^1}{\mathbf A}^{1\, \rm t}
\to X\times_{{\mathbf A}^1}{\mathbf A}^{1\, \rm t}$
be the base change of the open immersion $j\colon U\to X$.
Then, 
by the assumption that
${\cal F}$ is tamely ramified
along $D_1\cup D_2$,
the extension
$R^0j^{\rm t}_*{\cal F}$ is locally constant outside
the inverse image of $D_1\cap D_2$
and the restrictions
$(R^0j^{\rm t}_*{\cal F})|_{D_1^\circ}$
and
$(R^0j^{\rm t}_*{\cal F})|_{D_2^\circ}$
are tamely ramified
along $D_1\cap D_2$.
Since $X\sm (D_1\cap D_2)$ is smooth over ${\mathbf A}^1$,
we have isomorphisms
$(R^0j^{\rm t}_*{\cal F})|_{D_1^\circ}
\to \Psi|_{D_1^\circ}$ and
$(R^0j^{\rm t}_*{\cal F})|_{D_2^\circ}
\to \Psi|_{D_2^\circ}$ 
of locally constant sheaves
on $D_1^\circ$ and $D_2^\circ$.

2.
It suffices to show the isomorphisms
at each geometric point $x$ of $D_1\cap D_2$.
Let $X_x$ be the strict localization
and $U_x=X_x\times_{{\rm Spec}\, k[u]}
{\rm Spec}\, k(u)$
be the generic fiber.
Let $U_x^{\rm t}
=U_x\times_{{\rm Spec}\, k[s,t,u]/(st-u)}
{\rm Spec}\, \varinjlim_{p\nmid n}k[s^{1/n},t^{1/n},u^{1/n}]/
(s^{1/n}t^{1/n}-u^{1/n})$
be the universal tamely ramified covering.
By cohomological purity,
the canonical morphism
$\Lambda\to \Gamma(U_x^{\rm t},\Lambda)$
is an isomorphism (cf.~\cite[Th\'eor\`eme 3.3]{SGA7}).
The canonical morphism
$U\to {\mathbf A}^1$
is lifted to
$U_x^{\rm t}\to {\mathbf A}^{1\, \rm t}$.

We identify the Galois group ${\rm Gal}(U_x^{\rm t}/U_x)$
with the product $I_s\times I_t$ of
$I_s=I_t=\varprojlim_{p\nmid n}\mu_n$
acting on $s^{1/n},t^{1/n}$ by multiplication
respectively.
The morphism
$U_x^{\rm t}\to {\mathbf A}^{1\, \rm t}$
induces $I_s\times I_t
\to I_u={\rm Gal}({\mathbf A}^{1\, \rm t}/{\mathbf A}^1)$.
Let $I_0={\rm Ker}(I_s\times I_t\to I_u)$
be the kernel.
The pull-back of ${\cal F}$ on
$U_x^{\rm t}$ is a constant sheaf
and
the pull-back of ${\cal F}$ on
$U_x$ is the locally constant sheaf
corresponding to the representation
$M={\rm H}^0(U_x^{\rm t},{\cal F})$
of $I_s\times I_t$.
The isomorphism
$\Lambda\to \Gamma(U_x^{\rm t},\Lambda)$
induces an isomorphism
$\Gamma(I_0,M)=\Gamma(I_0,\Gamma(U_x^{\rm t},{\cal F}))
\to 
\Psi_x$.

For $i=1,2$,
let $D_{ix}=D_i\times_XX_x$
and $D^\circ_{ix}=D_i^\circ \times_X X_x=D_{ix}\sm \{x\}$.
Let $D_{ix}^{\circ\, \rm t}$ be the maximal tamely ramified
covering of $D^\circ_{ix}$.  We identify
${\rm Gal}(D_{1x}^{\circ\, \rm t}/D^\circ_{1x})=I_t$
and
${\rm Gal}(D_{2x}^{\circ\, \rm t}/D^\circ_{2x})=I_s$
with $I_0$
by the isomorphisms $I_0\to I_t$ and $I_0\to I_s$
induced by the projections.
The canonical morphisms
$\Lambda\to \Gamma(D_{ix}^{\circ\, \rm t},\Lambda)$
are also isomorphisms.
The pull-backs of
$\Psi$ on $D_{1x}^{\circ\, \rm t}$ and $D_{2x}^{\circ\, \rm t}$ 
correspond to the representation $M$ of $I_0$.
The isomorphisms
$\Lambda\to \Gamma(D_{ix}^{\circ\, \rm t},\Lambda)$
induce isomorphisms
$\Gamma(I_0,M)
=\Gamma(I_0,\Gamma(D_{1x}^{\circ\, \rm t},\Psi_1^\circ))
\to 
(j_{1*}\Psi_1^\circ)_x$ and
$\Gamma(I_0,M)
=\Gamma(I_0,\Gamma(D_{2x}^{\circ\, \rm t},\Psi_2^\circ))
\to 
(j_{2*}\Psi_2^\circ)_x$.
Hence we have canonical
isomorphisms 
$\Psi_x\to (j_{1*}\Psi_1^\circ)_x$
and
$\Psi_x\to (j_{2*}\Psi_2^\circ)_x$.
\qed

}

\medskip

Let $X={\mathbf A}^1={\rm Spec}\, k[x]$ 
and $Z=\{0\}\subseteq X$ be the reduced
closed subscheme consisting of the origin.
Let $(X\times {\mathbf A}^1)'
\to X\times {\mathbf A}^1={\rm Spec}\, k[x,u]$
be the blow-up at $(0,0)$
and $E={\mathbf P}^1$
be the exceptional divisor.
Let $(\{0\}\times {\mathbf A}^1)',
(X\times \{0\})'
\subseteq (X\times {\mathbf A}^1)'$ 
be the proper transforms
of $\{0\}\times {\mathbf A}^1$
and $X\times \{0\}$ 
respectively
and let
$0,\infty\in E$
be the intersection points with 
$(\{0\}\times {\mathbf A}^1)',
(X\times \{0\})'$.
The complement
$(X\times {\mathbf A}^1)'
\sm (X\times \{0\})'$
is the deformation
to normal bundle $A_ZX$
and 
$E\sm\{\infty\}$
is canonically identified with
the normal bundle
$T=T_ZX$.

\begin{lm}\label{lmpsiM}
Let ${\cal F}$ be a locally constant sheaf
on ${\mathbf G}_m={\rm Spec}\, k[x^{\pm1}]$
tamely ramified at $0$
and let $j\colon {\mathbf G}_m
\to X={\mathbf A}^1={\rm Spec}\, k[x]$
be the open immersion.
Let ${\rm pr}_1\colon X\times {\mathbf G}_m\to X$
be the projection and
$\Psi=\Psi {\rm pr}_1^*j_!{\cal F}$
be the nearby cycles complex
on $(X\times {\mathbf A}^1)'\times_{{\mathbf A}^1}0
=E\cup (X\times \{0\})'$
with respect to the morphism
$(X\times {\mathbf A}^1)'\to {\mathbf A}^1
={\rm Spec}\, k[u]$ at $0\in {\mathbf A}^1$.
Let ${\cal M}$ be the restriction
$\Psi|_{T^\times}$ on the ${\mathbf G}_m$-torsor
$T^\times=T_ZX\sm \{0\}
\subseteq T_ZX={\rm Spec}\, k[v]
\subseteq E$.

{\rm 1.}
Let $\eta_x$ be the generic point
of the henselization at $0\in X$
and let
$M$ be the representation
of the tame decomposition group $G_x^{\rm tame}
={\rm Gal}(\eta_{x,\rm tame}/\eta_x)$ corresponding to $\mathcal F|_{\eta_x}$.
Let $\eta_u$ be the generic point
of the henselization of
${\mathbf A}^1={\rm Spec}\, k[u]$ at $0$
and define 
$\pi_1^{\rm tame}(T^\times\times_k\eta_u)$ 
to be the fiber product
$\pi_1^{\rm tame}(T^\times)
\times_{{\rm Gal}(k_{\rm sep}/k)}
{\rm Gal}(\eta_{u,\rm tame}/\eta_u)$.
Define a morphism
$\pi_1^{\rm tame}(T^\times\times_k\eta_u)
\to G_x^{\rm tame}$
induced by $T^\times
\times ({\mathbf A}^1\sm \{0\})
={\rm Spec}\, k[v^{\pm 1},u^{\pm 1}]
\to {\mathbf G}_m={\rm Spec}\, k[x^{\pm 1}]$
given by $x=uv$.
Then the restriction ${\cal M}$
of $\Psi$ on $T^\times\times_k\eta_u={\mathbf G}_m\times_k\eta_u$
is a monodromic sheaf defined by
the pull-back of
$M$ to $\pi_1^{\rm tame}(T^\times\times_k\eta_u)$.

{\rm 2.}
Let $j_0\colon T^\times\to T$
and $j_\infty \colon T\to E$ be the open immersions.
Then, the restriction of $\Psi$ on $T$
is isomorphic to $\overline {\cal M}=
j_{0!}{\cal M}$.
The restriction of $\Psi$ on
the proper transform $(X\times \{0\})'$
is $j_*{\cal F}$
and
the restriction of $\Psi$ on $E$
is isomorphic to $j_{\infty*}\overline {\cal M}$.
\end{lm}

\proof{
1.
By Proposition \ref{prverdier}.1,
${\cal M}=\Psi|_{T^\times}$ is monodromic.
Let ${\rm Spec}\, k[u,v]=
A_ZX
\subseteq (X\times {\mathbf A}^1)'$
be the deformation to the normal bundle 
where $v=x/u$.
Then, 
${\rm pr}_1\colon X\times {\mathbf G}_m\to X$
is defined by $x=uv$.
Hence ${\cal M}$ corresponds to 
the pull-back of $M$
by 
$\pi_1^{\rm tame}(T^\times\times_k\eta_u)
\to G_x^{\rm tame}$.

2.
By Proposition \ref{prverdier}.2,
$\overline {\cal M}= j_{0!} {\cal M}\to \Psi|_T$ is 
an isomorphism.
The isomorphisms 
$\Psi|_{(X\times \{0\})'}
\to j_*{\cal F}$
and 
$\Psi|_E\to j_{\infty*}\overline {\cal M}$
follow from Lemma \ref{lmpsitame}.2.
\qed

}
\medskip

\subsection{Radon transform}\label{subsec:Radon-transform-notations}

\begin{definition}\label{def:radontransform}
Let $V$ be a fixed vector space over $k$ of dimension $n+1$ with $n\geq 1$, and put
$P=\mathbf{P}(V)$ and $P^\vee=\mathbf{P}(V^\vee)$.
Consider the incidence correspondence
\begin{equation}
Q=\{(x,H)\in P\times P^\vee:x\in H\}\subseteq P\times P^\vee.
\end{equation}
Note that $P^\vee$
parametrizes hyperplanes $H\subseteq P$
and $Q\subseteq P\times P^\vee$ is the universal family of hyperplanes.

Let $p:Q\to P$ and $p^\vee:Q\to P^\vee$ be the projection maps.
The Radon transform and the dual Radon transform are, respectively, the functors
\begin{align}
&\mathrm{Rad}=p^\vee_*p^*[n-1]:D_{\rm ctf}(P,\Lambda)\to D_{\rm ctf}(P^\vee,\Lambda),\\
&\mathrm{Rad}^\vee=p_*p^{\vee\ast}[n-1]:D_{\rm ctf}(P^\vee,\Lambda)\to D_{\rm ctf}(P,\Lambda).
\end{align}
We define the anti-Radon transform
\begin{equation}
\overline{\mathrm{Rad}}=\mathrm{Rad}^\vee(n-1):D_{\rm ctf}(P^\vee,\Lambda)\to D_{\rm ctf}(P,\Lambda).
\end{equation}
By Poincar\'e duality, $p^!=p^*[2n-2](n-1)$ and $p^{\vee!}=p^{\vee\ast}[2n-2](n-1)$. Consequently,  $\overline{\mathrm{Rad}}$ is both left and right adjoint to $\mathrm{Rad}$.
\end{definition}

It is well known that $\mathrm{Rad}$ and $\overline{\mathrm{Rad}}$ do not define an adjoint equivalence for $n\geq 2$ in general. Nevertheless, the following retraction statement holds.

\begin{proposition}\label{prop:Radon-retract}
Assume $n\geq 1$.
Let $\mathrm{U},\mathrm{C}$ denote the unit and counit of the adjunction $\mathrm{Rad}\vdash\overline{\mathrm{Rad}}$, and let $\overline{\mathrm{U}},\overline{\mathrm{C}}$ denote the unit and counit of the adjunction $\overline{\mathrm{Rad}}\vdash\mathrm{Rad}$.
For every ${\mathcal F}\in D_{\rm ctf}(P,\Lambda)$, the following identity holds:
\begin{align}\label{eq:composition-of-unit-and-anti-counit}
\overline{\mathrm{C}}_{\mathcal F}\circ\mathrm{U}_{\mathcal F}\simeq (-1)^{n-1}\mathrm{id}_{\mathcal F}.
\end{align}
In particular, ${\mathcal F}$ is a retract of $\overline{\mathrm{Rad}}(\mathrm{Rad}{\mathcal F})$.
\end{proposition}
\begin{proof} Unwinding the definitions, the left-hand side of \eqref{eq:composition-of-unit-and-anti-counit} is the composite
\begin{equation}\label{eq:expand-the-composition-of-unit-and-anti-counit}
\begin{aligned}
	&{\mathcal F}\to p_*p^*{\mathcal F}\to p_*p^{\vee!}p^\vee_*p^*{\mathcal F}\\
	\simeq &p_*p^{\vee\ast}p^\vee_*p^*{\mathcal F}[2n-2](n-1)\\
	\to& p_*p^*{\mathcal F}[2n-2](n-1)\simeq p_*p^!{\mathcal F}\to {\mathcal F}.
\end{aligned}
\end{equation}
We first evaluate the two arrows involving $p^\vee$. 
By \cite[Expos\'e XVI, D\'efinition 2.3.1 and Proposition 2.3.2]{ILO14},
the composite
\begin{equation}
p^*{\mathcal F}\to p^{\vee!}p^\vee_*p^*{\mathcal F}\simeq p^{\vee\ast}p^\vee_*p^*{\mathcal F}[2n-2](n-1)\to p^*{\mathcal F}[2n-2](n-1) 
\end{equation}
is given by cup product with the Chern class $c_{n-1}(T(Q/P^\vee))$ of the relative tangent bundle $T(Q/P^\vee)$.
By the projection formula, we get
\begin{equation}
\overline{\mathrm{C}}_{\mathcal F}\mathrm{U}_{\mathcal F}\simeq p_*c_{n-1}(T(Q/P^\vee))\cup \mathrm{id}_{\mathcal F}.
\end{equation}

It remains to compute the cohomology class
$p_*c_{n-1}(T(Q/P^\vee))\in H^0(P,\Lambda)$.
The closed immersion $j\colon Q\hookrightarrow P\times P^\vee$  is defined by the identity section 	(cf. \cite[3.2]{CC})
\begin{align}
{\rm id}\in 
{\rm End}(V)=V\otimes V^\vee=\Gamma(P\times P^\vee,\mathcal O(1,1)).
\end{align}
Thus, for the ideal sheaf $\mathcal I_{Q/(P\times P^\vee)}$ defining $j$, we have an isomorphism
\begin{align}
\mathcal I_{Q/(P\times P^\vee)}/\mathcal I_{Q/(P\times P^\vee)}^2\simeq \mathcal O_Q(-1,-1)=p^\ast \mathcal O_P(-1)\otimes p^{\vee\ast}\mathcal O_{P^\vee}(-1).
\end{align}
Now the morphisms $Q\xhookrightarrow{j} P\times P^\vee
\xrightarrow{{\rm pr}_2}P^\vee$ induce an exact sequence 
\begin{align}\label{eq:exactTQP}
0\to \mathcal I_{Q/(P\times P^\vee)}/\mathcal I_{Q/(P\times P^\vee)}^2\to j^\ast \Omega^1_{P\times P^\vee/P^\vee}\to \Omega^1_{Q/P^\vee}\to 0.
\end{align}
Since $ j^\ast \Omega^1_{P\times P^\vee/P^\vee}\simeq p^\ast\Omega^1_{P/k}$,  \eqref{eq:exactTQP} induces the following exact sequence
\begin{align}\label{eq:exactTQP-1}
0\to  \mathcal O_Q(-1,-1)\to p^\ast\Omega^1_{P/k}\to \Omega^1_{Q/P^\vee}\to 0.
\end{align}
Since all three terms are locally free, after taking duality, we get an exact sequence of vector bundles
\begin{align}
0\to T(Q/P^\vee)\to p^\ast TP\to \mathcal O_Q(1,1)\to 0. 
\end{align}
Now applying \cite[Lemme 8.4.1]{Jou77} to the above sequence, we get 
\begin{align}
p_\ast c_{n-1}(T(Q/P^\vee))=(-1)^{n-1 }\cdot p_\ast c_{n-1}(T(Q/P^\vee)^\vee)=(-1)^{n-1}.
\end{align}
This finishes the proof.
\end{proof}

\begin{corollary}\label{cor:Radon-retract-of-nuHom}
Let $V$ be a vector space over $k$ of dimension $n+1$ with $n\geq 1$, and put
$P=\mathbf{P}(V)$ and $P^\vee=\mathbf{P}(V^\vee)$. Let $Q\subseteq P\times P^\vee$ be the incidence 
correspondence. 
Let ${\mathcal F}\in D_{\rm ctf}(P,\Lambda)$, and set ${\mathcal G}=p^{\vee\ast}\mathrm{Rad}({\mathcal F})[n-1](n-1)\in D_{\rm ctf}(Q,\Lambda)$.
Then
$\mathscr{H}_{P}(\mathcal F,\mathcal F)$ is  a retract of $\mathscr{H}_{P}(p_*{\mathcal G},p_*{\mathcal G})$, and
$\nu\mathcal{H}om_P({\mathcal F},{\mathcal F}):=\nu_{\delta_P}\mathscr{H}_P({\mathcal F},{\mathcal F})$ is a retract of
\begin{align}
\nu_{\delta_P}\mathscr{H}_{P}(p_*{\mathcal G},p_*{\mathcal G})\overset{{\eqref{eq:proper-pushforward-of-Verdier-specialization}}}{\simeq }(h_p\times\mathrm{id}_{TP})_*\nu_{i_p}\mathscr{H}_{Q}(\mathcal G,\mathcal G).
\end{align}
\end{corollary}

\begin{proof}
Note that
\begin{equation}
p_*{\mathcal G}=p_*p^{\vee*}\mathrm{Rad}({\mathcal F})[n-1](n-1)=\overline{\mathrm{Rad}}\mathrm{Rad}({\mathcal F}).
\end{equation}
By Proposition~\ref{prop:Radon-retract}, ${\mathcal F}$ is a retract of $p_*{\mathcal G}$. The retraction is defined by
\begin{align}
{\mathcal F}\xrightarrow{i={\rm U}_{\mathcal F}}p_*{\mathcal G} \xrightarrow{r=(-1)^{n-1}\overline{\rm C}_{\mathcal F}}{\mathcal F},
\end{align}
where $ri={\rm id}$.
Now the results follow from Lemma \ref{lem:Radon-retract-of-nuHom} and \eqref{eq:hpidtpDef}.
\end{proof}

Consider	the following commutative diagram
\begin{align}\label{eq:collisionnormalbundleactionnotations}
\begin{aligned}
\xymatrix{
	Q\times_P Q\ar@{^(->}[r]^-{i_p}&Q\times Q\ar@{^(->}[r]^-\delta\ar@{=}[rd]&(Q\times Q)\times_{P^\vee\times P^\vee}(Q\times Q)\ar[d]_-{{\rm pr}_1}\ar[r]^-{{\rm pr}_2}& Q\times Q\ar[d]^-{p^\vee\times p^\vee}\\
	&&Q\times Q\ar[r]^{p^\vee\times p^\vee}&P^\vee\times P^\vee.
}
\end{aligned}
\end{align}
We apply Lemma~\ref{lem:invariant-structure-on-specialization} to the inclusion $i_p: Q\times_PQ\hookrightarrow Q\times Q$ and the morphism $p^\vee\times p^\vee: Q\times Q\to P^\vee\times P^\vee$.
We now describe the action \eqref{eq:actionofTXYZonNZX} in this setting.
There are canonical identifications
\begin{align}
&T(Q\times Q/P^\vee\times P^\vee)=\mathrm{pr}_{Q,1}^*T({Q/P^\vee})\oplus\mathrm{pr}_{Q,2}^*T({Q/P^\vee})\quad\text{over }Q\times Q,\\
&T_{Q\times_PQ}(Q\times Q)=Q\times_PQ\times_PTP\quad\text{over }Q\times_PQ.
\end{align}
Define a vector bundle $W$ over $Q\times_PQ$ by
\begin{equation}
W=T(Q\times Q/P^\vee\times P^\vee)|_{Q\times_PQ}=\mathrm{pr}_{p,1}^*T({Q/P^\vee})\oplus\mathrm{pr}_{p,2}^*T({Q/P^\vee}).
\end{equation}
Under these identifications, the action \eqref{eq:actionofTXYZonNZX} becomes an action of $W$ on $Q\times_PQ\times_PTP$ over $Q\times_PQ$:
\begin{equation}\label{eq:actionofWonTP}
\alpha:W\times_PTP=W\times_{Q\times_PQ}(Q\times_PQ\times_PTP)\to Q\times_PQ\times_PTP.
\end{equation}
More precisely, the action \eqref{eq:actionofWonTP} is induced by the linear morphism
\begin{equation}
\beta:W\to Q\times_PQ\times_P TP,(w_1,w_2)\mapsto Tp(w_2)-Tp(w_1),
\end{equation}
namely, the difference of the two morphisms
\begin{align}
&Tp\circ\mathrm{pr}_2:W\to\mathrm{pr}_2^*T({Q/P^\vee})\to T({Q/P^\vee})\xrightarrow{Tp}TP,\\
&Tp\circ\mathrm{pr}_1:W\to\mathrm{pr}_1^*T({Q/P^\vee})\to T({Q/P^\vee})\xrightarrow{Tp}TP.
\end{align}
Here $Tp: T(Q/P^\vee)\to TP$ is the composition $T(Q/P^\vee)\to T(P\times P^\vee/P^\vee)\simeq TP\times P^\vee\xrightarrow{{\rm pr}}TP$.

The natural morphism
\begin{equation}
\rho:T^*P\setminus 0_{T^*P}\to\mathbf{P}(T^*P)\simeq Q
\end{equation}
defines a tautological line subbundle on $\delta_p(Q)$:
\begin{equation}\label{eq:tautologicalLineSubbundle}
\mathcal{L}
=\{(x,H,H,\xi):x\in P,\xi\in \ell_{x,H}\}
\subseteq\delta_p(Q)\times_PT^*P,
\end{equation}
where $\ell_{x,H}\subseteq T_x^\ast P$ is the line represented by $(x,H)\in Q$.
We have a canonical isomorphism
\begin{align}
\mathcal{L}	&\simeq T^\ast_Q(P\times P^\vee).
\end{align}
\begin{lemma}\label{lem:calculatekerbetavee}
Consider the diagram \eqref{eq:collisionnormalbundleactionnotations}.

{\rm 1. }
Let ${\mathcal F}\in D_{\rm ctf}(P,\Lambda)$ and put ${\mathcal G}=p^{\vee\ast}\mathrm{Rad}({\mathcal F})[n-1](n-1)\in D_{\rm ctf}(Q,\Lambda)$. Then $\nu_{i_p}(\mathscr{H}_Q({\mathcal G},{\mathcal G}))$ on $Q\times_PQ\times_PTP$ is $W$-invariant.

{\rm 2.}
Consider the dual morphism
\begin{equation}
	\beta^\vee:Q\times_PQ\times_PT^*P\to W^\vee=\mathrm{pr}_{p,1}^*T^*({Q/P^\vee})\oplus\mathrm{pr}_{p,2}^*T^*({Q/P^\vee}),
\end{equation}
and the projection morphism
\begin{equation}
	h_p\times\mathrm{id}_{T^*P}:Q\times_PQ\times_PT^*P\to T^*P.
\end{equation}
The kernel $\ker(\beta^\vee)$ is the union of $Q\times_PQ\times_P0_{T^*P}$ and $\mathcal{L}$ $($cf. \eqref{eq:tautologicalLineSubbundle}$)$. 

\end{lemma}
\begin{proof}
We have isomorphisms
\begin{equation}
\begin{aligned}
	{D}_Q{\mathcal G}&= {D}_Q((p^\vee)^*\mathrm{Rad}({\mathcal F})[n-1](n-1))\simeq(p^\vee)^!{D}_{P^\vee}\mathrm{Rad}({\mathcal F})[1-n](1-n)\\
	&\simeq(p^\vee)^*{D}_{P^\vee}\mathrm{Rad}({\mathcal F})[n-1],
\end{aligned}
\end{equation}
\begin{equation}
\mathscr{H}_Q({\mathcal G},{\mathcal G})\simeq{D}_Q{\mathcal G}\boxtimes {\mathcal G}\simeq(p^\vee\times p^\vee)^*({D}_{P^\vee}(\mathrm{Rad}({\mathcal F}))\boxtimes\mathrm{Rad}({\mathcal F}))[2n-2](n-1).
\end{equation}
Now (1) follows from Lemma \ref{lem:invariant-structure-on-specialization}. We prove (2).
Considering $Q\subseteq P\times P^\vee\to P^\vee$,	we have isomorphisms
\begin{align}
&{\rm ker}({\rm pr}_1\circ\beta^\vee)\simeq (Q\times_PQ)\times_{{\rm pr}_1,Q} T^\ast_Q(P\times P^\vee),\\
&{\rm ker}({\rm pr}_2\circ\beta^\vee)\simeq (Q\times_PQ)\times_{{\rm pr}_2,Q} T^\ast_Q(P\times P^\vee).
\end{align}
Then $\ker(\beta^\vee)={\rm ker}({\rm pr}_1\circ\beta^\vee)\cap {\rm ker}({\rm pr}_2\circ\beta^\vee)=(Q\times_PQ\times_P0_{T^*P})\cup\mathcal{L}$.
\end{proof}

Let $f: X\to S$ be a morphism of noetherian schemes over $k$ and $\mathcal F\in D_c^b(X,\Lambda)$. We denote by $E_f(\mathcal F)$ the smallest closed subset of $X$ such that
$f: X\setminus E_f(\mathcal F)\to S$ is universally locally acyclic relatively to $\mathcal F$. We call $E_f(\mathcal F)$ the non-ULA locus of $\mathcal F$ with respect to $f$.

For later convenience, we recall a result of Beilinson on the identification of  the projectivization of singular support with the non-ULA locus.
Recall  the Legendre transform \cite{SS}:
\begin{equation}\label{eq:LegendreTransform}
\mathbf{P}(T^*P)\simeq Q\simeq \mathbf{P}(T^*P^\vee).
\end{equation}
\begin{lemma}[{\cite[Theorem 3.2]{SS}}]
\label{lem:Radon-criterion}
Let ${\mathcal F}\in D_{\rm ctf}(P,\Lambda)$, and let  ${SS}({\mathcal F})\subseteq T^*P$ be the singular support of ${\mathcal F}$.
Under the Legendre transform $\mathbf{P}(T^*P)= Q$, we have
\begin{align}\label{eq:Beilinson-PSS}
\mathbf{P}({SS}({{\mathcal F}}))=E_p(p^{\vee\ast}\mathrm{Rad}({\mathcal F})).
\end{align}
\end{lemma}

\section{Singular supports and microlocalization}\label{sec:SSMicrolocalization}
In this section, let $X$ be a separated smooth scheme
of finite type
over a field $k$ of characteristic $p>0$
and $\Lambda$ be a finite local ring with
residue characteristic $\ell\neq p$.
We fix a non-trivial character
$\psi\colon {\mathbf F}_p\to \Lambda^\times$.
\subsection{Construction}\label{ssSS}

We define the microlocalization of a constructible sheaf along a closed subscheme.
\begin{definition}
	Let $i:Z\hookrightarrow X$ be a closed immersion of smooth schemes, and let $\eta_0$ be the generic point
	of the henselization of ${\mathbf A}^1$
	at $0$.
	Recall the specialization functor \eqref{eqnZ}
	\begin{equation}
		\nu_{Z/X}:D_{\rm ctf}(X,\Lambda)\to D_{\rm ctf}(T_ZX\times_k\eta_0,\Lambda).
	\end{equation}
	For ${\mathcal F}\in D_{\mathrm{ctf}}(X,\Lambda)$, we define its
	microlocalization along $i$ to be the Fourier transform
	of its specialization:
	\begin{equation}
		\mu_{Z/X}({\mathcal F})=F_{\psi,T_ZX/Z}(\nu_{Z/X}({\mathcal F}))\in D_{\rm ctf}(T^*_ZX\times_k\eta_0,\Lambda).
	\end{equation}
In \eqref{eqnHom}, we define the specialization  $\nu\mathcal Hom$
 of the internal Hom object $\mathscr H=\mathcal Hom$ (cf. \eqref{eq:defHom}). Now we define its microlocal version.
	Let 
	\begin{align}
		F_\psi\colon
		D_{\rm ctf}(TX\times_k\eta_0,\Lambda)
		\to D_{\rm ctf}(T^*X\times_k\eta_0,\Lambda)
	\end{align}
	denote the Fourier transform associated with the tangent bundle (cf.
	Definition \ref{dfFD}.2).
	We define a bifunctor
	\begin{align}
		&\mu{\cal H}om=\mu{\cal H}om_{\psi,\Lambda}
		\colon 
		D_{\rm ctf}(X,\Lambda)^{\rm op}
		\times D_{\rm ctf}(X,\Lambda)
		\to
		D_{\rm ctf}(T^*X\times_k\eta_0,\Lambda),
		\label{eqmuHom}
	\end{align}
	by 
	$\mu{\cal H}om={F}_\psi\circ \nu{\cal H}om$.
\end{definition}

In the following, we use the notation of Convention~\ref{con:delignetopoi}.
\begin{proposition}\label{prop:mu_Z/X-as-a-retract-of-muHom}
	Let $i:Z\hookrightarrow X$ be a closed immersion of smooth $k$-schemes, and let $q^\vee: T_Z^\ast X\to T^\ast X|_Z=T^\ast X\times_XZ$ denote the canonical closed immersion of the conormal bundle into the restricted cotangent bundle.
	For all ${\mathcal F}\in D_{\mathrm{ctf}}(X,\Lambda)$, the complex
	\begin{equation}
		\mu_{Z/X}({\mathcal F})\in D_{\mathrm{ctf}}(T^*_ZX\times_k\eta_0,\Lambda)
	\end{equation}
	is a retract of
	\begin{equation}\label{eq:a-sheaf-whose-retract-is-mu_Z/X}
		(q^\vee)^*(\mu\mathcal{H}om_X({\mathcal F},{\mathcal F})|_{T^*X|_Z})\otimes\mu_{Z/X}({\mathcal F}).
	\end{equation}
\end{proposition}
Note that this defines a canonical morphism
\begin{align}
	(q^\vee)^*(\mu\mathcal{H}om_X({\mathcal F},{\mathcal F})|_{T^*X|_Z})\otimes\mu_{Z/X}({\mathcal F})\to \mu_{Z/X}({\mathcal F}).
\end{align}
\begin{proof}
	Applying Lemma \ref{lem:Fourier-transform-of-induced-action} to $q:TX|_Z\to T_ZX$ and the following diagram
	\begin{equation}
		\begin{tikzcd}
			&TX|_Z\times_ZT_ZX\arrow[dl,"\mathrm{pr}_{TX|_Z}"swap]\arrow[d,"\mathrm{pr}_{T_ZX}"]\arrow[r,"\alpha"]&T_ZX\\
			TX|_Z&T_ZX,
		\end{tikzcd}
	\end{equation}
	we see that \eqref{eq:a-sheaf-whose-retract-is-mu_Z/X} is the Fourier transform of \eqref{eq:a-sheaf-whose-retract-is-nu_Z/X}.
\end{proof}

\begin{lm}\label{lm:coeff-microlocalization}
	Let $X$ be a smooth scheme over a field $k$ of characteristic $p$, and let $\Lambda\to\Lambda^\prime$ be a finite \'etale
	faithfully flat local extension of finite local rings with residue
	characteristic different from $p$.
	Suppose that $\Lambda$ admits a non-trivial character
	$\psi:\mathbf F_p\to \Lambda^\times$, and let $\psi_{\Lambda^\prime}$ be its extension to $\Lambda^\prime$.
	For $\mathcal F,\mathcal G\in D_{\rm ctf}(X,\Lambda)$, put
	$\mathcal F_{\Lambda^\prime}=\mathcal F\otimes_\Lambda \Lambda^\prime$ and
	$\mathcal G_{\Lambda^\prime}=\mathcal G\otimes_\Lambda \Lambda^\prime$.
	Let $\delta:X\hookrightarrow X\times X$ be the
	diagonal.
	We have canonical isomorphisms
	\begin{align}\label{eq:coeff-microlocalization-3}
		&\mathscr H({\mathcal F,\mathcal G})\otimes_\Lambda \Lambda^\prime
		\simeq\mathscr H({\mathcal F_{\Lambda^\prime},\mathcal G_{\Lambda^\prime}})\,\,\,\,\,\,\qquad\qquad\qquad{\rm in}\quad D_{\rm ctf}(X\times_kX,\Lambda^\prime),\\
		\label{eq:coeff-microlocalization-2}	&\nu_\delta(\mathscr H({\mathcal F,\mathcal G}))\otimes_\Lambda \Lambda^\prime
		\simeq\nu_\delta(\mathscr H({\mathcal F_{\Lambda^\prime},\mathcal G_{\Lambda^\prime}}))\,\,\,\,\,\quad\qquad{\rm in}\quad D_{\rm ctf}(TX\times_k\eta_0,\Lambda^\prime),\\
		\label{eq:coeff-microlocalization}		&\mu{\cal H}om_{\psi,\Lambda}(\mathcal F,\mathcal G)\otimes_\Lambda \Lambda^\prime
		\simeq
		\mu{\cal H}om_{\psi_{\Lambda^\prime},\Lambda^\prime}(\mathcal F_{\Lambda^\prime},\mathcal G_{\Lambda^\prime})\quad{\rm in}\quad D_{\rm ctf}(T^*X\times_k\eta_0,\Lambda^\prime).
	\end{align}
\end{lm}

\begin{proof}
	Since $\Lambda^\prime$ is finite free of positive rank over $\Lambda$,
	the functor $-\otimes_\Lambda \Lambda^\prime$ is exact and conservative.
	It preserves constructibility and finite Tor-dimension and detects
	zero complexes, and is also compatible with the six operations, nearby cycles and specialization. Now \eqref{eq:coeff-microlocalization-3} and \eqref{eq:coeff-microlocalization-2} follow.

	Let $E\to Y$ be a vector bundle, and let ${\rm pr}_1$ and
	${\rm pr}_2$ be the projections from $E\times_Y E^\vee$.
	The formula
	\[
	F_{\psi,E/Y,\Lambda}(\mathcal N)
	={\rm pr}_{2!}\bigl(
	{\rm pr}_1^*\mathcal N\otimes_\Lambda
	\mathcal L_{\psi,\Lambda}(\langle\ ,\ \rangle)\bigr),
	\]
	together with
	$\mathcal L_{\psi,\Lambda}\otimes_\Lambda \Lambda^\prime\simeq\mathcal L_{\psi_{\Lambda^\prime},\Lambda^\prime}$,
	gives a canonical isomorphism
	\begin{equation}\label{eq:coeff-Fourier}
		F_{\psi,E/Y,\Lambda}(\mathcal N)\otimes_\Lambda \Lambda^\prime
		\simeq F_{\psi_{\Lambda^\prime},E/Y,\Lambda^\prime}(\mathcal N\otimes_\Lambda \Lambda^\prime).
	\end{equation}
	By the definition \eqref{eqmuHom}, the isomorphism
	\eqref{eq:coeff-microlocalization} follows from \eqref{eq:coeff-microlocalization-3}, \eqref{eq:coeff-microlocalization-2}
	and \eqref{eq:coeff-Fourier}.
\end{proof}

Let $T$ be a $k$-scheme. An object $\mathcal F\in D_c^b(T\times_k\eta_0,\Lambda)$ has an underlying object in
$D_c^b(T_{\bar k},\Lambda)$, which we still denote by $\mathcal F$. This complex carries a continuous semilinear
$G_{\eta_0}$-action, with the action on $T_{\bar k}$
induced by the quotient $G_{\eta_0}\to G_k$.
The semilinear action, together with the surjectivity of
$G_{\eta_0}\to G_k$, implies that
${\rm supp}(\mathcal F)\subseteq T_{\bar k}$ is $G_k$-stable.
Therefore by Galois descent for closed subsets, there is a  unique closed subset ${\rm supp}_k({\cal F})\subseteq T$ whose inverse image under $T_{\bar k}\to T$ is ${\rm supp}(\mathcal F)$:
\begin{align}\label{eq:defSuppPrime}
	({\rm supp}_k({\cal F}))_{\bar k}={\rm supp}({\cal F}).
\end{align}
Here ${\rm supp}({\cal F})$ is the closure of the set $\{t\in T_{\bar k}|\mathcal F_{\bar t}\not\simeq 0\}$.

\begin{df}\label{dfSSmu}
	Let $X$ be a smooth scheme over $k$, and let ${\cal F}\in D_{\rm ctf}(X,\Lambda)$. Recall the object $\mu{\cal H}om({\cal F},{\cal F})\in D_{\rm ctf}(T^\ast X\times_k\eta_0,\Lambda)$ defined by \eqref{eqmuHom}.
	We define  a closed subset $SS_\mu({\cal F})$
of $T^*X$ by
\begin{equation}
SS_\mu({\cal F}):={\rm supp}_k(\mu{\cal H}om({\cal F},{\cal F})).
\label{eqSSmu}
\end{equation}
\end{df}
\medskip

We say that
${\cal F}\in D^b_c(X,\Lambda)$
is locally constant
if every cohomology sheaf
${\cal H}^q{\cal F}$
is locally constant.

\begin{lm}\label{lmlcc}
Let ${\cal F}\in D_{\rm ctf}(X,\Lambda)$.

{\rm 1.}
The closed subset $SS_\mu({\cal F})
\subseteq T^*X$
is conical.

{\rm 2.}
Assume that ${\cal F}$ is locally constant
and that $X$ is irreducible.
Then, 
we have
$SS_\mu ({\cal F})=T^*_XX$ if ${\cal F}\neq 0$
and 
$SS_\mu ({\cal F})=\varnothing$ if ${\cal F}= 0$.

{\rm 3.}
Assume $\dim X=1$.
Then, we have
$SS_\mu({\cal F})
\subseteq 
SS ({\cal F})$.
\end{lm}

\proof{1.
The specialization
$\nu{\cal H}om({\cal F},{\cal F})$
and hence
the microlocalization
$\mu{\cal H}om({\cal F},{\cal F})$
are monodromic
by Proposition \ref{prverdier}.1
and \cite[Proposition 2.5(4)]{Zhou}.
Thus its support $SS_\mu({\cal F})$ is
conical.

2.
Since ${\cal H}om({\rm pr}_1^*{\cal F},
{\rm pr}_2^!{\cal F})$ is locally constant,
we have isomorphisms
$e^*({\cal H}om({\cal F},{\cal F})\otimes K_X)
\to \nu{\cal H}om({\cal F},{\cal F})$ 
and
$0^\vee_*{\cal H}om({\cal F},{\cal F})
\to \mu{\cal H}om({\cal F},{\cal F})$
(\ref{eqF0}).

3.
We may assume that $X$ is irreducible.
Let $U\subseteq X$ be the maximal open subset
on which ${\cal F}|_U$ is locally constant
and identify $T^*U=T^*X\times_XU$.
Then by 2,
we have $SS_\mu ({\cal F})\cap T^*U
=T^*_UU$ or $=\varnothing$
according to ${\cal F}|_U\neq 0$ or $=0$.
Since $SS({\cal F})$
is the union of
$SS_\mu ({\cal F})\cap T^*U$
and fibers
$T^*X\sm T^*U$,
the assertion follows.
\qed

}
\medskip

We show an equality 
for the direct image of singular support 
by closed immersions.

\begin{lm}\label{lmimm}
Let $i\colon X\to Y$ be a closed
immersion of smooth schemes
over $k$.
Then, we
have an equality
\begin{equation}
SS_\mu (i_*{\cal F})
=i_\circ SS_\mu({\cal F})
\label{eqimm}
\end{equation}
of closed subsets of $T^*Y$.
\end{lm}

\proof{
By Lemma \ref{lem:Saito-lemma-on-deformation-spaces}.3,
the morphism
$A_X(X\times X)
\to
A_Y(Y\times Y)$
is a closed immersion.
Hence
the base change morphism
$\nu{\cal H}om_Y(i_*{\cal F},i_*{\cal F})
\to Ti_*\nu{\cal H}om_X({\cal F},{\cal F})$
is an isomorphism.
Let 
\begin{equation}
	\begin{CD}T^*X@<{a=di}<<T^*Y\times_YX
		@>{b={\rm pr}}>> T^*Y\end{CD}
\end{equation}
be the canonical morphisms.
Then the isomorphism
$\nu{\cal H}om_Y(i_*{\cal F},i_*{\cal F})
\to Ti_*\nu{\cal H}om_X({\cal F},{\cal F})$
induces an isomorphism
$\mu{\cal H}om_Y(i_*{\cal F},i_*{\cal F})
\to b_*a^*\mu{\cal H}om_X({\cal F},{\cal F})$
by Proposition \ref{prFa}.
\qed

}

\begin{lm}\label{lmsm}
Let $h\colon W\to X$ be a smooth
morphism of smooth schemes
over $k$.
Then, we
have an equality
\begin{equation}
SS_\mu (h^*{\cal F})
=h^\circ SS_\mu({\cal F})
\label{eqSSh}
\end{equation}
of closed subsets of $T^*W$.
\end{lm}

\proof{
Let 
\begin{equation}
\begin{CD}
	T^*X@<{b={\rm pr}}<<T^*X\times_XW
	@>{a=dh}>> T^*W\\
	TX@<b^\vee={\rm pr}<< TX\times_XW@<a^\vee=Th<<TW
	\end{CD}
\label{eqabh}
\end{equation}
be the canonical morphisms.
By Lemma \ref{lem:Saito-lemma-on-deformation-spaces}.2,
the morphism
$A_W(W\times W)
\to
A_X(X\times X)$
is smooth.
Hence
the base change morphism
$a^{\vee!}b^{\vee\ast}\nu{\cal H}om_X({\cal F},{\cal F})
\to
\nu{\cal H}om_W(h^!{\cal F},h^!{\cal F})$
is an isomorphism
by the smooth base change theorem.
This induces an isomorphism
$a_!b^{*}\mu{\cal H}om_X({\cal F},{\cal F})
\to \mu{\cal H}om_W(h^!{\cal F},h^!{\cal F})=
\mu{\cal H}om_W(h^*{\cal F},h^*{\cal F})$
by Corollary \ref{cordv}.1.
\qed
}
\medskip

If $\psi$ does not exist, we can define $SS_\mu$ as follows.

\begin{pr}\label{pr:coeff-microlocal-support}
	Let $X$ be a smooth $k$-scheme, and let
	$\mathcal F\in D_{\rm ctf}(X,\Lambda)$.
	Let $\Lambda\to \Lambda'$ be a finite \'etale faithfully flat local extension of finite local rings with residue characteristic different from $p$. Assume there is a non-trivial character $\psi:\mathbf F_p\to \Lambda^{\prime\times}$.
	The following formula defines a closed conical subset of $T^*X$:
	\begin{equation}\label{eq:coeff-microlocal-support}
		SS_\mu(\mathcal F)
		:={\rm supp}_k\bigl(
		\mu{\cal H}om_{\psi,\Lambda'}
		(\mathcal F_{\Lambda'},\mathcal F_{\Lambda'})\bigr),
	\end{equation}
	which is independent of the choice
	of $\Lambda'$ and $\psi$.
	
	If $\Lambda$ admits a non-trivial character, this agrees with
	Definition~\ref{dfSSmu}.
\end{pr}

\begin{proof}
	Since $\Lambda'$ is henselian and $p$ is invertible in $A$,
	reduction induces a bijection between the $p$-th roots of unity in $\Lambda'$
	and those in its residue field.
	Every non-trivial character $\mathbf F_p\to\Lambda^{\prime\times}$ is therefore of the form
	$\psi_a(t)=\psi(at)$ for a unique $a\in\mathbf F_p^\times$.
	For a vector bundle $E\to Y$, let $m_a:E^\vee\to E^\vee$ be the
	multiplication by $a\in\mathbf F_p^\times$.
	We have the following isomorphism (cf. \cite[the remark (i) after Definition 3.8]{AS11})
	\begin{equation}\label{eq:coeff-character-scaling}
		F_{\psi_a,E/Y,\Lambda'}(\mathcal N)
		\simeq m_a^*F_{\psi,E/Y,\Lambda'}(\mathcal N).
	\end{equation}
	By Lemma~\ref{lmlcc}.1, 
	${\rm supp}_k(\mu{\cal H}om_{\psi,\Lambda'}(\mathcal F_{\Lambda'},\mathcal F_{\Lambda'}))$ is conical.
	Hence \eqref{eq:coeff-character-scaling} shows that the right hand side of \eqref{eq:coeff-microlocal-support} is independent
	of the non-trivial character $\psi$.
	
	Now consider two choices $(\Lambda'_1,\psi_1)$ and
	$(\Lambda'_2,\psi_2)$.
	Choose a nonzero local factor $B$ of
	$\Lambda'_1\otimes_\Lambda\Lambda'_2$.
	It is finite \'etale and faithfully flat over each $\Lambda'_i$.
	
	Since 
	the functors $-\otimes_{\Lambda_1^\prime} B$ and $-\otimes_{\Lambda_2^\prime} B$ are conservative and detect
	zero complexes, we have 
	\begin{align}
\begin{aligned}
			{\rm supp}_k\bigl(
	\mu{\cal H}om_{\psi_1,\Lambda^\prime_1}
	(\mathcal F_{\Lambda_1^\prime},\mathcal F_{\Lambda^\prime_1})\bigr)&={\rm supp}_k\bigl(
	\mu{\cal H}om_{\psi_{1B},B}
	(\mathcal F_{B},\mathcal F_{B})\bigr)\\
	&={\rm supp}_k\bigl(
	\mu{\cal H}om_{\psi_{2B},B}
	(\mathcal F_{B},\mathcal F_{B})\bigr)\\
	&={\rm supp}_k\bigl(
	\mu{\cal H}om_{\psi_2,\Lambda^\prime_2}
	(\mathcal F_{\Lambda_2^\prime},\mathcal F_{\Lambda_2^\prime})\bigr).
\end{aligned}
	\end{align}
	This proves independence of the choice
	of $\Lambda'$ and $\psi$, as well as agreement with
	Definition~\ref{dfSSmu} when a character exists over $\Lambda$.
\end{proof}

Let $k$ be any field, and let
$\Lambda$ be a finite local ring whose residue characteristic is invertible in $k$.
Let $u:\mathbf A_k^1\setminus\{1\}\hookrightarrow\mathbf A_k^1$
be the open immersion, and put $B=u_*\Lambda$.
Let $X$ be a noetherian scheme over $k$
and let $E$ be a vector bundle over $X$ with dual $E^\vee$.
Consider the projections and the canonical pairing
\begin{equation}
	\begin{CD}
		E@<{{\rm pr}_1}<<E\times_XE^\vee 
		@>{{\rm pr}_2}>>E^\vee\\
		@.@V\mu=<,> VV@.\\
		@.{\mathbf A}^1@.
	\end{CD}
\end{equation}
Following \cite[Section 2]{Wang}, but omitting the shift by
${\rm rank}\,E$, we define 
\begin{align}\label{eq:Wang-transform}
	F_{B,E/X}(\mathcal M)
	=
	{\rm pr}_{2!}
	\bigl({\rm pr}_1^*\mathcal M\otimes
	\langle\, ,\,\rangle^*B\bigr).
\end{align}
\begin{definition}\label{def:SSmuB}
	If $X$ is a smooth $k$-scheme $X$ and
	$\mathcal F\in D_{\rm ctf}(X,\Lambda)$,  we define
	\begin{align}\label{eq:Wang-muHom}
		\mu_B{\cal H}om(\mathcal F,\mathcal F)
		&:=
		F_{B,TX/X}\bigl(\nu{\cal H}om(\mathcal F,\mathcal F)\bigr),\\
		\label{eq:Wang-SSmu}	SS_{\mu,B}(\mathcal F)&
		:=
		{\rm supp}_k(\mu_B{\cal H}om(\mathcal F,\mathcal F)).
	\end{align}
\end{definition}
In positive characteristic, since  $\nu\mathcal Hom(\mathcal F,\mathcal F)$ is monodromic, this construction agrees with
\eqref{eq:coeff-microlocal-support}. Consequently, whenever  $\Lambda$ admits a nontrivial
character $\psi:\mathbf F_p\to\Lambda^\times$, it 
also agrees with
Definition~\ref{dfSSmu} .
\begin{proposition}\label{prop:Wang-SS-comparison}
Let $k$ be a field of characteristic $p>0$, and let
$\Lambda$ be a finite local ring of residue characteristic $\ell\neq p$.
Let $X$ be a smooth scheme over $k$, and 
	let $\mathcal F\in D_{\rm ctf}(X,\Lambda)$.
Then we have 
	\begin{equation}\label{eq:Wang-support-comparison}
		SS_{\mu}(\mathcal F)=SS_{\mu,B}(\mathcal F),
	\end{equation}
where the left hand side is defined by \eqref{eq:coeff-microlocal-support}.
\end{proposition}

\begin{proof}
	By Lemma~\ref{lm:coeff-character}, there exists a finite
	\'etale faithfully flat local $\Lambda$-algebra $\Lambda'$
	admitting a nontrivial character
	$\psi:\mathbf F_p\to(\Lambda')^\times$.
	Put
	$\mathcal F_{\Lambda'}=\mathcal F\otimes_\Lambda\Lambda'$.
	Since $\Lambda'$ is finite free over $\Lambda$, the functor $-\otimes_\Lambda\Lambda'$ commutes with specialization and with $F_B$.
Since $\Lambda'\to \Lambda$ is faithfully flat,  hence $-\otimes_\Lambda\Lambda'$ preserves  supports and  we have
	\begin{align}\label{eq:SSmuBbaseLambda}
		SS_{\mu,B}(\mathcal F)=SS_{\mu,B}(\mathcal F_{\Lambda'}).
	\end{align}
	By Proposition~\ref{prverdier}.1, the object $\nu{\cal H}om(\mathcal F_{\Lambda'},
	\mathcal F_{\Lambda'})$ is monodromic.
	By \cite[Corollary 4.5]{Wang}, the restrictions of $F_B$ and
	$F_\psi$ to geometric monodromic complexes over an algebraic closure of $k$ are naturally
	isomorphic, although not canonically. Thus we have
	\begin{align}
{\rm supp}_k(\mu_B{\cal H}om(\mathcal F_{\Lambda'},\mathcal F_{\Lambda'}))=	{\rm supp}_k\bigl(
	\mu{\cal H}om_{\psi,\Lambda'}
	(\mathcal F_{\Lambda'},\mathcal F_{\Lambda'})\bigr).
	\end{align}
Together with \eqref{eq:SSmuBbaseLambda} and \eqref{eq:coeff-microlocal-support}, we get
	\eqref{eq:Wang-support-comparison}.
\end{proof}

\subsection{Inclusion of singular supports}\label{sec:inclusion-of-singluar-supports}

In this subsection, we prove the inclusion $SS_\mu\subseteq SS$.
We note the following fact. Let $X$ be a smooth $k$-scheme, and let  $\mathcal F\in D_{\rm ctf}(X,\Lambda)$. 
Let $\Lambda\to \Lambda'$ be  any finite \'etale faithfully flat local extension. 
Since the functor $-\otimes_\Lambda \Lambda^\prime$ detects the vanishing of the vanishing cycles functor,
by \cite[Theorem 1.5]{SS}, we have 
\begin{align}
	SS(\mathcal F)=SS(\mathcal F\otimes_\Lambda\Lambda^\prime).
\end{align}
In order to prove $SS_\mu\subseteq SS$,
by Lemma \ref{lm:coeff-character} and Proposition \ref{pr:coeff-microlocal-support}, we may assume there exists a non-trivial character $\psi: \mathbf F_p\to \Lambda^\times$.

We first reduce the assertion to the case of projective spaces.
\begin{definition}
{\rm 1.} 
Let $X$ be a smooth $k$-scheme. We say that a complex ${\mathcal F}\in D_{\rm ctf}(X,\Lambda)$ satisfies the singular support inclusion if
\begin{equation}
	{SS}_\mu({\mathcal F})\subseteq{SS}({\mathcal F}).
\end{equation}

{\rm 2.}
Let $P$ be a projective space over $k$, and retain the notation of Definition~\ref{def:radontransform}. Set
\begin{equation}
	\rho:(T^*P)^\times=T^*P\setminus {T_P^*P}\to\mathbf{P}(T^*P)\overset{\eqref{eq:LegendreTransform}}{\simeq} Q,
\end{equation}
which is the composite of the projectivization map and the Legendre transform.
We say that a complex ${\mathcal F}\in D_{\rm ctf}(P,\Lambda)$ satisfies the projective singular support inclusion if
\begin{equation}
	{SS}_\mu({\mathcal F})^\times={SS}_\mu({\mathcal F})\setminus {T_P^*P}\subseteq\rho^{-1}(E_p(p^{\vee\ast}\mathrm{Rad}({\mathcal F})))\overset{\eqref{eq:Beilinson-PSS}}{=}\rho^{-1}(\mathbf P({SS}({\mathcal F}))).
\end{equation}
\end{definition}

\begin{lemma}\label{lem:reducetoprojectivespaces}
Assume that, for every $n\geq 1$, every $(n+1)$-dimensional vector space $V$, and every ${\mathcal F}\in D_{\rm ctf}(\mathbf P(V),\Lambda)$, the complex ${\mathcal F}$ satisfies the projective singular support inclusion.

Then, for every smooth scheme $X$ over $k$ and every ${\mathcal F}\in D_{\rm ctf}(X,\Lambda)$, the complex ${\mathcal F}$ satisfies the singular support inclusion. Equivalently, Theorem~\ref{thm:singular-support-inclusion} holds.
\end{lemma}
\begin{proof}
When ${\rm dim}X\leq 1$, Theorem \ref{thm:singular-support-inclusion} follows from Lemma~\ref{lmlcc}.2 and Lemma \ref{lmlcc}.3. 

Now assume ${\rm dim}X>1$.
Combining Lemmas~\ref{lmimm} and~\ref{lmsm} with \cite[Lemmas 2.3 and 2.5]{SS}, we may assume that $X=P=\mathbf P(V)$ is a projective space.
Let ${\mathcal F}\in D_{\rm ctf}(P,\Lambda)$.
Since ${\mathcal F}$ satisfies the projective singular support inclusion, and both ${SS}_\mu({\mathcal F})$ and ${SS}({\mathcal F})$ are conical (cf. Lemma~\ref{lmlcc} and \cite[1.3]{SS}), we have
\begin{equation}\label{eq:singularsupportinclusionnon-zeropart}
{SS}_\mu({\mathcal F})^\times\subseteq{SS}({\mathcal F})^\times.
\end{equation}
By \cite[Lemma 2.3]{SS}, the base of ${SS}({\mathcal F})$ is $\mathrm{supp}({\mathcal F})$. If ${\mathcal F}$ vanishes on an open neighborhood $U$ of $x\in P$, then $\mathscr{H}_{\mathcal F}$ vanishes on $U\times U$. Consequently, $\nu\mathcal{H}om_X({\mathcal F},{\mathcal F})$ vanishes over $TU$ and $\mu\mathcal{H}om_X({\mathcal F},{\mathcal F})$ vanishes over $T^*U$.
It follows that
\begin{equation}\label{eq:singularsupportinclusionzerosectionpart}
{SS}_\mu({\mathcal F})\cap T_P^\ast P\subseteq\mathrm{supp}({\mathcal F})={SS}({\mathcal F})\cap T_P^\ast P.
\end{equation}
Now the result follows from \eqref{eq:singularsupportinclusionnon-zeropart} and \eqref{eq:singularsupportinclusionzerosectionpart}.
\end{proof}

For the remainder of this subsection, let $P$ be a projective space over $k$ and let ${\mathcal F}\in D_{\rm ctf}(P,\Lambda)$.
Retain the notation of Subsection~\ref{subsec:Radon-transform-notations}, and set
\begin{equation}
{\mathcal G}=p^{\vee\ast}\mathrm{Rad}({\mathcal F})[n-1](n-1)\in D_{\rm ctf}(Q,\Lambda).
\end{equation}
Universal local acyclicity is unaffected by cohomological shifts and Tate twists. Therefore, by Lemma \ref{lem:Radon-criterion}, we have
\begin{equation}
E_p(p^{\vee\ast}\mathrm{Rad}({\mathcal F}))=E_p({\mathcal G}).
\end{equation}

\begin{lemma}\label{lem:nuH_G-is-relatively-constant-in-ULA-locus}
	Recall the closed immersion
	$i_p:Q\times_P Q\hookrightarrow Q\times Q$,
	and use the notation of Convention~\ref{con:delignetopoi}.
Let $U\subseteq Q$ be an open subscheme. If $p|_U:U\to P$ is  universally locally acyclic relatively to ${\mathcal G}|_U$, then the natural morphism
\begin{equation}\label{eq:specializationHGisisomorphism}
i_p^*\mathscr{H}_{\mathcal G}\boxtimes\Lambda_{TP}\to\nu_{i_p}\mathscr{H}_{\mathcal G}\quad\text{in }D_{\rm ctf}((Q\times_PQ\times_PTP)\times_k\eta_0,\Lambda)
\end{equation}
restricts to an isomorphism over $U\times_PU\times_PTP$.
\end{lemma}
\begin{proof}
By \cite[Theorem 2.16]{LuZheng}, the morphism $p|_U\times p|_U:U\times U\to P\times P$ is universally locally acyclic relatively to the complex
\begin{equation}
\mathscr{H}_{\mathcal G}|_{U\times U}\simeq {D}_U({\mathcal G}|_U)\boxtimes({\mathcal G}|_U)\in D_{\rm ctf}(U\times U,\Lambda).
\end{equation}
Applying Lemma \ref{lem:specialization-in-ULA-case} to the cartesian diagram
\begin{equation}
\begin{tikzcd}[column sep=40]
	U\times_PU\arrow[r,hook,"i_p|_{U\times_PU}"]\arrow[d,"h_p|_{U\times_PU}"swap]&U\times U\arrow[d,"p|_U\times p|_U"]\\
	P\arrow[r,hook,"\delta_P"]&P\times P,
\end{tikzcd}
\end{equation}
we obtain natural isomorphisms
\begin{equation}
\begin{aligned}
	&\nu_{i_p}\mathscr{H}_{\mathcal G}|_{U\times_PU\times_PTP}\simeq\nu_{i_p|_{U\times_PU}}(\mathscr{H}_{\mathcal G}|_{U\times U})\\
	\simeq&(\mathscr{H}_{\mathcal G}|_{U\times U})|_{U\times_PU\times_PTP}\simeq (i_p^*\mathscr{H}_{\mathcal G}\boxtimes\Lambda_{TP})|_{U\times_PU\times_PTP}.
\end{aligned}
\end{equation}
\end{proof}

Recall the notation $h_p: Q\times_PQ\to P$ in \eqref{eq:prphpp} and that $h_p\times\mathrm{id}_{TP}$ in \eqref{eq:hpidtpDef} denotes the projection
\begin{equation}
h_p\times\mathrm{id}_{TP}={\rm pr}_{TP}:Q\times_PQ\times_PTP\to TP.
\end{equation}
Let $h_p\times {\rm id}_{T^\ast P}$ be the projection
\begin{equation}
h_p\times\mathrm{id}_{T^\ast P}={\rm pr}_{T^\ast P}:Q\times_PQ\times_PT^\ast P\to T^\ast P.
\end{equation}
\begin{lemma}\label{lem:vanishofFouriertransform}
Let $U\subseteq Q$ be an open subscheme. Suppose that $p|_U$ is universally locally acyclic relatively to ${\mathcal G}|_U$. Then the complex
\begin{equation}
F_{\psi,TP/P}((h_p\times\mathrm{id}_{TP})_*\nu_{i_p}\mathscr{H}_{\mathcal G})
\simeq (h_p\times\mathrm{id}_{T^\ast P})_\ast F_{\psi,Q\times_PQ\times_PT P/Q\times_PQ}(\nu_{i_p}\mathscr{H}_{\mathcal G})
\end{equation}
vanishes over $\rho^{-1}(U)$.
\end{lemma}
\begin{proof}
There are canonical isomorphisms
\begin{equation}
\begin{aligned}
	{D}_Q{\mathcal G}&={D}_Q((p^\vee)^*\mathrm{Rad}({\mathcal F})[n-1](n-1))\simeq (p^\vee)^!{D}_{P^\vee}\mathrm{Rad}({\mathcal F})[1-n](1-n)\\
	&\simeq (p^\vee)^*{D}_{P^\vee}\mathrm{Rad}({\mathcal F})[n-1],\\
\mathscr{H}_{\mathcal G}&\simeq {D}_Q{\mathcal G}\boxtimes {\mathcal G}\simeq (p^\vee\times p^\vee)^*({D}_{P^\vee}\mathrm{Rad}({\mathcal F})\boxtimes\mathrm{Rad}({\mathcal F}))[2n-2](n-1).
\end{aligned}
\end{equation}
By Lemma~\ref{lem:calculatekerbetavee}, the complex $\nu_{i_p}\mathscr{H}_{\mathcal G}$ is $W$-invariant under the action of $W$ (cf. \eqref{eq:actionofWonTP}).
Let $\Sigma:={\rm supp}_k(F_{\psi,Q\times_PQ\times_PT P/Q\times_PQ}(\nu_{i_p}\mathscr{H}_{\mathcal G})) \subseteq Q\times_PQ\times_PT^\ast P$ (cf. \eqref{eq:defSuppPrime}).
By Proposition~\ref{prop:Fouriersupportofinvariantsheaf} and Lemma~\ref{lem:calculatekerbetavee}, we have
\begin{equation}\label{eq:supportofFourierclosedpart}
\Sigma	\subseteq\ker(\beta^\vee)=Q\times_PQ\times_P 0_{T^\ast P}\cup\mathcal L.
\end{equation}
Let $J: U\times_P U\hookrightarrow  Q\times_PQ$ be the inclusion, and let
\begin{align}
J_{TP}=J\times\mathrm{id}_{TP}:U\times_PU\times_PTP\hookrightarrow Q\times_PQ\times_PTP,
\end{align}
be its base change.
By Lemma~\ref{lem:nuH_G-is-relatively-constant-in-ULA-locus}, the complex
\begin{equation}
J_{TP}^*\nu_{i_p}\mathscr{H}_{\mathcal G}=(i_p^*\mathscr{H}_{\mathcal G})|_{U\times_PU}\boxtimes\Lambda_{TP}
\end{equation}
is the pull-back of $i_p^\ast\mathscr{H}_{\mathcal G}|_{U\times_PU}$
by the projection $U\times_PU\times_PTP\to U\times_PU$. Hence Lemma~\ref{lmF0}.3 gives
\begin{equation}\label{eq:supportofFourieropenpart}
\Sigma|_{U\times_PU}=	\mathrm{supp}_k(F_{\psi,U\times_PU\times_PTP/U\times_PU}(J_{TP}^*\nu_{i_p}\mathscr{H}_{\mathcal G}))\subseteq U\times_PU\times_P0_{T^*P}.
\end{equation}
By \eqref{eq:supportofFourierclosedpart} and \eqref{eq:supportofFourieropenpart}, we see that
\begin{align}
\Sigma\subseteq Q\times_PQ\times_P0_{T^\ast P}\cup \mathcal L|_{\delta_p(Q\setminus U)}.
\end{align}
Hence $\Sigma$ is disjoint from $Q\times_P Q\times_P \rho^{-1}(U)$.
\end{proof}
\begin{theorem}\label{thm:singular-support-inclusion}
Let $\Lambda$ be a finite local ring killed by a power of a prime $\ell\ne p$. Let $X$ be a smooth $k$-scheme, and let ${\mathcal F}\in D_{\rm ctf}(X,\Lambda)$. Then
\begin{equation}\label{eq:muSSsubseteqSS}
{SS}_\mu({\mathcal F})\subseteq{SS}({\mathcal F}).
\end{equation}
\end{theorem}
\begin{proof}
By Lemma~\ref{lem:reducetoprojectivespaces}, it suffices to prove that, for every projective space $P$ and every ${\mathcal F}\in D_{\rm ctf}(P,\Lambda)$, the complex ${\mathcal F}$ satisfies the projective singular support inclusion; that is,
\begin{equation}
{SS}_\mu({\mathcal F})^\times={SS}_\mu({\mathcal F})\setminus {T_P^*P}\subseteq\rho^{-1}(E_p(p^{\vee\ast}\mathrm{Rad}({\mathcal F}))).
\end{equation}
Set ${\mathcal G}=p^{\vee\ast}\mathrm{Rad}({\mathcal F})[n-1](n-1)$.
Let $U\subseteq Q$ be an open subscheme disjoint from $E_p(p^{\vee\ast}\mathrm{Rad}({\mathcal F}))$. Then $p|_U$ is universally locally acyclic relatively to ${\mathcal G}|_U$.
By Lemma~\ref{lem:vanishofFouriertransform}, the complex
\begin{equation}\label{eq:asheafwhoseretractismuHom}
F_{\psi,TP/P}((h_p\times\mathrm{id}_{TP})_*\nu_{i_p}\mathscr{H}_{\mathcal G})
\end{equation}
vanishes over $\rho^{-1}(U)$.
By Corollary~\ref{cor:Radon-retract-of-nuHom}, the complex
\begin{equation}
\mu\mathcal{H}om_P({\mathcal F},{\mathcal F})=F_{\psi,TP/P}\nu_{\delta_P}\mathscr{H}_{\mathcal F}
\end{equation}
is a retract of \eqref{eq:asheafwhoseretractismuHom}, and therefore also vanishes over $\rho^{-1}(U)$.
Therefore ${SS}_\mu({\mathcal F})$ is disjoint from $\rho^{-1}(U)$.
\end{proof}

\begin{theorem}\label{thm:support-of-microlocalization-to-closed-subscheme}
Let $i:Z\hookrightarrow X$ be a closed immersion of smooth schemes and let ${\mathcal F}\in D_{\rm ctf}(X,\Lambda)$. Then we have
\begin{equation}\label{eq:suppprimemuzx}
\mathrm{supp}_k(\mu_{Z/X}({\mathcal F}))\subseteq{SS}_\mu({\mathcal F})\cap T^*_ZX\subseteq{SS}({\mathcal F})\cap T^*_ZX,
\end{equation}
where $\mathrm{supp}_k(\mu_{Z/X}({\mathcal F}))\subseteq T_Z^\ast X$ is the unique closed subset such that $(\mathrm{supp}_k(\mu_{Z/X}({\mathcal F})))_{\bar k}=\mathrm{supp}(\mu_{Z/X}({\mathcal F}))$ (cf. \eqref{eq:defSuppPrime}).
\end{theorem}	
\begin{proof}
Let $U\subseteq T^*X$ be an open subset where $\mu\mathcal{H}om_X({\mathcal F},{\mathcal F})$ vanishes.
By Proposition \ref{prop:mu_Z/X-as-a-retract-of-muHom}, $\mu_{Z/X}({\mathcal F})$ is a retract of
\begin{equation}
(q^\vee)^*(\mu\mathcal{H}om_X({\mathcal F},{\mathcal F})|_{T^*X|_Z})\otimes\mu_{Z/X}({\mathcal F})
\end{equation}
which vanishes over $q^{\vee-1}(U)=U\cap T^*_ZX$ (note that $q^\vee: T^\ast_ZX\hookrightarrow T^\ast X|_Z$).
It follows that
\begin{align}
\begin{aligned}
	\mathrm{supp}_k(\mu_{Z/X}({\mathcal F}))&\subseteq\mathrm{supp}_k(\mu\mathcal{H}om_X({\mathcal F},{\mathcal F}))\cap T^*_ZX\\
	&={SS}_\mu({\mathcal F})\cap T^*_ZX\subseteq {SS}({\mathcal F})\cap T^\ast_ZX,
\end{aligned}
\end{align}
where the last inclusion follows from \eqref{eq:muSSsubseteqSS}.
\end{proof}

\subsection{Equality of singular supports}\label{sec:comparison-on-singular-supports}
In this subsection, we show that the inclusion $SS_\mu\subseteq SS$ in Theorem \ref{thm:singular-support-inclusion} is always an equality.
As in the beginning  of Subsection \ref{sec:inclusion-of-singluar-supports},
we may assume there exists a non-trivial character $\psi: \mathbf F_p\to \Lambda^\times$.

\begin{lemma}\label{lem:Kunneth-for-muHom}
Let $X,Y$ be smooth $k$-schemes, ${\mathcal F}_1,{\mathcal F}_2\in D_{\mathrm{ctf}}(X,\Lambda)$ and ${\mathcal G}_1,{\mathcal G}_2\in D_{\mathrm{ctf}}(Y,\Lambda)$. We have  canonical isomorphisms
\begin{align}
\label{eq:Kunneth-for-RHom}	&
\mathscr{H}_{X\times Y}({\mathcal F}_1\boxtimes {\mathcal G}_1,{\mathcal F}_2\boxtimes {\mathcal G}_2)\simeq \mathscr{H}_X({\mathcal F}_1,{\mathcal F}_2)\boxtimes \mathscr{H}_Y({\mathcal G}_1,{\mathcal G}_2),\\
\label{eq:Kunneth-for-nuHom}			&\nu\mathcal{H}om_X({\mathcal F}_1,{\mathcal F}_2)\boxtimes\nu\mathcal{H}om_Y({\mathcal G}_1,{\mathcal G}_2)\simeq \nu\mathcal{H}om_{X\times Y}({\mathcal F}_1\boxtimes {\mathcal G}_1,{\mathcal F}_2\boxtimes {\mathcal G}_2),\\
\label{eq:Kunneth-for-muHom}			&\mu\mathcal{H}om_X({\mathcal F}_1,{\mathcal F}_2)\boxtimes\mu\mathcal{H}om_Y({\mathcal G}_1,{\mathcal G}_2)\simeq \mu\mathcal{H}om_{X\times Y}({\mathcal F}_1\boxtimes {\mathcal G}_1,{\mathcal F}_2\boxtimes {\mathcal G}_2).
\end{align}
Consequently, for any ${\mathcal F}\in D_{\rm ctf}(X,\Lambda)$ and ${\mathcal G}\in D_{\rm ctf}(Y,\Lambda)$, we have
\begin{equation}
{SS}_\mu({\mathcal F}\boxtimes {\mathcal G})={SS}_\mu({\mathcal F})\times{SS}_\mu({\mathcal G}).
\end{equation}
\end{lemma}
\begin{proof}
The isomorphism \eqref{eq:Kunneth-for-RHom} follows from \cite[Proposition 2.11]{LuZheng}.
Applying Proposition \ref{prop:product-of-deformation-spaces} to $\delta_X$ and $\delta_Y$, we obtain an isomorphism
\begin{equation}
A_{X\times Y}(X\times X\times Y\times Y)\simeq A_X(X\times X)\times_{\mathbf{A}^1}A_Y(Y\times Y).
\end{equation}
Applying the 
specialization functor to the isomorphism \eqref{eq:Kunneth-for-RHom} along $\delta_{X\times Y}=\delta_X\times\delta_Y:X\times Y\to X\times X\times Y\times Y$ and then applying \cite[Proposition 3.1]{LuZheng}, we obtain the isomorphism \eqref{eq:Kunneth-for-nuHom}.
By Proposition \ref{prop:Kunneth-for-Fourier}, we obtain natural isomorphisms
\begin{equation}
\begin{aligned}
	&\mu\mathcal{H}om_{X\times Y}({\mathcal F}_1\boxtimes {\mathcal G}_1,{\mathcal F}_2\boxtimes {\mathcal G}_2)\\
	&\simeq F_{\psi,T(X\times Y)/X\times Y}(\nu\mathcal{H}om_{X\times Y}({\mathcal F}_1\boxtimes {\mathcal G}_1,{\mathcal F}_2\boxtimes {\mathcal G}_2))\\
	&\overset{\eqref{eq:Kunneth-for-nuHom}}{\simeq}F_{\psi,TX\times TY/X\times Y}(\nu\mathcal{H}om_X({\mathcal F}_1,{\mathcal F}_2)\boxtimes\nu\mathcal{H}om_Y({\mathcal G}_1,{\mathcal G}_2))\\
	&\overset{\eqref{eq:Fourier-product}}{\simeq}F_{\psi,TX/X}(\nu\mathcal{H}om_X({\mathcal F}_1,{\mathcal F}_2))\boxtimes F_{\psi,TY/Y}(\nu\mathcal{H}om_Y({\mathcal G}_1,{\mathcal G}_2))\\
	&\simeq\mu\mathcal{H}om_X({\mathcal F}_1,{\mathcal F}_2)\boxtimes\mu\mathcal{H}om_Y({\mathcal G}_1,{\mathcal G}_2).
\end{aligned}
\end{equation}
\end{proof}

\begin{lemma}\label{lem:mu-supported-on-zero-section-implies-transversal}
Let $i:Z\hookrightarrow X$ be a closed immersion of smooth $k$-schemes, and ${\mathcal F}\in D_{\mathrm{ctf}}(X,\Lambda)$. 
If ${\rm supp}_k(\mu_{Z/X}({\mathcal F}))$ (cf. \eqref{eq:suppprimemuzx}) is contained in the image of the zero section $0^\vee:Z\hookrightarrow T^*_ZX$, then $i$ is ${\mathcal F}$-transversal.
\end{lemma}
\begin{proof}
By Proposition~\ref{prFF}, there exists a complex ${\mathcal G}\in D_{\mathrm{ctf}}(Z\times_k\eta_0,\Lambda)$ such that
\begin{equation}
\nu_{Z/X}({\mathcal F})\simeq p^*{\mathcal G}.
\end{equation}
Applying \cite[Lemma 8.6.3]{CC} to $p:T_ZX\to Z$, we obtain that $0$ is $\nu_{Z/X}({\mathcal F})$-transversal. We conclude by Proposition \ref{prop:specialization-detect-transversality}.
\end{proof}

\begin{lemma}\label{lem:base-of-SSmu}
The base of ${SS}_\mu({\mathcal F})$ is equal to $\mathrm{supp}({\mathcal F})$.
\end{lemma}
\begin{proof}
Outside the inverse image of $\mathrm{supp}({\mathcal F})$, $\mu\mathcal Hom(\mathcal F,\mathcal F)$ is zero. Thus
the base of ${SS}_\mu({\mathcal F})$ is contained in  $\mathrm{supp}({\mathcal F})$.
Conversely, let $U\subseteq X$ be an open subset outside of the base of ${SS}_\mu({\mathcal F})$. Then $\mu\mathcal{H}om_X({\mathcal F},{\mathcal F})$ vanishes over $T^*U$, and hence
\begin{equation}
\nu\mathcal{H}om_U({\mathcal F}|_U,{\mathcal F}|_U)=0.
\end{equation}
By Proposition~\ref{prverdier},
\begin{equation}
\begin{aligned}
	0=0_{TU}^!\nu\mathcal{H}om_U({\mathcal F}|_U,{\mathcal F}|_U)&\simeq\delta_U^!\mathcal{H}om_{U\times U}(\mathrm{pr}_{U,1}^*{\mathcal F}|_U,\mathrm{pr}_{U,2}^!{\mathcal F}|_U)\\
	&\simeq \mathcal{H}om_U({\mathcal F}|_U,{\mathcal F}|_U).
\end{aligned}
\end{equation}
It follows that ${\mathcal F}|_U=0$.
\end{proof}

Let $X$ be a smooth $k$-scheme and $f:X\to\mathbf{A}^1$ be a function. Consider the cartesian diagram
\begin{equation}
\begin{tikzcd}
X\times_{\mathbf{A}^1}X\arrow[r,hook,"i_f"]\arrow[d,"h_f"]&X\times X\arrow[d,"f\times f"]\\
\mathbf{A}^1\arrow[r,hook,"\delta_{\mathbf{A}^1}"]&\mathbf{A}^1\times\mathbf{A}^1.
\end{tikzcd}
\end{equation}
We denote by
\begin{equation}
\mathrm{pr}_{f,1},\mathrm{pr}_{f,2}:X\times_{\mathbf{A}^1}X\to X
\end{equation}
the projections over $\mathbf{A}^1$, and by
\begin{equation}
\mathrm{pr}_1,\mathrm{pr}_2:X\times X\to X
\end{equation}
the projections over $k$.
\begin{lemma}\label{lem:ULA-equivalent-transversal}
Let $X$ be a smooth $k$-scheme and $f:X\to\mathbf{A}^1$ be a smooth morphism.
Let ${\mathcal F}\in D_{\mathrm{ctf}}(X,\Lambda)$. Recall that
$\mathscr{H}_{\mathcal F}=\mathcal{H}om_{X\times X}(\mathrm{pr}_1^*{\mathcal F},\mathrm{pr}_2^!{\mathcal F})={D}_X({\mathcal F})\boxtimes {\mathcal F}$.
Then $f$ is universally locally acyclic with respect to ${\mathcal F}$ if and only if $i_f$ is $\mathscr{H}_{\mathcal F}$-transversal.
\end{lemma}
\begin{proof}
We have a commutative diagram (cf. \cite[(3.54)]{YZ})
\begin{equation}
\begin{tikzcd}
	i_f^*\mathcal{H}om_{X\times X}(\mathrm{pr}_1^*{\mathcal F},\mathrm{pr}_2^!{\mathcal F})\otimes i_f^!\Lambda\arrow[r,"c_{i_f,\mathscr{H}_{\mathcal F}}"]&i_f^!\mathcal{H}om_{X\times X}(\mathrm{pr}_1^*{\mathcal F},\mathrm{pr}_2^!{\mathcal F})\arrow[dd,"\simeq"]\\
	i_f^*({D}_X({\mathcal F})\boxtimes {\mathcal F})\otimes i_f^!\Lambda\arrow[u,"\simeq"]\arrow[d,"\simeq"]\\
	{D}_{X/\mathbf{A}^1}({\mathcal F})\boxtimes_{\mathbf{A}^1}{\mathcal F}\arrow[r]&\mathcal{H}om_{X\times_{\mathbf{A}^1}X}(\mathrm{pr}_{f,1}^*{\mathcal F},\mathrm{pr}_{f,2}^!{\mathcal F}).
\end{tikzcd}
\end{equation}
We conclude using \cite[Theorem 2.16]{LuZheng}.
\end{proof}

\begin{lemma}\label{lem:SSmu-detects-ULA}
Let $X$ be a smooth $k$-scheme and ${\mathcal F}\in D_{\mathrm{ctf}}(X,\Lambda)$. Let $f:X\to\mathbf{A}^1$ be a smooth function. Let $u$ be the coordinate on $\mathbf{A}^1$. It induces a morphism
\begin{equation}
du:\mathbf{A}^1\to T^*\mathbf{A}^1\simeq\mathbf{A}^1\times\mathbf{A}^1.
\end{equation}
Consider the composition
\begin{equation}
X\xrightarrow{du}X\times_{\mathbf{A}^1}T^*\mathbf{A}^1\xrightarrow{df} T^*X,
\end{equation}
which is also denoted by $df:X\to T^*X$.
Assume that
\begin{equation}\label{eq:transversal-condition-for-SSmu}
df(X)\cap{SS}_\mu({\mathcal F})=\emptyset.
\end{equation}
Then $f$ is universally locally acyclic relatively to ${\mathcal F}$.
\end{lemma}
\begin{proof}
By Lemma \ref{lem:Kunneth-for-muHom}, we have
\begin{equation}
{SS}_\mu(\mathscr{H}_{\mathcal F})={SS}_\mu({D}_X{\mathcal F})\times{SS}_\mu({\mathcal F}).
\end{equation}
As $f$ is smooth, $X\times_{\mathbf{A}^1}X$ is smooth and $i_f$ is pure of codimension $1$. Consider the conormal bundle
\begin{equation}
T^*_{X\times_{\mathbf{A}^1}X}(X\times X)\simeq X\times_{\mathbf{A}^1}X\times_{\mathbf{A}^1}T^*\mathbf{A}^1\simeq X\times_{\mathbf{A}^1}X\times\mathbf{A}^1.
\end{equation}
Since ${SS}_\mu({\mathcal F})$ is conical, the condition \eqref{eq:transversal-condition-for-SSmu} implies that
\begin{equation}
{SS}_\mu(\mathscr{H}_{\mathcal F})\cap T^*_{X\times_{\mathbf{A}^1}X}(X\times X)\subseteq X\times_{\mathbf{A}^1}X\times_{\mathbf{A}^1}0_{T^*\mathbf{A}^1}.
\end{equation}
Then by Theorem \ref{thm:support-of-microlocalization-to-closed-subscheme}, we obtain that
\begin{equation}
\mathrm{supp}_k(\mu_{i_f}(\mathscr{H}_{\mathcal F}))\subseteq X\times_{\mathbf{A}^1}X\times_{\mathbf{A}^1}0_{T^*\mathbf{A}^1}.
\end{equation}
Applying Lemma \ref{lem:mu-supported-on-zero-section-implies-transversal} to $i_f$ and $\mathscr{H}_{\mathcal F}$, we know that $i_f$ is $\mathscr{H}_{\mathcal F}$-transversal. Now the result follows from
Lemma \ref{lem:ULA-equivalent-transversal}.
\end{proof}
\begin{theorem}\label{thm:singular-support-equality}
Let $\Lambda$ be a finite local ring killed by a power of a prime $\ell\ne p$. Let $X$ be a smooth $k$-scheme, and let ${\mathcal F}\in D_{\rm ctf}(X,\Lambda)$. Then we have
\begin{equation}\label{eq:muSS-equal-SS}
{SS}_\mu({\mathcal F})={SS}({\mathcal F}).
\end{equation}
\end{theorem}
\begin{proof}
By Theorem \ref{thm:singular-support-inclusion}, \cite[Theorem 1.5]{SS} and \cite[Theorem 4.4]{CC}, it suffices to show that
\begin{equation}
{SS}^w({\mathcal F})=SS(\mathcal F)\subseteq {SS}_\mu({\mathcal F}),
\end{equation}
where ${SS}^w({\mathcal F})$ is defined in \cite[1.5]{SS}.
In the following, if necessary, we may take a finite extension of $k$ by Lemma \ref{lmsm} and \cite[Theorem 1.4(iii)]{SS}.
Consider a test function $f:U\to\mathbf{A}^1$ on $X$ and a point $x\in U$, such that
\begin{equation}
(x,df(x))\notin{SS}_\mu({\mathcal F}|_U).
\end{equation}
If $x\notin\mathrm{supp}({\mathcal F})$, then ${\mathcal F}$ vanishes near $x$ and hence $f$ is universally locally acyclic relatively to ${\mathcal F}$ at $x$.
Otherwise, we may assume by Lemma \ref{lem:base-of-SSmu} that $df(x)\neq0$. After shrinking $U$, we may assume that $f:U\to\mathbf{A}^1$ is smooth and
\begin{equation}
df(U)\cap{SS}_\mu({\mathcal F}|_U)=\emptyset.
\end{equation}
By Lemma \ref{lem:SSmu-detects-ULA}, we conclude that $f$ is universally locally acyclic relatively to ${\mathcal F}$ at $x$. 
Thus by definition of $SS^w(\mathcal F)$, we have $(x,df(x))\notin{SS}^w({\mathcal F}).$
\end{proof}

\section{
Characteristic cycles
and microlocalization}\label{SXY:saito-section2}

In this section, let $X$ be a separated smooth scheme
of finite type
over a field $k$ of characteristic $p>0$
and $\Lambda$ be a finite local ring with
residue characteristic $\ell\neq p$.
We fix a non-trivial character
$\psi\colon {\mathbf F}_p\to \Lambda^\times$.

\subsection{Construction}\label{susec:constructionOfCCmu}

Let ${\cal F}\in D_{\rm ctf}(X,\Lambda)$
and
let ${\mathscr H}$ denote
${\cal H}om({\rm pr}_1^*{\cal F},{\rm pr}_2^!{\cal F})
\in D_{\rm ctf}(X\times X,\Lambda)$.
By \cite[(3.2.1)]{SGA5},
we have a canonical isomorphism
\begin{equation}
{\cal H}om({\cal F},{\cal F})
\to \delta^!{\mathscr H}.
\label{eqdel!}
\end{equation}
Hence
the identity $1_{\cal F}$
defines a canonical morphism
\begin{equation}
\Lambda 
\to \delta^!{\mathscr H}.
\label{eq1}
\end{equation}
Let $\boxtimes\colon 
D_{\rm ctf}(X,\Lambda)
\times
D_{\rm ctf}(X,\Lambda)
\to
D_{\rm ctf}(X\times X,\Lambda)$
denote the bifunctor
${\rm pr}_1^*-\otimes {\rm pr}_2^*-$.
The canonical isomorphism
$D_X{\cal F}\boxtimes{\cal F} \to {\mathscr H}$
\cite[(3.1.1)]{SGA5}
on $X\times X$
induces an isomorphism
\begin{equation}
D_X{\cal F}\otimes {\cal F}
\to 
\delta^*{\mathscr H}.
\label{eqdel*}
\end{equation}
Hence the evaluation morphism
$D_X{\cal F}\otimes {\cal F}
\to 
K_X$
defines a canonical morphism
\begin{equation}
\delta^*{\mathscr H}
\to K_X.
\label{eqev}
\end{equation}
The morphisms
(\ref{eq1}), (\ref{eqev})
and the canonical morphism
$\delta^!\to \delta^*$
define
\begin{equation}
\Lambda\to
\delta^!{\mathscr H}
\to 
\delta^*{\mathscr H}
\to K_X.
\label{eqccF}
\end{equation}

\begin{df}[{\cite[Definition 2.1.1]{AS}}]\label{dfccF}
The canonical 
class
$C_{X/k}({\cal F})\in {\rm H}^0(X,K_X)$
is defined by the composition
(\ref{eqccF}).
\end{df}

Under the identification
${\rm H}^0(X,K_X)
={\rm H}^0(T^*X,e^{\vee*}K_X)$,
the (cohomological) characteristic class
$C_{X/k}({\cal F})$
is known to be the cycle class
of $CC\, {\cal F}$ under the assumption
that $X$ is quasi-projective
\cite{YZ}.

The adjoint of (\ref{eqccF})
defines
\begin{equation}
\delta_!\Lambda\to
{\mathscr H}
\to \delta_*K_X.
\label{eqccFa}
\end{equation}
Taking the specialization and
the microlocalization,
we obtain
\begin{align}
&0_!\Lambda\to
\nu{\cal H}om({\cal F},{\cal F})
\to 0_*K_X,
\label{eqccFnm1}\\
&\Lambda\to
\mu{\cal H}om({\cal F},{\cal F})
\to e^{\vee*}K_X.
\label{eqccFnm}
\end{align}
Alternatively,
by Proposition \ref{prverdier}.2,
(\ref{eq1}) and (\ref{eqev}) define
the first and the last arrows in
\begin{equation}
\Lambda\to
0^!\nu_{X/X\times X}{\mathscr H}
\to 
0^*\nu_{X/X\times X}{\mathscr H}
\to K_X.
\label{eqnuF}
\end{equation}
The middle arrow is defined by
the morphism $0^!\to 0^*$ of functors.
Taking the adjoint, we obtain
(\ref{eqccFnm1}).

Let $i\colon SS_\mu{\cal F}\to T^*X$ be
the closed immersion, and let
$q:T^*X_{\bar k}\longrightarrow T^*X$
be the projection.
Since $(SS_\mu{\cal F})_{\bar k}\subseteq T^*X_{\bar k}$ is
 the support of
$\mu{\cal H}om({\cal F},{\cal F})$,
(\ref{eqccFnm})
defines
\begin{equation}
\Lambda\to
i^*\mu{\cal H}om({\cal F},{\cal F})
\to i^!e^{\vee*}K_X
\label{eqccFnmi}
\end{equation}
on $(SS_\mu{\cal F})_{\bar k}$
by adjunction. By the surjectivity  of $G_{\eta_0}\to G_{k}$,
the composition of \eqref{eqccFnmi} defines a $G_k$-invariant cohomology class in ${\rm H}^0_{(SS_\mu{\cal F})_{\bar k}}
(T^*X_{\bar k},q^\ast e^{\vee*}K_X)$.

Let $A$ and $A_{\bar k}$ denote the sets of irreducible components of dimension ${\rm dim}X$ of $SS_\mu(\mathcal F)$ and of $SS_\mu(\mathcal F)_{\bar k}$ respectively.
Applying semi-purity
\cite[Rappel 2.2.8]{Cycle} to each connected  component of $X$, we have canonical isomorphisms
\begin{align}
\begin{aligned}
	&{\rm H}^0_{SS_\mu{\cal F}}
	(T^*X,e^{\vee*}K_X)\simeq \bigoplus_{C_a\in A}\Lambda\cdot{\rm cl}[C_a],\\
	&{\rm H}^0_{(SS_\mu{\cal F})_{\bar k}}
	(T^*X_{\bar k},q^\ast e^{\vee*}K_X)\simeq \bigoplus_{C_b\in A_{\bar k}}\Lambda\cdot{\rm cl}[C_b],
\end{aligned}
\end{align}
Hence the canonical morphism
\begin{align}\label{eq:semipurityisomonH}
	q^*:
	{\rm H}^0_{SS_\mu{\cal F}}(T^*X,e^{\vee*}K_X)
	\xrightarrow{\ \sim\ }
	{\rm H}^0_{(SS_\mu{\cal F})_{\bar k}}
	(T^*X_{\bar k},q^\ast e^{\vee*}K_X)^{G_k}
\end{align}
is an isomorphism.
This allows us
to make the following definition.

\begin{df}[{cf.~\cite[Definition 9.4.1]{KS}}]\label{dfCCmu}
	Let ${\cal F}\in D_{\rm ctf}(X,\Lambda)$. 
	We define
	\begin{equation}\label{eqCC}
		CC_\mu({\cal F})=CC_{\mu,\psi,\Lambda}({\cal F})
		\in
		{\rm H}^0(SS_\mu({\cal F}),i^!e^{\vee*}K_X)
		=
		{\rm H}^0_{SS_\mu({\cal F})}(T^*X,e^{\vee*}K_X)
	\end{equation}
	to be the unique class whose image by \eqref{eq:semipurityisomonH} equals the class of the composition of \eqref{eqccFnmi}.
\end{df}
\begin{remark}\label{rem:Wang-CC-comparison}
Let $k$ be any field, and let
$\Lambda$ be a finite local ring whose residue characteristic is invertible in $k$.
Let $u:\mathbf A_k^1\setminus\{1\}\hookrightarrow\mathbf A_k^1$
be the open immersion, and put $B=u_*\Lambda$.
	Let $X$ be a smooth scheme over $k$, and 
	let $\mathcal F\in D_{\rm ctf}(X,\Lambda)$.
	Since $0^*B\simeq\Lambda$, we have canonical isomorphisms
$F_B(0_*A)\simeq e^{\vee*}A$.
Applying $F_B$ to \eqref{eqccFnm1},  we obtain
\begin{equation}\label{eq:Wang-characteristic-morphisms}
	\Lambda
	\longrightarrow
	\mu_B{\cal H}om(\mathcal F,\mathcal F)
	\longrightarrow
	e^{\vee*}K_X
\end{equation}
on $T^*X\times_k\eta_0$.
The composition defines a
$G_k$-invariant class supported on
$SS_{\mu,B}(\mathcal F)_{\bar k}$, as in
Definition~\ref{dfCCmu}. 	Thus by semi-purity and Galois descent as in
\eqref{eq:semipurityisomonH}, this class defines a unique class
$CC_{\mu,B}(\mathcal F) \in	{\rm H}^0_{SS_{\mu,B}({\cal F})}(T^*X,e^{\vee*}K_X)$ over $k$.
In future work, we will investigate the functorial properties of $CC_{\mu,B}(\mathcal F)$, and its relation to $CC_{\mu}(\mathcal F)$ in the case where $k$ has positive characteristic and $\Lambda$ admits a non-trivial character $\psi:\mathbf F_p\to \Lambda^\times$.
\end{remark}

\begin{lm}\label{lmlccC}
Assume that ${\cal F}$ is locally constant
of rank $r={\rm rank}\, {\cal F}$
and that $X$ is 
of dimension $n$.
Then, 
we have
$$CC_\mu {\cal F}=(-1)^n\cdot
{\rm rank}\, {\cal F}\cdot
{\rm cl}(T^*_XX).$$
\end{lm}

\proof{
Since ${\mathscr H}={\cal H}om({\rm pr}_1^*{\cal F},
{\rm pr}_2^!{\cal F})$ is locally constant,
we have an isomorphism
$e^*({\cal H}om({\cal F},{\cal F})\otimes K_X)
\to \nu{\cal H}om({\cal F},{\cal F})$.
The morphisms (\ref{eqccFnm}) and 
(\ref{eqccFnm1}) give a diagram
$$
\begin{CD}
F0_!\Lambda
@>{F{\rm adj}}>>
Fe^!{\cal H}om({\cal F},{\cal F})
@>{F({\rm res})}>>
F0_!({\cal H}om({\cal F},{\cal F})
\otimes K_X)
@>{F({\rm ev})}>>
F0_*K_X\\
@V{(\ref{eqF1})}VV@VV{(\ref{eqF0!})}V
@VV{(\ref{eqF1})}V
@VV{(\ref{eqF1})}V
\\\Lambda
@>{\rm res}>>
0^\vee_*{\cal H}om({\cal F},{\cal F})
@>{\rm adj}>>
e^{\vee*}({\cal H}om({\cal F},{\cal F})
\otimes K_X)
@>{\rm Tr}>>
e^{\vee*}K_X.
\end{CD}$$
The left square is commutative
by Proposition \ref{prEE'}.1
(\ref{eqresF64}).
The middle square is $(-1)^n$-commutative
by Proposition \ref{prEE'}.2
(\ref{eqresF1.30}).
The right square is commutative
by the functoriality of (\ref{eqF1}).
Since the class of the composition of the 
lower line is
$r\cdot
{\rm cl}(T^*_XX)$,
the assertion follows.
\qed

}

\begin{lm}\label{lmccF}
The pull-back of
$CC_\mu({\cal F})$
by the $0$-section
$$0^{\vee\ast}\colon
{\rm H}^0_{SS_\mu({\cal F})}(T^*X,
e^{\vee*}K_X)\to
{\rm H}^0(X,K_X)$$
equals the canonical class $C_{X/k}({\cal F})$.
\end{lm}

\proof{
The pull-back
$0^{\vee\ast}\colon
{\rm H}^0_{SS_\mu({\cal F})}(T^*X,
e^{\vee*}K_X)\to
{\rm H}^0(X,K_X)$
is the same as the composition
of the canonical morphism
${\rm H}^0_{SS_\mu({\cal F})}(T^*X,
e^{\vee*}K_X)\to
{\rm H}^0(T^*X,
e^{\vee*}K_X)$
and the inverse of the pull-back
$e^{\vee*}\colon{\rm H}^0(X,K_X)
\to 
{\rm H}^0(T^*X,e^{\vee*}K_X)$.
Since the composition of
(\ref{eqccFnm})
is the pull-back by $e^\vee$ of
that of
(\ref{eqccF}),
the assertion follows.
\qed

}

\begin{pr}\label{pradd}
Let $F$ be a filtration on ${\cal F}\in D_{\rm ctf}(X,\Lambda)$
such that ${\rm Gr}^F{\cal F}
=\bigoplus_i{\rm Gr}^F_i{\cal F}$ is in $D_{\rm ctf}(X,\Lambda)$.
Then, we have
$SS_\mu({\cal F})
\subseteq
S=SS_\mu({\rm Gr}^F{\cal F})$
and 
an equality
\begin{equation}
CC_\mu({\cal F})
=\sum_i CC_\mu ({\rm Gr}^F_i{\cal F})
\label{eqadd}
\end{equation}
in ${\rm H}^0_S
(T^*X,e^{\vee*}K_X)$.
\end{pr}

\proof{
	We may assume $k=\bar k$.
The filtration $F$ on ${\cal F}$
induces a filtration also denoted $F$ on
$\mu{\cal H}om({\cal F},{\cal F})$.
Since the identity $1_{\cal F}$
is in $F^0\mu{\cal H}om({\cal F},{\cal F})$,
the class
$CC_\mu({\cal F})$ is the image of
$1_{\cal F}
\in {\rm H}^0_{S}
(T^*X,F^0\mu{\cal H}om({\cal F},{\cal F}))$
by the evaluation morphism.
Since the restriction of the evaluation morphism
is the composition of
$F^0\mu{\cal H}om({\cal F},{\cal F})
\to
\bigoplus_i
\mu{\cal H}om({\rm Gr}^F_i{\cal F},{\rm Gr}^F_i{\cal F})
\to e^{\vee*}K_X$,
the assertion follows.
\qed

}

\begin{cor}\label{coradd}
Let $i\colon Z\to X$ be a closed
immersion and $j\colon U=X\sm Z\to X$
be the open immersion of the complement.
Then, we have
\begin{equation}
CC_\mu({\cal F})=
CC_\mu(j_!j^*{\cal F})+
CC_\mu(i_*i^*{\cal F})
\label{eqij}
\end{equation}
in ${\rm H}^0_{SS_\mu(
j_!j^*{\cal F}\oplus
i_*i^*{\cal F})}(T^*X,e^{\vee*}K_X)$.
\end{cor}

\proof{
We define a filtration
$F$ on ${\cal F}$ by $F^0{\cal F}=j_!j^*{\cal F}\subseteq {\cal F}$,
$F^{-1}=0$ and $F^1={\cal F}$.
Then ${\rm Gr}^F{\cal F}\in D_{\rm ctf}(X,\Lambda)$
and we obtain (\ref{eqij}).
\qed

}

\medskip

We define a variant for a not necessarily smooth
scheme admitting an immersion to a smooth scheme.
Let $i\colon Z\to X$ be a closed immersion
of schemes over $k$ and assume
that $X$ is smooth over $k$
and let $\gamma\colon Z\to Z\times X$
be the graph of $i$.
Let ${\cal F}\in D_{\rm ctf}(Z,\Lambda)$
and
let ${\mathscr H}_i$ denote
${\cal H}om({\rm pr}_1^*{\cal F},{\rm pr}_2^!i_*{\cal F})
\in D_{\rm ctf}(Z\times X,\Lambda)$.
We have a canonical isomorphism
\begin{equation}
{\mathscr H}_i\simeq 
(1_Z\times i)_*
{\cal H}om({\rm pr}_1^*{\cal F},{\rm pr}_2^!{\cal F}).
\label{eqHiZ}
\end{equation}
By \cite[(3.2.1)]{SGA5},
we have a canonical isomorphism
\begin{equation}
{\cal H}om({\cal F},{\cal F})
\to \gamma^!{\mathscr H}_i.
\label{eqgam!}
\end{equation}
Hence
the identity $1_{\cal F}$
defines a canonical morphism
\begin{equation}
\Lambda 
\to \gamma^!{\mathscr H}_i.
\label{eq1Z}
\end{equation}
Let $\boxtimes\colon 
D_{\rm ctf}(Z,\Lambda)
\times
D_{\rm ctf}(X,\Lambda)
\to
D_{\rm ctf}(Z\times X,\Lambda)$
denote the bifunctor
${\rm pr}_1^*-\otimes {\rm pr}_2^*-$.
The canonical isomorphism
$D_Z{\cal F}\boxtimes i_*{\cal F} \to {\mathscr H}_i$
\cite[(3.1.1)]{SGA5}
on $Z\times X$
induces an isomorphism
\begin{equation}
D_Z{\cal F}\otimes {\cal F}
\to 
\gamma^*{\mathscr H}_i.
\label{eqgam*}
\end{equation}
Hence the evaluation morphism
$D_Z{\cal F}\otimes {\cal F}
\to 
K_Z$
defines a canonical morphism
\begin{equation}
\gamma^*{\mathscr H}_i
\to K_Z.
\label{eqevZ}
\end{equation}
The morphisms
(\ref{eq1Z}), (\ref{eqevZ})
and the canonical morphism
$\gamma^!\to \gamma^*$
define
\begin{equation}
\Lambda\to
\gamma^!{\mathscr H}_i
\to 
\gamma^*{\mathscr H}_i
\to K_Z.
\label{eqccFZ}
\end{equation}
Since $i$ is a closed immersion,
the composition
(\ref{eqccFZ})
defines the canonical class
$C_{Z/k}({\cal F})\in {\rm H}^0(Z,K_Z)$.

Let $e^\vee_i\colon T^*X\times_XZ\to Z$
be the projection.
Then the adjoint of (\ref{eqccFZ})
defines
\begin{equation}
\gamma_!\Lambda\to
{\mathscr H}_i
\to \gamma_*K_Z.
\label{eqccFaZ}
\end{equation}
Since $\gamma$ is a regular immersion, the specialization functor to the normal bundle $TX\times_XZ$ is still well defined (cf. \cite[(8,2)]{Verdier}).
Taking the specialization and
the microlocalization, then using \cite[(8.5) and (8.6)]{Verdier},
we obtain
\begin{align}
&0_!\Lambda\to
\nu_{Z/Z\times X}{\mathscr H}_i
\to 0_*K_Z,
\label{eqccFnm1Z}\\
&\Lambda\to
\mu_{Z/Z\times X}{\mathscr H}_i
\to e_i^{\vee*}K_Z.
\label{eqccFnmZ}
\end{align}

Regard $SS_\mu (i_*{\cal F})\to T^*X$ 
as a subset of $T^*X\times_XZ$
and let $s\colon SS_\mu(i_*{\cal F})\to T^*X\times_XZ$ be
the closed immersion. Let $q: (T^*X\times_XZ)_{\bar k}\to T^*X\times_XZ$ be the projection.
Since $SS_\mu(i_*{\cal F})_{\bar k}\subseteq (T^*X \times_{X}Z)_{\bar k}$ 
equals the support of
$\mu_{Z/Z\times X}{\mathscr H}_i$,
(\ref{eqccFnmZ})
defines
\begin{equation}
\Lambda\to
s^*\mu_{Z/Z\times X}{\mathscr H}_i
\to s^!e_i^{\vee*}K_Z
\label{eqccFnmiZ}
\end{equation}
on $SS_\mu(i_*{\cal F})_{\bar k}$
by adjunction. 

Similarly to \eqref{eq:semipurityisomonH}
we have a canonical isomorphism
\begin{align}\label{eq:semipurityisomonH2}
	q^*:
	{\rm H}^0_{SS_\mu(i_\ast{\cal F})}(T^*X\times_XZ,e_i^{\vee*}K_Z)
	\xrightarrow{\ \sim\ }
	{\rm H}^0_{SS_\mu(i_\ast{\cal F})_{\bar k}}
	((T^*X\times_XZ)_{\bar k},q^\ast e_i^{\vee*}K_Z)^{G_k}.
\end{align}

\begin{df}\label{dfCCmuZ}
Let ${\cal F}\in D_{\rm ctf}(Z,\Lambda)$.
Let $i\colon Z\to X$ be a closed immersion to a smooth scheme over $k$
and $e^\vee_i\colon T^*X\times_XZ\to Z$
be the projection.
We define
\begin{equation}
CC_{\mu,i}({\cal F})
\in
{\rm H}^0(SS_\mu (i_*{\cal F}),s^!
e_i^{\vee*}K_Z)
=
{\rm H}^0_{SS_\mu (i_*{\cal F})}(T^*X\times_XZ,
e_i^{\vee*}K_Z)
\label{eqCCZ}
\end{equation}
to be the class whose image by \eqref{eq:semipurityisomonH2} equals the class of the composition
of (\ref{eqccFnmiZ}).
\end{df}

We prove a relation between
$CC_{\mu,i}({\cal F})$
and
$CC_{\mu}(i_*{\cal F})$
in Proposition \ref{prpush}.2.
By the same argument as in the proof of Lemma \ref{lmccF},
using the isomorphism (\ref{eqHiZ}),
we have the following.

\begin{lm}\label{lmccFZ}
The pull-back of
$CC_{\mu,i}({\cal F})$
by the $0$-section
$$0^*\colon
{\rm H}^0_{SS_\mu (i_*{\cal F})}(T^*X\times_XZ,
e_i^{\vee*}K_Z)\to
{\rm H}^0(Z,K_Z)$$
equals the canonical class $C_{Z/k}({\cal F})$.
\end{lm}

\subsection{Tamely ramified sheaves on curves}\label{sstame}

Let $X$ be a smooth irreducible curve over $k$
and let $x\in X$ be a $k$-rational point.
Let $j\colon U=X\sm \{x\}\to X$
be the open immersion
and ${\cal F}$ be a locally constant
constructible sheaf
of free $\Lambda$-modules
of rank $n$ on $U$.
By Lemma \ref{lmlcc},
we know that $SS_\mu(j_!{\cal F})
\subseteq T^*X$
is a closed conical subset of the union of
the $0$-section $T^*_XX$
and the fiber $T^*_xX$.
By the semi-purity
\cite[Rappel 2.2.8]{Cycle},
we have
${\rm H}^0_{
T^*_XX\cup T^*_xX}(T^*X,e^{\vee*}K_X)
=
\Lambda [T^*_XX]
\oplus \Lambda[T^*_xX]$.
By Lemma \ref{lmlccC}, we have
$CC_\mu (j_!{\cal F})
=-(n\cdot[T^*_XX]+a\cdot[T^*_xX])$
for some $a\in \Lambda$.
We give a method to compute 
the coefficient $a\in \Lambda$
in Lemma \ref{lmab}.

Let $\!+\colon
T^*_XX\cup T^*_xX
\to T^*X$ be the closed immersion.
Since $\mu{\cal H}om(j_!{\cal F},j_!{\cal F})$
is supported on $SS_\mu (j_!{\cal F})
\subseteq T^*_XX\cup T^*_xX$,
the adjoint of the morphism
$\mu{\cal H}om(j_!{\cal F},j_!{\cal F})
\to e^{\vee*}K_X$
(\ref{eqccFnm})
induced by ${\rm ev}$
defines
$\mu{\cal H}om(j_!{\cal F},j_!{\cal F})
\to +^!e^{\vee*}K_X.$
Here and in the following,
identify $\mu{\cal H}om(j_!{\cal F},j_!{\cal F})$
with its restriction to
$T^*_XX\cup T^*_xX$
by abuse of notation.
By the isomorphism $F_\psi {\cal N}_\Lambda
\to +^!e^{\vee*}K_X$ (\ref{eqFNL+}),
this is induced by a morphism
$\nu{\cal H}om(j_!{\cal F},j_!{\cal F})
\to {\cal N}_\Lambda.$

We define a subcomplex
${\cal N}_{\cal F}$
of $e^*(j_*{\cal E}nd{\cal F}\otimes K_X)$
on $TX$
similarly as ${\cal N}_\Lambda$
and construct a factorization
\begin{equation}
\nu{\cal H}om(j_!{\cal F},j_!{\cal F})
\to {\cal N}_{\cal F}\to {\cal N}_\Lambda
\label{eqNFL}
\end{equation}
such that the restriction to $TU$
is 
$\nu{\cal H}om({\cal F},{\cal F})
\to e_U^*({\cal E}nd\,{\cal F}\otimes K_U)
\to e_U^*K_U$
where the second arrow
is induced by 
${\rm Tr}\colon {\cal E}nd\,{\cal F}
\to \Lambda_U$
and $e_U\colon TU\to U$ is the projection.
The cohomology sheaves of
$j_*{\cal E}nd\,{\cal F}\otimes K_X$
are
${\cal H}^{-2}=R^0j_*{\cal E}nd{\cal F}(1)$,
${\cal H}^{-1}=R^1j_*{\cal E}nd{\cal F}(1)$
and vanish otherwise.
Let $I_x$ denote the inertia group
at $x\in X$.
The stalk
$(R^1j_*{\cal E}nd{\cal F}(1))|_x$
is canonically identified
with the Galois cohomology
${\rm H}^1(I_x,{\rm End}\, F(1))$.
Let $P\subseteq I_x$ denote the wild inertia subgroup.
Since $I_x/P$ is isomorphic to
$\widehat {\mathbf Z}'
=\varprojlim_{p\nmid n}\mu_n$
and $P$ is a pro-$p$ group,
the group
${\rm H}^1(I_x,{\rm End}\, F(1))$
is canonically identified
with the coinvariants $({\rm End}\, F)_{I_x}$.

Let $T=TX\times_Xx$ be the fiber at $x$
and $g\colon T^\times=T\sm \{0\}\to T$
be the open immersion.
Define a subcomplex
${\cal N}_{\cal F}$
of $e^*(j_*{\cal E}nd{\cal F}\otimes K_X)$
by replacing the cohomology sheaf
${\cal H}^{-1}=e^*R^1j_*{\cal E}nd{\cal F}(1)
=({\rm End}\, F)_{I_x, T}$
at the top degree
by the subsheaf
$g_!({\rm End}\, F)_{I_x, T^\times}$.
Here and in the following
$({\rm End}\, F)_{I_x, T}$
denotes the geometrically constant sheaf on $T$
with stalk $({\rm End}\, F)_{I_x}$.
The trace morphism
$e_U^*({\cal E}nd{\cal F}\otimes K_U)
\to
e_U^*K_U$
induces
$e^*(j_*{\cal E}nd{\cal F}\otimes K_X)
\to
e^*(j_*\Lambda\otimes K_X)$.
Define a morphism
${\rm Tr}\colon {\cal N}_{\cal F}\to {\cal N}_\Lambda$
to be the restriction.

\begin{lm}\label{lmNNL}
{\rm 1.}
The isomorphism
$\nu{\cal H}om(j_!{\cal F},j_!{\cal F})|_{TU}\to
e_U^*({\cal E}nd{\cal F}\otimes K_U)$
induces 
\begin{equation}
\nu{\cal H}om(j_!{\cal F},j_!{\cal F})\to {\cal N}_{\cal F}.
\label{eqnuNF}
\end{equation}

{\rm 2.}
The morphism
{\rm (\ref{eqnuNF})}
induces a factorization
{\rm (\ref{eqNFL})}.
\end{lm}

\proof{
1.
The isomorphism
$\nu{\cal H}om(j_!{\cal F},j_!{\cal F})|_{TU}=
\nu{\cal H}om({\cal F},{\cal F})\to
e_U^*({\cal E}nd{\cal F}\otimes K_U)$
induces a morphism
$\nu{\cal H}om(j_!{\cal F},j_!{\cal F})\to 
e^*(j_*{\cal E}nd{\cal F}\otimes K_X)$
by adjunction and the smooth base change theorem.
Since $0^*\nu{\cal H}om(j_!{\cal F},j_!{\cal F})
=\delta^*(D_X(j_!{\cal F})\boxtimes j_!{\cal F})
=D_X(j_!{\cal F})\otimes j_!{\cal F}
=j_!{\cal E}nd{\cal F}\otimes K_X$
by Proposition \ref{prverdier}.2,
the stalk $\nu {\cal H}om(j_!{\cal F},j_!{\cal F})_0$
at $0\in T\subseteq TX$ is $0$.
Hence by adjunction,
$\nu{\cal H}om(j_!{\cal F},j_!{\cal F})|_{TX\ssm\{0\}}\to 
e^*(j_*{\cal E}nd{\cal F}\otimes K_X)|_{TX\ssm\{0\}}
={\cal N}_{\cal F}|_{TX\ssm\{0\}}$
induces (\ref{eqnuNF}).

{\rm 2.}
The restriction to $TU$ of
$\nu{\cal H}om(j_!{\cal F},j_!{\cal F})\to {\cal N}_\Lambda$
is 
induced by the trace
$e_U^*({\cal E}nd{\cal F}\otimes K_U)
\to
e_U^*K_U$.
Hence the assertion follows
by adjunction
as in the proof of 1.
\qed

}
\medskip

We have
isomorphisms $j_*{\cal E}nd\, {\cal F}\to
{\cal E}nd\, j_!{\cal F}\to
0^!\nu{\cal H}om(j_!{\cal F},j_!{\cal F})$
by Proposition \ref{prverdier}.2
and (\ref{eqdel!}).
This and the identity $1_{\cal F}$
define the first two arrows in
\begin{equation}
\Lambda\to
j_*{\cal E}nd\, {\cal F}\to
0^!\nu{\cal H}om(j_!{\cal F},j_!{\cal F})
\to 0^!{\cal N}_{\cal F}\to 0^!{\cal N}_\Lambda.
\label{eqNFL0}
\end{equation}
The last two arrows are
induced by (\ref{eqNFL}).
Since the Fourier transform of its adjoint
$$0_!\Lambda\to
0_!{\cal E}nd\, j_!{\cal F}\to
\nu{\cal H}om(j_!{\cal F},j_!{\cal F})
\to {\cal N}_{\cal F}\to {\cal N}_\Lambda$$
defines
$\Lambda\to 
\mu{\cal H}om(j_!{\cal F},j_!{\cal F})
\to\! +^!e^{\vee*}K_X$
defining $CC_\mu (j_!{\cal F})$,
the composition of 
{\rm (\ref{eqNFL0})} 
determines $CC_\mu (j_!{\cal F})$.

Computations similar to those in
Lemma \ref{lmNL}.1 and
Corollary \ref{cor113}.1 work for
${\cal N}_{\cal F}$ and
$0^!{\cal N}_{\cal F}$.

\begin{lm}\label{lmNF}
Let $I_x$ denote the inertia group
at $x\in X$
and $I_0$ denote the inertia group
at $0\in T$.
Let $F$
be the representation 
of $I_x$ defined by ${\cal F}$
and consider
the coinvariant quotient
$({\rm End}\, F)_{I_x}$
as a trivial representation of $I_0$.

{\rm 1.}
The canonical morphism
${\cal N}_{\cal F}
\to e^*(j_*{\cal E}nd{\cal F}\otimes K_X)$
induces a distinguished triangle
\begin{equation}
{\cal N}_{\cal F}
\to 
e^*(j_*{\cal E}nd{\cal F}\otimes K_X)
\oplus
g_!({\rm End}\, F)_{I_x, T^\times}[1]
\to
({\rm End}\, F)_{I_x, T}[1]\to.
\label{eqNF}
\end{equation}

{\rm 2.}
The distinguished triangle
{\rm (\ref{eqNF})}
induces
a distinguished triangle
\begin{equation}
0^!{\cal N}_{\cal F}
\to 
j_*{\cal E}nd{\cal F}
\oplus
\Gamma(I_0,({\rm End}\, F)_{I_x})_0
\to
{\rm H}^1(I_0,({\rm End}\, F)_{I_x})_0
[-1]\to.
\label{eqNF0}
\end{equation}
The last two terms are placed
at $0\in T\cap X\subseteq TX$.
\end{lm}

\proof{
1.
The distinguished triangle (\ref{eqNF})
is clear from the construction of ${\cal N}_{\cal F}$.

2. 
For 
the geometrically constant sheaf ${\cal R}
=({\rm End}\, F)_{I_x, T}$ on $T$,
we have canonical isomorphisms
$0^!{\cal R}\to {\rm H}^1(I_0,{\cal R})_0[-2]$
and
$0^!g_!g^*{\cal R}\to \Gamma(I_0,{\cal R})_0[-1]$.
Hence (\ref{eqNF}) induces (\ref{eqNF0}).
\qed

}

\begin{lm}\label{lmab}
Suppose that the composition 
\begin{equation}
\Lambda\to
j_*{\cal E}nd\, {\cal F}\to
0^!\nu{\cal H}om(j_!{\cal F},j_!{\cal F})
\to 0^!{\cal N}_{\cal F}
\to
\Gamma(I_0,({\rm End}\, F)_{I_x})_0
\label{eqNL0c}
\end{equation}
of {\rm (\ref{eqNFL0})}
and the second component of the first arrow 
of {\rm (\ref{eqNF0})}
maps $1\in {\rm H}^0(X,\Lambda)$
to
$a\in 
{\rm H}^0(I_0,({\rm End}\, F)_{I_x})
=({\rm End}\, F)_{I_x}$
and let ${\rm Tr}\colon
({\rm End}\, F)_{I_x}\to \Lambda$
be the morphism induced by the trace.
Then, we have 
\begin{equation}
CC_\mu (j_!{\cal F})
=
-({\rm rank}\, {\cal F}\cdot
[T^*_XX]+{\rm Tr}\, a\cdot
[T^*_xX])
\label{eq-na}
\end{equation}
in
${\rm H}^0_{T^*_XX\cup T^*_xX}
(T^*X,e^{\vee*}K_X)$.
\end{lm}

\proof{
The trace morphism
${\cal E}nd\, {\cal F}\to \Lambda$
induces a commutative diagram
$$\begin{CD}
0^!{\cal N}_{\cal F}
@>>>
j_*{\cal E}nd{\cal F}
\oplus
\Gamma(I_0,({\rm End}\, F)_{I_x})_0
\\
@V{\rm Tr}VV@VV{{\rm Tr}\oplus{ \rm Tr}}V\\
0^!{\cal N}_\Lambda
@>>> 
j_*\Lambda
\oplus
\Gamma(I_0,\Lambda)_0.
\end{CD}$$
The composition
of {\rm (\ref{eqNFL0})}
and the first component of the first arrow 
of {\rm (\ref{eqNF0})}
maps $1\in {\rm H}^0(X,\Lambda)$
to $1_{\cal F}$.
The equality (\ref{eq-na}) follows from this
and Corollary \ref{cor113}.2.
\qed

}
\medskip

In this section, we prove the following.

\begin{pr}\label{prtame}
Let $j\colon {\mathbf G}_m\to X={\mathbf A}^1$
be the open immersion and 
let ${\cal F}$ be a locally constant constructible sheaf of
free $\Lambda$-modules of rank $n$
on ${\mathbf G}_m$
tamely ramified at $0\in X$.
Let $I_x$ denote the inertia group
at $0\in X$
and $I_0$ denote the inertia group
at $0\in T$.
Then,
the composition 
\begin{equation}
j_*{\cal E}nd\, {\cal F}\to
0^!\nu{\cal H}om(j_!{\cal F},j_!{\cal F})
\to 0^!{\cal N}_{\cal F}\to
\Gamma(I_0,({\rm End}\, F)_{I_x})_0
\label{eqNF0c}
\end{equation}
of {\rm (\ref{eqNFL0})}
and the second component of the first arrow 
of {\rm (\ref{eqNF0})}
induces 
the canonical morphism
${\rm H}^0(U,{\cal E}nd{\cal F})\to
{\rm H}^0(I_0,({\rm End}\, F)_{I_x})=({\rm End}\, F)_{I_x}.$
\end{pr}

\medskip

Proposition \ref{prtame}
implies the following equality
by Lemma \ref{lmab}.

\begin{cor}\label{cortame}
We have
\begin{equation}
CC_\mu (j_!{\cal F})
=
-n([T^*_XX]+[T^*_0X])
\label{eq-nn}
\end{equation}
in
${\rm H}^0_{T^*_XX\cup T^*_0X}
(T^*X,e^{\vee*}K_X)$.
\end{cor}

As a preparation of the proof
of Proposition \ref{prtame},
 we first prove a lemma on the cohomology of
a tame inertia group.

\begin{lm}\label{lmI}
Let $\widehat{\mathbf Z}'=
\varprojlim_{p\nmid n} {\mathbf Z}/n{\mathbf Z}$
and let $I$ be a profinite group
isomorphic to $\widehat{\mathbf Z}'^2$.
Let $\sigma, \tau\in I$
be elements defining an isomorphism
$\widehat{\mathbf Z}'^2\to I$
and let $I_1=\langle \sigma\rangle$,
$I_2=\langle \tau\rangle\subseteq I
=I_1\times I_2$
be the subgroups isomorphic to
$\widehat{\mathbf Z}'$.
Let $M$ be a 
$\Lambda$-module with
a continuous action of $I$
such that the action of $\tau\in I$
is trivial.

{\rm 1.}
The restriction morphisms
define an isomorphism
\begin{equation}
{\rm H}^1(I,M)
\to 
{\rm H}^1(I_1,M)
\oplus
{\rm H}^1(I_2,M)^{I_1}.
\label{eqH1}
\end{equation}

{\rm 2.}
Let $\rho=\sigma\tau$ and
$I_3=\langle\rho\rangle \subseteq I=I_3\times I_2$.
Then, we have a commutative diagram
\begin{equation}
\begin{CD}
{\rm H}^1(I,M)
@>{\rm (\ref{eqH1})}>>
{\rm H}^1(I_1,M)
\oplus
{\rm H}^1(I_2,M)^{I_1}
\\
@V{1}VV @VV
{\tiny
\begin{pmatrix}
i_{13}^*&i_{23}^*\\
0&1
\end{pmatrix}}V
\\
{\rm H}^1(I,M)
@>{\rm (\ref{eqH1})}>>
{\rm H}^1(I_3,M)
\oplus
{\rm H}^1(I_2,M)^{I_3}
\end{CD}
\label{eqH113}
\end{equation}
where 
$i_{13}^*\colon
{\rm H}^1(I_1,M)
\to
{\rm H}^1(I_3,M)$
and
$i_{23}^*\colon
{\rm H}^1(I_2,M)^{I_1}=
{\rm H}^1(I_2,M^{I_1})
\to
{\rm H}^1(I_3,M)$
are induced by the projections
$I_3\to I_1$
and
$I_3\to I_2$.
\end{lm}

\proof{
1.
Let $A=\Lambda[[I]]$ be the completed group algebra
and regard $M$ as an $A$-module.
By the Koszul free resolution
$A\overset {(\tau-1,-(\sigma-1))}\longrightarrow
A^2\overset {(\sigma-1,\tau-1)}\longrightarrow A$
of the $A$-module $\Lambda$,
the complex
$M\overset {(\sigma-1,\tau-1)}\longrightarrow
M^2\overset {(\tau-1,-(\sigma-1))}\longrightarrow M$
computes the cohomology
${\rm H}^*(I,M)$.
Since $\tau$ acts trivially on $M$,
we obtain an isomorphism
${\rm H}^1(I,M)
= {\rm Coker}(\sigma-1\colon M\to M)
\oplus {\rm Ker}(\sigma-1\colon M\to M)
\to
{\rm H}^1(I_1,M)
\oplus
{\rm H}^1(I_2,M)^{I_1}$.

2.
We have a commutative diagram
$$
\begin{CD}
M@>{(\sigma-1,\tau-1)}>>
M^2@>{(\tau-1,-(\sigma-1))}>>M
\\
@V1VV
@V{\tiny
\begin{pmatrix}
\tau&1\\
0&1
\end{pmatrix}}VV
@VV{\tau}V\\
M@>{(\rho-1,\tau-1)}>>
M^2@>{(\tau-1,-(\rho-1))}>>M
\end{CD}
$$
of complexes.
The middle vertical arrow induces
the right vertical arrow in (\ref{eqH113}).
\qed

}
\medskip

To prove Proposition \ref{prtame},
we prepare a geometric construction.
Let $X$ denote ${\mathbf A}^1={\rm Spec}\, k[x]$.
We define a commutative diagram
$$\begin{CD}
(X\times X\times {\mathbf A}^1)^\sharp
@>>>
(X\times X\times {\mathbf A}^1)^\natural\\
@VVV@VVV\\
(X\times X\times {\mathbf A}^1)'
@>>> X\times X\times {\mathbf A}^1
&={\rm Spec}\, k[x,y,u].
\end{CD}$$
The right vertical arrow is the blow-up at the closed point $(0,0,0)$.
The upper horizontal arrow is the blow-up
at the inverse image 
of $\delta(X)\times \{0\}$, in other words,
at the pull-back of the ideal $(y-x,u)$.
The lower horizontal arrow is the blow-up
at $\delta(X)\times \{0\}$.
The left vertical arrow is defined 
by the universality of blow-up.
The open subscheme
$A_X(X\times X)^\sharp
\subseteq (X\times X\times {\mathbf A}^1)^\sharp$
defined by the cartesian diagram
\begin{equation}
\begin{CD}
A_X(X\times X)^\sharp
@>>> (X\times X\times {\mathbf A}^1)^\sharp
\\
@VqVV@VVV
\\
A_X(X\times X)
@>>> (X\times X\times {\mathbf A}^1)'
\end{CD}
\end{equation}
is the complement 
of the proper transform of
the divisor $(u)
=X\times X\times \{0\}
\subseteq X\times X\times {\mathbf A}^1$.

The diagonal immersion
$\delta\colon X\times {\mathbf A}^1
\to 
X\times X\times {\mathbf A}^1$
is lifted to the immersion
$0\colon
(X\times {\mathbf A}^1)'
\to 
A_X(X\times X)^\sharp$
from the blow-up $(X\times {\mathbf A}^1)'
\to X\times {\mathbf A}^1$
at the closed point $(0,0)$.
The projections 
$X\times X\times {\mathbf A}^1
\to
X\times {\mathbf A}^1$
are lifted to 
smooth morphisms
$p_1,\, p_2\colon
A_X(X\times X)^\sharp
\to 
(X\times {\mathbf A}^1)'$.

On $A_X(X\times X)^\sharp_0=
A_X(X\times X)^\sharp\times_{{\mathbf A}^1}0$,
we have a cartesian diagram
\begin{equation}
\begin{CD}
&{\mathbf A}^1
@>>>
E
@>>>
(X\times {\mathbf A}^1)'_0
@<<<
X\\
&@V{\delta}VV@V{\overline \delta}VV
@VV0V
@VV0V
\\
A_0=\,&{\mathbf A}^2
@>>>
P
@>>>
A_X(X\times X)^\sharp_0
@<<<
TX
\end{CD}
\end{equation}
defined as follows.
The immersion $TX\to 
A_X(X\times X)$ is lifted to
a closed immersion $TX\to 
A_X(X\times X)^\sharp_0$.
The proper transform in
$A_X(X\times X)^\sharp$
of the exceptional divisor of
the second blow-up
$(X\times X\times {\mathbf A}^1)^\sharp
\to
(X\times X\times {\mathbf A}^1)^\natural$
is the complement $P
={\mathbf P}^{2\prime}\sm L'$
in the blow-up 
${\mathbf P}^{2\prime}$ of
${\mathbf P}^2$
at $(1,1,0)\in
L={\mathbf P}^2\sm {\mathbf A}^2$ of the proper transform
$L'$
of the line $L$
at infinity.
In the middle cartesian square,
$(X\times {\mathbf A}^1)'_0
\to
A_X(X\times X)^\sharp_0$
is the base change of
the lifting
$0\colon (X\times {\mathbf A}^1)'
\to A_X(X\times X)^\sharp$
of the morphism
$\delta\colon X\times {\mathbf A}^1
\to
X\times X\times {\mathbf A}^1$.
The closed fiber
$A_X(X\times X)^\sharp_0$
equals the union 
$P\cup TX$.
The intersection
$P\cap TX$
is the fiber
$T=TX\times_X0
\subseteq TX$
and is dense in the exceptional divisor
of the blow-up 
${\mathbf P}^{2\prime}\to
{\mathbf P}^2$.
The inverse image of $P$
is the exceptional divisor $E
$ of the blow-up
$(X\times {\mathbf A}^1)'
\to X\times {\mathbf A}^1$.
The immersion
$\overline \delta\colon E
\to P$
is an extension of
the diagonal morphism
$\delta\colon {\mathbf A}^1
\to {\mathbf A}^2$.

They are explicitly described as follows.
The scheme $A_X(X\times X)^\sharp$
has an open covering by
two affine spaces
$A={\mathbf A}^3=
{\rm Spec}\, k[x_1,y_1,u]$
and $B=
{\mathbf A}^3
={\rm Spec}\, k[x,w,v]$.
The second morphism in
$$
A_X(X\times X)^\sharp=A\cup B
\overset q\longrightarrow
A_X(X\times X)=
{\rm Spec}\, k[x,u,v]
\to
X\times X\times {\mathbf A}^1=
{\rm Spec}\, k[x,y,u]
$$ is given by
$y=x+uv$
and the first morphism is given by
$x=x_1u,
v=y_1-x_1$
and $u=xw$.
The closed subscheme
$TX\subseteq 
A_X(X\times X)^\sharp_0$
is identified with ${\rm Spec}\, k[x,v]
\subseteq B$.
The closed fiber $A_X(X\times X)^\sharp_0=
A_X(X\times X)^\sharp\times_{{\mathbf A}^1}0$
has an open covering by
$A_0={\mathbf A}^2=
{\rm Spec}\, k[x_1,y_1]$
and $B_0
={\rm Spec}\, k[x,w,v]/(xw)$.

The blow-up $(X\times {\mathbf A}^1)'$
is the union $A_1\cup B_1$
of two affine planes
$A_1={\mathbf A}^2=
{\rm Spec}\, k[x_1,u]$
and $B_1=
{\mathbf A}^2
={\rm Spec}\, k[x,w]$.
The morphism
$$(X\times {\mathbf A}^1)'
=A_1\cup B_1\to
X\times {\mathbf A}^1
={\rm Spec}\, k[x,u]$$
is given by
$x=x_1u$ and $u=xw$.
The immersion
$0\colon
(X\times {\mathbf A}^1)'
=A_1\cup B_1
\to
A_X(X\times X)^\sharp
=A\cup B$
is given by $y_1=x_1$ on $A_1\subseteq A$
and by $v=0$ on $B_1\subseteq B$.
The morphism $p_1
\colon A_X(X\times X)^\sharp
\to
(X\times {\mathbf A}^1)'$
is given by
the projections
$A={\rm Spec}\, k[x_1,y_1,u]
\to A_1={\rm Spec}\, k[x_1,u]$
and
$B={\rm Spec}\, k[x,w,v]
\to B_1={\rm Spec}\, k[x,w]$.

We fix notation for immersions of subschemes of
$A_X(X\times X)^\sharp_0$
as in the diagram
$$\begin{CD}
{\mathbf G}_m
@>{j_A}>> 
{\mathbf A}^1
@>{j_E}>> 
E
@.
T^\times
@.
X
@<j<<
U&\,=X\sm\{0\}\\
@.@V{\delta}VV
@V{\overline \delta}VV
@V{g}VV
@VV0V@VV0V\\
@.
{\mathbf A}^2
@>{j_P}>> 
P
@<{i_T}<<
T
@>>>
TX
@<{j_{TX}}<<
TU
\end{CD}$$
where
$T^\times=T\sm \{0\}$
is the complement of the origin.
The middle vertical arrow 
and the labeled horizontal arrows 
except $i_T$ are open immersions.
The other vertical arrows, the unlabeled horizontal
arrows and $i_T$ are
closed immersions.
We have
${\mathbf A}^2=
P\sm T$,
$TU=TX\sm T$
and $T=P\cap TX$.
Let $e\colon TX\to X$ be the canonical morphism.

Let ${\cal F}$ be a locally constant constructible sheaf of
free $\Lambda$-modules
on ${\mathbf G}_m$
tamely ramified at $0\in X={\mathbf A}^1$
as in Proposition \ref{prtame}.
Define ${\mathscr H}={\cal H}om({\rm pr}_1^*j_!{\cal F},
{\rm pr}_2^!j_!{\cal F})$ on $X\times X$
and let
${\mathscr H}^\sharp$
be the pull-back on
$A_X(X\times X)^\sharp$.
Let $\Psi {\mathscr H}^\sharp$ on
$A_X(X\times X)^\sharp_0$ 
be the nearby cycles complex 
with respect to
$A_X(X\times X)^\sharp
\to {\mathbf A}^1={\rm Spec}\, k[u]$ at $0$.
Let ${\cal M}=\Psi{\cal F}|_{{\mathbf G}_m}$ be the restriction 
on the complement
${\mathbf G}_m=E\sm \{0,\infty\}$
of the nearby cycles complex
$\Psi{\cal F}$
with respect to
$(X\times {\mathbf A}^1)'
\to {\mathbf A}^1={\rm Spec}\, k[u]$ at $0$
as in Lemma \ref{lmpsiM}.
Define ${\mathscr H}_{\cal M}=
{\cal H}om({\rm pr}_1^*j_{A!}{\cal M},
{\rm pr}_2^!j_{A!}{\cal M})$
on ${\mathbf A}^2$.

\begin{pr}\label{prpsiH}
The distinguished triangle
$\Psi{\mathscr H}^\sharp
\to \Psi{\mathscr H}^\sharp|_{TX}
\oplus
\Psi{\mathscr H}^\sharp|_P
\to 
\Psi{\mathscr H}^\sharp|_T\to $
defines a distinguished triangle
\begin{equation}
\Psi {\mathscr H}^\sharp
\to e^*(j_*{\cal E}nd{\cal F}\otimes K_X)
\oplus
j_{P*}{\mathscr H}_{\cal M}
\to 
(j_{P*}{\mathscr H}_{\cal M})_T
\to.
\label{eqpsisharp}
\end{equation}
We have a canonical isomorphism
\begin{equation}
e^*(j_*{\cal E}nd{\cal F}\otimes K_X)|_T
\to (j_{P*}{\mathscr H}_{\cal M})_T.
\label{eqjjM}
\end{equation}
\end{pr}

\proof{
The tensor product $\Psi{\mathscr H}^\sharp\otimes$
with the exact sequence
$0\to \Lambda_{
A_X(X\times X)^\sharp_0}
\to \Lambda_{TX}\oplus \Lambda_P
\to \Lambda_T\to 0$
defines a distinguished triangle
$\Psi{\mathscr H}^\sharp
\to \Psi{\mathscr H}^\sharp|_{TX}
\oplus
\Psi{\mathscr H}^\sharp|_P
\to 
\Psi{\mathscr H}^\sharp|_T\to $.
By identifying $A=
{\rm Spec}\, k[x_1,y_1,u]$
with the fiber product
${\rm Spec}\, k[x_1,u]
\times_{{\rm Spec}\, k[u]}
{\rm Spec}\, k[y_1,u]$
and by applying \cite[Th\'eor\`eme 4.7]{autour},
we obtain an isomorphism 
$\Psi{\mathscr H}^\sharp|_{{\mathbf A}^2}
\to
{\mathscr H}_{\cal M}$.
Since
$p_1,p_2\colon
A_X(X\times X)^\sharp
\to
(X\times {\mathbf A}^1)'$ are smooth,
we obtain isomorphisms 
$\Psi{\mathscr H}^\sharp|_P
\to j_{P*}{\mathscr H}_{\cal M}$,
$\Psi{\mathscr H}^\sharp|_{TX}
\to
j_{TX*}e^*({\cal E}nd\, {\cal F}
\otimes K_U)=
e^*(j_*{\cal E}nd{\cal F}\otimes K_X)$
and 
\begin{equation}
\Psi{\mathscr H}^\sharp|_T
\to (j_{P*}{\mathscr H}_{\cal M})|_T\to
e^*(j_*{\cal E}nd{\cal F}\otimes K_X)|_T
\label{eqRR}
\end{equation}
by Lemma \ref{lmpsitame}.
Hence
we obtain
(\ref{eqpsisharp}) and (\ref{eqjjM}).
\qed

}
\medskip

The morphism
$q\colon A_X(X\times X)^\sharp
\to 
A_X(X\times X)$
induces a morphism
$q_0\colon A_X(X\times X)^\sharp_0
\to TX={\rm Spec}\, k[x,v]$
to the tangent bundle.
On $A_0={\rm Spec}\, k[x_1,y_1]$,
it is given
by $v=y_1-x_1$, $x=0$
and on $B_0={\rm Spec}\, k[x,w,v]/(xw)$
by the injection $k[x,v]\to k[x,w,v]/(xw)$.
The restriction to $TX
\subseteq A_X(X\times X)^\sharp_0$ is the identity.
The restriction to ${\mathbf A}^2$
induces the morphism
$v\colon {\mathbf A}^2=A_0
={\rm Spec}\, k[x_1,y_1]
\to 
T={\rm Spec}\, k[v]
\subseteq
TX$
defined by $v=y_1-x_1$.
This is extended to a morphism
$\overline v\colon P\to T$.

\begin{cor}\label{corpsiH}
{\rm 1.}
The direct image of
{\rm (\ref{eqpsisharp})}
by $q_0\colon A_X(X\times X)^\sharp_0
\to TX$
defines a distinguished triangle
\begin{equation}
\nu {\cal H}om(j_!{\cal F},j_!{\cal F})
\to e^*(j_*{\cal E}nd{\cal F}\otimes K_X)
\oplus
v_*{\mathscr H}_{\cal M}
\to (j_{P*}{\mathscr H}_{\cal M})_T
\to.
\label{eqnuHom}
\end{equation}
The second component
\begin{equation}
v_*{\mathscr H}_{\cal M}
\to (j_{P*}{\mathscr H}_{\cal M})_T
\label{eqvj}
\end{equation}
of {\rm (\ref{eqnuHom})}
is given by the restriction 
$v_*{\mathscr H}_{\cal M}
=
\overline v_*j_{P*}{\mathscr H}_{\cal M}
\to (j_{P*}{\mathscr H}_{\cal M})_T$.

{\rm 2.}
Taking $0^!$ of
{\rm (\ref{eqnuHom})},
we obtain a distinguished triangle
\begin{align}
&0^!\nu {\cal H}om(j_!{\cal F},j_!{\cal F})
\to j_*{\cal E}nd\, {\cal F}
\oplus
\Gamma({\mathbf G}_{m,\overline k},{\cal E}nd\, {\cal M})_0
\to 
\Gamma({\mathbf G}_{m,\overline k},{\cal E}nd\, {\cal M})_0
\to.
\label{eqnuHom0}
\end{align}
The last two terms denote
the geometrically constant complexes 
placed at $\infty=E\cap X\subseteq A_X(X\times X)^\sharp_0$
and at $0=T\cap X\subseteq TX$.
\end{cor}

\proof{
1.
Since 
$q\colon A_X(X\times X)^\sharp
\to A_X(X\times X)$
is proper,
we have a canonical isomorphism
$\nu {\cal H}om(j_!{\cal F},j_!{\cal F})
\to q_{0*}\Psi{\mathscr H}^\sharp$.
Since the restriction of $q_0$ to $TX$ is the identity
and that on ${\mathbf A}^2$ is $v
=\overline v\circ j_P$,
the distinguished triangle
{\rm (\ref{eqpsisharp})}
induces (\ref{eqnuHom}).

The direct image $\bar v_*$ of the restriction morphism
$j_{P*}{\mathscr H}_{\cal M}
\to 
(j_{P*}{\mathscr H}_{\cal M})_T$ induces
(\ref{eqvj}).

2.
By the proper base change theorem
and \cite[(3.2.1)]{SGA5},
we have an isomorphism
\begin{align*}
&0^!v_*{\mathscr H}_{\cal M}
\to \Gamma({\mathbf A}^1,
0^!{\mathscr H}_{\cal M})
\to \Gamma({\mathbf A}^1_{\overline k},
{\cal E}nd\, j_{A!}{\cal M})
\to \Gamma({\mathbf G}_{m,\overline k},
{\cal E}nd\, {\cal M}).
\end{align*}
Hence
{\rm (\ref{eqnuHom0})}
is  induced by
{\rm (\ref{eqnuHom})}.
\qed

}
\medskip

We construct a commutative diagram to compute 
(\ref{eqNF0c}).

\begin{cor}\label{corpsiH0}
Let $I_E$ denote the tame inertia group
at $\infty \in E$, and let $I_T$ be the inertia group at $0\in T$.
Let $M$
be the representation of $I_E$ defined by ${\cal M}$
and consider
the coinvariant quotient
$({\rm End}\, M)_{I_E}$
as a trivial representation of $I_T$.
We identify $({\rm End}\, F)_{I_x}$
in {\rm (\ref{eqNF0c})} and 
$({\rm End}\, M)_{I_E}$
in {\rm (\ref{eqnuNF0t})} by the isomorphism
$e^*(j_*{\cal E}nd{\cal F}\otimes K_X)|_T
\to (j_{P*}{\mathscr H}_{\cal M})_T$
{\rm (\ref{eqjjM})}.
The morphism
$\nu{\cal H}om(j_!{\cal F},j_!{\cal F})\to {\cal N}_{\cal F}$
{\rm (\ref{eqnuNF})}
and the morphism
$v_*{\mathscr H}_{\cal M}
\to (j_{P*}{\mathscr H}_{\cal M})_T$
{\rm (\ref{eqvj})}
induce a commutative diagram
\begin{align}
&\begin{CD}
0^!\nu {\cal H}om(j_!{\cal F},j_!{\cal F})
@>{\rm (\ref{eqnuHom0})}>>
\Gamma({\mathbf G}_{m,\overline k},{\cal E}nd\, {\cal M})_0
\\
@V{(\ref{eqnuNF})}VV@VV{\rm (\ref{eqvj})}V
\\
0^!{\cal N}_{\cal F}
@>{\rm (\ref{eqNF0})}>>
\Gamma(I_T,({\rm End}\, M)_{I_E})_0
\end{CD}
\label{eqnuNF0t}
\end{align}
where the terms in the right column
are placed at $0=T\cap X\subseteq TX$.
\end{cor}

\proof{
The morphism
$\nu{\cal H}om(j_!{\cal F},j_!{\cal F})\to {\cal N}_{\cal F}$
{\rm (\ref{eqnuNF})}
is induced by
$\nu{\cal H}om(j_!{\cal F},j_!{\cal F})\to 
e^*(j_*{\cal E}nd{\cal F}\otimes K_X)$.
We identify
$e^*(j_*{\cal E}nd{\cal F}\otimes K_X)|_T$
with
$(j_{P*}{\mathscr H}_{\cal M})_T$
by the isomorphism
{\rm (\ref{eqjjM})}.
Then 
by 
(\ref{eqnuHom0})
and (\ref{eqNF0}),
we obtain (\ref{eqnuNF0t}).
\qed

}
\medskip

To complete the proof of Proposition \ref{prtame},
we compute the morphism 
\begin{equation}
{\rm H}^0({\mathbf G}_{m,\overline k},{\cal E}nd\, {\cal M})
\to
{\rm H}^0(I_T,({\rm End}\, M)_{I_E})
\label{eqHH}
\end{equation}
induced by the right vertical arrow
in the commutative diagram
(\ref{eqnuNF0t}),  in Lemma \ref{lmls} below.
By the canonical isomorphism
$0^!{\mathscr H}_{\cal M}=
j_{A*}{\cal E}nd\, {\cal M}$
\cite[(3.2.1)]{SGA5},
we obtain a distinguished triangle
$\Gamma({\mathbf G}_{m,\overline k},
{\cal E}nd\, {\cal M})
\to
\Gamma({\mathbf A}^2_{\overline k},{\mathscr H}_{\cal M})
\to
\Gamma({\mathbf A}^2_{\overline k}\sm \delta({\mathbf A}^1_{\overline k}),
{\mathscr H}_{\cal M})\to$.
We have
$\Gamma({\mathbf A}^2_{\overline k},
{\mathscr H}_{\cal M})=
\Gamma({\mathbf A}^1_{\overline k},
D(j_!{\cal M}))
\otimes
\Gamma({\mathbf A}^1_{\overline k},
j_!{\cal M})
=
0$
by Lemma \ref{lmjM}.2.
Hence for the left hand side
of (\ref{eqHH}),
we obtain an isomorphism
\begin{equation}
{\rm H}^{-1}
({\mathbf A}^2_{\overline k}\sm \delta({\mathbf A}^1_{\overline k}),
{\mathscr H}_{\cal M})\to
{\rm H}^0({\mathbf G}_{m,\overline k},{\cal E}nd\, {\cal M}).
\label{eqH0Gm}
\end{equation}
Since the action of $I_T$ is trivial,
the right hand
side of (\ref{eqHH})
is identified with
\begin{equation}
{\rm H}^0(I_T,({\rm End}\, M)_{I_E})
=({\rm End}\, M)_{I_E}
={\rm H}^1(I_E,{\rm End}\, M(1)).
\label{eqHI}
\end{equation}
Let $P_\infty$ be the strict localization of $P$ at $\infty=E\cap T$
and
let $V=P_\infty\times_P({\mathbf A}^2\sm \delta({\mathbf A}^1))
\subseteq P_\infty$.
We identify the tame fundamental group
$\pi_1(V)^{\rm tame}$ with $I_E\times I_T$
and 
$\Gamma(V,
{\mathscr H}_{\cal M})$
with
$\Gamma(I_E\times I_T,{\rm End}\, M(1))[2]$.
We define
\begin{equation}
{\rm H}^{-1}(V,
{\mathscr H}_{\cal M})
=
{\rm H}^1(I_E\times I_T,{\rm End}\, M(1))
\to 
{\rm H}^1(I_E,{\rm End}\, M(1))
\label{eqVM}
\end{equation}
to be the restriction morphism
for $I_E\to I_E\times I_T$.
Since the second component of the right
vertical arrow in
(\ref{eqnuNF0t})
is induced by 
$v_*{\mathscr H}_{\cal M}
\to (j_{P*}{\mathscr H}_{\cal M})|_T$,
we have a commutative diagram
\begin{equation}
\begin{CD}
{\rm H}^{-1}
({\mathbf A}^2_{\overline k}\sm \delta({\mathbf A}^1_{\overline k}),
{\mathscr H}_{\cal M})
@>>>
{\rm H}^{-1}(V,{\mathscr H}_{\cal M})
\\
@V{\rm (\ref{eqH0Gm})}VV@VV{\rm (\ref{eqVM})}V\\
{\rm H}^0({\mathbf G}_{m,\overline k},{\cal E}nd\, {\cal M})
@>{\rm (\ref{eqHH})}>>
{\rm H}^1(I_E,{\rm End}\, M(1)).
\end{CD}
\label{eqHH1}
\end{equation}
The upper horizontal arrow
is the pull-back by
$V\to {\mathbf A}^2_{\overline k}\sm \delta({\mathbf A}^1_{\overline k})$.

To compute the morphism
(\ref{eqHH}),
we introduce a geometric construction.
Let ${\mathbf P}^2_1$ be the strict localization of 
${\mathbf P}^2$ at $(1,1,0)\in L={\mathbf P}^2\sm
{\mathbf A}^2$
and
let $U={\mathbf P}^2_1
\times_{{\mathbf P}^2}({\mathbf A}^2\sm \delta({\mathbf A}^1))
\subseteq {\mathbf P}^2_1$.
Let $I_L$ be the tame inertia group
at $(1,1,0)\in L$
and let $I_{{\mathbf P}^1}$
be the tame inertia group
at $(1,1,0)\in {\mathbf P}^1$
in the closure ${\mathbf P}^1$
of $\delta({\mathbf A}^1)$.
We identify
the tame fundamental group
$\pi_1(U)^{\rm tame}$ with $I_{{\mathbf P}^1}\times I_L$.
We regard ${\rm End}\, M(1)$
as a representation of $I_{{\mathbf P}^1}\times I_L$
with the canonical action of $I_{{\mathbf P}^1}$
and the trivial action of $I_L$.
Then 
we have an identification
\begin{equation}
\Gamma(U,
{\mathscr H}_{\cal M})
=
\Gamma(I_{{\mathbf P}^1}\times I_L,{\rm End}\, M(1))[2].
\label{eqHHL}
\end{equation}
The restriction $P\to {\mathbf P}^2$
of the blow-up at $(1,1,0)$
induces a morphism $V\to U$.
We have a commutative diagram
\begin{equation}
\begin{CD}
\pi_1(U)^{\rm tame}@<<<
I_{{\mathbf P}^1}\times I_L
\\
@AAA@AA
{\tiny\begin{pmatrix}
1&0\\
1&1
\end{pmatrix}}
A
\\
\pi_1(V)^{\rm tame}@<<<
I_E\times I_T
\end{CD}
\label{eqpiI}
\end{equation}
of isomorphisms
for the tame fundamental groups.

We consider the diagram
\begin{equation}
\begin{CD}
{\rm H}^{-1}({\mathbf A}^2_{\overline k}\sm \delta({\mathbf A}^1_{\overline k}),
{\mathscr H}_{\cal M})
@>>>
{\rm H}^{-1}(U,
{\mathscr H}_{\cal M})
@>>>
{\rm H}^{-1}(V,
{\mathscr H}_{\cal M})
\\
@VVV@V{\rm (\ref{eqHHL})}VV
@V{=}V{(\ref{eqVM})}V\\
{\rm H}^0({\mathbf G}_{m,\bar k},
{\cal E}nd\, {\cal M})
@>>>
{\rm H}^1(I_{{\mathbf P}^1}\times I_L,{\rm End}\, M(1))
@>>>
{\rm H}^1(I_E\times I_T,{\rm End}\, M(1)).
\end{CD}
\label{eqrect}
\end{equation}
The upper horizontal arrows are
induced by $V\to U\to
{\mathbf A}^2\sm \delta({\mathbf A}^1)$.
The lower right horizontal arrow
is induced by the left vertical arrow
of (\ref{eqpiI})
and the right square is commutative.
To complete the proof of Proposition \ref{prtame},
it suffices to show the following lemma.

\begin{lm}\label{lmls}
We identify
$${\rm H}^0({\mathbf G}_{m,\overline k},{\cal E}nd\, {\cal M})
=
{\rm H}^0(I_{{\mathbf P}^1},{\rm H}^1(I_L,{\rm End}\, M(1)))
=
{\rm H}^0(I_E,{\rm H}^1(I_T,{\rm End}\, M(1))).$$

{\rm 1.}
Define 
${\rm H}^0({\mathbf G}_{m,\overline k},{\cal E}nd\, {\cal M})
\to
{\rm H}^1(I_{{\mathbf P}^1}\times I_L,{\rm End}\, M(1))$
to be the inclusion of
the direct summand
${\rm H}^0(I_{{\mathbf P}^1},{\rm H}^1(I_L,{\rm End}\, M(1)))
\subseteq {\rm H}^1(I_{{\mathbf P}^1}\times I_L,{\rm End}\, M(1))$
in Lemma {\rm \ref{lmI}.1}.
Then, the 
left square in the diagram
{\rm (\ref{eqrect})} is commutative.

{\rm 2.}
The morphism
{\rm (\ref{eqHH})}
equals the morphism 
\begin{equation}
{\rm H}^0(I_E,{\rm H}^1(I_T,{\rm End}\, M(1)))
={\rm H}^1(I_T,{\rm H}^0(I_E,{\rm End}\, M(1)))
\to 
{\rm H}^1(I_E,{\rm End}\, M(1))
\label{eqET}
\end{equation}
induced by 
the canonical isomorphism
$I_E\to I_T$ compatible with the inclusion
${\rm H}^0(I_E,{\rm End}\, M(1))
\to 
{\rm End}\, M(1)$.
\end{lm}

\proof{1.
Similarly as 
(\ref{eqHH1}),
we have a commutative diagram
\begin{equation}
\begin{CD}
\Gamma({\mathbf A}^2_{\overline k}
\sm\delta({\mathbf A}^1_{\overline k}),
{\mathscr H}_{\cal M})
@>>>
\Gamma(U,{\mathscr H}_{\cal M})\\
@VVV
@VVV\\
\Gamma({\mathbf G}_{m,\overline k},
{\cal E}nd\, {\cal M})[1]
@>>>
\Gamma(I_{{\mathbf P}^1},{\rm H}^1(I_L,{\rm End}\, M(1))[1]).
\end{CD}
\label{eqGamma}
\end{equation}
By Lemma {\rm \ref{lmI}.1}, it suffices to show that the composition
$$
{\rm H}^{-1}({\mathbf A}^2_{\overline k}
\sm\delta({\mathbf A}^1_{\overline k}),
{\mathscr H}_{\cal M})
\to
{\rm H}^{-1}(U,{\mathscr H}_{\cal M})
\to
{\rm H}^1(I_L\times I_{{\mathbf P}^1},{\rm End}\, M(1))
\to
{\rm H}^1(I_{{\mathbf P}^1},{\rm End}\, M(1))
$$
is the $0$-mapping.

Let $j_{{\mathbf P}^2}\colon {\mathbf A}^2\to {\mathbf P}^2$
be the open immersion
and set ${\cal G}=
(j_{{\mathbf P}^2*}{\mathscr H}_{\cal M})|_L$.
Let $0,1,\infty\in L$
be the intersections of
the closures 
of $0\times {\mathbf A}^1$,
$\delta({\mathbf A}^1)$,
${\mathbf A}^1\times 0$
with
$L$ in ${\mathbf P}^2$.
Let $L_1$ be the strict localization of $L$
at a geometric point above $1\in L$
and
set $U_1=L_1\times_L(L\sm \{1\})$.
Let $j_0\colon L\sm \{0\}
\to L$ be the open immersion.
Since ${\cal G}\to j_{0*}j_0^*{\cal G}$
is an isomorphism, we have a commutative diagram
\begin{equation}
\begin{CD}
{\rm H}^{-1}({\mathbf A}^2_{\overline k}
\sm\delta({\mathbf A}^1_{\overline k}),
{\mathscr H}_{\cal M})
@>>>
{\rm H}^{-1}(U,{\mathscr H}_{\cal M})
@>>>
{\rm H}^1(I_{{\mathbf P}^1},{\rm End}\, M(1))
\\
@VVV@VVV@VVV\\
{\rm H}^{-1}(L\sm \{0,1\},{\cal G})
@>>>
{\rm H}^{-1}(U_1,{\cal G})
@>>> {\rm H}^0(U_1,{\cal H}^{-1}{\cal G}).
\end{CD}
\label{eqAG}
\end{equation}
To show that the composition of the lower line is
$0$, we consider the commutative diagram
$$\begin{CD}
{\rm H}^{-1}(L\sm \{0,1\},{\cal G})
@>>>
{\rm H}^{-1}(U_1,{\cal G})\\
@VVV@VVV\\
{\rm H}^0(L\sm \{0,1\},{\cal H}^{-1}{\cal G})
@>>>
{\rm H}^0(U_1,{\cal H}^{-1}{\cal G})
\end{CD}$$
Let $j_\infty\colon L\sm \{\infty\}
\to L$ be the open immersion.
Since 
$j_{\infty!}j_{\infty}^*{\cal G}\to {\cal G}$
is an isomorphism,
we have
${\rm H}^0(L\sm \{0,1\},{\cal H}^{-1}{\cal G})
=0$
and the composition of the lower line of (\ref{eqAG}) is
$0$.
Since the right vertical arrow
in (\ref{eqAG}) is an isomorphism,
the composition of the upper line is also 0.

2.
Define
${\rm H}^0({\mathbf G}_{m,\overline k},{\cal E}nd\, {\cal M})
\to
{\rm H}^1(I_E\times I_T,{\rm End}\, M(1))$
to be the direct sum
\begin{align*}
{\rm H}^0(I_E,{\rm H}^1(I_T,{\rm End}\, M(1)))
\to 
&\,{\rm H}^1(I_E,{\rm End}\, M(1))
\oplus
{\rm H}^0(I_E,{\rm H}^1(I_T,{\rm End}\, M(1)))\\
&\,=
{\rm H}^1(I_E\times I_T,{\rm End}\, M(1))
\end{align*}
of (\ref{eqET}) and the identity
for the direct sum decomposition
in Lemma {\rm \ref{lmI}.1}.
Then,
the 
big rectangle in the diagram
{\rm (\ref{eqrect})} is commutative
by the commutative diagram
$$
\begin{CD}
I_{{\mathbf P}^1}\times I_L
@<{(\ref{eqpiI})}<<
I_E\times I_T\\
@A{1\times 1}AA@AA{(1,0)}A\\
I_E\times I_T
@<{\rm diag}<<
I_E,
\end{CD}$$
Lemma \ref{lmI}.2
and 1.
Hence (\ref{eqET})
at the position of (\ref{eqHH})
in (\ref{eqHH1})
makes the diagram (\ref{eqHH1})
commutative.
Since (\ref{eqH0Gm}) is an isomorphism,
such a morphism is unique and
(\ref{eqHH})
equals (\ref{eqET}).
\qed

}

\subsection{Proper push-forward}\label{sspush}

\begin{lm}\label{lmfcc}
Let 
\begin{equation}
\begin{CD}
C@>g>>D\\
@VcVV@VVdV\\
X@>f>>Y
\end{CD}
\label{eqCDXY}
\end{equation}
be a commutative diagram
of separated schemes of finite type over $k$.

{\rm 1.}
Define $f_!c_!c^!\to d_!d^!f_!$
to be the adjoint of the composition of the
canonical isomorphism
$f_!c_!c^!f^!\to 
d_!g_!g^!d^!$
and the morphism
$d_!g_!g^!d^!\to d_!d^!$
induced by ${\rm adj}_g$.
Then the diagram
\begin{equation}
\begin{CD}
f_!c_!c^!@>>>d_!d^!f_!
\\
@V{{\rm adj}_c}VV@VV{{\rm adj}_d}V\\
f_!@=f_!
\end{CD}
\label{eqfcc1}
\end{equation}
is commutative.

{\rm 2.}
Define $f_*c_*c^*\gets d_*d^*f_*$
to be the adjoint of the composition of the
canonical isomorphism
$f_*c_*c^*f^*\to 
d_*g_*g^*d^*$
and the morphism
$d_*g_*g^*d^*\to d_*d^*$
induced by ${\rm adj}_g$.
Then the diagram
\begin{equation}
\begin{CD}
f_*@=f_*
\\
@V{{\rm adj}_c}VV@VV{{\rm adj}_d}V\\
f_*c_*c^*@<<<d_*d^*f_*
\end{CD}
\label{eqfcc3}
\end{equation}
is commutative.
\end{lm}

\proof{
1. 
By the commutative 
diagram {\rm (\ref{eqCDXY})}, we have
a commutative diagram
$$\begin{CD}
f_!c_!c^!f^!@>{{\rm adj}_g}>>d_!d^!
\\
@V{{\rm adj}_c}VV@VV{{\rm adj}_d}V\\
f_!f^!@>{{\rm adj}_f}>>1_Y.
\end{CD}
$$
By taking the adjoint,
we obtain (\ref{eqfcc1}).

2.~Part 2 is proved
similarly as 1.
\qed

}
\medskip

Let $f\colon X\to Y$ be a 
morphism of smooth schemes over $k$
and $i\colon Z\hookrightarrow X$ be a closed immersion
such that  $g=fi\colon Z\to Y$ is proper.
Consider the following commutative diagram
	\begin{equation}
\begin{CD}
	T^*X@<{c^\vee={\rm pr}_{T^\ast X}}<<
	T^*X\times_XZ@<{a^\vee=df}<<
	T^*Y\times_YZ@>{b^\vee={\rm pr}_{T^\ast Y}}>>
	T^*Y\\
	@V{e^\vee_X}VV @V{e^\vee_i}VV 
	@VV{e^\vee_g}V @VV{e^\vee_Y}V
	\\
	X@<i<<Z@<{1_Z}<< Z@>g>> Y.
\end{CD}
\label{eqT*XXYdefinefshrink}
\end{equation}
Let ${\cal F}\in D_{\rm ctf}(Z,\Lambda)$.
By \cite[Lemma 2.2 (ii)]{SS} and Theorem \ref{thm:singular-support-equality}, we have
\begin{align}
SS_\mu(g_\ast \mathcal F)=
SS(g_\ast \mathcal F)\subseteq 
f_\circ SS(i_*\mathcal F)=
f_\circ SS_\mu(i_*\mathcal F)=b^\vee(a^{\vee-1}(SS_\mu(i_*\mathcal F))).
\end{align}
The base of $SS_\mu(i_*\mathcal F)$ is contained in $Z$ and we have $SS_\mu(i_*\mathcal F)\subseteq T^\ast X\times_XZ$. 
We define
\begin{align}
f_!: {\rm H}^0_{SS_\mu(i_\ast \mathcal F)}(T^\ast X, e_X^{\vee\ast}K_X)\to {\rm H}^0_{f_\circ SS_\mu(i_*\mathcal F)}(T^\ast Y, e_{Y}^{\vee\ast}K_Y)	
\end{align}
to be the following composition
\begin{align}
\begin{aligned}
{\rm H}^0_{SS_\mu(i_*\mathcal F)}(T^\ast X, e_X^{\vee\ast}K_X)&\simeq {\rm H}^0_{SS_\mu(i_*\mathcal F)}(T^\ast X\times_XZ, c^{\vee !}e_X^{\vee\ast}K_X)\\
&\simeq 
{\rm H}^0_{SS_\mu(i_*\mathcal F)}(T^\ast X\times_XZ, e_i^{\vee\ast}i^!K_X)\\
&\simeq {\rm H}^0_{SS_\mu(i_*\mathcal F)}(T^\ast X\times_XZ, e_i^{\vee\ast}K_Z)\\
&\xrightarrow{g_!=b^\vee_\ast a^{\vee\ast}} {\rm H}^0_{f_\circ SS_\mu(i_*\mathcal F)}(T^\ast Y, e_{Y}^{\vee\ast}K_Y).
\end{aligned}
\end{align}

\begin{pr}\label{prpush}
Let $f\colon X\to Y$ be a 
morphism of smooth schemes over $k$
and $i\colon Z\hookrightarrow X$ be a closed immersion
such that  $g=fi\colon Z\to Y$ is proper.
Let ${\cal F}\in D_{\rm ctf}(Z,\Lambda)$.

{\rm 1.
(\cite[Proposition 2.1.6]{AS}, 
\cite[Th\'eor\`eme 4.4, Corollaire 4.8]{SGA5})}
We have
\begin{equation}
C_{Y/k}(g_*{\cal F})=g_*C_{Z/k}({\cal F})
\label{eqpush}
\end{equation}
in ${\rm H}^0(Y,{ K}_Y)$.

{\rm 2.}
The canonical isomorphism
${\rm H}^0_{SS_\mu (i_*{\cal F})}
(T^*X\times_XZ,e_i^{\vee*}K_Z)
\to
{\rm H}^0_{SS_\mu (i_*{\cal F})}
(T^*X,e_X^{\vee*}K_X)$
induced by the isomorphism $K_Z\to i^!K_X$
maps $CC_{\mu, i}({\cal F})$
to $CC_\mu (i_*{\cal F})$.

{\rm 3.}
We have
\begin{equation}
CC_\mu (g_*{\cal F})=
f_! CC_\mu(i_*{\cal F})
\label{eqindexmu}
\end{equation}
in ${\rm H}^0_{f_\circ SS_\mu(i_*{\cal F})
\cup SS_\mu (g_*{\cal F})}
(T^*Y,e_Y^{\vee*}{ K}_Y)={\rm H}^0_{f_\circ SS_\mu(i_*{\cal F})}
(T^*Y,e_Y^{\vee*}{ K}_Y)$.
\end{pr}

\proof{
We apply Lemma \ref{lmfcc}
to the commutative diagram
\begin{equation}
\begin{CD}
	X@<i<< Z@>{1_Z}>> Z@>g>> Y\\
	@V{\delta_X}VV 
	@V{\gamma_i}VV 
	@VV{\gamma_g}V
	@VV{\delta_Y}V\\
	X\times X@<{i\times 1_X}<<
	Z\times X@>{1_Z\times f}>>
	Z\times Y@>{g\times 1_Y}>>
	Y\times Y
\end{CD}
\label{eqXXY}
\end{equation}
as follows.
Set
\begin{align}
\begin{aligned}
{\mathscr H}_X&\,=
{\cal H}om_{X\times X}({\rm pr}_1^*i_*{\cal F},
{\rm pr}_2^!i_*{\cal F}),
\quad
{\mathscr H}_i={\cal H}om_{Z\times X}
({\rm pr}_1^*{\cal F},
{\rm pr}_2^!i_*{\cal F}),
\\
{\mathscr H}_Y&\,=
{\cal H}om_{Y\times Y}({\rm pr}_1^*g_*{\cal F},
{\rm pr}_2^!g_*{\cal F}),
\quad
{\mathscr H}_g={\cal H}om_{Z\times Y}
({\rm pr}_1^*{\cal F},
{\rm pr}_2^!g_*{\cal F}),
\end{aligned}
\end{align}
and identify them
with
\begin{align}
\begin{aligned}
\boxtimes_X&\,=
D_X(i_*{\cal F})\boxtimes i_*{\cal F},\quad\,\,
\boxtimes_i=
D_Z({\cal F})\boxtimes i_*{\cal F},
\\
\boxtimes_Y&\,=
D_Y(g_*{\cal F})\boxtimes g_*{\cal F},\quad
\boxtimes_g=
D_Z({\cal F})\boxtimes g_*{\cal F},
\end{aligned}
\end{align}
by the canonical isomorphisms
\cite[(3.1.1)]{SGA5}.
The canonical morphism
$\boxtimes_X
\to(i\times 1_X)_*
\boxtimes_i$
is an isomorphism.

By applying Lemma \ref{lmfcc}.1
and by \cite[(3.2.1)]{SGA5},
we obtain commutative diagrams
\begin{equation}
\begin{CD}
	\delta_{X!}{\cal H}om_X(i_*{\cal F},i_*{\cal F})
	@<{\simeq}<<
	(i\times 1_X)_!\gamma_{i!}{\cal H}om_Z(
	{\cal F},
	{\cal F})
	\\
	@VVV@VVV
	\\
	{\mathscr H}_X
	@<{\simeq}<<
	(i\times 1_X)_!{\mathscr H}_i,
\end{CD}
\label{equp0}
\end{equation}
\begin{equation}
\begin{CD}
	(1_Z\times f)_!
	\gamma_{i!}{\cal H}om_Z(
	{\cal F},
	{\cal F})
	@>>>
	\gamma_{g!}{\cal H}om_Z(
	{\cal F},g^!g_!{\cal F})
	\\
	@VVV@VVV
	\\
	(1_Z\times f)_!{\mathscr H}_i
	@>{\simeq}>>{\mathscr H}_g,
\end{CD}
\label{equp1}
\end{equation}
\begin{equation}
\begin{CD}
	(g\times 1_Y)_!
	\gamma_{g!}{\cal H}om_Z(
	{\cal F},g^!g_!{\cal F})
	@>{\simeq}>>
	\delta_{Y!}{\cal H}om_Y(g_!{\cal F},g_!{\cal F})
	\\
	@VVV@VVV
	\\
	(g\times 1_Y)_!{\mathscr H}_g
	@>{\simeq}>>
	{\mathscr H}_Y.
\end{CD}
\label{equp2}
\end{equation}
Since $i$ and $g$ are proper,
we identify $i_!=i_*$
and $g_!=g_*$.
Defining vertical arrows by
$1_{i_*{\cal F}}$,
$1_{\cal F}
$,
${\rm adj}\colon 
{\cal F}\to g^!g_!{\cal F}$
and $1_{g_!{\cal F}}$,
we obtain commutative diagrams
\begin{equation}
\begin{CD}
	\delta_{X!}\Lambda
	@>{\rm res}>>
	(i\times 1_X)_!\gamma_{i!}\Lambda
	\\
	@V{1_{i_*{\cal F}}}VV
	@VV{1_{\cal F}}V
	\\
	\delta_{X!}{\cal H}om_X(i_*{\cal F},i_*{\cal F})
	@<{\simeq}<<
	(i\times 1_X)_!\gamma_{i!}{\cal H}om_Z(
	{\cal F},{\cal F}),
\end{CD}
\label{equp6}
\end{equation}
\begin{equation}
\begin{CD}
	(1_Z\times f)_!\gamma_{i!}\Lambda
	@>{\simeq}>>
	\gamma_{g!}\Lambda
	\\
	@V{1_{\cal F}}VV@VV{{\rm adj}}V
	\\
	(1_Z\times f)_!
	\gamma_{i!}{\cal H}om_Z(
	{\cal F},
	{\cal F})
	@>{\rm adj}>>
	\gamma_{g!}{\cal H}om_Z(
	{\cal F},g^!g_!{\cal F}),
\end{CD}
\label{equp3}
\end{equation}
\begin{equation}
\begin{CD}
	(g\times 1_Y)_!\gamma_{g!}\Lambda
	@<<<
	\delta_{Y!}\Lambda
	\\
	@V{{\rm adj}}VV@VV{1_{g_!{\cal F}}}V
	\\
	(g\times 1_Y)_!
	\gamma_{g!}{\cal H}om_Z(
	{\cal F},g^!g_!{\cal F})
	@>{\simeq}>>
	\delta_{Y!}{\cal H}om_Y(g_!{\cal F},g_!{\cal F}).
\end{CD}
\label{equp4}
\end{equation}

By applying Lemma \ref{lmfcc}.2,
we obtain commutative diagrams
\begin{equation}
\begin{CD}
	\boxtimes_X
	@>{\simeq}>>
	(i\times 1_X)_*\boxtimes_i
	\\
	@VVV@VVV
	\\
	\delta_{X*}(D_Xi_*{\cal F}\otimes i_*{\cal F})
	@>{\simeq}>>
	(i\times 1_X)_*\gamma_{i*}(D_Z{\cal F}\otimes {\cal F}),
\end{CD}
\label{eqdown0}
\end{equation}
\begin{equation}
\begin{CD}
	(1_Z\times f)_!\boxtimes_i
	@>{\simeq}>>\boxtimes_g
	\\
	@VVV@VVV
	\\
	(1_Z\times f)_!\gamma_{i*}(D_Z{\cal F}\otimes{\cal F})
	@<{\rm adj}<<
	\gamma_{g*}(D_Z{\cal F}\otimes g^*g_*{\cal F}),
\end{CD}
\label{eqdown1}
\end{equation}
\begin{equation}
\begin{CD}
	(g\times 1_Y)_*\boxtimes_g
	@>{\simeq}>>\boxtimes_Y
	\\
	@VVV@VVV
	\\
	(g\times 1_Y)_*\gamma_{g*}(D_Z{\cal F}\otimes g^*g_*{\cal F})
	@<{\simeq}<<
	\delta_{Y*}(D_Yg_*{\cal F}\otimes g_*{\cal F}).
\end{CD}
\label{eqdown2}
\end{equation}
The lower line in (\ref{eqdown2})
is induced by the isomorphism 
$g_*(D_Z{\cal F}\otimes g^*g_*{\cal F})
\to
g_*D_Z{\cal F}\otimes g_*{\cal F}$
given by the projection formula
and the canonical isomorphism
$g_*D_Z{\cal F}
\to
D_Yg_!{\cal F}\to
D_Yg_*{\cal F}$.

Defining vertical arrows by
${\rm ev}_X\colon D_Xi_*{\cal F}\otimes i_*{\cal F}
\to { K}_X$,
${\rm ev}_Z\colon D_Z{\cal F}\otimes {\cal F}
\to { K}_Z$,
${\rm ev}_Z\circ
(1\otimes {\rm adj})\colon D_Z{\cal F}\otimes g^*g_*{\cal F}
\to D_Z{\cal F}\otimes {\cal F}
\to { K}_Z$
and ${\rm ev}_Y\colon D_Yg_*{\cal F}\otimes g_*{\cal F}
\to { K}_Y$,
we obtain commutative diagrams
\begin{equation}
\begin{CD}
	\delta_{X*}(D_Xi_*{\cal F}\otimes i_*{\cal F})
	@>{\simeq}>>
	(i\times 1_X)_*\gamma_{i*}(D_Z{\cal F}\otimes {\cal F}),
	\\
	@V{{\rm ev}_X}VV
	@VV{{\rm ev}_Z}V
	\\
	\delta_{X*}{ K}_X
	@<{\rm adj}<<
	(i\times 1_X)_*\gamma_{i*}{ K}_Z,
\end{CD}
\label{eqdown6}
\end{equation}
\begin{equation}
\begin{CD}
	(1_Z\times f)_!\gamma_{i*}(D_Z{\cal F}\otimes {\cal F})
	@<{{\rm adj}}<<
	\gamma_{g*}(D_Z{\cal F}\otimes g^*g_*{\cal F})
	\\
	@V{{\rm ev}_Z}VV@VV{{\rm ev}_Z\circ 
		(1\otimes {\rm adj})}V
	\\
	(1_Z\times f)_!\gamma_{i*}{ K}_Z
	@>{\simeq}>>
	\gamma_{g*}{ K}_Z,
\end{CD}
\label{eqdown3}
\end{equation}
\begin{equation}
\begin{CD}
	(g\times 1_Y)_*\gamma_{g*}(D_Z{\cal F}\otimes g^*g_*{\cal F})
	@>>>
	\delta_{Y*}(D_Yg_*{\cal F}\otimes g_*{\cal F})
	\\
	@V{{\rm ev}_Z\circ 
		(1\otimes {\rm adj})}VV
	@VV{{\rm ev}_Y}V
	\\
	(g\times 1_Y)_*\gamma_{g*}{ K}_Z
	@>{\rm adj}>>
	\delta_{Y*}{ K}_Y.
\end{CD}
\label{eqdown4}
\end{equation}
The lower line in
(\ref{eqdown4}) is induced
by the adjunction $g_*{ K}_Z=g_!g^!{ K}_Y\to { K}_Y$.

Commutative diagrams
(\ref{equp0}),
(\ref{equp6}),
(\ref{eqdown0}) and
(\ref{eqdown6}) define
a commutative diagram
\begin{equation}
\begin{CD}\delta_{X!}\Lambda
	@>{\rm res}>>
	(i\times 1_X)_*
	\gamma_{i!}\Lambda
	\\
	@V{1_{i_*{\cal F}}}VV@VV{1_{\cal F}}V\\
	{\mathscr H}_X
	@>{\simeq}>>(i\times 1_X)_*{\mathscr H}_i\\
	@V{{\rm ev}_X}VV@VV{{\rm ev}_Z}V
	\\
	\delta_{X*}{ K}_X
	@<{\rm adj}<<
	(i\times 1_X)_*\gamma_{i*}{ K}_Z,
\end{CD}
\label{eqleft0}
\end{equation}
Similarly, commutative diagrams
(\ref{equp1}),
(\ref{equp3}),
(\ref{eqdown1}) and
(\ref{eqdown3}) define
a commutative diagram
\begin{equation}
\begin{CD}
	(1_Z\times f)_!\gamma_{i!}\Lambda
	@>{\simeq}>>
	\gamma_{g!}\Lambda
	\\
	@V{1_{\cal F}}VV@VV{{\rm adj}}V\\
	(1_Z\times f)_!{\mathscr H}_i
	@>{\simeq}>>{\mathscr H}_g\\
	@V{{\rm ev}_Z}VV@VV{{\rm ev}_g\circ {\rm adj}}V
	\\
	(1_Z\times f)_!\gamma_{i*}{ K}_Z
	@>{\simeq}>>
	\gamma_{g*}{ K}_Z.
\end{CD}
\label{eqleft}
\end{equation}
Similarly, commutative diagrams
(\ref{equp2}),
(\ref{equp4}),
(\ref{eqdown2}) and
(\ref{eqdown4}) define
a commutative diagram
\begin{equation}
\begin{CD}
	(g\times 1_Y)_*\gamma_{g!}\Lambda
	@<<<
	\delta_{Y!}\Lambda
	\\
	@V{\rm adj}VV@VV{1_{g_*{\cal F}}}V
	\\
	(g\times 1_Y)_*{\mathscr H}_g
	@>{\simeq}>>
	{\mathscr H}_Y\\
	@V{{\rm ev}\circ{\rm adj}}VV@VV{\rm ev}V
	\\
	(g\times 1_Y)_*\gamma_{g*}{ K}_Z
	@>>>
	\delta_{Y*}{ K}_Y.
\end{CD}
\label{eqright}
\end{equation}
Since $i$ is a closed immersion,
the diagrams 
(\ref{eqleft}) and (\ref{eqright})
imply (\ref{eqpush}).

The commutative diagram (\ref{eqXXY})
induces a commutative diagram
\begin{equation}
\begin{CD}
	X@<i<<Z@>{1_Z}>> Z@>g>> Y\\
	@V{0_X}VV @V{0_i}VV 
	@VV{0_g}V @VV{0_Y}V\\
	TX@<c<<
	TX\times_XZ@>a>>
	TY\times_YZ@>b>>
	TY.
\end{CD}
\label{eqTXXY}
\end{equation}
The vertical arrows are the $0$-sections.

Since the left square  in (\ref{eqXXY}) is cartesian with proper morphisms,
the base change morphism
$\nu_{X/X\times X}(i\times 1_X)_*
\to c_*\nu_{Z/Z\times X}$
(\ref{eqnufun2}) is an isomorphism.
Hence applying the functor $\nu_{X/X\times X}$ to (\ref{eqleft0}) 
defines
a commutative diagram
\begin{equation}
\begin{CD}
	0_{X!}\Lambda
	@>>>
	c_*0_{i!}\Lambda
	\\
	@VVV@VVV
	\\
	\nu_{X/X\times X}{\mathscr H}_X
	@>{\simeq}>>
	c_*\nu_{Z/Z\times X}{\mathscr H}_i
	\\
	@VVV@VVV
	\\
	0_{X*}{ K}_X
	@<<<
	c_*0_{i*}{ K}_Z.
\end{CD}
\label{eqnul0}
\end{equation}
We apply $\nu_{Z/Z\times Y}$ to
(\ref{eqleft})
and apply the morphism
$a_!\nu_{Z/Z\times X}
\to
\nu_{Z/Z\times Y}(1_Z\times f)_!$
of functors (\ref{eqnufun3}) to the left column.
Then, we obtain a commutative diagram
\begin{equation}
	\begin{CD}
		a_!0_{i!}\Lambda
		@>{\simeq}>>
		0_{g!}\Lambda
		\\
		@VVV@VVV
		\\
		a_!\nu_{Z/Z\times X}{\mathscr H}_i
		@>>>
		\nu_{Z/Z\times Y}{\mathscr H}_g
		\\
		@VVV@VVV
		\\
		a_!0_{i*}{ K}_Z
		@>{\simeq}>>
		0_{g*}{ K}_Z.
	\end{CD}
	\label{eqnul}
\end{equation}
Since the right square in (\ref{eqXXY}) is cartesian with proper morphisms,
the base change morphism
$\nu_{Y/Y\times Y}(g\times 1_Y)_*
\to b_*\nu_{Z/Z\times Y}$
(\ref{eqnufun2}) is an isomorphism.
Hence applying the functor $\nu_{Y/Y\times Y}$ to (\ref{eqright}) 
defines
a commutative diagram
\begin{equation}
	\begin{CD}
		b_*0_{g!}\Lambda
		@<<<
		0_{Y!}\Lambda
		\\
		@VVV@VVV
		\\
		b_*\nu_{Z/Z\times Y}{\mathscr H}_g
		@>{\simeq}>>
		\nu_{Y/Y\times Y}{\mathscr H}_Y\\
		@VVV@VVV
		\\
		b_*0_{g*}{ K}_Z
		@>>>
		0_{Y*}{ K}_Y.
	\end{CD}
	\label{eqnur}
\end{equation}

The lower line of the commutative diagram (\ref{eqTXXY})
induces the upper line of the commutative diagram
\begin{equation}
\begin{CD}
T^*X@<{c^\vee={\rm pr}_{T^\ast X}}<<
T^*X\times_XZ@<{a^\vee=df}<<
T^*Y\times_YZ@>{b^\vee={\rm pr}_{T^\ast Y}}>>
T^*Y\\
@V{e^\vee_X}VV @V{e^\vee_i}VV 
@VV{e^\vee_g}V @VV{e^\vee_Y}V
\\
X@<i<<Z@<{1_Z}<< Z@>g>> Y,
\end{CD}
	\label{eqT*XXY}
\end{equation}
where $a^\vee$ is the dual of $a$.
The vertical arrows are the projections.
By applying the Fourier transform to (\ref{eqnul0})
and the isomorphism (\ref{eqFb!}) to the left column
and using the isomorphisms (\ref{eqF1}), we obtain a 
commutative diagram
\begin{equation}
	\begin{CD}
		\Lambda
		@>>>
		c^\vee_*\Lambda
		\\
		@VVV@VVV
		\\
		\mu_{X/X\times X}{\mathscr H}_X
		@>>>
		c^\vee_*\mu_{Z/Z\times X}{\mathscr H}_i
		\\
		@VVV@VVV
		\\
		e_X^{\vee*}{ K}_X
		@<<<
		c^\vee_*e_i^{\vee*}{ K}_Z.
	\end{CD}
	\label{eqmul0}
\end{equation}
By the canonical isomorphism
$K_Z\to i^! K_X$,
the bottom arrow
induces an isomorphism
${\rm H}^0_{SS(i_*{\cal F})}
(T^*X,e^{\vee*}_XK_X)
\gets
{\rm H}^0_{SS(i_*{\cal F})}
(T^*X\times_XZ,e^{\vee*}_iK_Z)$.
By (\ref{eqmul0}),
the canonical isomorphism
${\rm H}^0_{SS_\mu (i_*{\cal F})}
(T^*X\times_XZ,e_i^{\vee*}K_Z)
\to
{\rm H}^0_{SS_\mu (i_*{\cal F})}
(T^*X,e_X^{\vee*}K_X)$
maps $CC_{\mu, i}{\cal F}$
to $CC_\mu i_*{\cal F}$.

Similarly
by applying the Fourier transform to (\ref{eqnul})
and
the isomorphism (\ref{eqd}) to the left column
and using the isomorphisms (\ref{eqF1}),
we obtain a commutative diagram
\begin{equation}
	\begin{CD}
		a^{\vee*}\Lambda
		@<{\simeq}<<
		\Lambda
		\\
		@VVV@VVV
		\\
		a^{\vee*}\mu_{Z/Z\times X}{\mathscr H}_i
		@>>>
		\mu_{Z/Z\times Y}{\mathscr H}_g
		\\
		@VVV@VVV
		\\
		a^{\vee*}e_i^{\vee*}{ K}_Z
		@<{\simeq}<<
		e_g^{\vee*}{ K}_Z.
	\end{CD}
	\label{eqmul}
\end{equation}
Similarly
by applying the Fourier transform to (\ref{eqnur})
and the isomorphism (\ref{eqFb!}) to the left column
and using the isomorphisms (\ref{eqF1}), we obtain a 
commutative diagram
\begin{equation}
	\begin{CD}
		b^\vee_*\Lambda
		@<<<
		\Lambda
		\\
		@VVV@VVV
		\\
		b^\vee_*\mu_{Z/Z\times Y}{\mathscr H}_g
		@>>>
		\mu_{Y/Y\times Y}{\mathscr H}_Y
		\\
		@VVV@VVV
		\\
		b^\vee_*e_g^{\vee*}{ K}_Z
		@>>>e_Y^{\vee*}{ K}_Y.
	\end{CD}
	\label{eqmur}
\end{equation}
Combining these three diagrams,
we obtain the equality (\ref{eqindexmu}).
\qed

}

\begin{cor}\label{corindex}
Let $Z$ be a proper scheme over $k$
and let $i\colon Z\to X$
be a closed immersion to a smooth scheme over $k$.
Let ${\cal F}\in D_{\rm ctf}(Z,\Lambda)$.
Then, we have
\begin{equation}
	\chi(Z_{\bar k},{\cal F})
	=(CC_{\mu,i}({\cal F}),T_X^*X\times_XZ)_{T^*X\times_XZ}
	\label{eqindex}
\end{equation}
in $\Lambda$.
The right hand side
denotes the image of
$CC_{\mu,i}({\cal F})$
by the composition
\begin{align}
	\begin{CD}
		{\rm H}^0_{SS_\mu(i_\ast{\cal F})}
		(T^*X\times_XZ,e_i^{\vee*}K_Z)
		@>{0_Z^*}>>
		{\rm H}^0
		(Z,K_Z)
		@>{{\rm Tr}_{Z/k}}>>
		\Lambda.
	\end{CD}
\end{align}
\end{cor}

\proof{
Let $f\colon X\to Y={\rm Spec}\, k$
be the canonical morphism and $g=fi$.
Then, we have
$\chi(Z_{\bar k},{\cal F})=
CC_\mu (g_*{\cal F})$.
Since
$(CC_{\mu,i} ({\cal F}),T^*_XX\times_XZ)_{T^*X\times_XZ}
=f_!CC_\mu i_*{\cal F}$,
the equality (\ref{eqindex})
follows from Proposition \ref{prpush}.3.
\qed

}

\subsection{Smooth pull-back}\label{sspull}

\begin{lm}\label{lmhcc}
Let 
\begin{equation}
\begin{CD}
C@<g<<D\\
@VcVV@VVdV\\
X@<h<<W
\end{CD}
\label{eqBCWX}
\end{equation}
be a commutative diagram
of separated schemes of finite type over $k$.

{\rm 1.} 
Define $h^!c_!c^!\gets d_!d^!h^!$
to be the adjoint of the composition of the
canonical isomorphism
$h_!d_!d^!h^!\to 
c_!g_!g^!c^!$
and the morphism
$c_!g_!g^!c^!\to c_!c^!$
induced by ${\rm adj}_g$.
Then the diagram
\begin{equation}
\begin{CD}
h^!c_!c^!@<<<d_!d^!h^!
\\
@V{{\rm adj}_c}VV@VV{{\rm adj}_d}V\\
h^!@=h^!.
\end{CD}
\label{eqhcc2}
\end{equation}
is commutative.

{\rm 2.}
Define $h^*c_*c^*\to d_*d^*h^*$
to be the adjoint of the composition of the
canonical isomorphism
$c_*g_*g^*c^*\to 
h_*d_*d^*h^*$
and the morphism
$c_*c^*\to c_*g_*g^*c^*$
induced by ${\rm adj}_g$.
Then the diagram
\begin{equation}
\begin{CD}
h^*@=h^*
\\
@V{{\rm adj}_c}VV@VV{{\rm adj}_d}V\\
h^*c_*c^*@>>>d_*d^*h^*
\end{CD}
\label{eqhcc4}
\end{equation}
is commutative.
\end{lm}

\proof{
Similar to Lemma \ref{lmfcc}.
\qed

}
\medskip

Let $h\colon W\to X$ be a smooth morphism of
relative dimension $r=\dim W-\dim X$ of
smooth schemes over $k$.
Let (\ref{eqabh}) be the canonical morphisms
of cotangent bundles
and 
let $e^\vee_X\colon T^*X\to X$,
$e^\vee_h\colon T^*X\times_XW\to W$
and $e^\vee_W\colon T^*W\to W$
be the projections.
For a closed conical subset $S\subseteq
T^*X$,
define a morphism
\begin{equation}
h^!\colon {\rm H}^0_S(T^*X,e_X^{\vee\ast}K_X)\to
{\rm H}^0_{h^\circ S}(T^*W,e_W^{\vee\ast}K_W)
\label{eqh!0S}
\end{equation}
to be $(-1)^r$-times
the composition of
$b^{*}\colon
{\rm H}^0_S(T^*X,e_X^{\vee *}K_X)
\to
{\rm H}^0_{b^{-1}S}(T^*X\times_XW,b^{\ast}e_X^{\vee *}K_X)$
and 
$a_*\colon
{\rm H}^0_{b^{-1}S}(T^*X\times_XW,b^{\ast}e_X^{\vee *}K_X)
\to
{\rm H}^0_{a(b^{-1}S)}(T^*W,e_W^{\vee *}K_W)$.
Note that the definition of $ h^!$ involves
the sign $(-1)^{\dim X-\dim W}.$

\begin{lm}\label{lmh!}
Let $h\colon W\to X$ be a smooth morphism of pure
relative dimension $r=\dim W-\dim X$ of
smooth schemes over $k$.
Define a morphism
\begin{equation}
h^!\colon {\rm H}^0(X,K_X)\to
{\rm H}^0(W,K_W)
\label{eqh!0}
\end{equation}
to be the composition
of $h^*\colon
{\rm H}^0(X,K_X)\to
{\rm H}^0(W,h^*K_X)$
and the cup-product
with the top Chern class 
$c_r(TW/X)=(-1)^rc_r(T^*W/X)
\in {\rm H}^0(W,h^!\Lambda)$
of the relative tangent bundle.
Then, the diagram
\begin{equation}
\begin{CD}
{\rm H}^0_S(T^*X,e_X^{\vee*}K_X)
@>{h^!}>>
{\rm H}^0_{h^\circ S}(T^*W,e_W^{\vee*}K_W)
\\
@VVV@VVV\\
{\rm H}^0(X,K_X)
@>{h^!}>>
{\rm H}^0(W,K_W)
\end{CD}
\label{eqh!cr}
\end{equation}
is commutative.
The vertical arrows
are defined by the pull-back by the $0$-sections.
\end{lm}

\proof{
It suffices to show that the diagram
\begin{equation}
\begin{CD}
	{\rm H}^0_S(T^*X,e_X^{\vee*}K_X)
	@>{b^{*}}>>
	{\rm H}^0_{b^{-1}S}(T^*X\times_XW,e_h^{\vee*}h^\ast K_X)
	@>{(-1)^ra_*}>>
	{\rm H}^0_{h^\circ S}(T^*W,e_W^{\vee*}K_W)
	\\
	@VVV@VVV@VVV\\
	{\rm H}^0(X,K_X)
	@>{h^*}>>
	{\rm H}^0(W,h^*K_X)
	@>{c_r(TW/X)}>>
	{\rm H}^0(W,K_W)
\end{CD}
\end{equation}
is commutative.
The commutativity of
the left square is clear from the functoriality.
By the exact sequence
$0\to h^*\Omega^1_X
\to \Omega^1_W
\to \Omega^1_{W/X}
\to 0$,
the morphism
$a_*\colon 
{\rm H}^0_{b^{-1}S}(T^*X\times_XW,e_h^{\vee*}K_X)
\to
{\rm H}^0_{h^\circ S}(T^*W,e_W^{\vee*}K_W)$
is compatible with
the cup-product
\begin{align}
\cup \, c_r(T^*W/X)
\colon
{\rm H}^0(W,h^*K_X)
\to
{\rm H}^0(W,K_W).	
\end{align}
Hence the assertion follows.
\qed

}

\begin{pr}\label{prpull}
Let $h\colon W\to X$ be a smooth
morphism of smooth schemes over $k$
and ${\cal F}\in D_{\rm ctf}(X,\Lambda)$.
Then,
we have
\begin{equation}
CC_\mu h^*({\cal F})=
h^! CC_\mu({\cal F})
\label{eqindexmuh}
\end{equation}
in ${\rm H}^0_{h^\circ(SS_\mu{\cal F})}
(T^*W,e_W^{\vee*}{ K}_W)$
and
\begin{equation}
C_{W/k}(h^*{\cal F})=
h^!C_{X/k}({\cal F})
\label{eqpull}
\end{equation}
in ${\rm H}^0(W,{ K}_W)$.
\end{pr}

\proof{
We apply Lemma \ref{lmhcc}
to the commutative diagram
\begin{equation}
\begin{CD}
X@<h<< W@<{1_W}<< W\\
@V{\delta_X}VV @V{\gamma_h}VV @VV{\delta_W}V\\
X\times X@<{h\times 1_X}<<
W\times X@<{1_W\times h}<<
W\times W
\end{CD}
\label{eqXWW}
\end{equation}
as follows.
Set
\begin{align}
\begin{aligned}
{\mathscr H}_X&\,=
{\cal H}om_{X\times X}({\rm pr}_1^*{\cal F},
{\rm pr}_2^!{\cal F}),\qquad
{\mathscr H}_h={\cal H}om_{W\times X}
({\rm pr}_1^*h^!{\cal F},
{\rm pr}_2^!{\cal F}),\\
{\mathscr H}_W&\,=
{\cal H}om_{W\times W}({\rm pr}_1^*h^!{\cal F},
{\rm pr}_2^!h^!{\cal F})
\end{aligned}
\end{align}
and identify them with
\begin{align}
\boxtimes_X=
D_X({\cal F})\boxtimes {\cal F}
,\
\boxtimes_h=
D_W(h^!{\cal F})\boxtimes {\cal F}
,\
\boxtimes_W=
D_W(h^!{\cal F})\boxtimes h^!{\cal F}
\end{align}
by the canonical isomorphisms
\cite[(3.1.1)]{SGA5}.
We have canonical isomorphisms
$(h\times 1_X)^*{\mathscr H}_X
\to {\mathscr H}_h$ and
$(1_W\times h)^!{\mathscr H}_h
\to {\mathscr H}_W$.
We consider the commutative diagram
\begin{equation}
\begin{CD}
(1_W\times h)^!{\mathscr H}_h
@>>>
{\mathscr H}_W\\
@AAA@AAA\\
(1_W\times h)^*
\boxtimes_h\otimes
{\rm pr}_2^*h^!\Lambda
@>>>
\boxtimes_W
\end{CD}
\label{h*!}
\end{equation}
where the upper horizontal arrow is an isomorphism.

By applying Lemma \ref{lmhcc}.1
and by \cite[(3.2.1)]{SGA5},
we obtain commutative diagrams
\begin{equation}
\begin{CD}
(h\times 1_X)^!\delta_{X!}{\cal H}om_X({\cal F},{\cal F})
@<<<
\gamma_{h!}({\cal H}om_W(
h^!{\cal F},h^!{\cal F})
\otimes h^!\Lambda)
\\
@VVV@VVV
\\
(h\times 1_X)^!{\mathscr H}_X
@<<<{\mathscr H}_h
\otimes{\rm pr}_1^*h^!\Lambda,
\end{CD}
\label{equp1h}
\end{equation}
\begin{equation}
\begin{CD}
(1_W\times h)^!\gamma_{h!}{\cal H}om_W(
h^!{\cal F},h^!{\cal F})
@<<<
\delta_{W!}{\cal H}om_W(h^!{\cal F},h^!{\cal F})
\\
@VVV@VVV
\\
(1_W\times h)^!{\mathscr H}_h
@>{\simeq}>>
{\mathscr H}_W.
\end{CD}
\label{equp2h}
\end{equation}
Since $h$ is smooth,
the canonical morphism
$(h\times 1_X)^*(-)
\otimes {\rm pr}_1^*h^!\Lambda
\to 
(h\times 1_X)^!$
of functors is an isomorphism
and
(\ref{equp1h}) defines a commutative diagram
\begin{equation}
\begin{CD}
(h\times 1_X)^*\delta_{X!}{\cal H}om_X({\cal F},{\cal F})
@<<<
\gamma_{h!}{\cal H}om_W(
h^!{\cal F},h^!{\cal F})
\\
@VVV@VVV
\\
(h\times 1_X)^*{\mathscr H}_X
@<<<{\mathscr H}_h
\end{CD}
\label{equp1hb}
\end{equation}
Since $h$ is smooth,
the canonical morphism
$(1_W\times h)^*(-)
\otimes {\rm pr}_2^*h^!\Lambda
\to 
(1_W\times h)^!$
is an isomorphism
and
(\ref{equp2h}) defines a commutative diagram
\begin{equation}
\begin{CD}
(1_W\times h)^*\gamma_{h!}{\cal H}om_W(
h^!{\cal F},h^!{\cal F})
\otimes {\rm pr}_2^*h^!\Lambda
@<<<
\delta_{W!}{\cal H}om_W(h^!{\cal F},h^!{\cal F})
\\
@VVV@VVV
\\
(1_W\times h)^*{\mathscr H}_h
\otimes {\rm pr}_2^*h^!\Lambda
@>{\simeq}>>
{\mathscr H}_W.
\end{CD}
\label{equp2hb}
\end{equation}
Defining vertical arrows by
$1_{\cal F}$ and $1_{h^!{\cal F}}$,
we obtain commutative diagrams
\begin{equation}
\begin{CD}
(h\times 1_X)^*\delta_{X!}\Lambda
@>{\simeq}>>
\gamma_{h!}\Lambda
\\
@V{1_{\cal F}}VV@VV{1_{h^!{\cal F}}}V
\\
(h\times 1_X)^*\delta_{X!}{\cal H}om_X({\cal F},{\cal F})
@>{\simeq}>>
\gamma_{h!}{\cal H}om_W(
h^!{\cal F},h^!{\cal F}),
\end{CD}
\label{equp3h}
\end{equation}
\begin{equation}
\begin{CD}
(1_W\times h)^!\gamma_{h!}\Lambda
@<{\rm adj}<<
\delta_{W!}\Lambda
\\
@V{1_{h^!{\cal F}}}VV@VV{1_{h^!{\cal F}}}V
\\
(1_W\times h)^!\gamma_{h!}{\cal H}om_W(
h^!{\cal F},h^!{\cal F})
@<<<
\delta_{W!}{\cal H}om_W(h^!{\cal F},h^!{\cal F}).
\end{CD}
\label{equp4h}
\end{equation}
The upper horizontal arrow
in (\ref{equp4h}) is the adjoint
of the isomorphism
$\gamma_{h!}\gets
(1_W\times h)_!\delta_{W!}$.

By applying Lemma \ref{lmhcc}.2,
we obtain commutative diagrams
\begin{equation}
\begin{CD}
(h\times 1_X)^*\boxtimes_X
@>{\simeq}>>\boxtimes_h
\\
@VVV@VVV
\\
(h\times 1_X)^*\delta_{X*}(D_X{\cal F}\otimes{\cal F})
@>{\simeq}>>
\gamma_{h*}(D_Wh^!{\cal F}\otimes h^*{\cal F}),
\end{CD}
\label{eqdown1h}
\end{equation}
\begin{equation}
\begin{CD}
(1_W\times h)^*\boxtimes_h
\otimes {\rm pr}_2^*h^!\Lambda
@>{\simeq}>>\boxtimes_W
\\
@VVV@VVV
\\
(1_W\times h)^*\gamma_{h*}(D_Wh^!{\cal F}\otimes h^*{\cal F}
\otimes
h^!\Lambda)
@>{\rm adj}>>
\delta_{W*}(D_Wh^!{\cal F}\otimes h^!{\cal F}).
\end{CD}
\label{eqdown2h}
\end{equation}
The lower line is induced by the canonical morphism
$h^*\otimes h^!\Lambda\to
h^!$ and
the adjoint
of the isomorphism
$\gamma_{h*}\to 
(1_W\times h)_*\delta_{W*}$.

Defining vertical arrows by
${\rm ev}_X\colon D_X{\cal F}\otimes {\cal F}
\to { K}_X$,
${\rm ev}_X\circ
(1\otimes {\rm adj})\colon D_X{\cal F}\otimes h^*h_*{\cal F}
\to D_X{\cal F}\otimes {\cal F}
\to { K}_X$
and ${\rm ev}_W\colon D_Wh^*{\cal F}\otimes h^*{\cal F}
\to { K}_W$,
we obtain commutative diagrams
\begin{equation}
\begin{CD}
(h\times 1_X)^*\delta_{X*}(D_X{\cal F}\otimes{\cal F})
@<{\simeq}<<
\gamma_{h*}(D_Wh^!{\cal F}\otimes h^*{\cal F})
\\
@V{\rm ev}VV@VV{{\rm ev}}V
\\
(h\times 1_X)^*\delta_{X*}{ K}_X
@>{\simeq}>>
\gamma_{h*}h^*{ K}_X,
\end{CD}
\label{eqdown3h}
\end{equation}
\begin{equation}
\begin{CD}
(1_W\times h)^*\gamma_{h*}(D_Wh^!{\cal F}\otimes h^*{\cal F}
\otimes
h^!\Lambda)
@>{\rm adj}>>
\delta_{W*}(D_Wh^!{\cal F}\otimes h^!{\cal F})
\\
@V{{\rm ev}}VV@VV{\rm ev}V
\\
(1_W\times h)^*
\gamma_{h*}(h^*{ K}_X\otimes
h^!\Lambda)
@>{\rm adj}>>
\delta_{W*}{ K}_W.
\end{CD}
\label{eqdown4h}
\end{equation}
The horizontal lines in (\ref{eqdown4h})
are induced by the morphism 
$h^*\otimes h^!\Lambda
\to h^!$.
The lower line is further induced
by the adjoint of the isomorphism
$\gamma_{h*}\to (1_W\times h)_* \delta_{W*}$
and $h^*{ K}_X\otimes h^!\Lambda\to { K}_W$.
Commutative diagrams
(\ref{equp1hb}),
(\ref{equp3h}),
(\ref{eqdown1h}) and
(\ref{eqdown3h}) define
a commutative diagram
\begin{equation}
\begin{CD}
(h\times 1_X)^*\delta_{X!}\Lambda
@>{\simeq}>>
\gamma_{h!}\Lambda
\\
@V{1_{\cal F}}VV@VV{1_{h^!\cal F}}V\\
(h\times 1_X)^*{\mathscr H}_X
@>{\simeq}>>{\mathscr H}_h\\
@V{\rm ev}VV@VV{{\rm ev}}V
\\
(h\times 1_X)^*\delta_{X*}{ K}_X
@>{\simeq}>>
\gamma_{h*}h^*{ K}_X.
\end{CD}
\label{eqlefth}
\end{equation}
Similarly, commutative diagrams
(\ref{equp2h}),
(\ref{equp4h}),
(\ref{eqdown2h}) and
(\ref{eqdown4h}) define
a commutative diagram
\begin{equation}
\begin{CD}
(1_W\times h)^*\gamma_{h!}h^!\Lambda
@<{\rm adj}<<
\delta_{W!}\Lambda
\\
@V{1_{h^!{\cal F}}}VV@VV{1_{h^!{\cal F}}}V
\\
(1_W\times h)^*{\mathscr H}_h
\otimes {\rm pr}_2^*h^!\Lambda
@>{\simeq}>>
{\mathscr H}_W\\
@V{{\rm ev}}VV@VV{\rm ev}V
\\
(1_W\times h)^*\gamma_{h*}(h^*{ K}_X
\otimes h^!\Lambda)
@>>>
\delta_{W*}{ K}_W.
\end{CD}
\label{eqrighth}
\end{equation}

The commutative diagram (\ref{eqXWW})
induces a commutative diagram
\begin{equation}
\begin{CD}
X@<h<< W@<{1_W}<< W\\
@V{0_X}VV @VV{0_h}V @VV{0_W}V\\
TX@<b<<
TX\times_XW@<a<<
TW.
\end{CD}
\label{eqTXWW}
\end{equation}
The vertical arrows are the $0$-sections.
Applying the functor $\nu_{W/W\times X}$ to 
(\ref{eqlefth}) and
the base change morphism
$b^*\nu_{X/X\times X}
\to
\nu_{W/W\times X}(h\times 1_X)^*$
(\ref{eqnufun1}) to the left column,
we obtain a commutative diagram
\begin{equation}
\begin{CD}
b^*0_{X!}\Lambda
@>{\simeq}>>
0_{h!}\Lambda
\\
@VVV@VVV
\\
b^*\nu_{X/X\times X}{\mathscr H}_X
@>>>
\nu_{W/W\times X}{\mathscr H}_h\\
@VVV@VVV
\\
b^*0_{X*}{ K}_X
@>{\simeq}>>
0_{h*}h^*{ K}_X.
\end{CD}
\label{eqnulh}
\end{equation}
Applying the functor $\nu_{W/W\times W}$ to 
(\ref{eqrighth}) and
the base change morphism
$a^*\nu_{W/W\times W}
\to
\nu_{W/W\times X}(1_W\times h)^*$
(\ref{eqnufun4}) to the left column,
we obtain a commutative diagram
\begin{equation}
\begin{CD}
a^!0_{h!}\Lambda
@<{{\rm adj}_a}<<
0_{W!}\Lambda
\\
@VVV@VVV
\\
a^!\nu_{W/W\times X}{\mathscr H}_h
@<<<
\nu_{W/W\times W}{\mathscr H}_W
\\
@VVV@VVV
\\
a^!0_{h*}h^*{ K}_X
@>{{\rm res}_a}>>
0_{W*}{ K}_W.
\end{CD}
\label{eqnurh}
\end{equation}

Let $b^\vee\colon T^*X\times_XW\to T^*X$
be the canonical morphism
and $a^\vee\colon T^*X\times_XW\to T^*W$ 
be the dual of $a$.
Let $e^\vee_X\colon T^*X\to X$,
$e^\vee_h\colon T^*X\times_XW\to W$
and $e^\vee_W\colon T^*W\to W$
be the projections.
By applying the Fourier transforms to (\ref{eqnulh})
and the isomorphism (\ref{eqFb*}) to the left column
and using the isomorphisms (\ref{eqF1}),
we obtain a commutative diagram
\begin{equation}
\begin{CD}
b^{\vee*}\Lambda
@>{\simeq}>>
\Lambda
\\
@VVV@VVV
\\
b^{\vee*}\mu_{X/X\times X}{\mathscr H}_X
@>>>
\mu_{W/W\times X}{\mathscr H}_h\\
@VVV@VVV
\\
b^{\vee*}e_X^{\vee *}{ K}_X
@>{\simeq}>>
e_h^{\vee*}h^*{ K}_X.
\end{CD}
\label{eqmulh}
\end{equation}
By applying the Fourier transforms to (\ref{eqnurh})
and the isomorphism (\ref{eqd2}) to the left column
and using the isomorphism (\ref{eqF1}),
we obtain a commutative diagram
\begin{equation}
\begin{CD}
a^\vee_!\Lambda
@<<<
\Lambda
\\
@VVV@VVV
\\
a^\vee_!\mu_{W/W\times X}{\mathscr H}_h
@<<<
\mu_{W/W\times W}{\mathscr H}_W
\\
@VVV@VVV
\\
a^\vee_!e_h^{\vee*}h^*{ K}_X
@>>>
e_W^{\vee*}{ K}_W.
\end{CD}
\label{eqmurh}
\end{equation}
By Proposition \ref{prFFdd}.2
and 1.,
the top and the bottom
horizontal arrows in (\ref{eqmurh})
are ${\rm res}_{a^\vee}$
and $(-1)^r{\rm adj}_{a^\vee}$ respectively.
Hence (\ref{eqmulh}) and (\ref{eqmurh})
imply (\ref{eqindexmuh}).
\qed

}

\subsection{External products}
\label{sec:external-products}

Let $X$ and $Y$ be smooth $k$-schemes.  Let
${\mathcal F}\in D_{\mathrm{ctf}}(X,\Lambda)$ and
${\mathcal G}\in D_{\mathrm{ctf}}(Y,\Lambda)$.
We identify
\begin{equation}\label{eq:cotangent-product}
T^*(X\times Y)=T^*X\times T^*Y
\end{equation}
by the canonical decomposition
$\Omega^1_{X\times Y/k}=
\operatorname{pr}_X^*\Omega^1_{X/k}\oplus
\operatorname{pr}_Y^*\Omega^1_{Y/k}$.

For cohomology classes with support, $\alpha\boxtimes\beta$ denotes the
cohomological external product followed by the canonical isomorphism  (cf. \cite[(3.2)]{YZ})
\begin{equation}\label{eq:dualizing-product}
K_{X/k}\boxtimes  K_{Y/k}\xrightarrow{\sim} K_{X\times Y/k}.
\end{equation}

\begin{lemma}
\label{lem:product-specialization}
Let
$\mathscr H_{\mathcal F}=\mathcal Hom(p_{1,X}^*{\mathcal F},p_{2,X}^!{\mathcal F})$ and
$\mathscr H_{\mathcal G}=\mathcal Hom(p_{1,Y}^*{\mathcal G},p_{2,Y}^!{\mathcal G})$.
There are canonical isomorphisms
\begin{align}
	\mathscr H_{{\mathcal F}\boxtimes {\mathcal G}}
	&\simeq
	\mathscr H_{\mathcal F}\boxtimes\mathscr H_{\mathcal G},
	\label{eq:kernel-product}\\
	\nu_{\delta_{X\times Y}}(\mathscr H_{{\mathcal F}\boxtimes {\mathcal G}})
	&\simeq
	\nu_{\delta_X}(\mathscr H_{\mathcal F})
	\boxtimes_{\eta_0}
	\nu_{\delta_Y}(\mathscr H_{\mathcal G}).
	\label{eq:specialization-product}
\end{align}
\end{lemma}
\begin{proof}
This follows from Lemma \ref{lem:Kunneth-for-muHom}.
\end{proof}

\begin{theorem}[Microlocal K\"unneth formula]
\label{thm:microlocal-Kunneth}
After the identification \eqref{eq:cotangent-product}, there is a canonical
isomorphism
\begin{equation}\label{eq:microlocal-Kunneth}
	{\mu\mathcal Hom}({\mathcal F}\boxtimes {\mathcal G},{\mathcal F}\boxtimes {\mathcal G})
	\xrightarrow{\sim}
	{\mu\mathcal Hom}({\mathcal F},{\mathcal F})\boxtimes{\mu\mathcal Hom}({\mathcal G},{\mathcal G}).
\end{equation}
Consequently, we have
\begin{align}
	\label{eq:SSmu-product}
	&{SS}_{\mu}({\mathcal F}\boxtimes {\mathcal G})={SS}_{\mu}({\mathcal F})\times{SS}_{\mu}({\mathcal G}),\\
	\label{eq:CCmu-product}
	&{CC}_{\mu}({\mathcal F}\boxtimes {\mathcal G})={CC}_{\mu}({\mathcal F})\boxtimes{CC}_{\mu}({\mathcal G}).
\end{align}
\end{theorem}

\begin{proof}
\eqref{eq:microlocal-Kunneth} and \eqref{eq:SSmu-product} have been proved in Lemma \ref{lem:Kunneth-for-muHom}.
The equality \eqref{eq:CCmu-product} follows from
Theorem \ref{thm:singular-support-equality}, Theorem \ref{prMilnor} 
and \cite[Theorem 2.2]{prod}. 
Here, we provide an independent proof of \eqref{eq:CCmu-product}.

Consider the kernel $\mathscr H_{\mathcal F}$.
The cohomological characteristic class of ${\mathcal F}$ is represented by the composition
\begin{equation}\label{eq:diagonal-trace-F}
	\Lambda_X\xrightarrow{1_{\mathcal F}}\delta_X^!\mathscr H_{\mathcal F}
	\longrightarrow\delta_X^*\mathscr H_{\mathcal F}
	\xrightarrow{\operatorname{ev}_{\mathcal F}} K_{X/k}.
\end{equation}
We have the analogous composites for ${\mathcal G}$ and ${\mathcal F}\boxtimes {\mathcal G}$.
The compatibility of \eqref{eq:kernel-product} with the identity morphisms
gives a commutative diagram
\begin{equation}\label{eq:identity-product-diagram}
	\begin{tikzcd}[column sep=large]
		\Lambda_X\boxtimes\Lambda_Y
		\arrow[r,"1_{\mathcal F}\boxtimes1_{\mathcal G}"]\arrow[d,"\simeq"']
		&\delta_X^!\mathscr H_{\mathcal F}\boxtimes\delta_Y^!\mathscr H_{\mathcal G}
		\arrow[d,"\simeq"]\\
		\Lambda_{X\times Y}\arrow[r,"1_{{\mathcal F}\boxtimes {\mathcal G}}"']
		&\delta_{X\times Y}^!\mathscr H_{{\mathcal F}\boxtimes {\mathcal G}}
	\end{tikzcd}
\end{equation}
The canonical morphisms $\delta^!\to\delta^*$ give a commutative diagram
\begin{equation}\label{eq:diagonal-shriek-star-product-diagram}
	\begin{tikzcd}[column sep=large]
		\delta_X^!\mathscr H_{\mathcal F}\boxtimes\delta_Y^!\mathscr H_{\mathcal G}
		\arrow[r]\arrow[d,"\simeq"']
		&\delta_X^*\mathscr H_{\mathcal F}\boxtimes\delta_Y^*\mathscr H_{\mathcal G}
		\arrow[d,"\simeq"]\\
		\delta_{X\times Y}^!\mathscr H_{{\mathcal F}\boxtimes {\mathcal G}}
		\arrow[r]
		&\delta_{X\times Y}^*\mathscr H_{{\mathcal F}\boxtimes {\mathcal G}}.
	\end{tikzcd}
\end{equation}
Its commutativity follows from the naturality of
$\delta^!\to\delta^*$ and its compatibility with external products.  The
evaluation morphisms give a commutative diagram
\begin{equation}\label{eq:evaluation-product-diagram}
	\begin{tikzcd}[column sep=large]
		\delta_X^*\mathscr H_{\mathcal F}\boxtimes\delta_Y^*\mathscr H_{\mathcal G}
		\arrow[r,"\operatorname{ev}_{\mathcal F}\boxtimes\operatorname{ev}_{\mathcal G}"]
		\arrow[d,"\simeq"']
		& K_{X/k}\boxtimes  K_{Y/k}\arrow[d,"\simeq"]\\
		\delta_{X\times Y}^*\mathscr H_{{\mathcal F}\boxtimes {\mathcal G}}
		\arrow[r,"\operatorname{ev}_{{\mathcal F}\boxtimes {\mathcal G}}"']
		& K_{X\times Y/k}.
	\end{tikzcd}
\end{equation}
The vertical morphisms are the external-product isomorphisms and
\eqref{eq:dualizing-product}.  Diagram
\eqref{eq:identity-product-diagram} follows from
$\operatorname{id}_{{\mathcal F}\boxtimes {\mathcal G}}=\operatorname{id}_{\mathcal F}\boxtimes\operatorname{id}_{\mathcal G}$.  To verify
\eqref{eq:evaluation-product-diagram}, use
$\mathscr H_{\mathcal F}\simeq { D}_X{\mathcal F}\boxtimes {\mathcal F}$ and
$\mathscr H_{\mathcal G}\simeq { D}_Y{\mathcal G}\boxtimes {\mathcal G}$.  Under
${ D}_{X\times Y}({\mathcal F}\boxtimes {\mathcal G})={ D}_X{\mathcal F}\boxtimes{ D}_Y{\mathcal G}$, the evaluation 
${\rm ev}_{{\mathcal F}\boxtimes {\mathcal G}}$ is the composite of the symmetry isomorphism
\begin{align}
	({ D}_X{\mathcal F}\boxtimes{ D}_Y{\mathcal G})\otimes({\mathcal F}\boxtimes {\mathcal G})
	\xrightarrow{\sim}
	({ D}_X{\mathcal F}\otimes {\mathcal F})\boxtimes({ D}_Y{\mathcal G}\otimes {\mathcal G})
\end{align}
and $\operatorname{ev}_{\mathcal F}\boxtimes\operatorname{ev}_{\mathcal G}$, followed by
\eqref{eq:dualizing-product}.  Hence
\eqref{eq:evaluation-product-diagram} is commutative.

By adjunction, \eqref{eq:diagonal-trace-F} defines a correspondence
\begin{equation}\label{eq:adjoint-diagonal-trace-F}
	\delta_{X!}\Lambda_X\longrightarrow\mathscr H_{\mathcal F}
	\longrightarrow\delta_{X*}  K_{X/k}.
\end{equation}
Let $0_X:X\hookrightarrow TX$ be the zero section.  Applying Verdier
specialization to \eqref{eq:adjoint-diagonal-trace-F}, we obtain
\begin{equation}\label{eq:specialized-trace-F}
	0_{X!}\Lambda_X\longrightarrow\nu_{\delta_X}(\mathscr H_{\mathcal F})
	\longrightarrow0_{X*}  K_{X/k}.
\end{equation}
Here we use the canonical isomorphisms
$\nu_{\delta_X}(\delta_{X!}\Lambda_X)=0_{X!}\Lambda_X$ and
$\nu_{\delta_X}(\delta_{X*}  K_{X/k})=0_{X*}  K_{X/k}$.
Applying the  Fourier transformation and the canonical isomorphisms
\begin{equation}\label{eq:Fourier-zero-identities}
	F_{\psi,TX/X}(0_{X!}\Lambda_X)\simeq\Lambda_{T^*X},
	\qquad
	F_{\psi,TX/X}(0_{X*} K_{X/k})\simeq e_X^{\vee *} K_{X/k}
\end{equation}
we obtain morphisms
\begin{equation}\label{eq:microlocal-trace-maps}
	\Lambda_{T^*X}\xrightarrow{u_{\mathcal F}}{\mu\mathcal Hom}({\mathcal F},{\mathcal F})
	\xrightarrow{v_{\mathcal F}}e_X^{\vee *} K_{X/k}.
\end{equation}
These are the two structural morphisms defining ${CC}_{\mu}({\mathcal F})$.

By  \cite[Proposition 3.1]{LuZheng} and Proposition \ref{prop:Kunneth-for-Fourier},
the three commutative diagrams
\eqref{eq:identity-product-diagram},
\eqref{eq:diagonal-shriek-star-product-diagram}, and
\eqref{eq:evaluation-product-diagram} can be specialized simultaneously. 
We obtain a commutative diagram
\begin{equation}\label{eq:microlocal-trace-product-diagram}
	\begin{tikzcd}
		\Lambda_{T^*X}\boxtimes\Lambda_{T^*Y}
		\arrow[r,"u_{\mathcal F}\boxtimes u_{\mathcal G}"]\arrow[d,"\simeq"']
		&\mu\mathcal{H}om({\mathcal F},{\mathcal F})\boxtimes \mu\mathcal{H}om({\mathcal G},{\mathcal G})
		\arrow[r,"v_{\mathcal F}\boxtimes v_{\mathcal G}"]\arrow[d,"\simeq"']
		&e_X^{\vee *}  K_{X/k}\boxtimes e_Y^{\vee *}  K_{Y/k}
		\arrow[d,"\simeq"]\\
		\Lambda_{T^*(X\times Y)}
		\arrow[r,"u_{{\mathcal F}\boxtimes {\mathcal G}}"']
		& \mu\mathcal{H}om({\mathcal F}\boxtimes {\mathcal G},{\mathcal F}\boxtimes {\mathcal G})
		\arrow[r,"v_{{\mathcal F}\boxtimes {\mathcal G}}"']
		&e_{X\times Y}^{\vee *}  K_{X\times Y/k}.
	\end{tikzcd}
\end{equation}
The left vertical
morphism is induced by the external product of constant sheaves, and the
right one is induced by \eqref{eq:dualizing-product}.  Thus
\eqref{eq:microlocal-trace-product-diagram} is compatible with both
structural morphisms \eqref{eq:microlocal-trace-maps} defining $CC_\mu$.

We now pass to cohomology with support.  Let $i_{\mathcal F}:{SS}_{\mu}({\mathcal F})\hookrightarrow T^\ast X $ and $i_{\mathcal G}: {SS}_{\mu}({\mathcal G}) \hookrightarrow T^\ast Y$ be the
closed immersions.  By \eqref{eq:SSmu-product}, we have
$SS_\mu({\mathcal F}\boxtimes {\mathcal G})={SS}_{\mu}({\mathcal F})\times {SS}_{\mu}({\mathcal G})$.  By definition, ${CC}_{\mu}({\mathcal F})$ is represented
on ${SS}_{\mu}({\mathcal F})$ by the composite
\begin{equation}\label{eq:supported-microlocal-trace-F}
	\Lambda\longrightarrow i_{\mathcal F}^* \mu\mathcal{H}om({\mathcal F},{\mathcal F})
	\longrightarrow i_{\mathcal F}^!e_X^{\vee *}  K_{X/k},
\end{equation}
obtained from \eqref{eq:microlocal-trace-maps} by adjunction (cf. \eqref{eqccFnmi}).  By
\eqref{eq:microlocal-trace-product-diagram}, the external product of
\eqref{eq:supported-microlocal-trace-F} and its analogue for ${\mathcal G}$ is the
corresponding composite on ${SS}_{\mu}({\mathcal F})\times {SS}_{\mu}({\mathcal G})$.  Here we use the canonical
isomorphism
\begin{align}
	&(i_{\mathcal F}^!e_X^{\vee *}  K_{X/k})\boxtimes
	(i_{\mathcal G}^!e_Y^{\vee *}  K_{Y/k})\xrightarrow{\sim}
	(i_{\mathcal F}\times i_{\mathcal G})^!
	e_{X\times Y}^{\vee *}  K_{X\times Y/k}.
\end{align}
Taking ${\rm H}^0$, the resulting class is
${CC}_{\mu}({\mathcal F})\boxtimes{CC}_{\mu}({\mathcal G})$.  By the lower row of
\eqref{eq:microlocal-trace-product-diagram}, it is also
${CC}_{\mu}({\mathcal F}\boxtimes {\mathcal G})$.  This proves \eqref{eq:CCmu-product}.
\end{proof}

\subsection{Characteristic cycles}\label{ssCC}

In this subsection, we prove that $CC_\mu={\rm cl}(CC)$. We first prove the curve case.

\begin{pr}\label{prdim1}
Let $X$ be a smooth curve over a perfect field $k$ and $\mathcal F\in D_{\rm ctf}(X,\Lambda)$. Then $CC_\mu(\mathcal F)$ is the cycle class of $CC(\mathcal F)$. 
\end{pr}

\proof{
Let $j\colon U\to X$ be a dense open immersion
such that $j^*{\cal F}$
is locally constant.
Since the assertion is local on $X$
and since we may replace $k$ by an algebraic closure,
we may assume that
$U=X\sm \{x\}$ is
the complement of
a $k$-rational point $x\in X$.
It suffices to show that
the coefficient $a_x$ of
the fiber $T^*_xX$ in $CC_\mu({\cal F})$
equals $-a_x({\cal F})$.

Let $i\colon x\to X$ be the closed immersion.
Then, by Corollary \ref{coradd},
we have
$CC_\mu({\cal F})=
CC_\mu (j_!j^*{\cal F})+
CC_\mu (i_*i^*{\cal F})$.
Since $CC_\mu (i_*i^*{\cal F})
={\rm rank}\, i^*{\cal F}\cdot [T^*_xX]$
by Lemma \ref{lmlccC} and
Proposition \ref{prpush}.2,
we may assume that $j_!j^*{\cal F}\to {\cal F}$
is an isomorphism.

Since $j^*{\cal F}\in D_{\rm ctf}(U,\Lambda)$,
by \cite[Lemme 4.5.1]{Rapp},
there exist a finite complex 
${\cal L}_\bullet$ of
locally constant sheaves
of free $\Lambda$-modules
of finite rank
and a quasi-isomorphism
${\cal L}_\bullet\to j^*{\cal F}$.
Hence by Proposition \ref{pradd},
we may further assume that
$j^*{\cal F}$ is a locally constant sheaf
of free $\Lambda$-modules
of finite rank.

In the case where $X={\mathbf A}^1$,
$x=0$ 
and ${\cal F}$ is a locally constant constructible
sheaf of free $\Lambda$-modules of rank $n$
on $U={\mathbf G}_m$
tamely ramified at $x$,
we have
$CC_\mu ({\cal F})=
-n([T^*_XX]+[T^*_xX])$
by Corollary \ref{cortame}
and hence $CC_\mu({\cal F})={\rm cl}\, CC({\cal F})$.

We show the general case
by reducing to this case. 
Since the assertion is \'etale local on $X$,
by \cite[Theorem 4.1]{KG},
we may assume that $X={\mathbf P}^1\sm \{0\}$,
$x=\infty$ and $U={\mathbf G}_m$
and that
$j^*{\cal F}$ is a locally constant constructible
sheaf of free $\Lambda$-modules of rank $n$
tamely ramified at $0$.
Let $\overline j\colon X={\mathbf P}^1\sm \{0\}
\to {\mathbf P}^1$ be the open immersion.
We have
$SS_\mu (\overline j_!{\cal F})
\subseteq
T^*_{\mathbf P^1}\mathbf P^1\cup T^*_0\mathbf P^1\cup T^*_\infty \mathbf P^1$.
By the tamely ramified case proved above
and by the semi-purity
\cite[Rappel 2.2.8]{Cycle},
we have
$CC_\mu (\overline j_!{\cal F})
=-n([T^*_{\mathbf P^1}{\mathbf P^1}]+[T^*_0\mathbf P^1])-a_\infty[T^*_\infty \mathbf P^1]$
for some $a_\infty\in \Lambda$.
We have
$\chi(\mathbf P^1_{\bar k},\overline j_!{\cal F})=\chi_c(X_{\overline k},{\cal F})=
n-a_\infty$
by Corollary \ref{corindex}.
Since
$\chi_c(X_{\overline k},{\cal F})=
n-a_\infty{\cal F}$
by the Grothendieck--Ogg--Shafarevich formula,
we obtain $a_\infty=a_\infty{\cal F}$
and $CC_\mu({\cal F})={\rm cl}\, CC({\cal F})$.
\qed

}

\begin{theorem}\label{prMilnor}
Let $X$ be a smooth connected scheme over a perfect field $k$
and ${\cal F}\in D_{\rm ctf}(X,\Lambda)$.
We have
\begin{equation}\label{eq:CC-cycleFinite}
{CC}_{\mu}({\mathcal F})=\operatorname{cl}({CC}({\mathcal F})),
\end{equation}
where ${\rm cl}: Z_{\mathrm{dim}(X)}({SS}({\mathcal F}))=\bigoplus\limits_{i\in I} \mathbf Z\cdot C_i\to {\rm H}^0_{{SS}({\mathcal F})}(T^\ast X, e_X^{\vee\ast}  K_{X/k})$ is the cycle class map and $C_i (i\in I)$ are irreducible components of $SS(\mathcal F)$.
\end{theorem}

\proof{
The proof is similar to that of
the functorial characterization of 
characteristic cycles
\cite[Proposition 8]{Saltlake}.
Let $C=SS_\mu({\mathcal F})=SS({\mathcal F})\subseteq T^*X$.
By \cite[1.3 Theorem (ii)]{SS},
we have $\dim C\leqq \dim X$.
Hence by the semi-purity
\cite[Rappel 2.2.8]{Cycle},
we have
${\rm H}^0_C(T^*X,e^{\vee*}K_X)
=\bigoplus_a\Lambda[C_a]$
where $C_a$ runs through irreducible components of
$C$ of dimension $\dim X$
and hence
$CC_\mu ({\mathcal F})=
\sum_a\mu_a[C_a]$ for some
$\mu_a\in \Lambda$.
Let 
$CC ({\mathcal F})=
\sum_am_a[C_a]$.
It suffices to show the equality
$\mu_a= m_a$ in $\Lambda$
for each irreducible component
$C_a$.

Since the question is local on $X$,
we may assume that $X$ is affine.
By applying Proposition \ref{prpush}
to an immersion $X\to {\mathbf A}^n$,
we may assume $X={\mathbf A}^n$.
Further,
since the question is local on $X$,
we may assume that $X$ is projective.
We fix a closed immersion $X\to {\mathbf P}^N$.

Let $C_a$ be an irreducible component
of $C$ of dimension $\dim X$.
Then, there exists a linear subvariety
$A\subseteq {\mathbf P}^N$ of codimension 2
and $x\in X\sm X\cap A$
satisfying the following conditions.
The intersection $X\cap A$
is transversal.
Let $X'\to X$ be the blow-up
at $X\cap A$ and let
$f\colon X'\to Y={\mathbf P}^1$
be the morphism to the projective line
parametrizing hyperplanes containing $A$.
The fibers of $f$ are the intersections with
hyperplanes.
The point $x$ 
is an isolated characteristic point 
of $f\colon X'\to Y$
and $x$ is the unique characteristic point 
in the fiber $f^{-1}(y)$ for $y=f(x)$.
The component 
$C_a\subseteq T^*X$ is the unique irreducible component
of $C$ meeting the section $df$ at $x$.
The intersection number
$(C_a,df)_x$ equals 1,
if necessary replacing $X$ by $X\times {\mathbf P}^1$
in the case $p=2$
\cite[Proposition 5.19]{CC}.

Let $m_a$ be the coefficient of
$C_a$ in $CC({\mathcal F})$
and $\mu_a\in \Lambda$
be the coefficient of
${\rm cl}\, C_a$ in $CC_\mu({\mathcal F})$.
Then, we have
$(CC({\mathcal F}),df)_x=m_a(C_a,df)_x
\in {\mathbf Z}$
and
$(CC_\mu({\mathcal F}),df)_x=\mu_a(C_a,df)_x\in \Lambda$.
By the Milnor formula,
for the complex of vanishing cycles
$\Phi_x({\cal F},f)$,
we have
$\dim{\rm tot}\, \Phi_x({\cal F},f)
=-
(CC({\mathcal F}),df)_x=-m_a(C_a,df)_x$.
Since $f_!CC_\mu({\mathcal F})
=CC_\mu (f_*{\cal F})$
by Proposition \ref{prpush},
the intersection product
$(CC_\mu({\mathcal F}),df)_x=\mu_a(C_a,df)_x$
equals the coefficient $-a_yf_*{\cal F}$
of ${\rm cl}\, T^*_yY$
in $f_!CC_\mu({\mathcal F})
=CC_\mu (f_*{\cal F})$.
By the distinguished triangle
$\Gamma(X_{\overline y},{\cal F})
\to
\Gamma(X_{\overline \eta},{\cal F})
\to
\Phi_x({\cal F},f)\to$
for a geometric generic point $\overline \eta$ of
the strict localization at $\overline y$,
we have
$\dim{\rm tot}\, \Phi_x({\cal F},f)
=a_yf_*{\cal F}$.
Thus we have
$\dim{\rm tot}\, \Phi_x({\cal F},f)
=-m_a(C_a,df)_x
=-\mu_a(C_a,df)_x$.
Since $(C_a,df)_x=1$,
we obtain
$m_a=\mu_a$ in $\Lambda$
and hence $CC_\mu({\mathcal F})={\rm cl}\, CC({\mathcal F})$.
\qed

}

\medskip

\begin{corollary}\label{cor:push-CC}
Let $f\colon X\to Y$ be a 
morphism of smooth schemes over a perfect field $k$
and $i\colon Z\hookrightarrow X$ be a closed immersion
such that  $g=fi\colon Z\to Y$ is proper.
Let ${\cal F}\in D_{\rm ctf}(Z,\Lambda)$.
We have
\begin{equation}
{\rm cl}CC (g_*{\cal F})=
f_!{\rm cl} CC(i_*{\cal F}) \qquad{\rm in}\qquad {\rm H}^0_{f_\circ SS(i_*{\cal F})}
(T^*Y,e_Y^{\vee*}{ K}_Y).
\label{eq:prpush-CC}
\end{equation}
In particular, if $Y={\rm Spec}\,k$, then we have
\begin{equation}\label{eq:indexCCinLambda}
\chi(Z_{\bar k},\mathcal F)=(CC(i_\ast \mathcal F), T^\ast_XX\times_XZ)_{T^\ast X\times_XZ}\qquad{\rm in}\qquad \Lambda.
\end{equation}
\end{corollary}
\begin{proof}
This follows from Proposition \ref{prpush}.3 and Theorem \ref{prMilnor}.
\end{proof}
\begin{remark}
In \cite[Conjecture 1]{Saltlake}, the first named author conjectures that we have equality $CC(g_\ast\mathcal F)=f_! CC(i_\ast \mathcal F)$ in ${\rm CH}_{{\rm dim}Y}(f_\circ SS(i_\ast\mathcal F))$. If ${\rm dim}f_\circ SS(i_\ast \mathcal F)\leq {\rm dim}Y$, then ${\rm H}^0_{f_\circ SS(i_\ast \mathcal F)}(T^\ast Y, e_{Y}^{\vee\ast}K_Y)\simeq {\rm CH}_{{\rm dim}Y}(f_\circ SS(i_\ast\mathcal F))\otimes\Lambda\simeq  {\rm Z}_{{\rm dim}Y}(f_\circ SS(i_\ast \mathcal F))\otimes\Lambda$.
When $f$ is quasi-projective and $Y$ is proper, \eqref{eq:prpush-CC} is proved in \cite[Theorem 2.2.5]{SaitoCond}. For \eqref{eq:indexCCinLambda},
we show in Corollary \ref{cor:index} that the equality holds in $\mathbf Z$.
\end{remark}

Let $f\colon X\to Y$ be a morphism from a smooth scheme $X$
to a smooth connected curve $Y$ over a perfect field $k$,
and let $i\colon Z\hookrightarrow X$ be a closed immersion
such that $g=fi\colon Z\to Y$ is proper.
Let ${\mathcal F}\in D_{\rm ctf}(Z,\Lambda)$.
Assume that $f$ is properly $SS(i_\ast\mathcal F)$-transversal over a dense open subset $V\subseteq Y$. 
Let $y\in Y\setminus V$ be a closed point.
We define an intersection number
$(CC(i_\ast\mathcal F),df)_{T^\ast X,Z_y}$ as follows.

Let $t$ be
a local parameter at $y$. 
After shrinking $Y$ around $y$, we may assume that $Y\setminus\{y\}\subseteq V$ and
that $dt$ is nowhere zero. Thus $dt$ defines a section
$dt\colon Y\to T^\ast Y$.
Let $\theta_{f,t}\colon X\to T^\ast X$ be the composition
\begin{equation}
	X\xrightarrow{dt}X\times_YT^\ast Y
	\xrightarrow{df}T^\ast X.
\end{equation}
Since $f$ is properly $SS(i_\ast\mathcal F)$-transversal over $Y\setminus\{y\}$, we have
	$\theta_{f,t}^{-1}{SS}(i_\ast{\mathcal F})\subseteq Z_y$.
Hence
\begin{align}
	\theta_{f,t}^\ast{\rm cl}{CC}(i_\ast{\mathcal F})
	\in {\rm H}^0_{Z_y}(X,K_{X/k})\qquad{\rm and} \qquad \theta_{f,t}^!CC(i_\ast \mathcal F)\in {\rm CH}_0(Z_y),
\end{align}
where $\theta_{f,t}^!$ is the refined Gysin pull-back along
$\theta_{f,t}$.
The compatibility between the cycle class map and refined Gysin
pull-back gives the commutative diagram (cf. \cite[(1.11)]{AS})
\begin{align}
\begin{aligned}
\xymatrix{
{\rm CH}_{{\rm dim}X}(SS(i_\ast\mathcal F))\ar[d]_-{\theta_{f,t}^!}\ar[r]^-{\rm cl}&{\rm H}^0_{{SS}(i_\ast{\mathcal F})}
(T^\ast X,e_X^{\vee\ast}K_{X/k})\ar[d]^-{\theta_{f,t}^\ast}\\
{\rm CH}_0(Z_y)\ar[r]_-{\rm cl}&{\rm H}^0_{Z_y}(X,K_{X/k}).
}
\end{aligned}
\end{align}
It follows that
\begin{equation}\label{eq:I-f-y}
	\theta_{f,t}^\ast{\rm cl}{CC}(i_\ast{\mathcal F})={\rm cl}(\theta_{f,t}^!CC(i_\ast \mathcal F))
	\qquad{\rm in}\qquad {\rm H}^0_{Z_y}(X,K_{X/k}).
\end{equation}
We define (cf. \cite[(2.5)]{SaitoCond})
\begin{align}\label{eq:I-f-y2}
	(CC(i_\ast\mathcal F),df)_{T^\ast X,Z_y}:= {\rm deg}_{k(y)} (\theta_{f,t}^!CC(i_\ast \mathcal F)).
\end{align}
Here $\deg_{k(y)}$ denotes the degree map
$\mathrm{CH}_0(Z_y)\to \mathrm{CH}_0(y)=\mathbf Z$,
which is defined since $Z_y$ is proper over $y$.
The resulting class and  its degree are independent
of the choice of the local parameter $t$, since
$CC(i_*\mathcal F)$ is conical.
\begin{proposition}
	\label{prop:NA-comparison}
	Let $f\colon X\to Y$ be a morphism from a smooth scheme $X$
	to a smooth connected curve $Y$ over a perfect field $k$,
	and let $i\colon Z\hookrightarrow X$ be a closed immersion
	such that $g=fi\colon Z\to Y$ is proper.
	Let ${\mathcal F}\in D_{\rm ctf}(Z,\Lambda)$.
	Let $y\in Y$ be a closed point. Assume that $f$ is properly $SS(i_\ast\mathcal F)$-transversal over a dense open subset $V\subseteq Y$.
	Then we have
	\begin{equation}\label{eq:NA-comparison}
		-a_y(g_\ast{\mathcal F})=(CC(i_\ast\mathcal F),df)_{T^\ast X,Z_y}
		\qquad\text{in}\qquad
		{\rm H}^0_y(Y,K_{Y/k})\simeq\Lambda,
	\end{equation}
	where $a_y$ is the Artin conductor.
\end{proposition}

Assuming an additional condition, we prove in Corollary~\ref{cor:bloch}
that the equality \eqref{eq:NA-comparison} holds in $\mathbf Z$.

\begin{proof}
	Let $t$ be
	a local parameter at $y$. 
	After shrinking $Y$ around $y$, we may assume that $Y\setminus\{y\}\subseteq V$ and
	that $dt$ is nowhere zero. 
	Consider the correspondence
	\begin{align}\label{eq:correspondenceTYYSTSTS}
		T^\ast X
		\xleftarrow{\,df\,}
		X\times_YT^\ast Y
		\xrightarrow{\,\operatorname{pr}\,}
		T^\ast Y.
	\end{align}
	By the assumption, we have
	$dt^{-1}\bigl(f_\circ{SS}(i_\ast{\mathcal F})\bigr)
	\subseteq\{y\}$ and a 
	commutative diagram
	\begin{equation}\label{eq:commutativeintheproofofcoreTYYSTSTS}
		\begin{tikzcd}[column sep=large,row sep=large]
			{\rm H}^0_{{SS}(i_\ast{\mathcal F})}
			(T^\ast X,e_X^{\vee\ast}K_{X/k})
			\arrow[r,"f_!"]
			\arrow[d,"\theta^\ast_{f,t}"']
			&
			{\rm H}^0_{f_\circ{SS}(i_\ast{\mathcal F})}
			(T^\ast Y,e_Y^{\vee\ast}K_{Y/k})
			\arrow[d,"\theta^\ast_{{\rm id},t}"]
			\\
			{\rm H}^0_{Z_y}(X,K_{X/k})
			\arrow[r,"g_\ast"']
			&
			{\rm H}^0_y(Y,K_{Y/k}).
		\end{tikzcd}
	\end{equation}
	By Corollary~\ref{cor:push-CC}, we have
	\begin{equation}
		f_!{\rm cl}{CC}(i_\ast{\mathcal F})
		=
		{\rm cl}{CC}(g_\ast{\mathcal F})
		\qquad\text{in}\qquad
		{\rm H}^0_{f_\circ{SS}(i_\ast{\mathcal F})}
		(T^\ast Y,e_Y^{\vee\ast}K_{Y/k}).
	\end{equation}
	On the curve $Y$, the coefficient of $T_y^\ast Y$ in
	${CC}(g_\ast{\mathcal F})$ is $-a_y(g_\ast{\mathcal F})$.
	Since $dt$ is disjoint from the zero section and meets $T_y^\ast Y$
	transversally with multiplicity one, it follows that
	\begin{align}\label{eq:I-f-y3}
		g_\ast\theta^\ast_{f,t}
		\bigl({\rm cl}{CC}(i_\ast{\mathcal F})\bigr)
		=
		-a_y(g_\ast{\mathcal F})\cdot{\rm cl}[y].
	\end{align}
The equality \eqref{eq:NA-comparison} now follows from \eqref{eq:I-f-y}, \eqref{eq:I-f-y2} and \eqref{eq:I-f-y3}.
\end{proof}

\subsection{An Artin--Schreier sheaf}\label{ssAS}

In this section, we compute $CC_\mu$
for an Artin--Schreier sheaf
with wild ramification,
directly without using
Proposition \ref{prdim1}.

\begin{lm}\label{lmqcc}
Let
$$\begin{CD}
C@>g>> D\\
@VcVV@VVdV\\
X@>q>>A
\end{CD}$$
be a cartesian diagram
of separated schemes of finite type over $k$.

{\rm 1.} Assume that the vertical arrows
are closed immersions
and define a morphism
$q^*d_!d^!\to c_!c^!q^*$
by
$q^*{\cal H}om(d_*\Lambda, -)
\to
{\cal H}om(c_*\Lambda, q^*-)$.
Then, the diagram
\begin{equation}
\begin{CD}
c_!c^!q^*
@<<<
q^*d_!d^!\\
@V{{\rm adj}_c}VV
@VV{{\rm adj}_d}V\\
q^*@=q^*
\end{CD}
\label{eqqcc}
\end{equation}
is commutative.

{\rm 2.}
The diagram
\begin{equation}
\begin{CD}
q^*@>{{\rm adj}_d}>>q^*d_*d^*\\
@V{{\rm adj}_c}VV@VV{\rm bc}V
\\
c_*c^*q^*
@<{\rm can}<<
c_*g^*d^*
\end{CD}
\label{eqccq}
\end{equation}
is commutative.
\end{lm}

\proof{
1.
This follows from the commutative diagram
$$\begin{CD}
q^*{\cal H}om(d_*\Lambda, -)
@>>>
{\cal H}om(c_*\Lambda, q^*-)\\
@VVV@VVV\\
q^*{\cal H}om(\Lambda, -)
@=
{\cal H}om(\Lambda, q^*-).
\end{CD}$$

2.
By adjunction
and by the definition of
$q^*d_*\to c_*g^*$, 
this follows from the commutative diagram
\begin{equation*}
\begin{CD}
1@>{{\rm adj}_d}>>d_*d^*\\
@V{{\rm adj}_{qc}}VV@VV{{\rm adj}_g}V
\\
q_*c_*c^*q^*
@<{\rm can}<<
d_*g_*g^*d^*.
\end{CD}
\end{equation*}

\qed

}

\begin{pr}\label{prAS}
Let $j_X\colon X\sm \{0\}=
{\rm Spec}\, k[x^{\pm 1}]
\to X={\mathbf A}^1=
{\rm Spec}\, k[x]$ be the open immersion
and let
${\cal F}=j_{X!}{\cal L}_\psi(1/x)$.
Then, we have
\begin{equation}
CC_\mu({\mathcal F})
=
-{\rm cl}([T^*_XX]+2\cdot [T^*_0X])
\label{eqAS}
\end{equation}
in ${\rm H}^0_{T^*_XX\cup T^*_0X}(
T^*X,e^{\vee*}K_X)$.
\end{pr}

\proof{
Set $X\times X={\rm Spec}\,k[x,y]$
and let 
$$A_X(X\times X)
={\mathbf A}^3={\rm Spec}\,
k[x,u,v]
\to
{\mathbf A}^1={\rm Spec}\,
k[u]$$
be the deformation 
to the normal bundle
$T=TX={\rm Spec}\, k[x,v]$
for $v=(y-x)/u$.
Let $q\colon X\to W={\rm Spec}\, k[w]$
be the morphism defined by $k[w]\to k[x]$
sending $w$ to $-x^2$.
We construct a commutative diagram
\begin{equation}
\begin{CD}
X\times {\mathbf A}^1
@>>>
W\times {\mathbf A}^1
&={\rm Spec}\, k[w,u]
@<1<<&
W\times {\mathbf A}^1\\
@V\delta VV@VV{0_A}V&&@VV{0_B}V\\
A_X(X\times X)@>{q_A}>>
A&={\rm Spec}\, k[w,u,v]
@<{\overline g}<<
\overline B\supset& B
&={\rm Spec}\, k[w,u,s].
\end{CD}
\label{eq103}
\end{equation}
The left vertical arrow
$\delta
\colon X\times {\mathbf A}^1={\rm Spec}\, k[x,u]
\to 
A_X(X\times X)={\rm Spec}\, k[x,u,v]$
is a lifting of the diagonal
$X\times {\mathbf A}^1
\to 
X\times X\times {\mathbf A}^1$
and is defined by the surjection
$k[x,u,v]\to k[x,u]$ sending $v$ to $0$.
The morphism
$q_A\colon 
A_X(X\times X)={\rm Spec}\, k[x,u,v]\to
A={\rm Spec}\, k[w,u,v]$
is defined by $w=-x(x+uv)$
and the left square is cartesian.
The morphism
$\overline g\colon \overline B\to A$
is the blow-up at the ideal $(w,v)$
and $B\subseteq \overline B$
is the complement of the
proper transform of the divisor $w=0$.
The restriction $g\colon B\to A$
of $\overline g$ is a morphism
of line bundles over $W\times{\mathbf A}^1$
defined by $v=ws$
and the vertical arrows 
in the left square are the $0$-sections.

The base change of the lower line
to the origin
$0\to {\mathbf A}^1={\rm Spec}\, k[u]$
gives the upper line of the diagram
\begin{equation}
\begin{CD}
TX@>{q_T}>>
E&={\rm Spec}\, k[w,v]@<{\overline g_F}<<
\overline F\supset
&F&={\rm Spec}\, k[w,s]\\
@VeVV@VV{e_E}V&&@VV{e_F}V\\
X
@>q>>
W
&@<1<<&
W
\end{CD}
\label{eq104}
\end{equation}
of line bundles.
The right square in (\ref{eq103})
is the base change of
that in (\ref{eq104}).

We have
an isomorphism
$$(j_X\times j_X)_!{\cal L}(1/y-1/x)
\otimes {\rm pr}_1^*K_X\to
{\cal H}om({\rm pr}_1^*{\cal F},
{\rm pr}_2^!{\cal F})$$
on $X\times X$.
Let
${\mathscr H}_X$ on $X\times X\times {\mathbf G}_m$
be the pull-back
of ${\cal H}om({\rm pr}_1^*{\cal F},
{\rm pr}_2^!{\cal F}).$
Let $e_A\colon A\to W$ be the projection
and $j_A\colon 
{\rm Spec}\, k[w^{\pm 1},u,v]\to A$
be the open immersion.
Define ${\mathscr H}_A$ on $A$ by
$${\mathscr H}_A= j_{A!}{\cal L}_\psi(uv/w)\otimes e_A^*K_W$$
and let ${\mathscr H}_{A^\circ}$ be
the restriction on $A^\circ =A\sm E$.
We extend the canonical isomorphism
$(q^*K_W)|_{X\ssm \{0\}}\to K_X|_{X\ssm \{0\}}$
uniquely to an
isomorphism $q^*K_W\to K_X$
and identify them
and let $q_{A^\circ }
\colon X\times X\times {\mathbf G}_m
\to A^\circ$ be the restriction of $q_A$.
Then,
since $$1/y-1/x=-(y-x)/xy=-(y-x)/(x(x+uv))=uv/w,$$
we have a canonical isomorphism
\begin{equation}
{\mathscr H}_X\gets q_{A^\circ}^*{\mathscr H}_{A^\circ}.
\label{eqqHA}
\end{equation}
By Lemma \ref{lmqcc}
applied to $q_{A^\circ}\colon
X\times X\times {\mathbf G}_m\to A^\circ$,
the isomorphism
(\ref{eqqHA})
defines a commutative diagram
\begin{equation}
\begin{CD}
\delta^\circ_!\delta^{\circ!}
{\mathscr H}_X@<<<
q_{A^\circ}^*0_{A^\circ!}0_{A^\circ}^!
{\mathscr H}_{A^\circ}\\
@VVV@VVV\\
{\mathscr H}_X@<<<
q_{A^\circ}^*{\mathscr H}_{A^\circ}\\
@VVV@VVV\\
\delta^\circ_*\delta^{\circ*}
{\mathscr H}_X@<<<
q_{A^\circ}^*0_{A^\circ*}0_{A^\circ}^*
{\mathscr H}_{A^\circ}.
\end{CD}
\label{eqHXA}
\end{equation}
Here and in the following,
the superscript $^\circ$ indicates the base change
by ${\mathbf G}_m\to {\mathbf A}^1$.
By the isomorphisms
\cite[(3.2.1) and (3.1.1)]{SGA5},
the identity of ${\cal F}$
and the evaluation define
isomorphisms
\begin{equation}
	j_{X\ast}\Lambda
	\to 
	{\cal H}om({\cal F},{\cal F})
	\to
	\delta^{\circ!}
	{\mathscr H}_X,\quad
	\delta^{\circ*}
	{\mathscr H}_X\to 
	D_X{\cal F}\otimes {\cal F}
	\to
	j_{X!}\Lambda.
	\label{eq478}
\end{equation}

Let $e_B\colon B\to W$ be the projection
and let $j_B\colon B\to \overline B$
be the open immersion.
Define 
${\mathscr H}_B$ on $B$
by
$${\mathscr H}_B= {\cal L}_\psi(us)\otimes e_B^*K_W$$
and
set ${\mathscr H}_{\overline B}= j_{B!}{\mathscr H}_B$.
Let $g^\circ\colon 
B^\circ =B\sm F
\to A^\circ =A\sm E$
be the restriction of $g$
and $\overline g^\circ\colon 
\overline B^\circ =\overline B\sm \overline F
\to A^\circ$
be the restriction of $\overline g$.
Let 
${\mathscr H}_{B^\circ}$
and ${\mathscr H}_{\overline B^\circ}$
be the restrictions of
${\mathscr H}_B$ 
and ${\mathscr H}_{\overline B}$.
Then, since $u$ is invertible on $B^\circ$,
we have isomorphisms
\begin{equation}
\overline g^\circ_*j_{B!}({\cal L}(us)|_{B^\circ})=
g^\circ_!({\cal L}(us)|_{B^\circ})=
{\cal L}(uv/w)|_{A^\circ}.
\label{eqgBL}
\end{equation}
This induces an isomorphism
\begin{equation}
{\mathscr H}_{A^\circ}\to \overline g^\circ _*{\mathscr H}_{\overline B^\circ}
=g^\circ_!{\mathscr H}_{B^\circ}.
\label{eqgHB}
\end{equation}

We consider the cartesian
diagram
$$\begin{CD}
A@<{\overline g}<< 
\overline B
@<{j_B}<<B\\
@A{0_A}AA
@AA{+_{\overline B}}A
@AA{+_B}A\\
W\times {\mathbf A}^1
@<<< 
\overline g^{-1}(W\times {\mathbf A}^1)
@<<< 
g^{-1}(W\times {\mathbf A}^1).
\end{CD}$$
The inverse image
$\overline g^{-1}(W\times {\mathbf A}^1)$
is the union of
the $0$-section $0_B(W\times {\mathbf A}^1)$
and the fiber
$\overline B_0=\overline B\times_W0$.
Similarly as (\ref{eqHXA}),
by Lemma \ref{lmfcc} applied to $\overline g^\circ \colon
\overline B^\circ \to A^\circ$,
the isomorphism (\ref{eqgHB})
defines a commutative diagram
\begin{equation}
\begin{CD}
0_{A^\circ !}0_{A^\circ}^!
{\mathscr H}_{A^\circ} 
@>>>
\overline g^\circ_*+_{\overline B!}^\circ
+_{\overline B}^{\circ!}
{\mathscr H}_{\overline B^\circ}
\\
@VVV@VVV\\
{\mathscr H}_{A^\circ}
@<{\simeq}<<
\overline g^\circ_*
{\mathscr H}_{\overline B^\circ}
\\
@VVV@VVV\\
0_{A^\circ *}0_{A^\circ}^*
{\mathscr H}_{A^\circ} 
@<<<
g_{*}^\circ
+_{\overline B*}^\circ
+_{\overline B}^{\circ*}
{\mathscr H}_{\overline B^\circ}
\end{CD}
\label{eqHAB}
\end{equation}
on $A^\circ$.
By proper base change theorem,
the horizontal arrows are isomorphisms.

We compute the top and the bottom of
the right column of
(\ref{eqHAB}).
Let $j_W\colon {\mathbf G}_m\to W$
and $j_{W\times {\mathbf A}^1}
\colon {\mathbf G}_m\times {\mathbf A}^1\to W\times {\mathbf A}^1$
be the open immersions.
The canonical morphism
$+_{\overline B}^*
{\mathscr H}_{\overline B^\circ}
\otimes
+_{\overline B}^!\Lambda
\to
+_{\overline B}^!
{\mathscr H}_{\overline B^\circ}$
is an isomorphism
and we have a distinguished triangle
$\Lambda_{\overline B_0}\to 
+_{\overline B}^!e^*_{\overline B}K_W
\to 0_{\overline B*}j_{W\times {\mathbf A}^1}\Lambda
\to $.
Since $\overline g^\circ_*({\cal L}(us)|_{\overline B_0^\circ})=0$
(\ref{eqgBL}),
we have canonical isomorphisms
\begin{align}
&\overline g^\circ_*
+_{\overline B!}^\circ
+_{\overline B}^{\circ!}
{\mathscr H}_{\overline B^\circ}
\to 0_{A^\circ!}j_{W\times {\mathbf A}^1*}\Lambda
=\overline g^\circ_*
0_{B^\circ!}j_{W\times {\mathbf A}^1*}\Lambda,
\label{eqg++*}
\\
&\overline g^\circ_*
+_{\overline B*}^\circ
+_{\overline B}^{\circ*}
{\mathscr H}_{\overline B^\circ}
\gets 0_{A^\circ*}j_{W\times {\mathbf A}^1!}K_W
=\overline g^\circ_*
0_{B^\circ!}j_{W\times {\mathbf A}^1!}K_W.
\label{eqg++}
\end{align}

Let $\eta_0$ be the generic point of
the henselization of ${\mathbf A}^1={\rm Spec}\, k[u]$
at $0\in {\mathbf A}^1$
and let
$\nu_A\colon D_{\rm ctf}(A^\circ,\Lambda)
\to D_{\rm ctf}(E\times_k\eta_0,\Lambda)$
be the specialization functor.
Let $j_F\colon F\to \overline F$ be
the open immersion.
By Proposition \ref{prGL}.1,
the projection $\overline B\to W\times {\mathbf A}^1$
is locally acyclic relatively to ${\mathscr H}_{\overline B}$.
Hence by \cite[Corollaire 2.7]{App},
the composition
$\overline B\to {\mathbf A}^1={\rm Spec}\, k[u]$
is locally acyclic relatively to ${\mathscr H}_{\overline B}$
and
we have
$\Psi {\mathscr H}_{\overline B}=
{\mathscr H}_{\overline B}|_{\overline F}
=
j_{F!}e_F^*{ K}_W$.
Here $\Psi$ denotes the nearby cycles functor
for $\overline B\to {\mathbf A}^1={\rm Spec}\, k[u]$
at $0\in  {\mathbf A}^1$.
Let $g_F\colon F\to E$ be
the restriction of $g\colon B\to A$.
Then, by proper base change theorem,
the isomorphisms (\ref{eqg++*}) and (\ref{eqg++})
induce isomorphisms
\begin{equation}
\nu_A \overline g^\circ_*\overline+^\circ_{\bar B!}\overline+_{\bar B}^{\circ!}
{\mathscr H}_{\overline B^\circ}
\to g_{F!}
0_{F!}j_{W*}\Lambda,
\quad
\nu_A \overline g^\circ_*
+_{\overline B*}^\circ
+_{\overline B}^{\circ*}
{\mathscr H}_{\overline B^\circ}
\to g_{F!}
0_{F!}j_{W!}K_W.
\label{eqng++}
\end{equation}
Hence
the specialization of
(\ref{eqHAB}) defines
a commutative diagram
\begin{equation}
\begin{CD}
\nu_A 0_{A^\circ!}0_{A^\circ}^!
{\mathscr H}_{A^\circ}
@>>>
g_{F!}0_{F!}j_{W*}\Lambda
\\
@VVV@VVV\\
\nu_A{\mathscr H}_{A^\circ}
@<\simeq<<
g_{F!}
e_F^*K_W
\\
@VVV@VVV\\
\nu_A0_{A^\circ*}0_{A^\circ}^*
{\mathscr H}_{A^\circ}
@<<<
g_{F!}0_{F*}j_{W!}K_W
\end{CD}
\label{eqnuAB}
\end{equation}
on $E$
where the horizontal arrows are isomorphisms.

Let $0_{\overline B^\circ}\colon W\times {\mathbf G}_m
\to \overline B^\circ$ be the $0$-section and
consider the canonical morphisms
\begin{equation}
0_{\overline B^\circ!}0_{\overline B^\circ}^!
{\mathscr H}_{\overline B^\circ}
\to
+_{\overline B!}^\circ
+_{\overline B}^{\circ!}
{\mathscr H}_{\overline B^\circ}
\to
{\mathscr H}_{\overline B^\circ}
\to
+_{\overline B*}^\circ
+_{\overline B}^{\circ*}
{\mathscr H}_{\overline B^\circ}
\to
0_{\overline B^\circ*}0_{\overline B^\circ}^*
{\mathscr H}_{\overline B^\circ},
\label{eq115}
\end{equation}
where the middle arrows induce the right columns
of (\ref{eqHAB}) and (\ref{eqnuAB}).
Then, the specialization of compositions
in (\ref{eq115})
defines
\begin{equation}
0_{F!}\Lambda
\to
e_F^*K_W
\to
0_{F*}K_W
\label{eq0F}
\end{equation}
where the first arrow
is the adjoint of the canonical isomorphism
$\Lambda\to 0_F^!e_F^*K_W$
and the second arrow is the restriction.

Let $\mu_A\colon D_{\rm ctf}(A^\circ,\Lambda)
\to D_{\rm ctf}(E^\vee\times_k\eta_0,\Lambda)$
denote the composition of $\nu_A$ 
with the Fourier transform $F_\psi$.
Let $g_F^\vee\colon E^\vee\to F^\vee$
be the dual of $g_F\colon F\to E$
and define a cartesian diagram
$$\begin{CD}
g_F^{\vee-1}(W)@>>> W\\
@V{+_{E^\vee}}VV@VV{0_{F^\vee}}V\\
E^\vee @>{g_F^\vee}>> F^\vee.
\end{CD}$$
Let $e_{E^\vee}\colon E^\vee\to W$
and $e_{F^\vee}\colon F^\vee\to W$
be the canonical morphisms.
Since $F_\psi g_{F!}
e_F^*K_W
=g_F^{\vee*}0_{F^\vee*}\Lambda
=+_{E^\vee*}\Lambda$
by Proposition \ref{prFa},
the Fourier transform of
(\ref{eqnuAB}) gives
\begin{equation}
\begin{CD}
\mu_A 0_{A^\circ!}0_{A^\circ}^!
{\mathscr H}_{A^\circ}
@>>>
e_{E^\vee}^*
j_{W*}\Lambda
\\
@VVV@VVV\\
\mu_A{\mathscr H}_{A^\circ}
@<\simeq<<
+_{E^\vee*}\Lambda
\\
@VVV@VVV\\
\mu_A0_{{A^\circ}*}0_{A^\circ}^*
{\mathscr H}_{A^\circ}
@<<<
e_{E^\vee}^*
j_{W!}K_W
\end{CD}
\label{eqmuAB}
\end{equation}
on the dual $E^\vee$.
To compute the Fourier transform of the
right column,
we apply the isomorphism
$Fg_{F!}\to g_F^{\vee*}F$ (\ref{eqd})
and the isomorphisms
(\ref{eqF1}) and
(\ref{eqF0}).
We consider the commutative diagram
\begin{equation}
\begin{CD}
e_{E^\vee}^*
\Lambda
@<<<
F_\psi 0_{E!}\Lambda
@<<< g_F^{\vee*}\Lambda
\\
@VVV@VVV@VVV\\
+_{E^\vee*}\Lambda
@<<<
F_\psi 
g_{F!}e_F^*K_W
@<<<
g_F^{\vee*}0_{F^\vee!}\Lambda
\\
@VVV@VVV@VVV\\
e_{E^\vee}^*
K_W
@<<<
F_\psi 0_{E*}K_W
@<<<
g_F^{\vee*}e_{F^\vee}^*
K_W.
\end{CD}
\label{eqFAB}
\end{equation}
The left column is obtained from
the right column of (\ref{eqmuAB})
by taking the compositions
with $\Lambda\to j_{W*}\Lambda$
and $j_{W!}K_W\to K_W$.
The right half is obtained by 
applying
the isomorphism $F_\psi g_{F!}\gets g_F^{\vee*}F_\psi$
(\ref{eqd})
to (\ref{eq0F}).
By the description above of (\ref{eq0F})
and Proposition \ref{prEE'},
the composition of the right column is $g_F^{\vee*}$ of
$-1$-times the cycle class of the $0$-section.
By the commutativity of the diagram, 
the left column is 
$-1$-times the cycle class of $g_F^{\vee-1}(W)$.

Let
$\nu_X\colon D_{\rm ctf}(X\times X\times {\mathbf G}_m,\Lambda)
\to D_{\rm ctf}(TX\times_k\eta_0,\Lambda)$
be the specialization functor
and let $\mu_X$ be the composition
with the Fourier transform.
Let $0_X\colon X\to TX$
be the $0$-section.
The specialization of the commutative diagram
(\ref{eqHXA})
defines a commutative diagram
\begin{equation}
\begin{CD}
0_{X!}0_X^!
\nu_X {\mathscr H}_X@<<<
q_T^*\nu_A 0_{A!}0_A^!
{\mathscr H}_A
\\
@VVV@VVV\\
\nu_X{\mathscr H}_X@<<<
q_T^*
\nu_A{\mathscr H}_A\\
@VVV@VVV\\
0_{X*}0_X^*
\nu_X{\mathscr H}_X@<<<
q_T^*\nu_A 0_{A*}0_A^*
{\mathscr H}_A.
\end{CD}
\label{eqnuXA}
\end{equation}
By Proposition \ref{prverdier}.2,
the isomorphisms (\ref{eq478})
induce
isomorphisms
\begin{equation}
	j_{X*}\Lambda
	\to 
	0_X^!\nu_X
	{\mathscr H}_X,\quad
	0_X^*\nu_X
	{\mathscr H}_X
	\to
	j_{X!}\Lambda.
	\label{eq4721}
\end{equation}
Let $q_{T^*}\colon T^*X\to E^\vee$
denote the morphism induced by $q_T$
on the dual.
Then the Fourier transform 
of (\ref{eqnuXA}) gives
\begin{equation}
\begin{CD}
e^{\vee *}j_{X*}\Lambda
@<<< q_{T^*}^*
\mu_A 0_{A!}0_A^!
{\mathscr H}_A
\\
@VVV@VVV\\
\mu_X{\mathscr H}_X
@<<<
q_{T^*}^*\mu_A{\mathscr H}_A
\\
@VVV@VVV\\
e^{\vee *}j_{X!}K_{X\ssm 0}
@<<<
q_{T^*}^*\mu_A 0_{A*}0_A^*
{\mathscr H}_A
\end{CD}
\label{eqmuXA}
\end{equation}
on $T^\ast X$.
Since the ramification index of
$q\colon X
\to W$ is 2,
the pull-back
$q_{T^*}^*\colon
{\rm H}^0_+(E^\vee,e_E^{\vee*}K_W)
\to
{\rm H}^0_+(T^*X,e_X^{\vee*}K_X)$
maps the class of the composition
of the left column in
(\ref{eqFAB})
to the class $-([T^*_XX]+2\cdot [T^*_0X])$.
Thus, combining
(\ref{eqmuAB})
and
(\ref{eqmuXA}),
we obtain
(\ref{eqAS}).
\qed

}

\subsection{Characteristic classes and characteristic cycles}
\label{sec:characteristic-classes}

In this subsection, we show that the cohomological characteristic class $C_{X/k}$ is the cycle class of the characteristic cycle $CC$. This proves the conjecture {\cite[Conjecture~6.8]{CC}} formulated by the first named author.

Let $X$ be a smooth scheme over a perfect field $k$. 
Let $e_X:TX\to X$ and $e^\vee_X:T^*X\to X$ be the projections, and let $0:X\hookrightarrow TX,0^\vee:X\hookrightarrow T^*X$ be the zero sections. Let $\delta:X\hookrightarrow X\times X$ be the diagonal morphism.

Let $i:Z\hookrightarrow X$ be a closed immersion.  
By transitivity of extraordinary pull-back, we have a canonical isomorphism $ K_{Z/k}= i^! K_{X/k}$.  Excision gives an isomorphism
\begin{equation}\label{eq:excision}
\iota_Z:{\rm H}^0(Z, K_{Z/k})\xrightarrow{\sim}{\rm H}_Z^0(X, K_{X/k}).
\end{equation}
Let $0_X:X\hookrightarrow T^*X$ be the zero section.  Let
${\mathcal F}\in D_{\mathrm{ctf}}(Z,\Lambda)$. Since both
${SS}_{\mu}(i_*{\mathcal F})$ and ${SS}(i_*{\mathcal F})$ lie over $Z$,   pull-back by the zero-section defines
a morphism
\begin{equation}\label{eq:supported-pullback}
0_{X,Z}^*:{\rm H}^0_{{SS}(i_\ast{\mathcal F})}(T^*X,e_X^{\vee *} K_{X/k})
\longrightarrow {\rm H}_Z^0(X, K_{X/k}).
\end{equation}
By Lemma \ref{lmccFZ} and Proposition \ref{prpush}.2, we have
\begin{equation}\label{eq:micro-zero}
	0_{X,Z}^*{CC}_{\mu}(i_*{\mathcal F})=\iota_Z(C_{Z/k}({\mathcal F})).
\end{equation}
\begin{lemma}
\label{lem:cycle-zero}
Let $i:Z\hookrightarrow X$ be a closed immersion to a  smooth scheme $X$ of dimension $d$ over $k$.   
Let $\operatorname{cl}_X: {\rm CH}_{d}({SS}(i_*{\mathcal F}))\to {\rm H}^0_{{SS}(i_*{\mathcal F})}(T^*X,e_X^{\vee *} K_{X/k})$ be the cycle-class map with support.  Then we have
\begin{equation}\label{eq:cycle-zero}
	0_{X,Z}^*\operatorname{cl}_X({CC}(i_*{\mathcal F}))
	=\iota_Z\bigl(\operatorname{cl}_Z(cc_{Z,0}({\mathcal F}))\bigr),
\end{equation}
where $cc_{Z,0}({\mathcal F})$ denotes the characteristic class associated with the cycle ${CC}(i_\ast {\mathcal F})$ $($see \cite[Definition 6.7]{CC}$)$.
\end{lemma}

\begin{proof}
The cycle ${CC}(i_*{\mathcal F})$ is a cycle on $T^*X$ whose base is contained in $Z$.
Since the zero section is a regular immersion, the compatibility of the
cycle-class map with refined Gysin pull-back gives a commutative diagram (cf. \cite[(1.11)]{AS})
\begin{equation}
	\begin{aligned}
		\begin{tikzcd}[column sep=large]
			{\rm CH}_{d}({SS}(i_*{\mathcal F}))\arrow[r,"\operatorname{cl}_X"]
			\arrow[d,"0_X^!"']
			&{\rm H}^0_{{SS}(i_*{\mathcal F})}(T^*X,e_X^{\vee *} K_{X/k})
			\arrow[d,"0_{X,Z}^*"]\\
			{\rm CH}_0(Z)\arrow[r,"\iota_Z\circ\operatorname{cl}_Z"']
			&{\rm H}_Z^0(X, K_{X/k}).
		\end{tikzcd}
	\end{aligned}
\end{equation}
The lower horizontal morphism is the cycle-class map followed by the excision
isomorphism, since the intersection is supported on $Z$.  Applying the
diagram to ${CC}(i_*{\mathcal F})$, we obtain
\begin{align}
	0_{X,Z}^*\operatorname{cl}_X({CC}(i_*{\mathcal F}))
	=\iota_Z\operatorname{cl}_Z(0_X^!{CC}(i_*{\mathcal F}))=\iota_Z\operatorname{cl}_Z(cc_{Z,0}({\mathcal F})),
\end{align}
where the second equality follows from
\cite[Definition~6.7.2 and Lemma~6.9.1]{CC}.
\end{proof}

\begin{theorem}[{\cite[Conjecture~6.8]{CC}}]
\label{thm:characteristic-class}
Let $Z$ be a $k$-scheme admitting a closed immersion into a smooth
$k$-scheme.  For every
${\mathcal F}\in D_{\mathrm{ctf}}(Z,\Lambda)$, we have
\begin{equation}\label{eq:full-characteristic}
	C_{Z/k}({\mathcal F})=\operatorname{cl}_Z(cc_{Z,0}({\mathcal F}))
	\quad\text{in }{\rm H}^0(Z, K_{Z/k}).
\end{equation}
\end{theorem}

\begin{proof}
Take a closed immersion $i:Z\hookrightarrow X$ with $X$ smooth.  
We first assume that there exists a non-trivial character $\psi:\mathbf F_p\to \Lambda^\times$.
Applying
Theorem~\ref{prMilnor}, we have
\begin{equation}\label{eq:ambient-CC}
	{CC}_{\mu}(i_*{\mathcal F})=\operatorname{cl}_X({CC}(i_*{\mathcal F}))
	\quad\text{in }{\rm H}^0_{{SS}(i_*{\mathcal F})}(T^*X,e_X^{\vee *} K_{X/k}).
\end{equation}
Apply $0_{X,Z}^*$ to \eqref{eq:ambient-CC}.  By 
\eqref{eq:micro-zero} and Lemma \ref{lem:cycle-zero}, the two sides are, respectively,
\begin{align}
	\iota_Z(C_{Z/k}({\mathcal F}))
	\quad\text{and}\quad
	\iota_Z(\operatorname{cl}_Z(cc_{Z,0}({\mathcal F}))).
\end{align}
Since $\iota_Z$ is an isomorphism, we obtain
\eqref{eq:full-characteristic}.

In general, let $(\Lambda',\psi)$ be as in Lemma \ref{lm:coeff-character}.
By \cite[(5.15)]{CC}, we have
\begin{align}\label{eq:CCLambda}
	CC(i_\ast\mathcal F)=CC(i_\ast \mathcal F_{\Lambda'}).
\end{align}
The characteristic class $C_{Z/k}(\mathcal F)$ is defined using the identity 
endomorphism, adjunction, and evaluation. These constructions commute with coefficient extension. Thus, the image of $C_{Z/k}(\mathcal F)$  under the canonical map
\begin{align}
\tau:	H^0(Z,\mathcal K_{Z/k})\to H^0(Z,\mathcal K_{Z/k})\otimes_\Lambda\Lambda^\prime\simeq H^0(Z,\mathcal K_{Z/k,\Lambda^\prime})
\end{align}
is $C_{Z/k}(\mathcal F_{\Lambda'})$. By the previous case, we have equality
in ${\rm H}^0(Z, K_{Z/k,\Lambda^\prime})$:
\begin{align}\label{eq:equalityunderchangeofring}
	\tau(C_{Z/k}({\mathcal F}))=C_{Z/k}({\mathcal F}_{\Lambda^\prime})=\operatorname{cl}_Z(cc_{Z,0}({\mathcal F}_{\Lambda^\prime}))\overset{\eqref{eq:CCLambda}}{=}\tau(\operatorname{cl}_Z(cc_{Z,0}({\mathcal F}))).
\end{align}
Since $\Lambda\to\Lambda^\prime$ is faithfully flat, the map $\tau$ is injective. From \eqref{eq:equalityunderchangeofring}, we get the equality \eqref{eq:full-characteristic}.
\end{proof}

\subsection{Microlocalization for \texorpdfstring{$\ell$}{ell}-adic coefficients}
\label{sec:NA-equality}
Let $E/\mathbf Q_\ell$ be a finite extension, let $\mathcal O_E$ be the ring of integers in $E$, and let
$\pi$ be a uniformizer. After replacing $E$ by a finite extension, we assume that there exists a non-trivial character $\psi: \mathbf F_p\to \mathcal O_E^\times$.

Let $Y$ be a scheme of finite type over $k$.
Let $D_{\mathrm{cons}}(Y_{\mathrm{pro\acute et}},\mathcal O_E)$ be the stable
$\infty$-category of constructible pro-\'etale
sheaves on $Y$ with $\mathcal O_E$-coefficients. We use analogous notation for $D_{\mathrm{cons}}(Y_{\mathrm{pro\acute et}},E)$ and $D_{\mathrm{cons}}(Y_{\mathrm{pro\acute et}},\mathcal O_E/\pi^m)$.
By \cite[Theorem 7.7]{HRS}, there are natural equivalences
\begin{align}\label{eq:equiBetweenEtaleAndProetale}
\begin{aligned}
&D_{\mathrm{cons}}(Y_{\mathrm{pro\acute et}},\mathcal O_E/\pi^m)\simeq D_{\mathrm{ctf}}(Y,\mathcal O_E/\pi^m),\\
&D_{\mathrm{cons}}(Y_{\mathrm{pro\acute et}},\mathcal O_E)\simeq D_{\mathrm{c}}^b(Y,\mathcal O_E),\\
&D_{\mathrm{cons}}(Y_{\mathrm{pro\acute et}},E)\simeq D_{\mathrm{c}}^b(Y,E).
\end{aligned}
\end{align}
By \cite[Theorem 1.6 (1) and (2)]{HRS}, there is
an equivalence of stable
$\infty$-categories 
\begin{align}
D_{\mathrm{cons}}(Y_{\mathrm{pro\acute et}},\mathcal O_E)
\simeq \varprojlim_m
D_{\mathrm{cons}}(Y_{\mathrm{pro\acute et}},\mathcal O_E/\pi^m).
\end{align}
By \cite[Corollary 2.4]{HansenScholze}, every constructible $E$-sheaf has an integral model. 
By
\cite[Lemma~2.1]{Barrett}, reduction modulo $\pi$ is conservative on derived
complete constructible objects. 

Assume $Y$ is smooth.
The operations used to form the diagonal endomorphism kernel, Verdier
specialization, and Fourier transform preserve constructibility and are
compatible with reduction.  Hence, for an integral model ${\mathcal F}_0$ of
${\mathcal F}\in D_{\mathrm{cons}}(Y_{\mathrm{pro\acute et}},E)$ and
${\mathcal F}_m={\mathcal F}_0\otimes^L_{\mathcal O_E}\mathcal O_E/\pi^m$, there are canonical
isomorphisms
\begin{align}
{\mu\mathcal Hom}_{\mathcal O_E}({\mathcal F}_0,{\mathcal F}_0)
\otimes^L_{\mathcal O_E}\mathcal O_E/\pi^m
&\xrightarrow{\sim}
{\mu\mathcal Hom}_{\mathcal O_E/\pi^m}({\mathcal F}_m,{\mathcal F}_m),
\label{eq:proetale-muhom-reduction}\\
{\mu\mathcal Hom}_E({\mathcal F},{\mathcal F})
&\xrightarrow{\sim}
{\mu\mathcal Hom}_{\mathcal O_E}({\mathcal F}_0,{\mathcal F}_0)\otimes_{\mathcal O_E}E.
\label{eq:proetale-muhom-rationalization}
\end{align}

By \cite[Proposition~3.6]{Barrett}, universal local acyclicity on the pro-\'etale site is equivalent
to dualizability in the relative correspondence category.  It is preserved
by smooth pull-back and proper push-forward.  Moreover, Beilinson's
singular support and its projective Radon description hold for constructible
$E$-sheaves (cf. \cite{Barrett}). Following \cite{UYZ}, the characteristic cycle of an $\ell$-adic sheaf ${\mathcal F}$ is defined by 
\begin{align}
{CC}({\mathcal F}):= {CC}({\mathcal F}_1),
\end{align}
where ${\mathcal F}_1={\mathcal F}_0\otimes^L \mathcal O_E/\pi$ for any integral model ${\mathcal F}_0$ of ${\mathcal F}$.

\begin{theorem}\label{thm:singular-support-inclusion-Ecoefficient}
Let $X$ be a smooth $k$-scheme, and let ${\mathcal F}\in D_{\mathrm{cons}}(X_{\mathrm{pro\acute et}},E)$. Then
\begin{equation}\label{eq:muSSsubseteqSS-Ecoefficient}
	{SS}_\mu({\mathcal F})\subseteq{SS}({\mathcal F}).
\end{equation}
\end{theorem}
\begin{proof}
The assertion is clear if ${\mathcal F}=0$. By
\cite[Theorem~1.5~\textup{(ii) and (v)}]{Barrett}, there exists an integral model
${\mathcal F}_0\in D_{\mathrm{cons}}(X_{\mathrm{pro\acute et}},\mathcal O_E)$
of ${\mathcal F}$ such that
\begin{equation}\label{eq:adapted-lattice-theorem82}
	{\mathcal F}_0\otimes_{\mathcal O_E}E\simeq {\mathcal F},
	\qquad
	{SS}({\mathcal F})={SS}({\mathcal F}_0)={SS}({\mathcal F}_1),
\end{equation}
where ${\mathcal F}_1={\mathcal F}_0\otimes_{\mathcal O_E}^{L}\mathcal O_E/\pi$.
We next compare the microlocal Hom complexes. 
The  isomorphisms
\eqref{eq:proetale-muhom-reduction} and
\eqref{eq:proetale-muhom-rationalization} give
\begin{equation}\label{eq:muhom-coefficient-change-theorem82}
	\begin{aligned}
		&{\mu\mathcal Hom}_{\mathcal O_E}({\mathcal F}_0,{\mathcal F}_0)\otimes_{\mathcal O_E}^{L}\mathcal O_E/\pi
		\xrightarrow{\ \sim\ }{\mu\mathcal Hom}_{\mathcal O_E/\pi}({\mathcal F}_1,{\mathcal F}_1),\\
		&{\mu\mathcal Hom}_{\mathcal O_E}({\mathcal F}_0,{\mathcal F}_0)\otimes_{\mathcal O_E}E
		\xrightarrow{\ \sim\ }{\mu\mathcal Hom}_E({\mathcal F},{\mathcal F}).
	\end{aligned}
\end{equation}
Let us recall why these compatibilities hold. The diagonal
endomorphism kernel can be written as
\begin{equation}
	\mathscr H_{\mathcal O_E}({\mathcal F}_0,{\mathcal F}_0)
	\simeq
D_{\mathcal O_E}({\mathcal F}_0)\boxtimes {\mathcal F}_0.
\end{equation}
For constructible complexes, Verdier duality and external products
commute with reduction modulo $\pi$ and with rationalization.
Specialization along $\delta_X$ has the same coefficient
compatibility. Finally, the Fourier transform is constructed
from pull-back, tensor product with the Artin--Schreier sheaf, and
compactly supported push-forward, all of which commute with these
changes of coefficients. This proves
\eqref{eq:muhom-coefficient-change-theorem82}.

We claim that
\begin{equation}\label{eq:SSmu-F-contained-F1}
	{SS}_{\mu}({\mathcal F})\subseteq{SS}_{\mu}({\mathcal F}_1).
\end{equation}
Let $U=T^*X\setminus{\rm supp}_k({\mu\mathcal Hom}_{\mathcal O_E/\pi}({\mathcal F}_1,{\mathcal F}_1))$.
By definition, $\left.{\mu\mathcal Hom}_{\mathcal O_E/\pi}({\mathcal F}_1,{\mathcal F}_1)\right|_U=0$.
Using the first isomorphism in
\eqref{eq:muhom-coefficient-change-theorem82}, we obtain
$	\left.{\mu\mathcal Hom}_{\mathcal O_E}({\mathcal F}_0,{\mathcal F}_0)\right|_U
\otimes_{\mathcal O_E}^{L}\mathcal O_E/\pi=0$.
The complex
\begin{align}
	{\mathcal N}=\left.{\mu\mathcal Hom}_{\mathcal O_E}({\mathcal F}_0,{\mathcal F}_0)\right|_U
\end{align}
is constructible and derived $\pi$-complete. Derived Nakayama
therefore gives ${\mathcal N}=0$. More explicitly, the distinguished triangle
\begin{align}
	{\mathcal N}\xrightarrow{\ \pi\ }{\mathcal N}
	\longrightarrow
	{\mathcal N}\otimes_{\mathcal O_E}^{L}\mathcal O_E/\pi
	\longrightarrow
\end{align}
shows that multiplication by $\pi$ on ${\mathcal N}$ is an isomorphism. Hence
${\mathcal N}\otimes_{\mathcal O_E}^{L}\mathcal O_E/\pi^m=0$ for every
$m\geq1$.
Since ${\mathcal N}$ is derived $\pi$-complete, we have
${\mathcal N}\simeq R\varprojlim_m
\left({\mathcal N}\otimes_{\mathcal O_E}^{L}\mathcal O_E/\pi^m\right)=0$.

Rationalizing and using the second isomorphism in
\eqref{eq:muhom-coefficient-change-theorem82}, we obtain
\begin{align}
	\left.{\mu\mathcal Hom}_E({\mathcal F},{\mathcal F})\right|_U
	\simeq
	\left.{\mu\mathcal Hom}_{\mathcal O_E}({\mathcal F}_0,{\mathcal F}_0)\right|_U
	\otimes_{\mathcal O_E}E
	=0.
\end{align}
It follows that
\begin{align}
	\operatorname{supp}_k({\mu\mathcal Hom}_E({\mathcal F},{\mathcal F}))
	\subseteq
	\operatorname{supp}_k({\mu\mathcal Hom}_{\mathcal O_E/\pi}({\mathcal F}_1,{\mathcal F}_1)),
\end{align}
which proves \eqref{eq:SSmu-F-contained-F1}.

Finally, ${\mathcal F}_1$ is a constructible complex with coefficients in the
finite field $\mathcal O_E/\pi$. Under the equivalence \eqref{eq:equiBetweenEtaleAndProetale} between constructible
torsion sheaves on the pro-\'etale and the \'etale sites,
Theorem~\ref{thm:singular-support-equality} applies to ${\mathcal F}_1$ and
gives	${SS}_{\mu}({\mathcal F}_1)={SS}({\mathcal F}_1)$.
Together with \eqref{eq:SSmu-F-contained-F1} and
\eqref{eq:adapted-lattice-theorem82}, we obtain
\begin{align}
	{SS}_{\mu}({\mathcal F})
	\overset{\eqref{eq:SSmu-F-contained-F1}}{\subseteq}
	{SS}_{\mu}({\mathcal F}_1)
	\overset{\eqref{eq:muSS-equal-SS}}{=}
	{SS}({\mathcal F}_1)
	\overset{\eqref{eq:adapted-lattice-theorem82}}{=}
	{SS}({\mathcal F}),
\end{align}
as required.
\end{proof}

Within the six-functor formalism, the microlocal characteristic cycle ${CC}_\mu$ for an object ${\mathcal F}\in D_{\mathrm{cons}}(X_{\mathrm{pro\acute et}},E)$ is also well-defined.

\begin{theorem}\label{thm:CC-cycle}
Let $X$ be a smooth connected scheme over a perfect field $k$, and let ${\mathcal F}\in D_{\mathrm{cons}}(X_{\mathrm{pro\acute et}},E)$. Then
we have
\begin{equation}\label{eq:CC-cycle}
	{CC}_{\mu}({\mathcal F})=\operatorname{cl}({CC}({\mathcal F})),
\end{equation}
where ${\rm cl}: Z_{{\rm dim}(X)}({SS}({\mathcal F}))\to {\rm H}^0_{{SS}({\mathcal F})}(T^\ast X, {e}_X^{\vee\ast}  K_{X/k})$ is the cycle class map.

In particular, 	if $X$ is a smooth curve and ${\mathcal F}$ is smooth outside a finite set $Z\subseteq X$, then we have (cf. \cite[Lemma 5.11.3]{CC})
\begin{align}\label{eq:GOSforCC}
{CC}_\mu({\mathcal F})= -{\rm rank}{\mathcal F}\cdot {\rm cl}[T_X^\ast X]-\sum_{x\in Z}a_x({\mathcal F})\cdot {\rm cl}[T_x^\ast X],
\end{align}
where $a_x({\mathcal F})$ is the Artin conductor of ${\mathcal F}$ at $x$.
\end{theorem}

\begin{proof}
By Proposition~\ref{pradd} and \cite[Lemma 5.13.1]{CC}, both ${CC}_{\mu}$ and ${CC}$ are additive for
distinguished triangles.   By \cite[Theorem~1.5\textup{(viii)}]{Barrett}, we have
\begin{align}\label{eq:SSaddper}
	{SS}({\mathcal F})=\bigcup_i{SS}({}^p{\rm H}^i{\mathcal F}).
\end{align}
By additivity and \eqref{eq:SSaddper}, it suffices to prove the assertion when $\mathcal F$ is a nonzero perverse $E$-sheaf.

Take a torsion-free perverse
integral model ${\mathcal F}_0$ of $\mathcal F$. Put ${\mathcal F}_m={\mathcal F}_0\otimes^L \mathcal O_E/\pi^m$.
By
\cite[Theorem~1.5\textup{(v)}]{Barrett} and \cite[Corollary 4.5.4]{CC}, we have
\begin{equation}\label{eq:barrett-adapted-lattice}
	{SS}({\mathcal F})={SS}({\mathcal F}_0)={SS}({\mathcal F}_m).
\end{equation}
By definition, 
\begin{equation}\label{eq:defCC-lattice}
	{CC}({\mathcal F})={CC}({\mathcal F}_0)={CC}({\mathcal F}_m).
\end{equation}
By \cite[Theorem 1.3]{SS}, we have ${\rm dim}{SS}({\mathcal F})= {\rm dim}X$. By semi-purity, we have
\begin{align}\label{eq:semipurityresult}
\begin{aligned}
&{\rm H}_{SS(\mathcal F)}^r\!\left(T^*X,e_X^{\vee *}K_{X/k,\mathcal O_E/\pi^m}\right)=0 \qquad(r<0), \\
&{\rm H}_{SS(\mathcal F)}^0\!\left(T^*X,e_X^{\vee *}K_{X/k,\mathcal O_E/\pi^m}\right) \simeq \bigoplus_a \mathcal O_E/\pi^m\cdot \operatorname{cl}_{\mathcal O_E/\pi^m}[C_a],\\
&{\rm H}^0_{{SS}({\mathcal F})}(T^\ast X, e_X^{\vee\ast} K_{X,E})=\bigoplus_a E\cdot {\rm cl}_E[C_a],	
\end{aligned}
\end{align}
where $C_a$ runs through the irreducible components of ${SS}({\mathcal F})$ (of dimension ${\rm dim}X$). These identifications are compatible with reduction.
By \eqref{eq:semipurityresult}, we have an isomorphism
	\begin{align}\label{eq:isomforHSS}
	{\rm H}_{{SS}({\mathcal F})}^0(T^*X,e_X^{\vee *}  K_{X/k,\mathcal O_E})
	\xrightarrow{\sim}
	\varprojlim_m
	{\rm H}_{{SS}({\mathcal F})}^0(T^*X,e_X^{\vee *}  K_{X/k,\mathcal O_E/\pi^m}).
\end{align}
Indeed, the relevant supported complexes lie in \(D^{\geq0}\), so the
possible \(\varprojlim^1{\rm H}^{-1}\)-term vanishes.

Since the evaluation morphism and identity section are compatible with reduction, by the definition of $CC_\mu$ (cf. \eqref{eqccFnmi}) and the isomorphism \eqref{eq:proetale-muhom-reduction}, we have compatibility
\begin{align}
\begin{aligned}
{\rm H}^0_{SS(\mathcal F)}(T^\ast X,e_X^{\vee *}  K_{X/k,\mathcal O_E})&\rightarrow {\rm H}_{{SS}({\mathcal F})}^0(T^*X,e_X^{\vee *}  K_{X/k,\mathcal O_E/\pi^m})\rightarrow {\rm H}_{{SS}({\mathcal F})}^0(T^*X,e_X^{\vee *}  K_{X/k,\mathcal O_E/\pi^{m-1}})\\
CC_{\mu,\mathcal O_E}(\mathcal F_0) &\longmapsto CC_{\mu,\mathcal O_E/\pi^m}(\mathcal F_m)\longmapsto CC_{\mu,\mathcal O_E/\pi^{m-1}}(\mathcal F_{m-1}).
\end{aligned}
\end{align}
Similarly, for rationalization, we have 
\begin{align}
\begin{aligned}
{\rm H}^0_{SS(\mathcal F)}(T^\ast X,e_X^{\vee *}  K_{X/k,\mathcal O_E})\otimes_{\mathcal O_E} E&\xrightarrow{\simeq} {\rm H}_{{SS}({\mathcal F})}^0(T^*X,e_X^{\vee *}  K_{X/k,E})\\
CC_{\mu,\mathcal O_E}(\mathcal F_0) &\longmapsto CC_{\mu,E}(\mathcal F) 
\end{aligned}
\end{align}
By Theorem \ref{prMilnor}, we have 
\begin{align}
      	CC_\mu(\mathcal F_m)={\rm cl}_{\mathcal O_E/\pi^m}[CC(\mathcal F_m)]\qquad{\rm in}\qquad 	{\rm H}_{{SS}({\mathcal F})}^0(T^*X,e_X^{\vee *}  K_{X/k,\mathcal O_E/\pi^m}).
\end{align}
Passing to the inverse limit and then inverting $\pi$, we get ${CC}_{\mu}({\mathcal F})=\operatorname{cl}({CC}({\mathcal F}))$.
\end{proof}

\begin{corollary}\label{cor:pushforwardOfCC-Ecoefficient}
Let $f\colon X\to Y$ be a 
morphism of smooth schemes over a perfect field $k$
and $i\colon Z\hookrightarrow X$ be a closed immersion
such that  $g=fi\colon Z\to Y$ is proper.
Let ${\cal F}\in D_{\rm cons}(Z_{\rm pro\acute{e}t},E)$.
We have 
\begin{align}\label{eq:pushCC2}
	{\rm cl}CC(g_\ast\mathcal F)=f_! {\rm cl}CC(i_\ast \mathcal F)\qquad{\rm in}\qquad {\rm H}^0_{f_\circ SS(i_\ast \mathcal F)}(T^\ast Y, e_{Y}^{\vee\ast}K_Y).
\end{align}
\end{corollary}
\begin{proof}
By Theorem \ref{thm:CC-cycle}, we only need to show the following equality:
\begin{align}\label{eq:pushCC1}
CC_\mu (g_*{\cal F})=
f_! CC_\mu(i_*{\cal F}) \qquad{\rm in}\qquad {\rm H}^0_{f_\circ SS(i_\ast \mathcal F)}(T^\ast Y, e_{Y}^{\vee\ast}K_Y).
\end{align}
Note that the proof of Proposition \ref{prpush}.3 is within the six-functor formalism. It still works in the $\ell$-adic case, and the same argument proves \eqref{eq:pushCC1}.
\end{proof}

\begin{corollary}\label{cor:condforProper}
Let $f: X\to Y$ be a morphism from a  smooth scheme $X$ to a smooth connected curve $Y$ over a perfect field $k$.
Let $i\colon Z\hookrightarrow X$ be a closed immersion
such that  $g=fi\colon Z\to Y$ is proper.
Let ${\cal F}\in D_{\rm cons}(Z_{\rm pro\acute{e}t},E)$ and
let $y\in Y$ be a closed point.
Assume that $f$ is properly $SS(i_\ast\mathcal F)$-transversal over a dense open subset $V\subseteq Y$.
Then we have
\begin{align}\label{eq:corConductor}
-a_y(g_\ast\mathcal F)=({CC}(i_\ast {\mathcal F}),df)_{T^\ast X,Z_y}\qquad{\rm in}\qquad E.
\end{align}
\end{corollary}
In particular, this extends \cite[Theorem 2.2.3]{SaitoCond} by removing its quasi-projectivity assumption on $f$.
\begin{proof}
This follows from Corollary \ref{cor:pushforwardOfCC-Ecoefficient} (cf. proof of \cite[Lemma 3]{Saltlake}). For more details, see the proof of Proposition \ref{prop:NA-comparison}.
\end{proof}

We apply the preceding results to establish two results for finite coefficients.

\begin{corollary}\label{cor:index}
Let $X$ be a  smooth scheme over a perfect field  $k$ of characteristic $p>0$, and let $\Lambda$ be a finite local ring over $\mathbf Z/\ell^n\mathbf Z$ for a prime number $\ell\neq p$. 
Let  $i\colon Z\hookrightarrow X$ be a closed immersion and $\mathcal F\in D_{\rm ctf}(Z,\Lambda)$.
Let $cc_{Z,0}(\mathcal F)=0_X^! CC(i_\ast \mathcal F)\in {\rm CH}_0(Z)$ be the refined Gysin pull-back of $CC(i_\ast \mathcal F)$ by the zero section $0_X\colon X\to T^\ast X$. 
Assume $Z$ is proper over $k$.
Then we have an equality in $\mathbf Z$
\begin{align}\label{eq:indexFormulaforProper}
\chi(Z_{\bar k},\mathcal F)={\rm deg}cc_{Z,0}(\mathcal F).
\end{align}
\end{corollary}
When $Z$ is projective and smooth, this is proved in \cite[Theorem 7.13]{CC}.
\begin{proof}
By compatibility of characteristic cycles with extension of the base field, we may assume that $k$ is algebraically closed. 
Let $\Lambda_0$ be the residue field of $\Lambda$. Since  $\chi(Z_{\bar k},\mathcal F)=\chi(Z_{\bar k},\mathcal F\otimes^L_\Lambda\Lambda_0)$ and $cc_{Z,0}(\mathcal F)=cc_{Z,0}(\mathcal F\otimes^L_\Lambda\Lambda_0)$,  we may assume that $\Lambda$ is a finite field of characteristic $\ell\neq p$.
Let $K(Z,\overline{\mathbf Q}_\ell)$ be the Grothendieck group of the category of constructible $\overline{\mathbf Q}_\ell$-sheaves (cf. \cite[5.2.2]{UYZ}), and similarly for  $K(Z,\overline{\mathbf F}_\ell)$.
Note that the decomposition homomorphism $d_Z: K(Z,\overline{\mathbf Q}_\ell)\to K(Z,\overline{\mathbf F}_\ell)$ defined in \cite[5.2.6]{UYZ} is surjective.
Let  $\mathcal G\in  K(Z,\overline{\mathbf Q}_\ell)$ such that $d_Z(\mathcal G)=\mathcal F\otimes_\Lambda\overline{\mathbf F}_\ell$.
We have $\chi(Z_{\bar k},\mathcal F)=\chi(Z_{\bar k},\mathcal G)$ and $cc_{Z,0}(\mathcal F)=cc_{Z,0}(\mathcal G)$.
By Theorem \ref{thm:CC-cycle} and Corollary \ref{cor:pushforwardOfCC-Ecoefficient},
we have an equality in $\overline{\mathbf Q}_\ell$:
\begin{align}
\chi(Z_{\bar k},\mathcal G)={\rm deg}cc_{Z,0}(\mathcal G).
\end{align}
From this, we obtain \eqref{eq:indexFormulaforProper}.
\end{proof}

\begin{corollary}\label{cor:bloch}
Let $f: X\to Y$ be a morphism from a  smooth scheme $X$ to a smooth connected curve $Y$ over a perfect field $k$ of characteristic $p>0$, and let $\Lambda$ be a finite local ring over $\mathbf Z/\ell^m\mathbf Z$ for a prime number $\ell\neq p$.  
Let $i\colon Z\hookrightarrow X$ be a closed immersion
such that  $g=fi\colon Z\to Y$ is proper.
Let ${\cal F}\in D_{\rm ctf}(Z,\Lambda)$ and let $y\in Y$ be a closed point.
Assume that there is a dense open subset $V\subseteq Y$ such that
the restriction $g_V:Z_V=Z\times_YV\to  V$ is a smooth morphism and $\mathcal F|_{Z_V}$ is locally constant.
Then we have
\begin{align}\label{eq:bloch}
-a_y(g_\ast\mathcal F)=(CC(i_\ast \mathcal F),df)_{T^\ast X,Z_y}\qquad{\rm in}\qquad {\mathbf Z}.
\end{align}
\end{corollary}	

\begin{proof}
By compatibility of characteristic cycles and Artin conductors with extension of the base field, we may assume that $k$ is algebraically closed. 
Let $\Lambda_0$ be the residue field of $\Lambda$. Since  $a_y(g_\ast \mathcal F)=a_y(g_\ast \mathcal F\otimes^L_\Lambda\Lambda_0)$ and $CC(i_\ast \mathcal F)=CC(i_\ast \mathcal F\otimes^L_\Lambda\Lambda_0)$,  we may assume that $\Lambda$ is a finite field of characteristic $\ell\neq p$.

Let $A(\overline{\mathbf F}_\ell)$ (resp. $A(\overline{\mathbf Q}_\ell)$)
be the subgroups of $K(Z,\overline{\mathbf F}_\ell)$ (resp. $K(Z,\overline{\mathbf Q}_\ell)$)
generated by the classes of constructible sheaves $\mathcal G$ such that  $\mathcal G|_{Z_V}$ are smooth. 
Let $L(Z_V,\overline{\mathbf F}_\ell)$ (resp. $L(Z_V,\overline{\mathbf Q}_\ell)$)
be the subgroups of $K(Z_V,\overline{\mathbf F}_\ell)$ (resp. $K(Z_V,\overline{\mathbf Q}_\ell)$)
generated by the classes of smooth  sheaves $\mathcal G$ on $Z_V$. 
Then, we have
\begin{align}
A(\overline{\mathbf F}_\ell)&=K(Z\setminus Z_V,\overline{\mathbf F}_\ell)\oplus L(Z_V,\overline{\mathbf F}_\ell),\\
A(\overline{\mathbf Q}_\ell)&=K(Z\setminus Z_V,\overline{\mathbf Q}_\ell)\oplus L(Z_V, \overline{\mathbf Q}_\ell).
\end{align}
The decomposition homomorphism $d$ defined in \cite[5.2.6]{UYZ} induces a surjective map
\begin{align}
d\colon A(\overline{\mathbf Q}_\ell)\to A(\overline{\mathbf F}_\ell).
\end{align}
Let  $\mathcal G\in  A(\overline{\mathbf Q}_\ell)$ such that $d(\mathcal G)=\mathcal F\otimes_\Lambda\overline{\mathbf F}_\ell$. By Corollary \ref{cor:condforProper}, we have
\begin{align}
		-a_y(g_\ast\mathcal G)=(CC(i_\ast \mathcal G),df)_{T^\ast X,Z_y}\qquad{\rm in}\qquad \overline{\mathbf Q}_\ell .
\end{align}
Since  $a_y(g_\ast\mathcal G)=a_y(g_\ast\mathcal F)$ and $CC(i_\ast \mathcal G)=CC(i_\ast \mathcal F)$, we obtain \eqref{eq:bloch}.
\end{proof}	
\begin{remark}
Retain the hypotheses of Corollary \ref{cor:bloch}, except for the local constancy of $\mathcal F|_{Z_V}$. Suppose that $g_V:Z_V\to V$ is smooth and that $D\subset Z_V$ is a relative divisor with simple normal crossings. Let
$j:Z_V\setminus D\hookrightarrow Z_V$
be the open immersion. Assume that
$\mathcal F|_{Z_V}\simeq j_!\mathcal L,$
where $\mathcal L\in D_{\rm ctf}(Z_V\setminus D,\Lambda)$ is smooth and tamely ramified along $D$. Then the same argument shows \eqref{eq:bloch}.
\end{remark}

\begin{corollary}[{cf. \cite[Conjecture]{Blo87}}]\label{cor:bloch-constant}
Let $f: X\to Y$ be a proper flat morphism of smooth schemes over a perfect field $k$ of characteristic $p>0$, and let $\Lambda$ be a finite local ring over $\mathbf Z/\ell^m\mathbf Z$ for a prime number $\ell\neq p$.  Let $y\in Y$ be a closed point.
Assume that ${\rm dim}X=n$,  ${\rm dim}Y=1$ and that there is a dense open subset $V\subseteq Y$ such that
the base change $f_V:X\times_YV\to  V$ is smooth.
Then we have
\begin{align}\label{eq:bloch-constant}
-a_y(f_\ast\Lambda)=(-1)^n{\rm deg}_{k(y)}\left( {c_n}^X_{X_y}(\Omega^1_{X/Y})\cap[X]\right)\qquad{\rm in}\qquad {\mathbf Z},
\end{align}
where ${c_n}^X_{X_y}(\Omega^1_{X/Y})$ is the localized Chern class defined in \cite[Section 1]{Blo87}.
\end{corollary}
When $f$ is projective, this is proved in  \cite[Corollary 2.2.4]{SaitoCond}.
\begin{proof}
By Corollary \ref{cor:bloch}, we have $-a_y(f_\ast\Lambda)=(-1)^n(T^\ast_X X,df)_{T^\ast X,X_y}$.
By \cite[Lemma 2.1.4]{SaitoCond}, we have 
$(T^\ast_XX,df)_{T^\ast X,X_y}={\rm deg}\left( {c_n}^X_{X_y}(\Omega^1_{X/Y})\cap[X]\right)$, which proves \eqref{eq:bloch-constant}.
\end{proof}

\begin{proposition}\label{prop:perverse}
For every perverse $E$-sheaf $\mathcal P$ on a smooth $k$-scheme $X$, we have
\begin{equation}\label{eq:equalitySSforPerverse}
{SS}_{\mu}(\mathcal P)={SS}(\mathcal P).
\end{equation}
\end{proposition}
\begin{proof}
By definition of $SS_\mu$, we have $SS_\mu(\mathcal P)_{\bar k}=SS_\mu(\mathcal P|_{X_{\bar k}})$.
By \cite[Theorem 1.5(xi)]{Barrett}, we have $SS(\mathcal P)_{\bar k}=SS(\mathcal P|_{X_{\bar k}})$.
Therefore, we may assume $k=\bar k$ is perfect.
The assertion may be checked on the connected
components of $X$. We may therefore assume that $X$ is connected, and
put $n=\dim X$. 
The assertion is clear if $\mathcal P=0$.
By
Theorem~\ref{thm:singular-support-inclusion-Ecoefficient}, we have
\begin{equation}\label{eq:SSmu-contained-in-SS-perverse}
{SS}_\mu(\mathcal P)\subseteq {SS}(\mathcal P).
\end{equation}
We first establish the positivity property of ${CC}(\mathcal P)$.
By
\cite[Proposition~2.6 and the proof of
Theorem~1.5\textup{(v)}]{Barrett}, we may choose an integral model
$	\mathcal P_0\in\operatorname{Perv}(X,\mathcal O_E)$ of $\mathcal P$
which is torsion-free as a perverse sheaf. Put
$\mathcal P_1
=\mathcal P_0\otimes_{\mathcal O_E}^{L}\mathcal O_E/\pi$.
The distinguished triangle
\begin{align}
\mathcal P_0\xrightarrow{\pi}\mathcal P_0
\longrightarrow\mathcal P_1\longrightarrow
\end{align}
shows that $\mathcal P_1$ is perverse. Indeed,
$
{}^p{\rm H}^{-1}(\mathcal P_1)
=\ker\bigl(\pi:\mathcal P_0\to\mathcal P_0\bigr)=0
$
because $\mathcal P_0$ is torsion-free, while
$
{}^p{\rm H}^0(\mathcal P_1)
=\operatorname{coker}
\bigl(\pi:\mathcal P_0\to\mathcal P_0\bigr)$.
By \cite[Theorem~1.5\textup{(ii),(v)}]{Barrett}, we have
\begin{equation}\label{eq:SS-integral-reduction-perverse}
{SS}(\mathcal P_1)
={SS}(\mathcal P_0)
={SS}(\mathcal P).
\end{equation}
By \cite[Proposition~5.14]{CC}, the characteristic cycle
${CC}(\mathcal P_1)$ is effective and
\begin{align}
\operatorname{supp}{CC}(\mathcal P_1)
={SS}(\mathcal P_1).
\end{align}
Since the characteristic cycle with $E$-coefficients is defined by
reduction of an integral model, if
${SS}(\mathcal P)=\bigcup_{a\in A}C_a$
is the decomposition into irreducible components, then
\begin{equation}\label{eq:positive-CC-perverse-E}
{CC}(\mathcal P)
={CC}(\mathcal P_1)
=\sum_{a\in A}m_a[C_a],
\qquad m_a\in\mathbf Z_{>0}.
\end{equation}
In particular, ${SS}(\mathcal P)$ is purely $n$-dimensional, and every $m_a$ has
nonzero image in the characteristic-zero field $E$.

We now recall the consequence of semi-purity needed below. If
$W\subseteq T^\ast X$ is a closed subset of dimension at most $n$, then the
cycle-class map induces an isomorphism
\begin{equation}\label{eq:semipurity-top-cycles-E}
\operatorname{cl}_{W,E}:
Z_n(W)\otimes_{\mathbf Z}E
\xrightarrow{\ \sim\ }
{\rm H}_W^0\bigl(T^\ast X,e_X^{\vee\ast} K_{X/k}\bigr),
\end{equation}
where $Z_n(W)$ denotes the group of $n$-dimensional cycles supported
on $W$.

Indeed, since $X$ is smooth of dimension $n$, we have
$e_X^{\vee\ast} K_{X/k}\simeq E(n)[2n]$.
A closed subset of $W$ of dimension strictly smaller than $n$ has
codimension at least $n+1$ in the smooth scheme $T^\ast X$ of dimension
$2n$. Semi-purity therefore shows that such a subset contributes
neither to degree $0$ nor to degree $1$ cohomology with coefficients
in $e_X^{\vee\ast} K_{X/k}$. After removing a closed subset of dimension
less than $n$, the $n$-dimensional irreducible components of $W$
become pairwise disjoint regular closed subschemes of codimension
$n$. Absolute purity then gives one copy of $E$ for each such
component, proving \eqref{eq:semipurity-top-cycles-E}. Moreover, since
these identifications are induced by cycle classes, they are
compatible with extension of supports.

Let $C={SS}(\mathcal P),C_\mu={SS}_\mu(\mathcal P)$ and
\[
\rho_{C_\mu,C}:
{\rm H}_{C_\mu}^0\bigl(T^\ast X,e_X^{\vee\ast} K_{X/k}\bigr)
\longrightarrow
{\rm H}_C^0\bigl(T^\ast X,e_X^{\vee\ast} K_{X/k}\bigr)
\]
be the extension-of-supports morphism. Since $C_\mu\subseteq C$ (cf. \eqref{eq:muSSsubseteqSS-Ecoefficient}) and
$C$ is purely $n$-dimensional, every $n$-dimensional irreducible
component of $C_\mu$ is one of the $C_a$. Consequently,
\eqref{eq:semipurity-top-cycles-E} gives
\begin{equation}\label{eq:image-extension-supports-perverse}
\operatorname{Im}(\rho_{C_\mu,C})
\simeq
\bigoplus_{\{a\in A\,:\,C_a\subseteq C_\mu\}}
E\cdot\operatorname{cl}_{C,E}([C_a])
\subseteq
{\rm H}_C^0\bigl(T^\ast X,e_X^{\vee\ast} K_{X/k}\bigr).
\end{equation}
By Theorem~\ref{thm:CC-cycle}, interpreted after extension of
supports along $C_\mu\subseteq C$, we have
\begin{align}
\rho_{C_\mu,C}\bigl({CC}_{\mu}(\mathcal P)\bigr)
&=\operatorname{cl}_{C,E}\bigl({CC}(\mathcal P)\bigr) \notag\\
&=\sum_{a\in A}
m_a\,\operatorname{cl}_{C,E}([C_a])
\quad\text{in}\quad
{\rm H}_C^0\bigl(T^\ast X,e_X^{\vee\ast} K_{X/k}\bigr).
\label{eq:CCmu-equals-positive-cycle-class}
\end{align}
The left-hand side of
\eqref{eq:CCmu-equals-positive-cycle-class} belongs to the subspace
described in \eqref{eq:image-extension-supports-perverse}. On the
right-hand side, however, every coefficient $m_a$ is nonzero in $E$.
It follows that
$C_a\subseteq C_\mu$
for every $a\in A$.
Therefore $C\subseteq C_\mu$. Together with
\eqref{eq:SSmu-contained-in-SS-perverse}, this proves
${SS}_{\mu}(\mathcal P)={SS}(\mathcal P)$.
\end{proof}

\begin{lemma}\label{lem:mu-t-exact}
Let $X$ be a smooth $k$-scheme of pure dimension $n$. 
The functor
\begin{equation}
\mu_{\delta_X}[n]:D_{\rm cons}((X\times X)_{\mathrm{pro\acute et}},E)\to D_{\rm cons}(T^*X_{\mathrm{pro\acute et}},E)
\end{equation}
is perverse $t$-exact.
\end{lemma}
\begin{proof}
It suffices to show that both $\nu_{\delta_X}$ and $F_{\psi,TX/X}[n]$ are perverse $t$-exact. The first assertion follows from \cite[Corollaire 4.5]{autour} and the second assertion follows from \cite[Theorem 1.3.2.3]{La}.
\end{proof}
\begin{lemma}\label{lem:perverse-decomposition-of-muHom}
Let $X$ be a smooth $k$-scheme of pure dimension $n$, and let ${\mathcal F},{\mathcal G}\in D_{\rm cons}(X_{\mathrm{pro\acute et}},E)$. For every integer $r\in\mathbf{Z}$, we have a canonical isomorphism
\begin{equation}\label{eq:perdecmpofmuhom}
{}^p{\rm H}^r(\mu\mathcal{H}om_X({\mathcal F},{\mathcal G})[n])\simeq\bigoplus\limits_{j-i=r}\mu\mathcal{H}om_X({}^p{\rm H}^i{\mathcal F},{}^p{\rm H}^j{\mathcal G})[n].
\end{equation}
In particular, we have
$\bigcup\limits_i{SS}_\mu({}^p{\rm H}^i{\mathcal F})\subseteq{SS}_\mu({\mathcal F})$.
\end{lemma}
\begin{proof}
By \cite[Propositions 1.3.21 and 4.2.8]{BBD},
we have isomorphisms
\begin{align}\label{eq:homdecomp}
\begin{aligned}
	{}^p{\rm H}^r(\mathscr{H}_X({\mathcal F},{\mathcal G}))&\simeq {}^p{\rm H}^r( {D}_X{\mathcal F}\boxtimes {\mathcal G})\\
	&\simeq\bigoplus\limits_{j-i=r} {D}_X({}^p{\rm H}^i{\mathcal F})\boxtimes{}^p{\rm H}^j{\mathcal G}\\
	&\simeq\bigoplus\limits_{j-i=r}\mathscr{H}_X({}^p{\rm H}^i{\mathcal F},{}^p{\rm H}^j{\mathcal G}).
\end{aligned}
\end{align}
By Lemma  \ref{lem:mu-t-exact}, the functor $\mu_{\delta_X}[n]$
is perverse $t$-exact. Applying it  to \eqref{eq:homdecomp}, we obtain the isomorphism \eqref{eq:perdecmpofmuhom}.

Now take $\mathcal G=\mathcal F$ and $r=0$. Each $\mu\mathcal{H}om_X({}^p{\rm H}^i{\mathcal F},{}^p{\rm H}^i{\mathcal F})[n]$
is a direct summand of ${}^p{\rm H}^0(\mu\mathcal{H}om_X({\mathcal F},{\mathcal F})[n])$. Thus $\bigcup\limits_i{SS}_\mu({}^p{\rm H}^i{\mathcal F})\subseteq{SS}_\mu({\mathcal F})$.
\end{proof}
\begin{theorem}\label{thm:singular-support-equality-Ecoeff}
Let $X$ be a smooth $k$-scheme and ${\mathcal F}\in D_{\rm cons}(X_{\rm pro\acute{e}t},E)$. We have
\begin{equation}\label{eq:SSequalityforEcoeff}
{SS}_\mu({\mathcal F})={SS}({\mathcal F}).
\end{equation}
\end{theorem}
\begin{proof}
By \cite[Theorem 1.5.(viii)]{Barrett}  and Lemma \ref{lem:perverse-decomposition-of-muHom}, we have
\begin{align}\label{eq:inverseInclusion}
\begin{aligned}
	SS(\mathcal F)=\bigcup\limits_i{SS}({}^p{\rm H}^i{\mathcal F})\overset{\eqref{eq:equalitySSforPerverse}}{=}\bigcup\limits_i{SS}_\mu({}^p{\rm H}^i{\mathcal F})\subseteq{SS}_\mu({\mathcal F}).
\end{aligned}
\end{align}
Combining \eqref{eq:muSSsubseteqSS-Ecoefficient} and \eqref{eq:inverseInclusion}, we obtain \eqref{eq:SSequalityforEcoeff}.
\end{proof}

\end{document}